\documentclass[11pt,reqno]{amsart}
\usepackage{fullpage,amsmath,amsthm,amssymb,stmaryrd,color,enumerate,xcolor}
\usepackage[all]{xy}
\usepackage{url}

\usepackage[bbgreekl]{mathbbol}
\usepackage{amsfonts}

\DeclareSymbolFontAlphabet{\mathbb}{AMSb} 
\DeclareSymbolFontAlphabet{\mathbbl}{bbold}

\usepackage{enumitem}
\usepackage[colorlinks=true,hyperindex, linkcolor=magenta, pagebackref=false, citecolor=cyan,pdfpagelabels]{hyperref}

\newcommand{\cosimp}[3]{\xymatrix@1{#1 \ar@<.4ex>[r] \ar@<-.4ex>[r] & {\ }#2 \ar@<0.8ex>[r] \ar[r] \ar@<-.8ex>[r] & {\ } #3 \ar@<1.2ex>[r] \ar@<.4ex>[r] \ar@<-.4ex>[r] \ar@<-1.2ex>[r] & \cdots }}

\newcommand{\colim}{\mathop{\mathrm{colim}}}

\newcommand{\adjunction}[4]{\xymatrix@1{#1{\ } \ar@<0.3ex>[r]^{ {\scriptstyle #2}} & {\ } #3 \ar@<0.3ex>[l]^{ {\scriptstyle #4}}}}

\newcommand{\calO}{\mathcal{O}}
\newcommand{\calC}{\mathcal{C}}
\newcommand{\calN}{\mathcal{N}}
\newcommand{\T}{\mathbb{T}}
\newcommand{\HH}{\mathrm{HH}}
\newcommand{\THH}{\mathrm{THH}}
\newcommand{\TC}{\mathrm{TC}}
\newcommand{\HC}{\mathrm{HC}}
\newcommand{\Fil}{\mathrm{Fil}}

\newcommand{\TP}{\mathrm{TP}}
\newcommand{\TF}{\mathrm{TF}}
\newcommand{\TR}{\mathrm{TR}}
\newcommand{\HP}{\mathrm{HP}}
\newcommand{\gr}{\mathrm{gr}}

\newcommand{\dotimes}{{\otimes}^{\mathbb L}}

\newcommand{\Prism}{{ \mathbbl{\Delta}}}

\newcommand{\f}{{\star}}

\DeclareMathOperator{\Spec}{Spec}

\newcommand{\sub}[1]{{\mbox{\rm \scriptsize #1}}}

\newcommand{\Qs}{\mathrm{QSyn}}
\newcommand{\qs}{\mathrm{qSyn}}
\newcommand{\qsp}{\mathrm{qrsPerfd}}
\newcommand{\Qsp}{\mathrm{QRSPerfd}}

\begin{document}

\newtheorem{theorem}{Theorem}[section]
\newtheorem*{theorem*}{Theorem}
\newtheorem*{definition*}{Definition}
\newtheorem{proposition}[theorem]{Proposition}
\newtheorem{lemma}[theorem]{Lemma}
\newtheorem{corollary}[theorem]{Corollary}
\newtheorem{conjecture}[theorem]{Conjecture}
\newtheorem{definition}[theorem]{Definition}

\theoremstyle{definition}
\newtheorem{question}[theorem]{Question}
\newtheorem{remark}[theorem]{Remark}
\newtheorem{warning}[theorem]{Warning}
\newtheorem{example}[theorem]{Example}
\newtheorem{notation}[theorem]{Notation}
\newtheorem{convention}[theorem]{Convention}
\newtheorem{construction}[theorem]{Construction}
\newtheorem{claim}[theorem]{Claim}
\newtheorem{assumption}[theorem]{Assumption}
\newtheorem{variant}[theorem]{Variant}

\setcounter{tocdepth}{1}

\title{Topological Hochschild homology and integral $p$-adic Hodge theory}
\author{Bhargav Bhatt, Matthew Morrow and Peter Scholze}
\begin{abstract}
In mixed characteristic and in equal characteristic $p$ we define a filtration on topological Hochschild homology and its variants. This filtration is an analogue of the filtration of algebraic $K$-theory by motivic cohomology. Its graded pieces  are related in mixed characteristic to the complex $A\Omega$ constructed in our previous work, and in equal characteristic $p$ to crystalline cohomology. Our construction of the filtration on $\THH$ is via flat descent to semiperfectoid rings.

As one application, we refine the construction of the $A\Omega$-complex by giving a cohomological construction of Breuil--Kisin modules for proper smooth formal schemes over $\mathcal O_K$, where $K$ is a discretely valued extension of $\mathbb Q_p$ with perfect residue field. As another application, we define syntomic sheaves $\mathbb Z_p(n)$ for all $n\geq 0$ on a large class of $\mathbb Z_p$-algebras, and identify them in terms of $p$-adic nearby cycles in mixed characteristic, and in terms of logarithmic de~Rham-Witt sheaves in equal characteristic $p$.
\end{abstract}
\date{\today}
\maketitle

\tableofcontents

\section{Introduction}

This paper proves various foundational results on topological cyclic homology $\TC$ (and its cousins $\THH$, $\TC^-$, $\TP$) in mixed characteristic, notably the existence of certain filtrations mirroring the motivic filtration on algebraic $K$-theory. As a concrete arithmetic consequence, we refine the $A_{\inf}$-cohomology theory from \cite{BMS} to proper smooth (formal) schemes defined over $\calO_K$, where $K$ is a \emph{discretely} valued extension of $\mathbb Q_p$ with perfect residue field $k$, and relate this cohomology theory to topological cyclic homology, leading to new computations of algebraic $K$-theory.

\subsection{Breuil-Kisin modules}

Let us start by explaining more precisely the arithmetic application. Let $\mathfrak X$ be a proper smooth (formal) scheme over $\calO_K$. If $C$ is a completed algebraic closure of $K$ with ring of integers $\calO_C$, then in \cite{BMS} we associate to the base change $\mathfrak X_{\calO_C}$ a cohomology theory $R\Gamma_{A_{\inf}}(\mathfrak X_{\calO_C})$ with coefficients in Fontaine's period ring $A_{\inf}$. We recall that there is a natural surjective map $\theta: A_{\inf}\to \calO_C$ whose kernel is generated by a non-zero-divisor $\xi$, and a natural Frobenius automorphism $\varphi: A_{\inf}\to A_{\inf}$. Then $\tilde\xi=\varphi(\xi)\in A_{\inf}$ is a generator of the kernel of $\tilde\theta = \theta\circ \varphi^{-1}: A_{\inf}\to \calO_C$. The main properties of this construction are as follows.

\begin{enumerate}
\item The complex $R\Gamma_{A_{\inf}}(\mathfrak X_{\calO_C})$ is a perfect complex of $A_{\inf}$-modules, and each cohomology group is a finitely presented $A_{\inf}$-module that is free over $A_{\inf}[\tfrac 1p]$ after inverting $p$.
\item There is a natural Frobenius endomorphism $\varphi: R\Gamma_{A_{\inf}}(\mathfrak X_{\calO_C})\to R\Gamma_{A_{\inf}}(\mathfrak X_{\calO_C})$ that is semilinear with respect to $\varphi: A_{\inf}\to A_{\inf}$, and becomes an isomorphism after inverting $\xi$ resp.~$\tilde\xi = \varphi(\xi)$:
\[
\varphi: R\Gamma_{A_{\inf}}(\mathfrak X_{\calO_C})[\tfrac 1\xi]\simeq R\Gamma_{A_{\inf}}(\mathfrak X_{\calO_C})[\tfrac 1{\tilde{\xi}}]\ .
\]
\item After scalar extension along $\theta: A_{\inf}\to \calO_C$, one recovers de~Rham cohomology:
\[
R\Gamma_{A_{\inf}}(\mathfrak X_{\calO_C})\dotimes_{A_{\inf}} \calO_C\simeq R\Gamma_{\mathrm{dR}}(\mathfrak X_{\calO_C}/\calO_C)\ .
\]
\item After inverting a generator $\mu \in A_{\inf}$ of the kernel of the canonical map $A_{\inf} \to W(\mathcal{O}_C)$, one recovers \'etale cohomology:
\[
R\Gamma_{A_{\inf}}(\mathfrak X_{\calO_C})[\tfrac 1\mu]\simeq R\Gamma_{\mathrm{\acute{e}t}}(\mathfrak X_C,\mathbb Z_p)\otimes_{\mathbb Z_p} A_{\inf}[\tfrac 1\mu]\ ,
\]
where the isomorphism is $\varphi$-equivariant, where the action on the right-hand side is only via the action on $A_{\inf}[\tfrac 1\mu]$.
\item After scalar extension along $A_{\inf}\to W(\overline{k})$, one recovers crystalline cohomology of the special fiber,
\[
R\Gamma_{A_{\inf}}(\mathfrak X_{\calO_C})\dotimes_{A_{\inf}} W(\overline{k})\simeq R\Gamma_{\mathrm{crys}}(\mathfrak X_{\overline{k}}/W(\overline{k}))\ ,
\]
$\varphi$-equivariantly.
\end{enumerate}

In particular, the first two parts ensure that each $H^i_{A_{\inf}}(\mathfrak X_{\calO_C}):=H^i(R\Gamma_{A_{\inf}}(\mathfrak X_{\calO_C}))$ is a Breuil-Kisin-Fargues module in the sense of \cite[Definition 4.22]{BMS}.

On the other hand, in abstract $p$-adic Hodge theory, there is the more classical notion of Breuil-Kisin modules as defined by Breuil, \cite{Breuil} and and studied further by Kisin \cite{Kisin}. The theory depends on the choice of a uniformizer $\varpi\in \calO_K$. One gets the associated ring $\mathfrak S=W(k)[[z]]$, which has a surjective $W(k)$-linear map $\tilde\theta: \mathfrak S\to \calO_K$ sending $z$ to $\varpi$. The kernel of $\tilde\theta$ is generated by $E(z)\in \mathfrak S$, where $E$ is an Eisenstein polynomial for $\varpi$. Also, there is a Frobenius $\varphi: \mathfrak S\to \mathfrak S$ which is the Frobenius on $W(k)$ and sends $z$ to $z^p$. One can regard $\mathfrak S$ as a subring of $A_{\inf}$ by using the Frobenius on $W(k)$ and sending $z$ to $[\varpi^\flat]^p$ for a compatible choice of $p$-power roots $\varpi^{1/p^n}\in \calO_C$. This embedding is compatible with $\varphi$ and $\tilde{\theta}$.

\begin{definition} A Breuil-Kisin module is a finitely generated $\mathfrak S$-module $M$ together with an isomorphism
\[
M\otimes_{\mathfrak S,\varphi} \mathfrak S[\tfrac 1E]\to M[\tfrac 1E]\ .
\]
\end{definition}

Our first main theorem states roughly that there exists a well-behaved cohomology theory on proper smooth formal schemes $\mathfrak{X}/\calO_K$ that is valued in Breuil-Kisin modules, and recovers most other standard $p$-adic cohomology theories attached to $\mathfrak{X}$ by a functorial procedure; the existence of this construction geometrizes the results of \cite{Kisin} attaching Breuil-Kisin modules to lattices in crystalline Galois representations, proving a conjecture of Kisin; see \cite{CaisLiu, BarThesis} for some prior work on this question.

\begin{theorem}\label{thm:main1} There is a $\mathfrak S$-linear cohomology theory $R\Gamma_{\mathfrak S}(\mathfrak X)$ equipped with a $\varphi$-linear Frobenius map $\varphi: R\Gamma_{\mathfrak S}(\mathfrak X)\to R\Gamma_{\mathfrak S}(\mathfrak X)$, with the following properties:
\begin{enumerate}
\item After base extension to $A_{\inf}$, it recovers the $A_{\inf}$-cohomology theory:
\[
R\Gamma_{\mathfrak S}(\mathfrak X)\otimes_{\mathfrak S} A_{\inf}\simeq R\Gamma_{A_{\inf}}(\mathfrak X_{\calO_C})\ .
\]
In particular, $R\Gamma_{\mathfrak S}(\mathfrak X)$ is a perfect complex of $\mathfrak S$-modules, and $\varphi$ induces an isomorphism
\[
R\Gamma_{\mathfrak S}(\mathfrak X)\otimes_{\mathfrak S,\varphi} \mathfrak S[\tfrac 1E]\simeq R\Gamma_{\mathfrak S}(\mathfrak X)[\tfrac 1E]\ ,
\]
and so all $H^i_{\mathfrak S}(\mathfrak X):=H^i(R\Gamma_{\mathfrak S}(\mathfrak X))$ are Breuil-Kisin modules. Moreover, after scalar extension to $A_{\inf}[\tfrac 1\mu]$, one recovers \'etale cohomology.
\item After scalar extension along $\theta:=\tilde\theta\circ \varphi: \mathfrak S\to \calO_K$, one recovers de~Rham cohomology:
\[
R\Gamma_{\mathfrak S}(\mathfrak X)\dotimes_{\mathfrak S} \calO_K\simeq R\Gamma_{\mathrm{dR}}(\mathfrak X/\calO_K)\ .
\]
\item After scalar extension along the map $\mathfrak S\to W(k)$ which is the Frobenius on $W(k)$ and sends $z$ to $0$, one recovers crystalline cohomology of the special fiber:
\[
R\Gamma_{\mathfrak S}(\mathfrak X)\dotimes_{\mathfrak S} W(k)\simeq R\Gamma_{\mathrm{crys}}(\mathfrak X_k/W(k))\ .
\]
\end{enumerate}
\end{theorem}

The constructions of \cite{BMS} are not enough to get such a descent. In this paper, we deduce this theorem as a consequence of a different construction of the $A_{\inf}$-cohomology theory, in terms of topological Hochschild homology. This alternative construction was actually historically our first construction, except that we were at first unable to make it work.

\begin{remark} In \cite{BS}, this theorem will be reproved using the theory of the \emph{prismatic site}, which is something like a mixed-characteristic version of the crystalline site. In particular, this gives a proof of Theorem~\ref{thm:main1} that is independent of topological Hochschild homology. Moreover, that approach clarifies the various Frobenius twists that appear: In parts (1), (2) and (3), one always uses a base change along the Frobenius map of $W(k)$, which may seem confusing.
\end{remark}

\begin{remark}
\label{rmk:TorsiondR}
The Frobenius twists appearing in Theorem~\ref{thm:main1} are not merely an artifact of the methods, but have concrete implications for torsion in de Rham cohomology. Let us give one example. Take $K = \mathbb{Q}_p(p^{1/p})$. In this case, the map $\theta$ appearing in Theorem~\ref{thm:main1} (2) carries $z$ to $(p^{1/p})^p = p$, and thus factors over $\mathbb{Z}_p \subset \mathcal{O}_K$. It follows from the theorem that the $\calO_K$-complex $R\Gamma_{\mathrm{dR}}(\mathfrak{X}/\calO_K)$ is the pullback of a complex defined over $\mathbb{Z}_p$. In particular, the length of each $H^i_{\mathrm{dR}}(\mathfrak{X}/\calO_K)_{\mathrm{tors}}$ is a multiple of $p$; in fact, each indecomposable summand of this group has length a multiple of $p$. 
\end{remark}

\subsection{Quick reminder on topological Hochschild homology (THH)}

The theory of topological Hochschild homology was first introduced in \cite{Boekstedt}, following on some ideas of Goodwillie. Roughly, it is the theory obtained by replacing the ring $\mathbb{Z}$ with the sphere spectrum $\mathbb{S}$ in the definition of Hochschild homology. We shall use this theory to prove Theorem~\ref{thm:main1}. Thus, let us recall the essential features of this theory from our perspective; we shall use \cite{NikolausScholze} as our primary reference.

$\THH(-)$ is a functor that takes an associative ring spectrum $A$ and builds a spectrum $\THH(A)$ that is equipped with an action of the circle group $\mathbb T=S^1$ and a $\T\cong \T/C_p$-equivariant Frobenius map
\[
\varphi_p: \THH(A)\to \THH(A)^{tC_p}\ ,
\]
where $C_p\cong\mathbb Z/p\mathbb Z\subset \T$ is the cyclic subgroup of order $p$, and $(-)^{tC_p}$ is the Tate construction, i.e.~the cone of the norm map $(-)_{hC_p}\to (-)^{hC_p}$ from homotopy orbits to homotopy fixed points. These Frobenius maps are an essential feature of the topological theory, and are (provably) not present in the purely algebraic theory of Hochschild homology.

In the commutative case, the definition of $\THH$ is relatively easy to give (cf. \cite[\S IV.2]{NikolausScholze}), as we now recall. Say $A$ is an $E_\infty$-ring spectrum. Then:

\begin{enumerate}
\item $\THH(A)$ is naturally a $\T$-equivariant $E_\infty$-ring spectrum equipped with a non-equivariant map $i:A \to \THH(A)$ of $E_\infty$-ring spectra,  and is initial among such.
\item The map $i: A\to \THH(A)$ induces a $C_p$-equivariant map $A\otimes\dots\otimes A\to \THH(A)$ of $E_\infty$-ring spectra, given by $a_0\otimes\dots\otimes a_{p-1}\mapsto i(a_0) \sigma(i(a_1))\cdots \sigma^{p-1}(i(a_{p-1}))$, where $\sigma\in C_p$ is the generator, and the $C_p$-action on the left permutes the tensor factors cyclically. Applying the Tate construction, this induces a map
\[
(A\otimes\cdots\otimes A)^{tC_p}\to \THH(A)^{tC_p}
\]
of $E_\infty$-ring spectra. Moreover, there is a canonical map of $E_\infty$-ring spectra
\[
\Delta_p: A\to (A\otimes\cdots\otimes A)^{tC_p}
\]
given by the Tate diagonal of \cite[\S IV.1]{NikolausScholze}. This gives a map $A\to \THH(A)^{tC_p}$ of $E_\infty$-ring spectra, where the target has a natural residual $\T/C_p\cong \T$-action. By the universal property of $\THH(A)$, this factors over a unique $\T\cong \T/C_p$-equivariant map $\varphi_p: \THH(A)\to \THH(A)^{tC_p}$.
\end{enumerate}

We shall ultimately be interested in other functors obtained from $\THH$. Thus, recall that using the circle action, one can form the homotopy fixed points $\TC^-(A) = \THH(A)^{h\T}$ and the periodic topological cyclic homology $\TP(A) = \THH(A)^{t\T}$, both of which are again $E_\infty$-ring spectra. There is a {\em canonical} map 
\[ \mathrm{can}:\TC^-(A)\to \TP(A)\]
relating these constructions, arising from the natural map $(-)^{h\T} \to (-)^{t\T}$. Writing $\THH(-;\mathbb Z_p)$ for the $p$-completion of $\THH(-)$ (and similarly for other spectra), it is easy to see that if $A$ is connective, then $\pi_0 \TC^-(A;\mathbb Z_p) \cong \pi_0 \TP(A;\mathbb Z_p)$ via the canonical map. On the other hand, again assuming that $A$ is connective, there is also a {\em Frobenius} map
\[
\varphi_p^{h\T}: \TC^-(A;\mathbb Z_p)\to (\THH(A)^{tC_p})^{h\T} \simeq  \TP(A;\mathbb Z_p)
\]
induced by $\varphi_p$; the displayed equivalence comes from \cite[Lemma II.4.2]{NikolausScholze}. Combining these observations, the ring $\pi_0 \TC^-(A; \mathbb{Z}_p)$ acquires a ``Frobenius'' endomorphism $\varphi = \varphi_p^{h\T}: \pi_0 \TC^-(A;\mathbb Z_p)\to \pi_0 \TP(A;\mathbb Z_p) \cong \pi_0 \TC^-(A;\mathbb Z_p)$. This map is the ultimate source of the endomorphism $\varphi$ of $R\Gamma_{\mathfrak{S}}(\mathfrak{X})$ in Theorem~\ref{thm:main1}. In certain situations, our results show that this map is a lift of the Frobenius modulo $p$, justifying its name.

\begin{remark}
In this paper, we shall apply the preceding constructions only in the case that $A$ is (the Eilenberg-MacLane spectrum corresponding to) a usual commutative ring. Even in this relatively simple case, $\THH(A)$ is manifestly an $E_\infty$-ring spectrum,  so we are relying on some rather heavy machinery from algebraic topology. On the other hand, we rely largely on the ``formal'' aspects of this theory, with the only exception being B\"{o}kstedt's calculation that $\pi_* \THH(\mathbb{F}_p) = \mathbb{F}_p[u]$ for $u \in \pi_2 \THH(\mathbb{F}_p)$. Moreover, we will extract our desired information from the level of homotopy groups; in fact, understanding $\pi_0 \TC^{-}(-; \mathbb{Z}_p)$ will suffice for the application to Theorem~\ref{thm:main1}.
\end{remark}

\subsection{From $\THH$ to Breuil-Kisin modules}
\label{ss:THHtoBK}

Let us first explain how to recover the $A\Omega$-complexes of \cite{BMS} from $\THH$, and then indicate the modifications necessary for Theorem~\ref{thm:main1}.  We begin with the following theorem, which is due to Hesselholt \cite{HesselholtOC}  if $R=\calO_{\mathbb C_p}$, and was the starting point for our investigations.

\begin{theorem}[cf.\ \S \ref{sec:TCperfectoid}]\label{thm:TCperfectoid} Let $R$ be a perfectoid ring in the sense of \cite[Definition 3.5]{BMS}. Then there is a canonical (in $R$) $\varphi$-equivariant isomorphism
\[
\pi_0 \TC^-(R;\mathbb Z_p)\cong A_{\inf}(R)\ .
\]
In fact, one can explicitly identify $\pi_* \TC^{-}(R;\mathbb{Z}_p)$, $\pi_* \TP(R;\mathbb{Z}_p)$, $\pi_* \THH(R;\mathbb{Z}_p)$ as well as the standard maps relating them, cf.~\S \ref{sec:TCperfectoid}.
\end{theorem}

Now let $A$ be the $p$-adic completion of a smooth $\calO_C$-algebra as in \cite{BMS}. We will recover $A\Omega_A$ via flat descent from $\pi_0 \TC^-(-;\mathbb Z_p)$ by passage to a perfectoid cover $A\to R$. A convenient home for the rings encountered while performing the descent (such as $R\widehat{\otimes}_A R$) is provided by the following:

\begin{definition}[The quasisyntomic site, cf.~Definition \ref{defqs}] 
A ring $A$ is {\em quasisyntomic}\footnote{It would be better to write ``$p$-completely quasisyntomic".} if it is $p$-complete, has bounded $p^\infty$-torsion (i..e, the $p$-primary torsion is killed by a fixed power of $p$), and $L_{A/\mathbb Z_p}\dotimes_{A} A/pA\in D(A/pA)$ has Tor-amplitude in $[-1,0]$. A map $A\to B$ of such rings is a {\em quasisyntomic map (resp. cover)} if $A/p^nA\to B/p^n B$ is flat (resp. faithfully flat) for all $n\geq 1$ and $L_{(B/pB)/(A/pA)}\in D(B/pB)$ has Tor-amplitude in $[-1,0]$.

Let $\Qs$ be the category of quasisyntomic rings. For $A \in \Qs$, let $\qs_{A}$ denote the category of all quasisyntomic $A$-algebras $B$. Both these categories are endowed with a site structure with the topology defined by quasisyntomic covers.

For any abelian presheaf $F$ on $\qs_{A}$, we write $R\Gamma_\sub{syn}(A,F) = R\Gamma(\qs_{A},F)$ for the cohomology of its sheafification.
\end{definition}

The category $\Qs$ contains many Noetherian rings of interest, for example all $p$-complete regular rings; even more generally, all $p$-complete local complete intersection rings are in $\Qs$. It also contains the objects encountered above, i.e., $p$-adic completions of smooth $\calO_C$-algebras, as well as perfectoid rings.

The association $B\mapsto \pi_0 \TC^-(B;\mathbb Z_p)$ defines a presheaf of rings on $\qs_A$. The next result identifies the cohomology of this presheaf with the $A\Omega$-complexes:

\begin{theorem}[cf.\ Theorem \ref{main_theorem}]\label{thm:main2} Let $A$ be an $\calO_C$-algebra that can be written as the $p$-adic completion of a smooth $\calO_C$-algebra. There is a functorial (in $A$) $\varphi$-equivariant isomorphism of $E_\infty$-$A_{\inf}$-algebras
\[
A\Omega_A \simeq R\Gamma_\sub{syn}(A,\pi_0 \TC^-(-;\mathbb Z_p))\ .
\]
\end{theorem}

\begin{remark}
While proving Theorem~\ref{thm:main2}, we will actually show that on a base of the site $\qs_A$ (given by the quasiregular semiperfectoid rings $S$), the presheaf $\pi_0 \TC^-(-;\mathbb{Z}_p)$ is already a sheaf with vanishing higher cohomology.
\end{remark}

There is also the following variant of this theorem in equal characteristic $p$, recovering crystalline cohomology.

\begin{theorem}[cf.\ \S \ref{subsection_TC_Acrys}]\label{thm:main3} Let $k$ be a perfect field of characteristic $p$, and $A$ a smooth $k$-algebra. There is a functorial (in $A$) $\varphi$-equivariant isomorphism of $E_\infty$-$W(k)$-algebras
\[
R\Gamma_{\mathrm{crys}}(A/W(k)) \simeq R\Gamma_\sub{syn}(A, \pi_0 \TC^-(-;\mathbb Z_p))\ .
\]
\end{theorem}

In this case, this is related to Fontaine-Messing's approach to crystalline cohomology via syntomic cohomology, \cite{FontaineMessing}. More precisely, they identify crystalline cohomology with syntomic cohomology of a certain sheaf $\mathbb A_{\mathrm{crys}}$. The previous theorem is actually proved by identifying the sheaf $\pi_0 \TC^-(-;\mathbb Z_p)$ with (the Nygaard completion of) the sheaf $\mathbb A_{\mathrm{crys}}$. It is also possible to deduce Theorem~\ref{thm:main3} from Theorem~\ref{thm:main2} and the results of \cite{BMS}.

\begin{remark}\label{remark_prism}
The topological perspective seems very well-suited to handling certain naturally arising filtrations on both crystalline cohomology and $A_{\inf}$-cohomology, as we now explain. For any quasisyntomic ring $A$, define the $E_\infty$-$\mathbb{Z}_p$-algebra
\[
\widehat{\Prism}_A = R\Gamma_\sub{syn}(A,\pi_0 \TC^-(-;\mathbb Z_p))\ .
\]
The homotopy fixed point spectral sequence endows $\pi_0 \TC^-(-;\mathbb Z_p)$ with a natural abutment filtration. Passing to cohomology, we learn that $\widehat{\Prism}_A$ comes equipped with a natural complete filtration $\calN^{\geq \f}\widehat{\Prism}_A$ called the {\em Nygaard filtration}. We shall identify this filtration with the classical Nygaard filtration \cite{Nygaard} on crystalline cohomology in the situation of Theorem~\ref{thm:main3}, and with a mixed-characteristic version of it in the situation of Theorem~\ref{thm:main2}. In fact, these identifications are crucial to our proof strategy for both theorems. The notation $\widehat{\Prism}_A$ here is chosen in anticipation of the prismatic cohomology defined in \cite{BS}, where we will prove that $\widehat{\Prism}_A$ agrees with the Nygaard completion of the cohomology of the structure sheaf on the prismatic site.
\end{remark}

For a proper smooth formal scheme $\mathfrak{X}/\mathcal{O}_C$, Theorem~\ref{thm:main2} gives an alternate construction of the cohomology theory $R\Gamma_{A_{\inf}}(\mathfrak{X})$ from \cite{BMS} without any recourse to the generic fibre: one can simply define it as
\[ R\Gamma_{A_{\inf}}(\mathfrak{X}) := R\Gamma_\sub{syn}(\mathfrak{X}, \pi_0 \TC^{-}(-;\mathbb{Z}_p))\]
where one defines the quasisyntomic site $\qs_{\mathfrak{X}}$ in the natural way.  Similarly, for a proper smooth formal scheme $\mathfrak{X}/\mathcal{O}_K$ as in Theorem~\ref{thm:main1}, one can construct the cohomology theory $R\Gamma_{\mathfrak{S}}(\mathfrak{X})$ from Theorem~\ref{thm:main1} in essentially the same way: using the choice of the uniformizer $\varpi \in \mathcal{O}_K$, we produce a complex of $\mathfrak S$-modules by repeating the above construction, replacing $\THH(-)$ by its relative variant $\THH(-/\mathbb{S}[z])$, where $\mathbb{S}[z]$ is a polynomial ring\footnote{Formally, one can define $\mathbb{S}[z]$ as the free $E_\infty$-algebra generated by the $E_\infty$-monoid $\mathbf{N}$.  In particular, $\pi_*(\mathbb{S}[z]) \simeq \pi_*(\mathbb{S})[z]$.  We caution the reader that $\mathbb{S}[z]$ is {\em not} the free $E_\infty$-ring on one generator: the latter coincides with $\oplus_{n \geq 0} \mathbb S_{h\Sigma_n}$ as a spectrum, so its $\pi_*$ is not flat over $\pi_*(\mathbb{S})$. } over $\mathbb{S}$, i.e., we work with
\[ R\Gamma_{\sub{syn}}(\mathfrak{X}, \pi_0 \TC^{-}(-/\mathbb{S}[z];\mathbb{Z}_p)).\]
There is a slight subtlety here due to the non-perfectoid nature of $\mathcal{O}_K$: the above complex is actually $\varphi^* R\Gamma_{\mathfrak{S}}(\mathfrak{X})$ where $\varphi:\mathfrak{S} \to \mathfrak{S}$ is the Frobenius; the Frobenius descended object $R\Gamma_{\mathfrak{S}}(\mathfrak{X})$ is  then constructed using an analog of the Segal conjecture, cf. \S \ref{sec:breuilkisin}.

\subsection{``Motivic'' filtrations on $\THH$ and its variants}

In the proof of Theorem~\ref{thm:main1} as sketched above, we only needed $\pi_0 \TC^{-}(-;\mathbb{Z}_p)$ locally on $\Qs$. In the next result, we show that by considering the entire Postnikov filtration of $\TC^{-}(-;\mathbb{Z}_p)$ (and variants), we obtain a filtration of $\TC^{-}(-;\mathbb{Z}_p)$ that is reminiscent of the motivic filtration on algebraic $K$-theory whose graded pieces are motivic cohomology, cf. \cite{FSMotAtiHir}. In fact, one should expect a precise relation between the two filtrations through the cyclotomic trace, but we have not addressed this question. Our precise result is as follows; the existence of the filtration mentioned below has been conjectured by Hesselholt.   

\begin{theorem}[cf.\ \S \ref{sec:pAdicNygaard}]\label{thm:main5} Let $A$ be a quasisyntomic ring.
\begin{enumerate}
\item Locally on $\qs_{A}$, the spectra $\THH(-;\mathbb Z_p)$, $\TC^-(-;\mathbb Z_p)$ and $\TP(-;\mathbb Z_p)$ are concentrated in even degrees.
\item Define
\[\begin{aligned}
\Fil^n \THH(A; \mathbb Z_p) &= R\Gamma_\sub{syn}(A,\tau_{\geq 2n} \THH(-;\mathbb Z_p))\\
\Fil^n \TC^-(A; \mathbb Z_p) &= R\Gamma_\sub{syn}(A,\tau_{\geq 2n} \TC^-(-;\mathbb Z_p))\\
\Fil^n \TP(A; \mathbb Z_p) &= R\Gamma_\sub{syn}(A,\tau_{\geq 2n} \TP(-;\mathbb Z_p))\ .
\end{aligned}\]
These are complete exhaustive decreasing multiplicative $\mathbb Z$-indexed filtrations.

\item The filtered $E_\infty$-ring $\widehat{\Prism}_A = \gr^0 \TC^-(A; \mathbb Z_p) = \gr^0 \TP(A;\mathbb Z_p)$ with its Nygaard filtration $\calN^{\geq \f}\widehat{\Prism}_A$ is an $E_\infty$-algebra in the completed filtered derived category $\widehat{DF}(\mathbb Z_p)$ (cf. \S \ref{subsection_FDC}). We write $\calN^n\widehat{\Prism}_A$ for the $n$-th graded piece of this filtration.

The complex $\widehat{\Prism}_A\{1\} = \gr^1 \TP(A;\mathbb Z_p)[-2]$ with the Nygaard filtration $\calN^{\geq \f}\widehat{\Prism}_A\{1\}$ (defined via quasisyntomic descent of the abutment filtration) is a module in $\widehat{DF}(\mathbb Z_p)$ over the filtered ring $\widehat{\Prism}_A$, and is invertible as such.\footnote{We warn the reader that this statement does not imply that $\widehat{\Prism}_A\{1\}$ is an invertible module over the non-filtered ring $\widehat{\Prism}_A$, as that deduction would need the passage to an inverse limit over all filtration steps.} In particular, for any $n\geq 1$, $\widehat{\Prism}_A\{1\}/\calN^{\geq n}\widehat{\Prism}_A\{1\}$ is an invertible $\widehat{\Prism}_A/\calN^{\geq n}\widehat{\Prism}_A$-module. If $A$ admits a map from a perfectoid ring, then in fact $\widehat{\Prism}_A\{1\}$ is isomorphic to $\widehat{\Prism}_A$.

The base change $\widehat{\Prism}_A\{1\}\otimes_{\widehat{\Prism}_A} A$ is canonically trivialized to $A$, where the map $\widehat{\Prism}_A\to A$ is the map $\gr^0 \TC^-(A;\mathbb Z_p)\to \gr^0\THH(A;\mathbb Z_p)=A$. For any $\widehat{\Prism}_A$-module $M$ in $\widehat{DF}(\mathbb Z_p)$, we denote by $M\{i\} = M\otimes_{\widehat{\Prism}_A} \widehat{\Prism}_A\{1\}^{\otimes i}$ for $i\in \mathbb Z$ its Breuil-Kisin twists.

\item There are natural isomorphisms
\[\begin{aligned}
\gr^n \THH(A;\mathbb Z_p)&\simeq \calN^n \widehat{\Prism}_A\{n\}[2n]\simeq \calN^n\widehat{\Prism}_A[2n]\ ,\\
\gr^n \TC^-(A;\mathbb Z_p)&\simeq \calN^{\geq n}\widehat{\Prism}_A\{n\}[2n]\ ,\\
\gr^n \TP(A;\mathbb Z_p)&\simeq \widehat{\Prism}_A\{n\}[2n]\ .
\end{aligned}\]
These induce multiplicative spectral sequences
\[\begin{aligned}
E_2^{ij} = H^{i-j}(\calN^{-j}\widehat{\Prism}_A)&\Rightarrow \pi_{-i-j} \THH(A;\mathbb Z_p)\\
E_2^{ij} = H^{i-j}(\calN^{\geq -j} \widehat{\Prism}_A\{-j\})&\Rightarrow \pi_{-i-j} \TC^-(A;\mathbb Z_p)\\
E_2^{ij} = H^{i-j}(\widehat{\Prism}_A\{-j\})&\Rightarrow \pi_{-i-j} \TP(A;\mathbb Z_p)\ .
\end{aligned}\]
\item The map $\varphi: \TC^-(A;\mathbb Z_p)\to \TP(A;\mathbb Z_p)$ induces natural maps $\varphi: \Fil^n \TC^-(A;\mathbb Z_p)\to \Fil^n \TP(A;\mathbb Z_p)$, thereby giving a natural filtration
\[
\Fil^n \TC(A;\mathbb Z_p) = \mathrm{hofib}(\varphi - \mathrm{can}: \Fil^n\TC^-(A;\mathbb Z_p)\to \Fil^n \TP(A;\mathbb Z_p))
\]
on topological cyclic homology
\[
\TC(A;\mathbb Z_p) = \mathrm{hofib}(\varphi - \mathrm{can}: \TC^-(A;\mathbb Z_p)\to \TP(A;\mathbb Z_p))\ .
\]
The graded pieces
\[
\mathbb Z_p(n)(A) := \gr^n\TC(A;\mathbb Z_p)[-2n]
\]
are given by
\[
\mathbb Z_p(n)(A) = \mathrm{hofib}(\varphi - \mathrm{can}: \calN^{\geq n}\widehat{\Prism}_A\{n\}\to \widehat{\Prism}_A\{n\})\ ,
\]
where $\varphi: \calN^{\geq n}\widehat{\Prism}_A\{n\}\to \widehat{\Prism}_A\{n\}$ is a natural Frobenius endomorphism of the Breuil-Kisin twist. In particular, there is a spectral sequence
\[
E_2^{ij} = H^{i-j}(\mathbb Z_p(-j)(A))\Rightarrow \pi_{-i-j} \TC(A;\mathbb Z_p)\ .
\]
\end{enumerate}
\end{theorem}

\begin{remark} Our methods can be extended to give similar filtrations on the spectra $\TR^r(A;\mathbb Z_p)$ studied in the classical approach to cyclotomic spectra. In this case, one gets a relation to the de~Rham-Witt complexes $W_r\Omega_{A/\mathbb F_p}$ if $A$ is of characteristic $p$, and the complexes $\widetilde{W_r\Omega}_A$ of \cite{BMS} if $A$ lives over $\calO_C$.
\end{remark}

\begin{remark} In the situation of (5), if $A$ is an $R$-algebra for a perfectoid ring $R$, then after a trivialization $A_{\inf}(R)\{1\}\cong A_{\inf}(R)$ of the Breuil-Kisin twist, a multiple $\tilde\xi^n \varphi: \calN^{\geq n}\widehat{\Prism}_A\{n\}\to \widehat{\Prism}_A\{n\}$ gets identified with the restriction to $\calN^{\geq n}\widehat{\Prism}_A\subset \widehat{\Prism}_A$ of the Frobenius endomorphism $\varphi: \widehat{\Prism}_A\to \widehat{\Prism}_A$; in other words, $\varphi: \calN^{\geq n}\widehat{\Prism}_A\{n\}\to \widehat{\Prism}_A\{n\}$ is a divided Frobenius, identifying the complexes $\mathbb Z_p(n)$ with a version of (what is traditionally called) syntomic cohomology.
\end{remark}

It follows from the definition that $\mathbb Z_p(n)$ is locally on $\qs_{A}$ concentrated in degrees $0$ and $1$. We expect that the contribution in degree $1$ vanishes after sheafification, but we can currently only prove this in characteristic $p$, or when $n\leq 1$. In fact, in degree $0$, one can check that $\mathbb Z_p(0) = \mathbb Z_p=\varprojlim_r \mathbb Z/p^r\mathbb Z$ is the usual (``constant'') sheaf; in degree $1$ we prove that $\mathbb Z_p(1)\simeq T_p\mathbb G_m[0]$; and for $n<0$, the complexes $\mathbb Z_p(n)=0$ vanish. Meanwhile, in characteristic $p$, the trace map from algebraic $K$-theory induces an identification $K_{2n}(-;\mathbb Z_p)[0]\simeq \mathbb Z_p(n)$.

Assuming that $\mathbb Z_p(n)$ is indeed locally concentrated in degree $0$, one can write
\[
\mathbb Z_p(n)(A) = R\Gamma_\sub{syn}(A,\mathbb Z_p(n))
\]
as the cohomology of a sheaf on the (quasi)syntomic site of $A$, justifying the name syntomic cohomology.

We identify the complexes $\mathbb Z_p(n)(A)$ when $A$ is a smooth $k$-algebra or the $p$-adic completion of a smooth $\calO_C$-algebra. In the formulation we use the pro-\'etale site \cite{BhattScholzePro} as we work with $p$-adic coefficients.

\begin{theorem}[cf.~Corollary~\ref{cor:logdRW}, Theorem~\ref{thm:nearbycycles}]\label{thm:main6}$ $
\begin{enumerate}
\item Let $A$ be a smooth $k$-algebra, where $k$ is a perfect field of characteristic $p$. Then there is an isomorphism of sheaves of complexes on the pro-\'etale site of $X=\Spec A$,
\[
\mathbb Z_p(n)\simeq W\Omega_{X,\sub{log}}^n[-n]\ .
\]
\item Let $A$ be the $p$-adic completion of a smooth $\calO_C$-algebra, where $C$ is an algebraically closed complete extension of $\mathbb Q_p$. Then there is an isomorphism of sheaves of complexes on the pro-\'etale site of $\mathfrak X=\mathrm{Spf}\ A$,
\[
\mathbb Z_p(n)\simeq \tau^{\leq n} R\psi \mathbb Z_p(n)\ ,
\]
where on the right-hand side, $\mathbb Z_p(n)$ denotes the usual (pro-)\'etale sheaf on the generic fibre $X$ of $\mathfrak X$, and $R\psi$ denotes the nearby cycles functor.
\end{enumerate}
\end{theorem}

Theorem~\ref{thm:main6} (1) is closely is related to the results of Hesselholt \cite{Hesselholtptypical} and Geisser--Hesselholt \cite{GeisserHesselholt1999}. Theorem~\ref{thm:main6}  (2) gives a description of $p$-adic nearby cycles as syntomic cohomology that works integrally; this description is related to the results of Geisser--Hesselholt \cite{GeisserHesselholt}. We expect that at least in the case of good reduction, this will yield refinements of earlier results relating $p$-adic nearby cycles with syntomic cohomology, such as Fontaine--Messing \cite{FontaineMessing}, Tsuji \cite{Tsuji1999}, and Colmez--Nizio{\l} \cite{ColmezNiziol}.

\begin{remark} If $X$ is a smooth $\calO_K$-scheme, \'etale sheaves of complexes $\mathfrak{T}_r(n)$ on $X$ have been defined by Schneider, \cite{Schneider}, and the construction has been extended to the semistable case by Sato, \cite{Sato} (and we follow Sato's notation). A direct comparison with our construction is complicated by a difference in the setups as all our rings are $p$-complete, but modulo this problem we expect that $\mathfrak{T}_r(n)$ is the restriction of the syntomic sheaves of complexes $\mathbb Z/p^r\mathbb Z(n) = \mathbb Z_p(n)/p^r$ to the \'etale site of $X$. In particular, we expect a canonical isomorphism in case $A=\calO_K$:
\[
\mathfrak{T}_r(n)(\calO_K)\simeq \mathbb Z/p^r\mathbb Z(n)(\calO_K)\ .
\]
If $k$ is algebraically closed, then in light of Schneider's definition of $\mathfrak T_r(n)(\calO_K)$ and passage to the limit over $r$, this means that there should be a triangle
\[
\mathbb Z_p(n-1)(k) = W\Omega_{k,\sub{log}}^{n-1}[-n+1]\to \mathbb Z_p(n)(\calO_K)\to \tau^{\leq n} R\Gamma_\sub{\'et}(\Spec K,\mathbb Z_p(n))
\]
where on the right-most term, $\mathbb Z_p(n)$ denotes the usual (pro-)\'etale sheaf in characteristic $0$.\footnote{If $k$ is not algebraically closed, one needs to interpret the objects as sheaves on the pro-\'etale site of $\mathrm{Spf}\ \calO_K$. More generally, one can expect a similar triangle involving logarithmic de~Rham-Witt sheaves of the special fibre, the complexes $\mathbb Z_p(n)$, and truncations of $p$-adic nearby cycles, on the pro-\'etale site of smooth formal $\calO_K$-schemes; this would give the comparison to the theory of \cite{Schneider}. Comparing with the theory of \cite{Sato} would then correspond to a generalization of this picture to semistable formal $\calO_K$-schemes.} Via comparison with the cofiber sequence
\[
K(k;\mathbb Z_p)\to K(\calO_K;\mathbb Z_p)\to K(K;\mathbb Z_p)
\]
in $K$-theory and the identifications $K(k;\mathbb Z_p)=\TC(k;\mathbb Z_p)$, $K(\calO_K;\mathbb Z_p)=\TC(\calO_K;\mathbb Z_p)$, such a comparison should recover the result of Hesselholt-Madsen, \cite{Hesselholt2003}, that $K(K;\mathbb Z_p)$ has a filtration with graded pieces
\[
\tau^{\leq n} R\Gamma_\sub{\'et}(\Spec K,\mathbb Z_p(n))\ ,
\]
verifying the Lichtenbaum-Quillen conjecture in this case.
\end{remark}

\subsection{Complements on cyclic homology}
We can also apply our methods to usual Hochschild homology. In that case, we get a relation between negative cyclic homology and de~Rham cohomology that seems to be slightly finer than the results in the literature, and is related to a question of Kaledin, \cite[\S 6.5]{KaledinCocyclic}.

\begin{theorem}[cf.~\S \ref{subsection_de_rham_neg}]\label{thm:main7} Fix $R\in \Qs$ and $A\in \qs_{R}$. There are functorial (in $A$) complete exhaustive decreasing multiplicative $\mathbb Z$-indexed filtrations
\[\begin{aligned}
\Fil^n \HC^-(A/R;\mathbb Z_p) &= R\Gamma_\sub{syn}(A,\tau_{\geq 2n} \HC^-(-/R;\mathbb Z_p))\\
\Fil^n \HP(A/R;\mathbb Z_p) &= R\Gamma_\sub{syn}(A,\tau_{\geq 2n} \HP(-/R;\mathbb Z_p))
\end{aligned}\]
on $\HC^-(A/R;\mathbb Z_p)$ and $\HP(A/R;\mathbb Z_p)$ with
\[\begin{aligned}
\gr^n \HC^-(A/R;\mathbb Z_p)&\simeq \widehat{L\Omega}^{\geq n}_{A/R}[2n]\ ,\\
\gr^n \HP(A/R;\mathbb Z_p)&\simeq \widehat{L\Omega}_{A/R}[2n]\ ,
\end{aligned}\]
where the right-hand side denotes the $p$-adically and Hodge completed derived (naively truncated) de~Rham complex. In particular, there are multiplicative spectral sequences
\[\begin{aligned}
E_2^{ij} = H^{i-j}(\widehat{L\Omega}^{\geq -j}_{A/R})\Rightarrow \pi_{-i-j} \HC^-(A/R;\mathbb Z_p)\ ,\\
E_2^{ij} = H^{i-j}(\widehat{L\Omega}_{A/R})\Rightarrow \pi_{-i-j} \HP(A/R;\mathbb Z_p)\ .
\end{aligned}\]
\end{theorem}

We would expect that this results holds true without $p$-completion.\footnote{Antieau, \cite{AntieauFiltration}, has recently obtained such results.} In fact, rationally, one gets such a filtration by using eigenspaces of Adams operations; in that case, the filtration is in fact split \cite[\S5.1.12]{Loday}. We can actually identify the action of the Adams operation $\psi_m$ on $\gr^n \THH(-;\mathbb Z_p)$ and its variants $\gr^n \TC^-(-;\mathbb Z_p)$, $\TP(-;\mathbb Z_p)$ and $\gr^n \HC^-(-;\mathbb Z_p)$, for any integer $m$ prime to $p$ (acting via multiplication on $\mathbb T$); it is given by multiplication by $m^n$, cf.~Proposition~\ref{prop:adams}.

\subsection{Overview of the paper}
Now let us briefly summarize the contents of the different sections. We start in \S \ref{sec:THH} by recalling very briefly the basic definitions on Hochschild homology and topological Hochschild homology. In \S \ref{sec:descent} we prove that all our theories satisfy flat descent, which is our central technique. In \S \ref{sec:sites} we then set up the quasisyntomic sites that we will use to perform the flat descent. Moreover, we isolate a base for the topology given by the quasiregular semiperfectoid rings; essentially these are quotients of perfectoid rings by regular sequences, and they come up in the Cech nerves of the flat covers of a smooth algebra by a perfectoid algebra. As a first application of these descent results, we construct the filtration on $\HC^-$ by de Rham cohomology in \S \ref{sec:HCvsdeRham}. For the proof, we use some facts about filtered derived $\infty$-categories that we recall, in particular the Beilinson $t$-structure.

Afterwards, we start to investigate topological Hochschild homology. We start with a description for perfectoid rings, proving Theorem~\ref{thm:TCperfectoid} in \S \ref{sec:TCperfectoid}. Morever, in the same section, we identify $\THH$ of smooth algebras over perfectoid rings. This information is then used in \S \ref{sec:pAdicNygaard} to control the $\THH$, $\TC^-$ and $\TP$ of quasiregular semiperfectoid rings, and prove Theorem~\ref{thm:main5}. At this point, we have defined our new complexes $\widehat{\Prism}_A$, and it remains to compare them to the known constructions.

In \S \ref{sec:charp}, we handle the case of characteristic $p$, and prove Theorem~\ref{thm:main3}, and the characteristic $p$ case of Theorem~\ref{thm:main6}. Afterwards, in \S \ref{sec:mixedchar}, we show that this recovers the $A\Omega$-theory by proving Theorem~\ref{thm:main2}. As an application of this comparison, we also identify the Adams operations. In \S \ref{sec:nearby}, we identify the sheaves of complexes $\mathbb Z_p(n)$ in terms of $p$-adic nearby cycles, proving the second part of Theorem~\ref{thm:main6}. Finally, we use relative $\THH$ to construct Breuil-Kisin modules by proving Theorem~\ref{thm:main1} in \S \ref{sec:breuilkisin}.

\begin{remark}[Comparison with \cite{BMS}]
As made clear by this introduction, a number of results are proved in this paper that go beyond the problems addressed in \cite{BMS}; the intersection is largely restricted to the application to the construction of the Breuil-Kisin cohomology theory. In particular:
\begin{enumerate}
\item The methods used in \cite{BMS} (such as perfectoid spaces, the $L\eta$-operator) lie squarely within arithmetic geometry. On the other hand, the methods used in this paper (such as $\infty$-categories, the formalism surrounding $\THH$) are much closer to homotopy theory.

\item The construction in \cite{BMS} was engineered to admit a comparison with \'etale cohomology; this comparison is crucial for applications to the cohomology of algebraic varieties over $C$. In contrast, the present construction begins life close to de Rham cohomology, and there is no easy way to compare to \'etale cohomology.

\item A primary goal of the present paper is to construct the motivic filtration on $\THH$ and its variants (Theorem~\ref{thm:main5}). The comparison with \cite{BMS} shows that the methods of $p$-adic Hodge theory have impact on questions in algebraic topology and algebraic $K$-theory. For example, we can compute algebraic $K$-theory in new cases, cf.~Theorem~\ref{cor:Ktheory}. Another application to $K$-theory is the calculation that $L_{K(1)} K(\mathbb{Z}/p^n\mathbb{Z}) \simeq 0$ for $n \geq 1$ (to appear in forthcoming work of the first author with Clausen and Mathew).
\end{enumerate}

We see that the fact that both approaches yield the same information yields interesting new information on both sides.
\end{remark}

\subsection*{Acknowledgments} The main ideas behind this paper were already known to the authors in 2015. However, the results of this paper are significantly clearer when expressed in the language of \cite{NikolausScholze}; giving a formula for $\widehat{\Prism}$ in terms of the spectra $\TR^r$ is a nontrivial translation, and the discussion of the syntomic complexes $\mathbb Z_p(n)$ is more transparent using the new formula for $\TC$. Therefore, we only formalized our results now. We apologize for the resulting long delay. It is a pleasure to thank Clark Barwick, Sasha Beilinson, Dustin Clausen, Vladimir Drinfeld, Saul Glasman, Lars Hesselholt, Igor Kriz, Jacob Lurie, Akhil Mathew, Thomas Nikolaus and Wiesia Nizio{\l} for discussions of the constructions of this paper. 

During the period in which this work was carried out, the first author was supported by National Science Foundation under grant number 1501461 and a Packard fellowship, the second author was  partly supported by the Hausdorff Center for Mathematics, and the third author was supported by a DFG Leibniz grant.

\subsection*{Conventions} We will freely use $\infty$-categories and the language, methods and results of \cite{LurieHTT} and \cite{HA} throughout the paper. All our derived categories are understood to be the natural $\infty$-categorical enhancements. When we work with usual rings, we reserve the word module for usual modules, and write $\dotimes$ for the derived tensor product. However, when we work over $E_\infty$-ring spectra, the only notion of module is a module spectrum and corresponds to a complex of modules in the case of discrete rings; similarly, only the derived tensor product retains meaning. For this reason, we simply say ``module'' and write $\otimes$ for the derived notions when working over an $E_\infty$-ring spectrum. Often, in particular when applying $\THH$ to a usual ring, we regard usual rings as $E_\infty$-ring spectra and also usual modules as module spectra, via passage to the Eilenberg--MacLane spectrum; this is often denoted via $R\mapsto HR$ and $M\mapsto HM$. Under this translation, the derived $\infty$-category $D(R)$ of a usual ring $R$ agrees with the $\infty$-category of module spectra over the $E_\infty$-ring spectrum $HR$. In this paper, we will omit the symbol $H$ in order to lighten notation.

Given a complex $C$, we denote by $C[1]$ its shift satisfying $C[1]^n=C^{n+1}$ (or, in homological indexing, $C_n[1]=C_{n-1}$); under this convention, the Eilenberg--MacLane functor takes the shift $[1]$ to the suspension $\Sigma$.

We tend to use the notation $\cong$ to denote isomorphisms between $1$-categorical objects such as rings, modules and actual complexes; conversely, $\simeq$ is used for equivalences in $\infty$-categories (and, in particular, quasi-isomorphisms between complexes).

Degrees of complexes, simplicial objects, etc.\ are denoted by $\bullet$, e.g., $K^\bullet$; for graded objects we use $*$, e.g., $\pi_*\THH(A)$; finally, for filtrations we use $\f$, e.g., $\Omega_{R/k}^{\le\f}$ denotes the Hodge filtration on the complex $\Omega_{R/k}^\bullet$. When we regard an actual complex $K^\bullet$ as an object of the derived category, we write simply $K$; so for example $\Omega_{R/k}$ denotes the de~Rham complex considered as an object of the derived category.

We often need to use completions of modules or spectra with respect to a finitely generated ideal. For the latter, we use the existing notion in stable homotopy theory (see, e.g., \cite[\S 7.3]{SAG} for a modern presentation). For modules and chain complexes, we use the notion of derived completions (see, e.g., \cite[\S 3.4, 3.5]{BhattScholzePro}, \cite[\S 6.2]{BMS}).

\newpage
\section{Reminders on the cotangent complex and (topological) Hochschild homology}
\label{sec:THH}

In this section, we recall the basic results about the cotangent complex and (topological) Hoch\-schild homology that we will use.

\subsection{The cotangent complex}
\label{subsection_AQ}
Let $R$ be a commutative ring. Given a commutative $R$-algebra $A$, let $P_\bullet\to A$ be a simplicial resolution of $A$ by polynomial $R$-algebras, and define (following \cite{Quillen1970}) the cotangent complex of $R\to A$ to be the simplicial $A$-module $L_{A/R}:= \Omega_{P_\bullet/R}^1\otimes_{P_\bullet}A$. Its wedge powers will be denoted by $\bigwedge_A^i L_{A/R}=\Omega_{P_\bullet/R}^i\otimes_{P_\bullet}A$ for each $i\ge 1$. Note that the map $P_\bullet \to A$ is an equivalence in the $\infty$-category of simplicial commutative $R$-algebras, and thus the object in $D(R)$ defined by $L_{A/R}$ coincides with that attached to the simplicial $R$-module $\Omega^1_{P_\bullet/R}$ via the Dold-Kan correspondence (and similarly for the wedge powers).

As such constructions appear repeatedly in the sequel, let us give an $\infty$-categorical account, following \cite[\S 5.5.9]{LurieHTT}.

\begin{construction}[Non-abelian derived functors on commutative rings as left Kan extensions]
\label{LKE}
\

\noindent Consider the category $\mathrm{CAlg}_R^{poly}$ of finitely generated polynomial $R$-algebras and the $\infty$-category $s\mathrm{CAlg}_R$, so one has an obvious fully faithful embedding $i:\mathrm{CAlg}_R^{poly} \to s\mathrm{CAlg}_R$. Using \cite[Corollary 5.5.9.3]{LurieHTT}, one can identify $s\mathrm{CAlg}_R$ as the $\infty$-category $\mathcal{P}_{\Sigma}(\mathrm{CAlg}_R^{poly})$ obtained from $\mathrm{CAlg}_R^{poly}$ by freely adjoining sifted colimits: this amounts to the observation that any contravariant set-valued functor $F$ on $\mathrm{CAlg}_R^{poly}$ that carries coproducts in $\mathrm{CAlg}_R^{poly}$ to products of sets is representable by the commutative ring $F(R[x])$. By \cite[Proposition 5.5.8.15]{LurieHTT}, one has the following universal property of the inclusion $i$: for any $\infty$-category $\mathcal{D}$ that admits sifted colimits, any functor $f:\mathrm{CAlg}_R^{poly} \to \mathcal{D}$ extends uniquely to a sifted colimit preserving functor $F:s\mathrm{CAlg}_R \to \mathcal{D}$. In this case, we call $F$ the left Kan extension of $f$ along the inclusion $i$. If $f$ preserves finite coproducts, then $F$ preserves all colimits.
\end{construction}

In this language, the cotangent complex construction has a simple description.

\begin{example}[The cotangent complex as a left Kan extension]
Applying Construction~\ref{LKE} to $\mathcal{D} := D(R)$ the derived $\infty$-category of $R$-modules and $f = \Omega^1_{-/R}$, one obtains a functor $F:s\mathrm{CAlg}_R \to D(R)$. We claim that $F(-) = L_{-/R}$. To see this, note that $F$ commutes with sifted colimits and agrees with $\Omega^1_{-/R}$ on polynomial $R$-algebras. It follows that if $P_\bullet \to A$ is a simplicial resolution of $A \in \mathrm{CAlg}_R$ by a simplicial polynomial $R$-algebra $P_\bullet$, then $F(A) \simeq |\Omega^1_{P_\bullet/R}| \simeq L_{A/R}$, as asserted. We can summarize this situation by saying that the cotangent complex functor $L_{-/R}$ is obtained by left Kan extension of the K\"{a}hler differentials functor $\Omega^1_{-/R}$ from polynomial $R$-algebras to all simplicial commutative $R$-algebras. Similarly, for each $j \geq 0$, the functor $\wedge^j L_{-/R}$ is the left Kan extension of $\Omega^j_{-/R}$ from polynomial $R$-algebras to all simplicial commutative $R$-algebras.
\end{example}

Next, recall from \cite[\S III.3.1]{IllusieCC1} that if $A$ is a smooth $R$-algebra, then the adjunction map $\wedge^i_A L_{A/R} \to\Omega^i_{A/R}$ is an isomorphism for each $i\ge0$.

Finally, if $A\to B\to C$ are homomorphisms of commutative rings, then from \cite[\S II.2.1]{IllusieCC1} one has an associated transitivity triangle
\[ L_{B/A}\dotimes_B C\to L_{C/A}\to L_{C/B}\]
in $D(C)$. Taking wedge powers induces a natural filtration on $\wedge^i L_{C/A}$ as in \cite[\S V.4]{IllusieCC1}:
\[
\wedge^i_C L_{C/A}= {\rm Fil}^0\wedge^i_C L_{C/A}\leftarrow {\rm Fil}^1\wedge^i_C L_{C/A}\leftarrow\ldots\leftarrow {\rm Fil}^i\wedge^i_C L_{C/A}=\wedge^i_B L_{B/A}\dotimes_B C\leftarrow {\rm Fil}^{i+1}\wedge^i_C L_{C/A}=0\]
of length $i$, with graded pieces
\[
\gr^j\wedge^i_C L_{C/A}\simeq (\wedge^j_B L_{B/A}\dotimes_B C)\dotimes_C\wedge^{i-j}_C L_{C/B}.\tag{$j=0,\dots,i$}
\]

\subsection{Hochschild homology}
\label{subsection_HH}

Let $R$ be a commutative ring. Let $A$ be a commutative $R$-algebra.\footnote{Hochschild homology can also be defined for noncommutative (and nonunital) $R$-algebras, but we will not need this generality.} Following \cite{Loday}, the ``usual'' Hochschild homology of $A$ is defined to be $\HH^{\mathrm{usual}}(A)=C_\bullet(A/R)$, where $C_\bullet(A/R)=\{[n]\mapsto A^{\otimes_R {n+1}}\}$ is the usual simplicial $R$-module. However, we will work throughout with the derived version of Hochschild homology (also known as Shukla homology following \cite{Shukla}), which we now explain. Letting $P_\bullet\to A$ be a simplicial resolution of $A$ by flat $R$-algebras, let $\HH(A/R)$ denote the diagonal of the bisimplicial $R$-module $C_\bullet(P_\bullet/R)$; the homotopy type of $\HH(A/R)$ does not depend on the choice of resolution. The canonical map $\HH(A/R)\to \HH^{\mathrm{usual}}(A/R)$ is an isomorphism if $A$ is flat over $R$. When $R=\mathbb Z$ we omit it from the notation, so $\HH(A) = \HH(A/\mathbb Z)$.

\begin{remark}[Hochschild homology as a left Kan extension]
On the category $\mathrm{CAlg}_R^{poly}$ from Construction~\ref{LKE}, consider the functor $A \mapsto \HH^{\mathrm{usual}}(A/R)$ valued in the derived $\infty$-category $D(R)$. As polynomial $R$-algebras are $R$-flat, we have $\HH^{\mathrm{usual}}(A/R) \simeq \HH(A/R)$ for polynomial $R$-algebras $A$. The left Kan extension of this functor along $i:\mathrm{CAlg}_R^{poly} \to s\mathrm{CAlg}_R$ defines a functor $F:s\mathrm{CAlg}_R \to D(R)$ that coincides with the $\HH(-/R)$ functor introduced above: the functor $F$ commutes with sifted colimits as in Construction~\ref{LKE}, so we have $F(A) \simeq |\HH(P_\bullet/R)| \simeq \HH(A/R)$ for all commutative $R$-algebras $A$ and simplicial resolutions $P_\bullet \to A$. 
\end{remark}

As $C_\bullet(A/R)$ is actually a cyclic module, there are natural $\T=S^1$-actions on $\HH(A/R)$ considered as an object of the $\infty$-derived category $D(R)$ for all $R$-algebras $A$, and the negative cyclic and periodic cyclic homologies are defined by
\[
\HC^-(A/R) = \HH(A/R)^{h\T},\ \HP(A/R) = \HH(A/R)^{t\T} = \mathrm{cofib}(\mathrm{Nm}: \HH(A/R)_{h\T}[1]\to \HH(A/R)^{h\T})\ .
\]
For a comparison with the classical definitions via explicit double complexes, see \cite{Hoyois}. The homotopy fixed point and Tate spectral sequences
\[E_2^{ij}=H^i(\T, \pi_{-j}\HH(A/R))\Rightarrow \pi_{-i-j} \HC^-(A/R)\ ,\ E_2^{ij}=\widehat H^i(\T, \pi_{-j}\HH(A/R))\Rightarrow \pi_{-i-j} \HP(A/R)\]
are basic tools to analyse $\HC^-(A/R)$ and $\HP(A/R)$.

\begin{remark}[The universal property of $\HH(A/R)$]
In anticipation of \S \ref{subsection_intro_to_THH}, let us explain a higher algebra perspective on $\HH(-/R)$. The simplicial $R$-module $C_\bullet(A/R)$ is naturally a simplicial commutative $R$-algebra, and the multiplication is compatible with the $\T$-action. By left Kan extension from the flat case, we learn:
\begin{enumerate}
\item $\HH(A/R)$ is naturally a $\T$-equivariant-$E_\infty$-$R$-algebra.
\item One has a non-equivariant $E_\infty$-$R$-algebra map $A \to \HH(A/R)$ induced by the $0$ cells.
\end{enumerate}
In fact, one can also show that $\HH(A/R)$ is initial with respect to these features; we will use this perspective when introducing topological Hochschild homology. To see this, write $A \otimes_{E_\infty\text-R} \T$ for the universal\footnote{Even more generally, for any $\infty$-groupoid $X$ and any $E_\infty$-$R$-algebra $A$, the $\infty$-groupoid valued functor $B \mapsto \mathrm{Map}(X, \mathrm{Map}_{E_\infty\text-R}(A,B))$ on the $\infty$-category of $E_\infty$-$R$-algebras preserves limits, and is thus corepresented by some $A \otimes_{E_\infty\text-R} X$. The resulting functor $X \mapsto A \otimes_{E_\infty\text-R} X$ commutes with all colimits by construction. The discussion above pertains to the special case $X = \T$. In this case, the $\T$-action on $A \otimes_{E_\infty\text-R} \T$ is induced by functoriality, and the unit element $1 \in \T$ defines a map $A \to A \otimes_{E_\infty\text-R} \T$.} $\T$-equivariant-$E_\infty$-$R$-algebra equipped with a non-equivariant map $A \to A \otimes_{E_\infty\text-R} \T$. Then one has a natural $\T$-equivariant map $A \otimes_{E_\infty\text-R} \T \to \HH(A/R)$ of $E_\infty$-$R$-algebras by universality. To show this is an equivalence, it is enough to do so for $A\in \mathrm{CAlg}_R^{poly}$ and work non-equivariantly. In this case, if we write $\T$ as the colimit of $\ast \gets \ast \sqcup \ast \to \ast$, then it follows that $A \otimes_{E_\infty\text-R} \T \simeq A \dotimes_{A \otimes_R A} A$, which is also the object in $D(R)$ being computed by $C_\bullet(A/R)$.
\end{remark}

Recall that the Hochschild--Kostant--Rosenberg theorem asserts that if $A$ is smooth over $R$, then the antisymmetrisation map $\Omega^n_{A/R}\to\HH_n(A/R)$ is an isomorphism for each $n\ge0$. Here, in degree $1$, the element $da$ maps to $a\otimes 1 - 1\otimes a$ for any $a\in A$. Hence by left Kan extension of the Postnikov filtration, it follows that the functor $\HH(-/R)$ on $s\mathrm{CAlg}_R$ comes equipped with a $\T$-equivariant complete descending $\mathbb{N}$-indexed filtration $\mathrm{Fil}_\sub{HKR}^n$ with $\gr^i_\sub{HKR} \HH(-/R) \simeq \wedge^i L_{-/R}[i]$ (with the trivial $\T$-action).  As each $\wedge^i L_{-/R}[i]$ is $i$-connective, the HKR filtration gives a weak Postnikov tower $\{\HH(-/R)/\mathrm{Fil}_\sub{HKR}^n\}$ with limit $\HH(-/R)$ in the sense of the forthcoming Lemma~\ref{Postnikov}.

\subsection{Topological Hochschild homology}\label{subsection_intro_to_THH}

Topological Hochschild homology is the analogue of Hochschild homology over the base ring $\mathbb S$ given by the sphere spectrum. This is not a classical ring, but an $E_\infty$-ring spectrum (or equivalently an $E_\infty$-algebra in the $\infty$-category of spectra). The definition is due to B\"okstedt, \cite{Boekstedt}, and the full structure of topological Hochschild homology as a cyclotomic spectrum was obtained by B\"okstedt--Hsiang--Madsen, \cite{BHM}, cf.~also \cite{Hesselholt1997}. We will use the recent discussion in \cite{NikolausScholze} as our basic reference.

In particular, if $A$ is an $E_\infty$-ring spectrum, then $\THH(A)$ is a $\T$-equivariant $E_\infty$-ring spectrum with a non-equivariant map $A\to \THH(A)$ of $E_\infty$-ring spectra, and $\THH(A)$ is initial with these properties. Moreover, if $C_p\subset \T$ is the cyclic subgroup of order $p$, then there is a natural Frobenius map
\[
\varphi_p: \THH(A)\to \THH(A)^{tC_p}
\]
that is a map of $E_\infty$-ring spectra which is equivariant for the $\T$-actions, where the target has the $\T$-action coming from the residual $\T/C_p$-action via the isomorphism $\T\cong \T/C_p$. The Frobenius maps exist only on $\THH(A)$, not on $\HH(A)$, cf.\ \cite[Remark III.1.9]{NikolausScholze}.

The negative topological cyclic and periodic topological cyclic homologies are given by
\[
\TC^-(A) = \THH(A)^{h\T}\ ,\ \TP(A) = \THH(A)^{t\T}= \mathrm{cofib}(\mathrm{Nm}: \THH(A)_{h\T}[1]\to \THH(A)^{h\T})\
\]
There are homotopy fixed point and Tate spectral sequences analogous to those for Hochschild homology.

We will often work with the corresponding $p$-completed objects. We denote these by $\THH(A;\mathbb Z_p)$, $\HH(A;\mathbb Z_p)$, $\TC^-(A;\mathbb Z_p)$ etc. We note that if $A$ is connective, then
\[
\TC^-(A;\mathbb Z_p) = \THH(A;\mathbb Z_p)^{h\T}\ ,\ \TP(A;\mathbb Z_p) = \THH(A;\mathbb Z_p)^{t\T}\ .
\]
Here, we use that if a spectrum $X$ is $p$-complete, then so is $X^{h\T}$ by closure of $p$-completeness under limits. If $X$ is moreover assumed to be homologically bounded below, then the homotopy orbit spectrum $X_{h\T}$ (and thus also the Tate construction $X^{t\T}$) is also $p$-complete: by writing $X$ as the limit of $\tau_{\leq n} X$ (and $X_{h\T}$ as the limit of $(\tau_{\leq n} X)_{h\T}$, using Lemma~\ref{Postnikov}), this reduces by induction to the case that $X$ is concentrated in degree $0$, in which case the result follows by direct computation.

Interestingly, if $A$ is connective, there is a natural equivalence
\[
\TP(A;\mathbb Z_p)\simeq (\THH(A)^{tC_p})^{h\T} = (\THH(A;\mathbb Z_p)^{tC_p})^{h\T}
\]
by \cite[Lemma II.4.2]{NikolausScholze}, and therefore $\varphi_p$ induces a map
\[
\varphi_p^{h\T}: \TC^-(A;\mathbb Z_p)\to (\THH(A;\mathbb Z_p)^{tC_p})^{h\T} \simeq \TP(A;\mathbb Z_p)
\]
of $p$-completed $E_\infty$-ring spectra. As $p$ is fixed throughout the paper, we will often write abbreviate $\varphi=\varphi_p^{h\T}$.

We use only ``formal'' properties of $\THH$ throughout the paper, with the one exception of B\"okstedt's computation of $\pi_\ast \THH(\mathbb F_p)$. In particular, we do not need B\"okstedt's computation of $\pi_\ast \THH(\mathbb Z)$.

We will often use the following well-known lemma, which we briefly reprove in the language of this paper.

\begin{lemma}\label{THHvsHH} For any commutative ring $A$, there is a natural $\T$-equivariant isomorphism of $E_\infty$-ring spectra
\[
\THH(A)\otimes_{\THH(\mathbb Z)} \mathbb Z\simeq \HH(A)\ .
\]
Moreover, this induces an isomorphism of $p$-complete $E_\infty$-ring spectra
\[
\THH(A;\mathbb Z_p)\otimes_{\THH(\mathbb Z)} \mathbb Z\simeq \HH(A;\mathbb Z_p)\ .
\]
The homotopy groups $\pi_i \THH(\mathbb Z)$ are finite for $i>0$.\footnote{In fact, they are $\mathbb Z$ for $i=0$ and $\mathbb Z/j\mathbb Z$ if $i=2j-1>0$ is odd and $0$ else, by B\"okstedt, \cite{BoekstedtZ}.}
\end{lemma}

\begin{proof} The final statement follows from the description of $\THH(\mathbb Z)$ as the colimit of the simplicial spectrum with terms $\mathbb Z\otimes_{\mathbb S}\ldots\otimes_{\mathbb S} \mathbb Z$ and the finiteness of the stable homotopy groups of spheres. The first statement follows from the universal properties of $\THH(A)$, respectively\ $\HH(A)$, as the universal $\T$-equivariant $E_\infty$-ring spectrum, respectively $\T$-equivariant $E_\infty$-$\mathbb Z$-algebra, equipped with a non-equivariant map from $A$. The statement about $p$-completions follows as soon as one checks that $\THH(A;\mathbb Z_p)\otimes_{\THH(\mathbb Z)} \mathbb Z$ is still $p$-complete, which follows from finiteness of $\pi_i \THH(\mathbb Z)$ for $i>0$.
\end{proof}
\newpage

\section{Flat descent for cotangent complexes and Hochschild homology}

In this section, we prove flat descent for topological Hochschild homology via reduction to the case of (exterior powers of) the cotangent complex, which was first proved by the first author in \cite{Bhatt2012}.

\label{sec:descent}
\begin{theorem}
\label{FlatDescentCC}
Fix a base ring $R$. For each $i \geq 0$, the functor $A \mapsto \wedge^i_A L_{A/R}$ is an fpqc sheaf with values in the $\infty$-derived category $D(R)$ of $R$-modules, i.e., if $A \to B$ is a faithfully flat map of $R$-algebras, then the natural map gives an equivalence
\[ \wedge^i L_{A/R} \simeq \lim \big(\cosimp {\wedge^i L_{B/R}}{\wedge^i L_{(B \otimes_A B)/R}}{\wedge^i L_{(B \otimes_A B \otimes_A B)/R}} \big)\]
\end{theorem}

\begin{proof}
The $i=0$ case amounts to faithfully flat descent. We explain the $i=1$ case in depth, and then indicate the modifications necessary to tackle larger $i$.

Write $B^\bullet$ for the Cech nerve of $A \to B$. The transitivity triangle for $R \to A \to B^\bullet$ is a cosimplicial exact triangle
\[ L_{A/R} \otimes_A B^\bullet \to L_{B^\bullet/R} \to L_{B^\bullet/A}. \]
Thus, to prove the theorem for $i=1$, we are reduced to showing the following two assertions:
\begin{enumerate}
\item The map $A \to B^\bullet$ induces an isomorphism $L_{A/R} \to \lim \big( L_{A/R} \otimes_A B^\bullet \big)$.
\item One has $\mathrm{Tot} (L_{B^\bullet/A}) \simeq 0$.
\end{enumerate}

Assertion (1) holds true more generally for any $M \in D(A)$ (with the assertion above corresponding to $M = L_{A/R}$) by fpqc descent. 

We now prove assertion (2). By the convergence of the Postnikov filtration, it is enough to show that for each $i$, the $A$-cochain complex corresponding to $\pi_i L_{B^\bullet/A}$ under the Dold-Kan equivalence is acyclic. By faithful flatness of $A \to B$, this reduces to showing acyclicity of  $(\pi_i L_{B^\bullet/A}) \otimes_A B \simeq \pi_i (L_{B^\bullet/A} \otimes_A B)$. If we set $B \to C^\bullet$ to be the base change of $A \to B^\bullet$ along $A \to B$, then by flat base change for the cotangent complex, we have reduced to showing that $\pi_i L_{C^\bullet/B}$ is acyclic. But $B \to C^\bullet$ is a cosimplicial homotopy-equivalence of $B$-algebras: it is the Cech nerve of the map $B \to B \otimes_A B$, which has a section. It follows that for any abelian group valued functor $F(-)$ on $B$-algebras, we have an induced cosimplicial homotopy equivalence $F(B) \to F(C^\bullet)$. Taking $F = \pi_i L_{-/B}$ then shows that the cochain complex $\pi_i L_{C^\bullet/B}$ is homotopy-equivalent to the abelian group $\pi_i L_{B/B} \simeq 0$, as wanted.

To handle larger $i$, one follows the same steps as above with the following change: instead of using the transitivity triangle to reduce to proving (1) and (2) above, one uses the length $i$ filtration of $\wedge^i L_{B^\bullet/R}$ induced by applying $\wedge^i$ to the transitivity triangle above used above (and induction on $i$) to reduce to proving the analog of (1) and (2) for exterior powers of the cotangent complex.
\end{proof}

\begin{remark}
We do not know if the functors appearing in Theorem~\ref{FlatDescentCC} satisfy hyperdescent: if $A \to B^\ast$ is a hypercover for the faithfully flat topology on (the opposite of) the category of $R$-algebras,  is the natural map $L_{A/R} \to \lim L_{B^\ast/R}$ an equivalence?
\end{remark}

\begin{lemma}
\label{Postnikov}
Let $S$ be a connective $E_1$-ring spectrum. Assume $\{M_n\}$ is a weak Postnikov tower of connective $S$-module spectra, i.e., the fiber of $M_{n+1} \to M_n$ is $n$-connected. Write $M$ for the inverse limit. Then for any right $t$-exact functor $F:D(S) \to \mathrm{Sp}$, the tower $\{F(M_n)\}$ is a weak Postnikov tower with limit $F(M)$.
\end{lemma}

\begin{proof}
The assertion that $\{F(M_n)\}$ is a weak Postnikov tower is immediate from the exactness hypothesis on $F$.  For the rest, note that the fiber of $M \to M_n$ is $n$-connected: it is the inverse limit of the fibers $P_k$ of $M_{n+k} \to M_n$, and each $P_k$ is $n$-connected with $P_{k+1} \to P_k$ being an isomorphism on $\pi_{n+1}(-)$. Then $F(M) \to F(M_n)$ also has an $n$-connected fiber by the exactness hypothesis on $F$, and hence $F(M) \to \lim F(M_n)$ is an equivalence, as wanted.
\end{proof}

\begin{corollary}
\label{FlatDescentHH}$ $
\begin{enumerate}
\item For any commutative ring $R$, the functors
\[
\HH(-/R)\ ,\ \HC^{-}(-/R)\ ,\ \HH(-/R)_{h\T}\ ,\ \HP(-/R)
\]
on the category of commutative $R$-algebras are fpqc sheaves.
\item Similarly, the functors
\[
\THH(-)\ ,\ \TC^{-}(-)\ ,\ \THH(-)_{h\T}\ ,\ \TP(-)
\]
on the category of commutative rings are fpqc sheaves.
\end{enumerate}
\end{corollary}
\begin{proof}
(1) Theorem~\ref{FlatDescentCC} and induction imply that each $\HH(-/R)/\mathrm{Fil}_\sub{HKR}^n$ is a sheaf. Taking the limit over $n$ then implies $\HH(-/R)$ is a sheaf. Likewise, $\HC^{-}(-/R)$ is a sheaf as it is the limit of a diagram of sheaves.
	
For $\HH(-/R)_{h\T}$, using Lemma~\ref{Postnikov} for $S = R[\T]$ and $F = (-)_{h\T}$ applied to the weak Postnikov tower $\{\HH(-/R)/\mathrm{Fil}_\sub{HKR}^n\}$ with limit $\HH(-/R)$, it suffices to prove that each $(\HH(-/R)/\mathrm{Fil}_\sub{HKR}^n)_{h\T}$ is a sheaf. Using the filtration, this immediately reduces to checking that $(\gr_\sub{HKR}^i \HH(-/R))_{h\T}$ is a sheaf. But the $\T$-action on $\gr^i_\sub{HKR}\HH(-/R)$ is trivial, so we can write $(\gr^i_\sub{HKR} \HH(-/R))_{h\T} \simeq \gr^i_\sub{HKR} \HH(-/R) \dotimes_R R_{h\T}$. Using that $\gr^i_\sub{HKR} \HH(-/R)\simeq \wedge^i L_{-/R}[i]$ is connective, we see that the tower $\{\gr^i_\sub{HKR} \HH(-/R)\dotimes_R \tau_{\leq n} R_{h\T}\}_n$ is a weak Postnikov tower, and so we reduce to showing that $\gr^i_\sub{HKR} \HH(-/R)\dotimes_R \tau_{\leq n} R_{h\T}$ is a sheaf. But each $\tau_{\leq n} R_{h\T}$ is a perfect $R$-complex, so the claim follows from Theorem~\ref{FlatDescentCC} again.
	
Combining the assertions for $\HC^{-}(-/R)$ and $\HH(-/R)_{h\T}$ then trivially implies that $\HP(-/R)$ is also a sheaf.
	
(2) As above, it is enough to prove the claims for $\THH(-)$ and $\THH(-)_{h\T}$. For $\THH(-)$, we use that for any commutative ring $A$, the tower $\{\THH(A) \otimes_{\THH(\mathbb{Z})} \tau_{\leq n} \THH(\mathbb{Z})\}$ is a weak Postnikov tower with limit $\THH(A)$ by applying Lemma~\ref{Postnikov} to $S = \THH(\mathbb{Z})$ and $F = \THH(A) \otimes_{\THH(\mathbb{Z})} -$ to the Postnikov tower $\{ \tau_{\leq n} \THH(\mathbb{Z}) \}$ with limit $\THH(\mathbb{Z})$. We are thus reduced to checking that $\THH(-) \otimes_{\THH(\mathbb{Z})} \tau_{\leq n} \THH(\mathbb{Z})$ is a sheaf for each $n$. By induction, this immediately reduces to checking that $\THH(-) \otimes_{\THH(\mathbb{Z})} \pi_n \THH(\mathbb{Z}) \simeq \big(\THH(-) \otimes_{\THH(\mathbb{Z})} \mathbb{Z}\big) \otimes_{\mathbb{Z}} \pi_n \THH(\mathbb Z)$ is a sheaf for each $n$. Now $\pi_n \THH(\mathbb{Z})$ is a perfect $\mathbb{Z}$-complex by Lemma \ref{THHvsHH}, so we reduce to showing that $\THH(-) \otimes_{\THH(\mathbb{Z})} \mathbb{Z}$ is a sheaf. But this last functor is $\HH(-)=\HH(-/\mathbb Z)$, so we are done by reduction to (1). 

For $\THH(-)_{h\T}$, one repeats the argument in the previous paragraph by applying $(-)_{h\T}$ to the $\T$-equivariant weak Postnikov tower $\{\THH(R) \otimes_{\THH(\mathbb{Z})} \tau_{\leq n} \THH(\mathbb{Z})\}$ to reduce to the case of $\HH(-)_{h\T}$, which was handled in (1).
\end{proof}

\begin{remark} The previous proof also applies to $\HH(-/R)_{hC_n}$ and $\THH(-)_{hC_n}$ for any finite subgroup $C_n\subset \T$, and thus to $\HH(-/R)^{tC_n}$, $\THH(-)^{tC_n}$ and by induction all $\TR^n(-)$, using the isotropy separation squares.
\end{remark}

\newpage
\section{The quasisyntomic site}
\label{sec:sites}

This section studies the category of quasisyntomic rings, which is where all our later constructions take place. In \S \ref{qsynsite}, we define the notion of a quasisyntomic ring as well as the quasisyntomic site. An extremely important class of examples comes from quasiregular semiperfectoid rings: these are roughly  the quasisyntomic rings whose mod $p$ reduction is semiperfect (i.e., has a surjective Frobenius), and are studied in \S \ref{qsprings}. The key result of this section is that quasiregular semiperfectoid rings form a basis for the quasisyntomic topology (Proposition~\ref{qsqspextend}). As $p$-adic completions show up repeatedly in the sequel, we spend some time in \S \ref{pflat} exploring the interaction of $p$-adic completion with notions such as flatness.

\subsection{$p$-complete flatness}
\label{pflat}

Let us begin by defining a notion of $p$-complete Tor amplitude\footnote{Recall the classical definitions: given $a, b \in \mathbb{Z} \cup \{\pm \infty\}$, a commutative ring $R$ and $M \in D(R)$, we say that $M$ has Tor amplitude in $[a,b]$ if for any $R$-module $N$, we have $M \dotimes_R N \in D^{[a,b]}(R)$. If $a=b$, then we say that $M$ has Tor amplitude concentrated in degree $a$; note that if $a=b=0$, then the condition simply says that $M$ is concentrated in degree $0$ and flat. These conditions are preserved under base change, and can be checked after faithfully flat base change.} for complexes over commutative rings.\footnote{For a more general discussion of such matters, see work of Yekutieli, \cite{Yekutieli}.}

\begin{definition}
Let $A$ be a commutative ring, fix $M \in D(A)$, and $a,b \in \mathbb{Z} \sqcup \{\pm \infty\}$.
\begin{enumerate}
\item We say that $M \in D(A)$ has {\em $p$-complete Tor amplitude in $[a,b]$} if $M \dotimes_A A/pA \in D(A/pA)$ has Tor amplitude in $[a,b]$. If $a=b$, we say that $M \in D(A)$ has {\em $p$-complete Tor amplitude concentrated in degree $a$}.
\item We say that $M \in D(A)$ is {\em $p$-completely (faithfully) flat} if $M \dotimes_A A/pA \in D(A/pA)$ is concentrated in degree $0$, and a (faithfully) flat $A/pA$-module.
\end{enumerate}

In particular, by definition $M\in D(A)$ has $p$-complete Tor amplitude $[0,0]$ if and only if it is $p$-completely flat.
\end{definition}

\begin{remark}\label{rem:toramplitudethickening}
One may replace $A/pA$ with $A/p^nA$ for any $n \geq 1$ in the above definition without changing its meaning. Indeed, if $R\to S=R/I$ with $I^2=0$ is a square-zero thickening, then $M\in D(R)$ has Tor-amplitude in $[a,b]$ if and only if $M\dotimes_R S\in D(S)$ has Tor-amplitude in $[a,b]$: The forward direction follows from the stability of Tor-amplitude under base change, and for the converse one uses the triangle
\[
(M\dotimes_R S)\dotimes_S I\to M\to M\dotimes_R S
\]
in $D(R)$.
\end{remark}

\begin{remark}
We will see in Lemma~\ref{BoundedTorsionTA} below that if $A$ has bounded $p^\infty$-torsion, then if $M\in D(A)$ has $p$-complete Tor-amplitude in $[a,b]$ and is derived $p$-complete, one has $M\in D^{[a,b]}(A)$. In particular, if $M\in D(A)$ is derived $p$-complete and $p$-completely flat, then $M$ is an $A$-module concentrated in degree $0$. In that case, the condition implies that $M/p^nM$ is a flat $A/p^nA$-module for all $n$, and a precise characterization of the $p$-completely flat $A$-modules is given by Lemma~\ref{BoundedTorsionFlat}.
\end{remark}

\begin{lemma}\label{lem:completionTorampl} Fix a ring $A$, an $M \in D(A)$ and $a,b \in \mathbb{Z} \sqcup \{\pm \infty\}$. Let $\widehat{M} \in D(A)$ be the derived $p$-completion of $M$. The following are equivalent:
\begin{enumerate}
\item  $M \in D(A)$ has $p$-complete Tor amplitude in $[a,b]$ (resp. is $p$-completely (faithfully) flat)
\item  $\widehat{M} \in D(A)$ has $p$-complete Tor amplitude in $[a,b]$ (resp. is $p$-completely (faithfully) flat).
\end{enumerate}
\end{lemma}

\begin{proof} The map $M\to \widehat{M}$ induces an isomorphism $M\dotimes_{\mathbb Z} \mathbb Z/p\mathbb Z\simeq \widehat{M}\dotimes_{\mathbb Z} \mathbb Z/p\mathbb Z$. In particular, after further base change along $A\dotimes_{\mathbb Z}\mathbb Z/p\mathbb Z\to A/pA$, it induces an isomorphism
\[
M\dotimes_A A/pA\simeq \widehat{M}\dotimes_A A/pA\ ,
\]
which immediately gives the result.
\end{proof}

\begin{lemma}
Fix a map $A \to B$ of rings, a complex $M \in D(A)$ and $a,b \in \mathbb{Z} \sqcup \{\pm \infty\}$. 
\begin{enumerate}
\item If $M \in D(A)$ has $p$-complete Tor amplitude in $[a,b]$ (resp. is $p$-completely (faithfully) flat), then the same holds true for $M \dotimes_A B \in D(B)$. 
\item If $A \to B$ is $p$-completely faithfully flat, then the converse to (1) holds true.
\end{enumerate}
\end{lemma}

\begin{proof}
This is immediate from the corresponding assertions in the discrete case, noting that the condition in part (2) implies in particular that $A/pA\to B/pB$ is faithfully flat.
\end{proof}

\begin{lemma}
\label{BoundedTorsionTA}
Fix a ring $A$ with bounded $p^\infty$-torsion and a derived $p$-complete $M \in D(A)$ with $p$-complete Tor amplitude in $[a,b]$ for $a,b \in \mathbb{Z}$. Then $M \in D^{[a,b]}(A)$.
\end{lemma}

Here we say that an abelian group $N$ has {\em bounded $p^\infty$-torsion} if $N[p^\infty] = N[p^c]$ for $c \gg 0$. In this case, the derived $p$-completion $\widehat{N}$ of $N$ coincides with the classical $p$-adic completion $\lim_n N/p^nN$, and this completion also has bounded $p^\infty$-torsion. In fact, the pro-systems $\{N/p^n N\}$ and $\{N\dotimes_{\mathbb Z} \mathbb Z/p^n \mathbb Z\}$ in $D(\mathbb Z)$ are pro-isomorphic.

\begin{proof}
As $A$ has bounded $p^\infty$-torsion, the pro-systems $\{A/p^n A\}$ and $\{A\dotimes_{\mathbb Z} \mathbb Z/p^n \mathbb Z\}$ are pro-iso\-mor\-phic. Thus, $M$ being derived $p$-complete implies that $M$ is the derived limit of $M\dotimes_A A/p^n A$. On the other hand, by assumption all $M\dotimes_A A/p^n A\in D^{[a,b]}(A/p^n A)$, and the transition maps on the highest degree $H^b(M\dotimes_A A/p^n A)$ are surjective. By passage to the limit, we get the result.
\end{proof}

Over rings with bounded $p^\infty$-torsion, we can describe $p$-completely flat complexes as modules:

\begin{lemma}
\label{BoundedTorsionFlat}
Fix a ring $A$ with bounded $p^\infty$-torsion.
\begin{enumerate}
\item If a derived $p$-complete $M \in D(A)$ is $p$-completely flat, then $M$ is a classically $p$-complete $A$-module concentrated in degree $0$, with bounded $p^\infty$-torsion, such that $M/p^nM$ is flat over $A/p^n A$ for all $n\geq 1$. Moreover, for all $n\geq 1$, the map
\[
M\otimes_A A[p^n]\to M[p^n]
\]
is an isomorphism.

\item Conversely, if $N$ is a classically $p$-complete $A$-module with bounded $p^\infty$-torsion such that $N/p^n N$ is flat over $A/p^nA$ for all $n\geq 1$, then $N \in D(A)$ is $p$-completely flat.
\end{enumerate}
\end{lemma}

\begin{proof}
(1) Lemma~\ref{BoundedTorsionTA} implies that $M$ is an $A$-module concentrated in degree $0$. The condition that $M$ is $p$-completely flat implies that $M/p^n M = M\dotimes_A A/p^n A$ is a flat $A/p^n A$-module for all $n\geq 1$. Moreover, $M$ is the limit of $M\dotimes_A A/p^n A = M/p^n M$, so $M$ is classically $p$-complete. It remains to prove that $M[p^n] = M\otimes_A A[p^n]$ for all $n\geq 1$; this implies boundedness of $p^\infty$-torsion. To see this, consider the $A_n = A\dotimes_{\mathbb Z} \mathbb Z/p^n \mathbb Z$-module $M_n = M\dotimes_{\mathbb Z} \mathbb Z/p^n \mathbb Z$; tensoring the triangle
\[
(A[p^n])[1]\to A_n\to A/p^n A
\]
in $D(A_n)$ with $M_n$ gives a triangle
\[
M_n\dotimes_{A_n} (A[p^n])[1]\to M_n\to M_n\dotimes_{A_n} A/p^nA\ .
\]
Here, $M_n\dotimes_{A_n} A/p^n A = M\dotimes_A A/p^n A = M/p^n M$ is concentrated in degree $0$ and flat over $A/p^nA$, and then also $M_n\dotimes_{A_n} A[p^n] = (M_n\dotimes_{A_n} A/p^nA)\dotimes_{A/p^nA} A[p^n]$ is concentrated in degree $0$. Thus, using the above triangle in the second equality,
\begin{align*}
M[p^n] = H^{-1}(M_n) = H^0(M_n\dotimes_{A_n} A[p^n]) &= H^0(M/p^nM\dotimes_{A/p^nA} A[p^n])\\ &= M/p^n M\otimes_{A/p^nA} A[p^n] = M\otimes_A A[p^n]\ ,
\end{align*}
as desired.

(2) Note first that the pro-system $\mathbb Z/p^n\mathbb Z\dotimes_{\mathbb Z}\mathbb Z/p\mathbb Z\in D(\mathbb Z/p\mathbb Z)$ is pro-isomorphic to $\mathbb Z/p\mathbb Z$. We may extend scalars to $A/pA$ to get a pro-isomorphism
\[
A/pA\simeq \{\mathbb Z/p^n \mathbb Z\dotimes_{\mathbb Z} A/pA\}_n
\]
in $D(A/pA)$, in particular in $D(A)$. Taking the derived tensor product $N\dotimes_A -$, we get a pro-isomorphism
\[
N\dotimes_A A/pA\simeq \{(N\dotimes_{\mathbb Z} \mathbb Z/p^n\mathbb Z)\dotimes_A A/pA\}_n\ .
\]
On the other hand, the pro-system $\{N\dotimes_{\mathbb Z} \mathbb Z/p^n\mathbb Z\}_n$ is pro-isomorphic to $\{N/p^n N\}_n$ as $N$ has bounded $p^\infty$-torsion. Thus, we get a pro-isomorphism
\[
N\dotimes_A A/pA\simeq \{N/p^nN\dotimes_A A/pA\}_n\ .
\]
We need to see that $N\dotimes_A A/pA$ is concentrated in degree $0$. But
\[
N/p^n N\dotimes_A A/pA = N/p^n N\dotimes_{A/p^nA} (A/p^nA\dotimes_A A/pA)
\]
and $N/p^n N$ is a flat $A/p^n A$-module, so it suffices to see that $\{H^{-i}(A/p^nA\dotimes_A A/pA)\}_n$ for $i<0$ is pro-zero. But the above discussion for $N=A$ shows that this pro-system is pro-isomorphic to $H^{-i}(A/pA) = 0$, as desired.
\end{proof}

\begin{corollary}\label{cor:torsionfflat} Let $A\to B$ be a map of derived $p$-complete rings.
\begin{enumerate}
\item If $A$ has bounded $p^\infty$-torsion and $A\to B$ is $p$-completely flat, then $B$ has bounded $p^\infty$-torsion.
\item Conversely, if $B$ has bounded $p^\infty$-torsion and $A\to B$ is $p$-completely faithfully flat, then $A$ has bounded $p^\infty$-torsion.
\item Assume that $A$ and $B$ have bounded $p^\infty$-torsion. Then the map $A\to B$ is $p$-completely flat (resp.~$p$-completely faithfully flat) if and only if $A/p^nA\to B/p^n B$ is flat (resp.~faithfully flat) for all $n\geq 1$.
\end{enumerate}
\end{corollary}

\begin{proof} Parts (1) and (3) follow immediately from Lemma~\ref{BoundedTorsionFlat}. For part (2), note that the proof of Lemma~\ref{BoundedTorsionFlat} shows that if $B$ is $p$-completely flat, then
\[
A[p^n]\otimes_{A/p^nA} B/p^nB\to B[p^n]
\]
is an isomorphism. By faithful flatness of $A/p^nA\to B/p^nB$, this implies that $A[p^n]\subset B[p^n]$ for all $n\geq 1$, which implies that if $B[p^c]=B[p^\infty]$, then also $A[p^c]=A[p^\infty]$, as desired.
\end{proof}

\begin{remark} In particular, if $R$ is some ring and $A\to B$ is a $p$-completely faithfully flat map of derived $p$-complete $R$-algebras with bounded $p^\infty$-torsion, then
\[
\wedge^i L_{A/R}^\wedge \simeq \lim \big(\cosimp {\wedge^i L_{B/R}^\wedge}{\wedge^i L_{(B \otimes_A B)/R}^\wedge}{\wedge^i L_{(B \otimes_A B \otimes_A B)/R}^\wedge} \big).\]
Indeed, this follows by applying Theorem~\ref{FlatDescentCC} to $A/p^n\to B/p^n$ and passing to the limit over $n$, noting that $\{A/p^n\}$ and $\{A\otimes^L_{\mathbb Z} \mathbb Z/p^n\}$ are pro-isomorphic.
\end{remark}

\subsection{The quasisyntomic site}
\label{qsynsite}

In the following, we will work with $p$-complete rings $A$ with bounded $p^\infty$-torsion (in which case classical and derived $p$-completeness are equivalent). To state our results in optimal generality, it will be convenient to generalize the usual notions of smooth and syntomic morphisms by omitting finiteness conditions and merely requiring good behaviour of the cotangent complex.

We note that it would be better to say ``$p$-completely quasisyntomic/quasismooth'' in place of ``quasisyntomic/quasismooth'' below; but that would get excessive for the purposes of this paper.

\begin{definition}[The quasisyntomic site]  We need the following notions.
\label{defqs}
\begin{enumerate}
\item A ring $A$ is quasisyntomic if the following conditions are satisfied.
\begin{enumerate}
\item The ring $A$ is $p$-complete with bounded $p^\infty$-torsion.
\item $L_{A/\mathbb{Z}_p} \in D(A)$ has $p$-complete Tor amplitude in $[-1,0]$.
\end{enumerate}
We denote by $\Qs$ the category of quasisyntomic rings.
\end{enumerate}

Let $A \to B$ be a map of $p$-complete rings with bounded $p^\infty$-torsion.
\begin{enumerate}[resume]
\item We say that $A \to B$ is a {\em quasismooth\footnote{This notion is distinct from other notions with the same name used sometimes in the literature. For instance, it differs from Berthelot's notion of a quasismooth map (which is in terms of a lifting property with respect to nilideals, cf. \cite[IV.1.5]{Berthelot1974}) or Lurie's notion of a quasismooth map of derived schemes (which, in fact, is much closer to our notion of quasisyntomic, cf. \cite[page 9]{SAG}).} map (resp. cover)} if:
\begin{enumerate}
\item $B$ is $p$-completely flat over $A$ (resp. $p$-completely faithfully flat over $A$).
\item $L_{B/A} \in D(B)$ is $p$-completely flat.
\end{enumerate}

\item We say that $A \to B$ is a {\em quasisyntomic map (resp. cover)} if:
\begin{enumerate}
\item $B$ is $p$-completely flat over $A$ (resp. $p$-completely faithfully flat over $A$).
\item $L_{B/A} \in D(B)$ has $p$-complete Tor amplitude in $[-1,0]$.
\end{enumerate} 
\end{enumerate}

We endow $\Qs^\sub{op}$ with the structure of a site via the quasisyntomic covers, cf.~Lemma~\ref{QsSite} below.
\end{definition}

\begin{remark}[Relation to Quillen's definition]
\label{remark_syntomic}
In \cite{Quillen1970}, Quillen defines a notion of quasiregular and regular ideals $I\subset A$. In particular, by \cite[Theorem 6.13]{Quillen1970}, an ideal $I\subset A$ is quasiregular if and only if $L_{(A/I)/A}$ has Tor-amplitude concentrated in degree $-1$. A map of rings $A\to B$ is classically called syntomic if it is flat and a local complete intersection. Note that if $A,B$ are Noetherian and $A\to B$ is of finite type, then it is a local complete intersection if and only if $L_{B/A}$ has Tor amplitude in $[-1,0]$ \cite[Theorem 5.5]{Quillen1970}; this equivalence even remains valid without the finite type hypothesis provided that ``local complete intersection'' is replaced by a more general notion for non-finite type morphisms of Noetherian rings, \cite{Avramov}. Thus, ignoring $p$-completion issues, the above definition of quasisyntomic is designed to extend the usual notion of syntomic to the non-Noetherian, non-finite-type setting.
\end{remark}

\begin{example}
\label{ex:SmoothQsyn}
(1) The $p$-adic completion of a smooth algebra over a perfectoid ring lies in~$\Qs$ by Example~\ref{Qspsperf} below and Corollary~\ref{cor:torsionfflat}.

(2) Any $p$-complete local complete intersection Noetherian ring $A$ lies in $\Qs$ (cf.~\cite[Tag 09Q3]{StacksProject} for the definition of a local complete intersection ring; it is equivalent to the map $\mathbb Z\to A$ being a local complete intersection in the sense of \cite{Avramov}). The boundedness of the $p^\infty$-torsion is clear as $A$ is Noetherian. The rest follows from (the easy direction of) the following theorem of Avramov.

\begin{theorem}[{\cite[Theorem 1.2]{Avramov}}] A Noetherian ring $A$ is a local complete intersection if and only if $L_{A/\mathbb Z}$ has Tor-amplitude in $[-1,0]$.
\end{theorem}
\end{example}

\begin{remark}[HKR for quasismooth maps]
\label{HKRQuasismooth}
Say $A \to B$ is a map of $p$-complete rings with bounded $p^\infty$-torsion. Consider the $p$-completion $\HH(B/A;\mathbb Z_p)$ of the Hochschild complex. By $p$-completing the HKR filtration from  \S \ref{subsection_HH}, we obtain a $\T$-equivariant complete descending $\mathbb{N}$-indexed filtration $\mathrm{Fil}_\sub{HKR}^n$ with $\gr^i_\sub{HKR} \HH(B/A;\mathbb Z_p)\simeq (\wedge^i L_{-/A}[i])^{\wedge}_p$ (with the trivial $\T$-action).  If $A \to B$ is quasismooth, then each $(\wedge^i L_{B/A}[i])^{\wedge}_p \simeq (\Omega^i_{B/A})^\wedge_p[i]$ is concentrated in homological degree $i$ by Lemma~\ref{BoundedTorsionFlat}. Thus, it follows that $\pi_* \HH(B/A;\mathbb Z_p)$ is the $p$-completion of the exterior algebra $\Omega^*_{B/A}$, thus giving a $p$-complete HKR theorem for quasismooth maps.
\end{remark}

\begin{lemma}
\label{QsCoverQs}
Let $A \to B$ be a quasisyntomic cover of $p$-complete rings. Then $A \in \Qs$  if and only if $B \in \Qs$. 
\end{lemma}

\begin{proof} By Corollary~\ref{cor:torsionfflat}, we can assume that $A$ and $B$ have bounded $p^\infty$-torsion. Assuming $A \in \Qs$, the transitivity triangle for $\mathbb{Z}_p \to A \to B$ and the quasisyntomicity of $A$ and $A \to B$ imply that $B \in \Qs$. Conversely, assume $B \in \Qs$. The transitivity triangle for $\mathbb{Z}_p \to A \to B$ and the quasisyntomicity of $B$ and $A \to B$ show that $L_{A/\mathbb{Z}_p} \dotimes_A B \in D(B)$ has $p$-complete Tor amplitude in $[-1,1]$. By connectivity, this trivially improves to $[-1,0]$. As $A \to B$ is $p$-completely faithfully flat, it follows that $L_{A/\mathbb{Z}_p} \in D(A)$ also has $p$-complete Tor amplitude in $[-1,0]$.
\end{proof}

\begin{lemma} 
\label{QSProperties}
All rings below are assumed to be $p$-complete with bounded $p^\infty$-torsion.
\begin{enumerate}
\item If $A \to B$ and $B \to C$ are quasisyntomic (resp.~quasismooth), then $A \to C$ is quasisyntomic (resp.~quasismooth). If $A\to B$ and $B\to C$ are covers, then so is $A\to C$.
\item If $A \to B$ is quasisyntomic (resp.~quasismooth) and $A \to C$ is arbitrary, then the $p$-adically completed pushout $C \to D := B \widehat{\otimes}_A C$ of $A \to B$ is also quasisyntomic (resp.~quasismooth). If $A\to B$ is a cover, then so is $C\to D$.
\end{enumerate}
\end{lemma}

\begin{proof} 
(1) This is clear from the transitivity triangle and stability of faithful flatness under composition.

(2) Let $D' := \widehat{B \dotimes_A C}$, so $D'$ is $p$-completely flat over $C$ by base change and $D = H^0(D')$. As $D'$ is $p$-completely flat over $C$, it is discrete by Lemma~\ref{BoundedTorsionFlat}, so $D \simeq D'$. As the formation of cotangent complexes commutes with derived base change, we get that $C \to D$ is quasisyntomic (resp.~quasismooth). Moreover, faithful flatness is preserved under base change.
\end{proof}

We follow the conventions of \cite[Tag 00VG]{StacksProject} for sites. Recall that the axioms for a covering family are: (a) isomorphisms are covers, (b) covers are stable under compositions, and (c) the pushout of a cover along an arbitrary map is required to exist and be a cover.

\begin{lemma}
\label{QsSite}
The category $\Qs^\sub{op}$ forms a site.
\end{lemma}

\begin{proof}
The only nontrivial assertion is the existence of pushouts of covers. Fix a diagram $C \gets A \to B$ in $\Qs$ with $A \to B$ being a quasisyntomic cover. Let $D := B \widehat{\otimes}_A C$ be the pushout in $p$-complete rings. As $C$ has bounded $p^\infty$-torsion, Lemma~\ref{QSProperties} implies that $C \to D$ is a quasisyntomic cover. Lemma~\ref{QsCoverQs} then implies that $D \in \Qs$. It is then immediate that $C \to D$ provides a pushout of $A \to B$ in $\Qs$.
\end{proof}

\subsection{Perfectoid rings}
\label{subsec:perfectoid}

For later reference, we recall a few facts about perfectoid rings in the sense of \cite[Definition 3.5]{BMS}, sometimes also called integral perfectoid rings to distinguish them from the perfectoid Tate rings usually studied in relation to perfectoid spaces.

\begin{definition} A ring $R$ is perfectoid if it is $p$-adically complete, there is some $\pi\in R$ such that $\pi^p=pu$ for some unit $u\in R^\times$, the Frobenius map $x\mapsto x^p$ on $R/p$ is surjective, and the kernel of the map $\theta: A_{\inf}(R)\to R$ is generated by one element. Here $A_{\inf}(R) = W(R^\flat)$ where $R^\flat$ is the inverse limit of $R/p$ along the Frobenius map, and $\theta: A_{\inf}(R)\to R$ is Fontaine's map (written $\theta_R$ if there is any chance of confusion), cf.~\cite[Lemma 3.2]{BMS}.
\end{definition}

The main properties of perfectoid rings that we need are summarized in the following proposition.

\begin{proposition}\label{prop:basicperfd} Let $R$ be a perfectoid ring.
\begin{enumerate}
\item The kernel of $\theta: A_{\inf}(R)\to R$ is generated by a non-zero-divisor $\xi$ of the form $p+[\pi^\flat]^p\alpha$, where $\pi^\flat=(\pi,\pi^{1/p},\ldots)\in R^\flat$ is a system of $p$-power roots of an element $\pi$ as in the definition and $\alpha\in A_{\inf}(R)$ is some element.
\item The cotangent complex $L_{R/\mathbb Z_p}$ has $p$-complete Tor-amplitude concentrated in degree $-1$, and its derived $p$-completion is isomorphic to $R[1]$.
\item The $p^\infty$-torsion in $R$ is bounded. More precisely, $R[p^\infty]=R[p]$.
\end{enumerate}

In particular, $R\in \Qs$.
\end{proposition}

\begin{proof} Part (1) follows from the proof of \cite[Lemma 3.10]{BMS}, in particular the construction of the element $\xi$ in the beginning. For part (2), it is enough to see that the $p$-completion of $L_{R/\mathbb Z_p}$ is isomorphic to $R[1]$ by Lemma~\ref{lem:completionTorampl}. We use the transitivity triangle for $\mathbb{Z}_p \to A_{\inf}(R) \xrightarrow{\theta} R$ to see that the $p$-completions of $L_{R/\mathbb Z_p}$ and $L_{R/A_{\inf}(R)}$ agree, as $\mathbb{Z}_p \to A_{\inf}(R)$ is relatively perfect modulo $p$. But $L_{R/A_{\inf}(R)} \cong \ker(\theta)/\ker(\theta)^2[1] \cong R[1]$ as $\ker \theta$ is generated by a non-zero-divisor.

For part (3), we give two proofs. We start with a proof by ``overkill''. As valuation rings have bounded $p^\infty$-torsion (as they are domains), it suffices to show that $R$ embeds into a product of perfectoid valuation rings. When $R$ has characteristic $p$, this is clear: any reduced ring embeds into a product of domains, and any domain embeds into a valuation ring, and any valuation ring of characteristic $p$ embeds into its perfection. In general, we use $v$-descent techniques from \cite{BhattScholzeWitt}. Let $R^\flat \to S$ be a $v$-cover of $R^\flat$ with each connected component of $S$ being a perfect valuation ring (see \cite[Lemma 6.2]{BhattScholzeWitt}). If $S^\ast$ denotes the Cech nerve of $R^\flat \to S$, then $R^\flat \simeq \lim S^\ast$ by the v-descent result \cite[Theorem 4.1]{BhattScholzeWitt} for the structure sheaf. Applying the functor $W(-) \otimes^L_{W(R^\flat)} R$ (which commutes with limits) then shows that $R \simeq \lim S^{\ast,\sharp}$. Now any distinguished element (in the sense of \cite[Remark 3.11]{BMS}) in $W(S)$ with $S$ perfect is automatically a nonzerodivisor (see \cite[Lemma 3.10]{BMS}). So each $S^{i,\sharp}$ is concentrated in degree $0$. In particular, the map $R \to S^{0,\sharp} = S^{\sharp}$ is injective. As $S$ was a product of perfect valuation rings, $S^\sharp$ is a product of perfectoid valuation rings, to the claim follows.

Now we give a more elementary proof. Write $R = A_{\inf}(R)/\xi$. If $x\in R[p^n]$, then $x$ lifts to $\tilde{x} \in A_{\inf}(R)$ with $p^n \tilde{x} \in (\xi)$. It is thus enough to show the following: if $f \in A_{\inf}(R)$ and $p^2 f \in (\xi)$, then $pf \in (\xi)$.  Assume $p^2 f = g \xi$ for some $g \in A_{\inf}(R)$. Write $g = \sum_{i \geq 0} [g_i] p^i$ and $\xi = \sum_{i \geq 0} [a_i] p^i$ for the $p$-adic expansions of $g$ and $\xi$ with $g_i,a_i \in R^\flat$. We shall show that $g_0 = 0 \in R^\flat$; this will imply $p \mid g$, and hence $pf = \frac{g}{p} \xi \in (\xi)$ as $A_{\inf}(R)$ is $p$-torsionfree, as wanted. We can write
\[ g \xi = [a_0 g_0] + ([a_0 g_1] + [a_1 g_0]) p + hp^2 \]
for some $h \in A_{\inf}(R)=W(R^\flat)$. As $p^2 \mid g \xi$, we get
\[ a_0 g_0 = 0 \quad \text{and then} \quad a_0g_1 + a_1g_0 = 0\]
in $R^\flat$. Multiplying the second equation by $g_0$ and using the first equation yields $a_1 g_0^2 = 0 \in R^\flat$. But $a_1 \in R^{\flat,\ast}$ by the choice of $\xi$ in (1), so $g_0^2 = 0$, which implies $g_0 = 0$ as $R^\flat$ is perfect.
\end{proof}

Note that $A_{\inf}(R) = W(R^\flat)$ carries a natural Frobenius automorphism $\varphi$. We will also often use the map $\tilde\theta = \theta\circ \varphi^{-1}: A_{\inf}(R)\to R$, whose kernel is generated by $\tilde\xi = \varphi(\xi)$.

\subsection{Quasiregular semiperfectoid rings}
\label{qsprings}

A basis for the topology of $\Qs$ is given by the quasiregular semiperfectoid rings, defined as follows.

\begin{definition}
\label{defqsp}
A ring $S$ is {\em quasiregular semiperfectoid} if:
\begin{enumerate}
\item The ring $S$ is quasisyntomic, i.e.~$S\in \Qs$.
\item There exists a map $R \to S$ with $R$ perfectoid.
\item The Frobenius of $S/pS$ is surjective, i.e.~$S/pS$ is semiperfect.
\end{enumerate}
Write $\Qsp$ for the category of quasiregular semiperfectoid rings. We equip $\Qsp$ with the topology determined by quasisyntomic covers.
\end{definition}

\begin{remark}\label{rem:qsptor}
For $S \in \Qsp$, condition (3) in the definition ensures that $\Omega^1_{(S/pS)/\mathbb{F}_p} = 0$, and thus $L_{S/R} \dotimes_S S/pS \in D^{\leq -1}(S/pS)$ for any map $R \to S$. This observation shall be used often in the sequel. Moreover, it implies that $L_{S/\mathbb Z_p}$ has $p$-complete Tor amplitude concentrated in degree $-1$.
\end{remark}

\begin{remark} Conditions (2) and (3) can be replaced by the condition that there exists a surjective map $R\to S$ from a perfectoid ring $R$. This is clearly sufficient; conversely, $R\widehat{\otimes}_{\mathbb Z_p} W(S^\flat)$ is a perfectoid ring surjecting onto $S$, where $S^\flat$ is the inverse limit perfection of $S/pS$.
\end{remark}

\begin{remark} In Definition~\ref{defqsp}, condition (2) is not implied by the other conditions. For example, the ring $\mathbb{Z}_p$ itself satisfies (1) and (3) but not (2). 
\end{remark}

\begin{example}
\label{Qspsperf}
Any perfectoid ring $R$ lies in $\Qsp$. Indeed, conditions (2) and (3) in Definition~\ref{defqsp} are automatic. For (1), use Proposition~\ref{prop:basicperfd}.
\end{example}

\begin{lemma}
\label{QspAlternative}
Fix a $p$-complete ring $S$ with bounded $p^\infty$-torsion such that $S/pS$ is semiperfect. Then $S$ is quasiregular semiperfectoid if and only if there exists a map $R\to S$ with $R$ perfectoid such that $L_{S/R} \in D(S)$ has $p$-complete Tor amplitude concentrated in degree $-1$. In this case, the latter condition holds true for every map $R \to S$ with $R$ perfectoid.
\end{lemma}

In particular, a ring $S$ is quasiregular semiperfectoid if and only if it is $p$-complete with bounded $p^\infty$-torsion and can be written as the quotient $S=R/I$ of a perfectoid ring $R$ by a ``$p$-completely quasiregular'' ideal $I$ (i.e.~$L_{S/R}$ has $p$-complete Tor amplitude concentrated in degree $-1$); in that case, whenever $S=R/I$ for some perfectoid ring $R$, the ideal $I$ is $p$-completely quasiregular.

\begin{proof}
Assume that there exists a map $R\to S$ with $R$ perfectoid such that $L_{S/R} \in D(S)$ has $p$-complete Tor amplitude concentrated in degree $-1$. Then the transitivity triangle for $\mathbb{Z}_p \to R \to S$ and Example~\ref{Qspsperf} ensure that $L_{S/\mathbb{Z}_p} \in D(S)$ also has $p$-complete Tor amplitude concentrated in degree $-1$, whence $S$ is quasisyntomic, and thus satisfies Definition~\ref{defqsp}.

Conversely, assume that $S$ is quasiregular semiperfectoid. Fix a map $R\to S$ with $R$ perfectoid. We shall show that $L_{S/R} \in D(S)$ has $p$-complete Tor amplitude concentrated in degree $-1$. The transitivity triangle for $\mathbb{Z}_p \to R \to S$, base changed to $S/pS$, gives
\[ L_{R/\mathbb{Z}_p} \dotimes_R S/pS \xrightarrow{\alpha_S} L_{S/\mathbb{Z}_p} \dotimes_S S/pS \to L_{S/R} \dotimes_S S/pS.\]
As the first two terms have Tor amplitude concentrated in degree $-1$ (by Example~\ref{Qspsperf} and the assumption $S \in \Qsp$), it is sufficient to show that the map $\beta_S := \pi_1(\alpha_S)$ of flat $S/pS$-modules is pure (i.e., injective after tensoring with any discrete $S/pS$-module). We shall use the following criterion:

\begin{lemma}
Let $A$ be a commutative ring. Fix a map $\beta:F \to N$ of $A$-modules with $F$ finite free and $N$ flat. Assume that $\beta \otimes_A k$ is injective for every field $k$. Then $\beta$ is pure.
\end{lemma}

\begin{proof}
Write $N$ as filtered colimit $\colim_i N_i$ with $N_i$ finite free (by Lazard's theorem). By finite presentation of $F$, we may choose a map $\beta_i:F \to N_i$ factoring $\beta$. For $j \geq i$, write $\beta_j:F \to N_j$ for the resulting map that also factors $\beta$. The assumption on $\beta$ trivially implies that $\beta_j \otimes_A k$ is also injective for every residue field $k$ of $A$ and all $j \geq i$. But then $\beta_j$ must be split injective for $j \geq i$ as both $F$ and $N_j$ are finite free. The claim follows as filtered colimits of split injective maps are pure.
\end{proof}

As the $p$-completion of $L_{R/\mathbb{Z}_p}$ coincides with $\ker(\theta_R)/\ker(\theta_R^2)[1] \cong R[1]$ (cf.~Example~\ref{Qspsperf}), $\beta_S$ can be viewed as the map
\[ \ker(\theta_R)/\ker(\theta_R)^2 \otimes_R S/pS \xrightarrow{\beta_S} \pi_1 (L_{S/\mathbb{Z}_p} \dotimes_S S/pS)\]
of $S/pS$-modules. Note that the source of this map is a free $S/pS$-module whose formation commutes with base change in $S$, and the target is flat over $S/pS$. By the above lemma, it is enough to show that $\beta_S \otimes k$ is injective for all perfect fields $k$ under $S/pS$. But $\beta_S \otimes k$ factors $\beta_k$ by functoriality, and $\beta_k$ is an isomorphism (as $\alpha_k$ is so for any perfectoid ring $k$ by \cite[Lemma 3.14]{BMS}). This gives injectivity for $\beta_S \otimes k$, as wanted.
\end{proof}

\begin{lemma}
\label{QspSite}
The category $\Qsp^\sub{op}$ forms a site.
\end{lemma}
\begin{proof}
The only nontrivial assertion is the existence of pushouts of covers. Fix a diagram $C \gets A \to B$ in $\Qsp$ with $A \to B$ be a quasisyntomic cover. Let $D := B \widehat{\otimes}_A C$ be the pushout in $p$-complete rings. Lemma~\ref{QSProperties} implies that $C \to D$ is a quasisyntomic cover. It is then enough to check that $D \in \Qsp$. Lemma~\ref{QSProperties} implies that $D$ has bounded $p^\infty$-torsion as the same holds for $C$. It is also clear that $D$ receives a map from a perfectoid ring. Finally, the formula $D/pD = B/pB \otimes_{A/pA} C/pC$ shows that the Frobenius is surjective on $D/pD$ as the same holds true for $B/pB$ and $C/pC$.
\end{proof}

\begin{lemma}
\label{QSynQSperfCover}
A $p$-complete ring $A$ lies in $\Qs$ exactly when there exists a quasisyntomic cover $A \to S$ with $S \in \Qsp$.
\end{lemma}

\begin{proof}
If there exists a quasisyntomic cover $A \to S$ with $S\in \Qsp$, then $A \in \Qs$ by Lemma~\ref{QsCoverQs}.

Conversely, assume $A \in \Qs$. Choose a free $p$-complete algebra $F = \widehat{\mathbb{Z}_p[\{x_i\}_{i \in I}]}$ on a set $I$ with a surjection $F \to A$. Let $F \to F_\infty$ be the quasisyntomic cover obtained by formally adjoining $p$-power roots of $\{p\} \sqcup \{x_i\}_{i \in I}$ in the $p$-complete sense, so $F_\infty$ is perfectoid. Let $A \to S$ be the base change of $F \to F_\infty$ along $F \to A$ in the $p$-complete sense; we shall check that $A \to S$ solves the problem. By Lemma~\ref{QSProperties}, $A \to S$ is a quasisyntomic cover and thus $S\in \Qs$ by Lemma~\ref{QsCoverQs}. To finish proving $S \in \Qsp$, it is now enough to observe that the ring $F_\infty$ is perfectoid, and the map $F_\infty \to S$ is surjective.
\end{proof}

\begin{remark}
\label{QspQsFreeCover}
The construction of the cover $A \to S$ in the second paragraph of the proof of Lemma~\ref{QSynQSperfCover} shows a bit more: the map $A/pA \to S/pS$ displays $S/pS$ as a free $A/pA$-module, and $L_{(S/pS)/(A/pA)}[-1]$ is a free $S/pS$-module (as the analogous assertions are true for $F \to F_\infty$). Moreover, the ring $S \in \Qsp$ receives a map from a $p$-torsionfree perfectoid ring. 
\end{remark}

\begin{lemma}
\label{QspCechCoversQs}
Let $A \to S$ be a quasisyntomic cover in $\Qs$ with $S \in \Qsp$. Then all terms of the Cech nerve $S^\bullet$ lie in $\Qsp$.
\end{lemma}
\begin{proof}
Each term $S^i$ is a quasisyntomic cover of $S$. In particular, each $S^i$ has bounded $p^\infty$-torsion by Lemma~\ref{QSProperties} and receives a map from a perfectoid ring (as $S$ does). As $S^i/pS^i$ is a quotient of $(S/pS)^{\otimes_{\mathbb F_p} (i+1)}$, its Frobenius is surjective.
\end{proof}

\begin{proposition}
\label{qsqspextend}
Restriction along $u:\Qsp^\sub{op} \to \Qs^\sub{op}$ induces an equivalence 
\[ \mathrm{Shv}_{\calC}(\Qs^\sub{op}) \simeq \mathrm{Shv}_{\calC}(\Qsp^\sub{op})\]
for any presentable $\infty$-category $\calC$.
\end{proposition}

Denote the inverse $\mathrm{Shv}_{\calC}(\Qsp^\sub{op}) \to \mathrm{Shv}_{\calC}(\Qs^\sub{op})$ by $F \mapsto F^\beth$; we shall call $F^\beth$ the {\em unfolding} of $F$.  Explicitly, given $A \in \Qs$, one computes $F^\beth(A)$ as the totalization of $F(S^\bullet)$ where $S^\bullet$ is chosen as in Lemma~\ref{QspCechCoversQs}.

\begin{proof}
It is enough to see that the corresponding $\infty$-topoi $\mathrm{Shv}(\Qsp^\sub{op})$ and $\mathrm{Shv}(\Qs^\sub{op})$ are equivalent (corresponding to the case where $\calC$ is the $\infty$-category of spaces); both sides are equivalent to the contravariant functors from the corresponding $\infty$-topos to $\calC$ taking colimits to limits by \cite[Proposition 1.3.17]{SAG}. We define an inverse functor $\mathrm{Shv}(\Qsp^\sub{op})\to \mathrm{Shv}(\Qs^\sub{op})$ as follows. There is a functor $\Qs^\sub{op}\to \mathrm{Shv}(\Qsp^\sub{op})$ sending any $A\in \Qs^\sub{op}$ to the sheaf $h_A$ it represents on $\Qsp^\sub{op}$. This functors takes covers to effective epimorphisms (as pullbacks of quasisyntomic maps are quasisyntomic, and can be covered by quasiregular semiperfectoids), and preserves their Cech nerves. This implies that for any $F\in \mathrm{Shv}(\Qsp^\sub{op})$, the presheaf $A\mapsto \mathrm{Hom}_{\mathrm{Shv}(\Qsp^\sub{op})}(h_A,F)$ defines a sheaf on $\Qs^\sub{op}$, defining the desired functor $\mathrm{Shv}(\Qsp^\sub{op})\to \mathrm{Shv}(\Qs^\sub{op})$. It is clear that the composite $\mathrm{Shv}(\Qsp^\sub{op})\to \mathrm{Shv}(\Qs^\sub{op})\to \mathrm{Shv}(\Qsp^\sub{op})$ is the identity. In the other direction, the composite
\[
\mathrm{Shv}(\Qs^\sub{op})\to \mathrm{Shv}(\Qsp^\sub{op})\to \mathrm{Shv}(\Qs^\sub{op})
\]
is the identity by using the previous lemma: For any $F\in \mathrm{Shv}(\Qs^\sub{op})$ and $A\in \Qs^\sub{op}$, we have to show that
\[
F(A) = \mathrm{Hom}_{\mathrm{Shv}(\Qsp^{\mathrm{op}})}(h_A,F|_{\Qsp^{\mathrm{op}}})\ ,
\]
noting that there is a natural map from left to right. But $h_A$ is the colimit of the Cech nerve $h_{S^\bullet}$ as in the previous lemma, and thus
\begin{align*}
\mathrm{Hom}_{\mathrm{Shv}(\Qsp^{\mathrm{op}})}(h_A,F|_{\Qsp^{\mathrm{op}}}) &= \lim \mathrm{Hom}_{\mathrm{Shv}(\Qsp^{\mathrm{op}})}(h_{S^\bullet},F|_{\Qsp^{\mathrm{op}}})\\ &= \lim F(S^\bullet)= F(A)\ ,
\end{align*}
as desired.
\end{proof}

\begin{remark}
We shall often use Proposition~\ref{qsqspextend} for the complete filtered derived category $\calC = \widehat{DF}(R)$, which we will be recalled in \S \ref{subsection_FDC} (and which the reader should consult for the ensuing notation). Therefore we remark that the unfolding process is compatible with the evaluation and associated graded functors for such sheaves, i.e., if $F \in \mathrm{Shv}_{\widehat{DF}(R)}(\Qsp)$ unfolds to $F^\beth$, then $F^\beth(i) = F(i)^\beth$ and $\gr^i (F^\beth) = (\gr^i F)^\beth$. In particular, if $F$ corresponds to an $\mathbb{N}$-filtered object (i.e., $\gr^i = 0$ for $i < 0$), then passage to the underlying non-filtered sheaf is also compatible with unfolding, i.e., $F^\beth(-\infty) = F(-\infty)^\beth$ as they both coincide with $F(0)^\beth$ by the previous observations.
\end{remark}

\subsection{Variants}

In applications, we shall often need to restrict attention to smaller subcategories of $\Qs$ and $\Qsp$ which are still related by an analog of Proposition~\ref{qsqspextend}; in particular, we will often fix a base ring.

\begin{variant}[Slice categories, I]
\label{QSynSlice}
Fix a ring $A$. We can consider the category $\Qs_A$ of maps $A \to B$ with $B\in \Qs$ as well as the full subcategory $\Qsp_A \subset \Qs_A$ spanned by maps $A\to S$ with $S\in \Qsp$. One can then check that the analogs of Lemma~\ref{QspSite}, Lemma~\ref{QsSite}, Lemma~\ref{QSynQSperfCover} and Proposition~\ref{qsqspextend} hold true for these categories. The following lemma is quite useful in working in these categories in practice:

\begin{lemma}
\label{QSynCC}
Assume $A$ is perfectoid or $A=\mathbb Z_p$. For any $B \in \Qs_A$, the complex $L_{B/A} \in D(B)$ has $p$-complete Tor amplitude in $[-1,0]$. Hence, the $p$-adic completion of $\wedge^i L_{B/A}[-i]$ lies in $D^{\geq 0}(B)$.
\end{lemma}

\begin{proof} Let us explain the assertion about the cotangent complex first. If $A=\mathbb Z_p$, this is true by definition, so assume that $A$ is perfectoid. Choose a quasisyntomic cover $B \to S$ with $S\in \Qsp$. By Lemma~\ref{QspAlternative}, we know that $L_{S/A} \in D(C)$ has $p$-complete Tor amplitude concentrated in degree $-1$. The transitivity triangle for $A \to B \to S$ and the quasisyntomicity of $B \to S$ then shows that $L_{B/A} \dotimes_B S \in D(S)$ has $p$-complete Tor amplitude in $[-1,1]$, and thus in $[-1,0]$ by connectivity. We conclude using $p$-complete faithful flatness of $B \to S$.

For exterior powers: it follows formally from the previous paragraph (and the corresponding statement over $B/pB$) that $\wedge^i_B L_{B/A}$ has $p$-complete Tor amplitude in $[-i,0]$. The claim now follows from Lemma~\ref{BoundedTorsionTA}.
\end{proof}
\end{variant}

\begin{variant}[Slice categories, II]
\label{QSynSlice2}
There is another variant of the slice category. Fix a quasisyntomic ring $A$. We consider the category $\qs_{A}$ of quasisyntomic $A$-algebras, with the quasisyntomic topology. Again, it has a full subcategory $\qsp_{A}\subset \qs_{A}$, and the previous results including Proposition~\ref{qsqspextend} stay true. In fact, all statements about covers of $A$ in $\Qs$ or $\Qsp$ are immediately statements about covers in $\qs_{A}$ and $\qsp_{A}$.

For a map $A\to B$ of quasisyntomic rings, there is an associated functor $\qs_{A}\to \qs_{B}$ sending $C$ to $C\hat{\otimes}_A B$. It is however not clear that this induces a morphism of topoi, as our sites do not have finite limits. For this reason, we prefer to work in big sites like $\Qs$ or $\Qs_A$ to get functoriality of our constructions in the algebras.
\end{variant}

\begin{variant}[Restricting to topologically free objects over $\calO_C$]
\label{FreeQSP}
In this variant, we specialize to working over $\calO_C$ where $C$ is a perfectoid field of characteristic $0$, and explain an analog of the preceding theory where, roughly, all instances of ``flat'' are replaced by ``projective''; this will be used in the proof of Theorem \ref{main_theorem}. Define a map $A \to B$ of $p$-complete and $p$-torsionfree rings to be a {\em proj-quasisyntomic map (resp. cover)} if the following properties hold:
\begin{enumerate}
\item $B/p$ is a projective (resp. projective and faithfully flat) $A/p$-module. 
\item $L_{(B/p)/(A/p)} \in D(B/p)$ has projective amplitude\footnote{A complex $K$ over a commutative ring $R$ has projective amplitude in $[a,b]$ if it can be represented by a complex of projective modules located in degrees $a,...,b$. This is equivalent to requiring that $\mathrm{Ext}^i_R(K,N) = 0$ for any $R$-module $N$ whenever $i \notin [-b,-a]$; see \cite[Tag 05AM]{StacksProject} for more.} in $[-1,0]$.
\end{enumerate}
Let $\qs_{\calO_C}^\sub{proj} \subset \qs_{\calO_C}$ be the full subcategory spanned by proj-quasisyntomic $\calO_C$-algebras. Let $\qsp_{\calO_C}^\sub{proj} := \qs_{\calO_C}^\sub{proj} \cap \Qsp_{\calO_C}$, so $L_{(S/p)/(\calO_C/p)}[-1]$ is a projective $S/p$-module for $S\in \qsp_{\calO_C}^\sub{proj}$. Note that $p$-adic completions of smooth $\calO_C$-algebras lie in $\qs_{\calO_C}^\sub{proj}$: the condition on cotangent complexes is clear, and any finitely presented flat $\calO_C/p$-algebra is free\footnote{Write $\calO_C/p$ as a direct limit of artinian local rings $R_i \subset \calO_C/p$. Then any finitely presented flat $\calO_C/p$-algebra $A$ descends to a finitely presented flat $R_i$-algebra $A_i$ for some $i \gg 0$. As flat modules over artinian local rings are free, $A_i$ is free over $R_i$, and hence $A$ is free over $\calO_C/p$.}.

We equip (the opposites of) $\qs_{\calO_C}^\sub{proj}$ and $\qsp_{\calO_C}^\sub{proj}$ with the topology determined by proj-quasisyntomic covers. It is easy to see that proj-quasisyntomic maps (resp. covers) are stable under base change and composition, which gives analogs of Lemma~\ref{QspSite} and Lemma~\ref{QsSite}. Remark~\ref{QspQsFreeCover} then ensures that objects in $\qs^\sub{proj,op}_{\calO_C}$ can be covered by those in $\qsp^\sub{proj,op}_{\calO_C}$, giving an analog of Lemma~\ref{QSynQSperfCover}. It is then easy to see that the analog of Proposition~\ref{qsqspextend} holds true for these categories.
\end{variant}

\newpage
\section{Negative cyclic homology and de~Rham cohomology}
\label{sec:HCvsdeRham}

The goal of this section is to prove  Theorem~\ref{thm:main7}. As this theorem concerns the existence of filtrations on objects of the derived category, we start with some reminders about the filtered derived category in \S \ref{subsection_FDC}; the main results here are the existence of a Beilinson $t$-structure (Theorem~\ref{TruncationBeilinson}) and the interaction of this $t$-structure with the Berthelot-Ogus-Deligne $L\eta$-functor (Proposition~\ref{LetaDF}). With this language in place, we study some important examples of sheaves on the quasisyntomic site (such as de Rham complexes or negative cyclic homology) in  \S \ref{subsection_de_rham_neg} and prove Theorem~\ref{thm:main7}.

\subsection{Recollections on the filtered derived category}\label{subsection_FDC}

We review some formalism surrounding the filtered derived category\footnote{In our applications, it will be useful to work with unbounded complexes with unbounded filtrations. Moreover, since we use the $\infty$-categorical perspective of \cite{NikolausScholze}, we also need the filtered derived category as an $\infty$-category. For these reasons, we adopt the language of $\infty$-categories when discussing the filtered derived category, instead of the more classical language used to discuss this notion, e.g., as in \cite{BBD}.}. Recall the following notion, where $\mathbb Z^\sub{op}$ is the category whose objects are the integers $n\in \mathbb Z$, and there is at most one map $n\to m$, which exists precisely when $n\geq m$.

\begin{definition}[Filtered derived category]
For any $E_\infty$-ring $R$, write 
\[ DF(R)  := \mathrm{Fun}(\mathbb{Z}^\sub{op}, D(R))\]
for {\em the filtered derived category of $R$}; write $DF = DF(\mathbb{S})$. We view these as symmetric monoidal presentable stable $\infty$-categories via the Day convolution symmetric monoidal structure, cf.~\cite{GwilliamDF}. Recall that this means that, for $F,G\in DF(R)$, one has
\[
(F\dotimes_R G)(i) = \colim_{j+k\geq i} F(j)\dotimes_R G(k)\ .
\]
Given $F \in DF(R)$, we call $F(-\infty) := \colim_i F(i)$ the {\em underlying spectrum} with $F(i) \to F(-\infty)$ specifying the $i$-th filtration level. Such an $F$ is called {\em complete} if $F(\infty) := \lim_i F(i)$ vanishes; in this case, we have $F(-\infty) \simeq \lim_i F(-\infty)/F(i)$. Write $\widehat{DF}(R) \subset DF(R)$ for the {\em complete filtered derived category}, i.e., the full subcategory spanned by complete objects. 
\end{definition}

For $F \in DF(R)$, write $\gr^i(F) = F(i)/F(i-1)$. We shall often denote $F \in \widehat{DF}(R)$ as $(F(-\infty), F(\f))$ or simply $F(\f)$; the former notation reflects the intuition that $F$ gives a complete descending $\mathbb{Z}$-indexed filtration $F(\f)$ on the underlying spectrum $F(-\infty)$, and will typically be used only in the $\mathbb{N}$-indexed case (i.e., when $\gr^i F = 0$ all $i < 0$, whence $F(0) \simeq F(-\infty)$). The next lemma summarizes the basic properties of the filtered derived category that we shall use repeatedly, and is well-known.

\begin{lemma} 
\label{DFbasics}
With notation as above:
\begin{enumerate}
\item The collection of functors given by $\{\gr^i(-)\}_{i \in \mathbb{Z}}$ and $F \mapsto F(\infty)$ is conservative on $DF(R)$. On the subcategory $\widehat{DF}(R)$, the collection $\{\gr^i(-)\}_{i \in \mathbb{Z}}$ is already conservative.
\item The inclusion $\widehat{DF}(R) \subset DF(R)$ has a left-adjoint $F \mapsto \widehat{F}$ called completion. Explicitly, this is given by the formula $\widehat{F}(i) = F(i)/F(\infty)$ for all $i$. The completion functor commutes with the associated graded functors $\gr^i(-)$.
\item Both $DF(R)$ and $\widehat{DF}(R)$ have all limits and colimits. The evaluation functors $F \mapsto F(i)$ and the associated graded functors $\gr^i(-)$ commute with all limits and colimits in $DF(R)$. The associated graded functors $\gr^i(-)$ commute with all limits and colimits in $\widehat{DF}(R)$. 
\item There is a unique symmetric monoidal structure on $\widehat{DF}(R)$ compatible with the one on $DF(R)$ under the completion map.
\item For $F,G \in DF(R)$ or $F,G\in \widehat{DF}(R)$, we have a functorial isomorphism $\gr^n(F \dotimes_R G) \simeq \oplus_{i+j=n} \gr^i(F) \dotimes_R \gr^j(G)$.
\end{enumerate}
\end{lemma}

\begin{proof}
See the first 8 pages of \cite{GwilliamDF}.
\end{proof}

The next results of this section are an elaboration of \cite[Appendix A]{BeilinsonDF}. From now on, assume that $R$ is connective. Recall that the $\infty$-category $D(R)$ carries a natural $t$-structure whose connective objects $DF^{\leq 0}(R)$ are those $R$-module spectra whose underlying spectrum is connective (\cite[Proposition 7.1.1.13]{HA}). In the following, we explain why this endows $DF(R)$ with a natural $t$-structure as well. 

\begin{definition}
\label{Beiltstr}
Let $DF^{\leq 0}(R) \subset DF(R)$ be the full subcategory spanned those $F$'s with $\gr^i(F) \in D^{\leq i}(R)$ for all $i$; dually, $DF^{\geq 0}(R) \subset DF(R)$ is the full subcategory spanned by those $F$'s with $F(i) \in D^{\geq i}(R)$ for all $i$.  We shall refer to the pair $(DF^{\leq 0}(R), DF^{\geq 0}(R))$ as the {\em Beilinson $t$-structure} on $DF(R)$; this name is justified by Theorem~\ref{TruncationBeilinson} below.
\end{definition}

The Beilinson $t$-structure is not left-complete: the $\infty$-connected objects of $DF(R)$ (i.e., objects in $\cap_i DF^{\leq -i}(R)$) are exactly those $F$'s with $\gr^i(F) = 0$ for all $i$, i.e., constant diagrams. In particular, no complete objects are $\infty$-connected. The next result summarizes the existence of this $t$-structure and describes the truncation and homology functors. Note that it is a statement about the homotopy category of $DF(R)$, i.e.~the usual filtered derived category as a triangulated category.

\begin{theorem}[Beilinson] 
\label{TruncationBeilinson}
With notation as above.
\begin{enumerate}
\item The Beilinson $t$-structure $(DF^{\leq 0}(R), DF^{\geq 0}(R))$ is a $t$-structure on $DF(R)$. This $t$-structure is compatible with the symmetric monoidal structure, i.e., $DF^{\leq 0}(R) \subset DF(R)$ is a symmetric monoidal subcategory.
\item If $\tau^{\leq 0}_B$ denotes the connective cover functor for the $t$-structure from (1), then there is a natural isomorphism $\gr^i \circ \tau^{\leq 0}_B(-) \simeq \tau^{\leq i} \circ \gr^i(-)$. 
\item Assume $R$ is discrete, i.e.~$\pi_i R=0$ for $i\neq 0$. The heart $DF(R)^\heartsuit := DF^{\leq 0}(R) \cap DF^{\geq 0}(R)$ is equivalent to the abelian category $\mathrm{Ch}(R)$ of chain complexes of $R$-modules via the following recipe: given $F \in DF(R)$, its $0$-th cohomology $H^0_B(F)$ in the Beilinson $t$-structure corresponds to the chain complex $(H^\bullet(\gr^\bullet(F)), d)$ where $d$ is induced as the boundary map for the standard triangle
\[ \gr^{i+1}(F) := F(i+1)/F(i+2) \to F(i)/F(i+2) \to \gr^i(F) := F(i)/F(i+1)\]
by shifting. The resulting functor $H^0_B:DF^{\leq 0}(R) \to DF(R)^\heartsuit \simeq \mathrm{Ch}(R)$ is symmetric monoidal with respect to the standard symmetric monoidal structure on the category of chain complexes.
\end{enumerate}
\end{theorem}

\begin{remark}
\label{rmk:FiltDec}
At the level of explicit filtered complexes, the formation of connective covers in the Beilinson $t$-structure is implemented by Deligne's construction of the filtration d\'ecal\'ee for any filtered complex (see \cite[\S 1.3.3]{DeligneHodgeII}). Thus, even though the language of $t$-structures was invented later, \cite{DeligneHodgeII} already contained an essential idea of the proof of Theorem~\ref{TruncationBeilinson}.
\end{remark}

\begin{proof}
(1) Let us explain why we get a $t$-structure. As each $D^{\leq i}(R)  \subset D(R)$ is stable under colimits, and because each $\gr^i(-)$ commutes with colimits, $DF^{\leq 0}(R) \subset DF(R)$ is also closed under colimits. Thus, by presentability, there is a functor $R:DF(R) \to DF^{\leq 0}(R)$ that is right adjoint to the inclusion. For any $Y \in DF(R)$, this gives an exact triangle
\[ R(Y) \to Y \to Q(Y)\]
defining $Q(Y)$. We must check that $Q(Y) \in DF^{> 0}(R)$, i.e., $Q(Y)(i) \in D^{> i}(R)$ or equivalently that $\mathrm{Map}(X, Q(Y)(i)) = 0$ if $X \in D^{\leq i}(R)$. The functor $F \mapsto F(i)$ has a left-adjoint $L_i$ such that $L_i(X)(j)$ equals $0$ if $j > i$ and equals $X$ if $j \leq i$ (with all transition maps being the identity). In particular, we have $\gr^i(L_i(X)) = X$ and $\gr^j(L_i(X)) = 0$ for $j \neq i$. Thus, if $X \in D^{\leq i}$, then $L_i(X) \in DF^{\leq 0}(R)$. By adjointness, we have an identification $\mathrm{Map}_{DF(R)}(L_i(X), Q(Y)) = \mathrm{Map}_{D(R)}(X, Q(Y)(i))$, so it is enough to show that each map $\eta:L_i(X) \to Q(Y)$ is nullhomotopic if $X \in D^{\leq i}(R)$. Pulling back the preceding fiber sequence along $\eta$ gives a map of fiber sequences
\[ \xymatrix{ R(Y) \ar[r] \ar@{=}[d] & F \ar[r] \ar[d] & L_i(X) \ar[d]^-{\eta} \\
		 R(Y) \ar[r] & Y \ar[r] & Q(Y). }\]
As $X \in D^{\leq i}(R)$, we have $L_i(X) \in DF^{\leq 0}(R)$. Also, $R(Y) \in DF^{\leq 0}(R)$ by construction. Stability of $DF^{\leq 0}(R)$ under extensions shows that $F \in DF^{\leq 0}(R)$. But then by the defining property of $R(Y) \to Y$, the map $F \to Y$ above factors uniquely over $R(Y) \to Y$. As the left vertical map is identity, this implies that $F$ splits uniquely as $R(Y) \oplus L_i(X)$, and thus the first fiber sequence above is split (i.e., has $0$ boundary map). On the other hand, since $F \to Y$ factors over $R(Y)$, it follows that $\eta$ factors over the boundary $L_i(X) \to R(Y)[1]$; as we just explained that the latter is $0$, we must also have $\eta = 0$, as wanted.

The assertion about symmetric monoidal structures follows from Lemma~\ref{DFbasics} (5).\vspace{0.1in}

(2) We shall use the following fact: any exact and $t$-exact functor between stable $\infty$-categories equipped with $t$-structures commutes with the truncation functors associated to the $t$-structures.  Now for each $i \in \mathbb{Z}$, by definition of the $t$-structure, the exact functor $\gr^i:DF(R) \to D(R)$ is $t$-exact if $DF(R)$ is equipped with the Beilinson $t$-structure $(DF^{\leq 0}(R), DF^{\geq 0}(R))$ and $D(R)$ is equipped with the shift $(D^{\leq i}(R), D^{\geq i}(R))$ of the usual $t$-structure. The desired formula now follows immediately from the previous quoted fact about stable $\infty$-categories with $t$-structures.\vspace{0.1in}

(3) The heart comprises those $F$ with $\gr^i(F) \in D^{\leq i}(R)$ and $F(i) \in D^{\geq i}(R)$. It is easy to see that this forces the following:
\begin{enumerate}
\item[(a)] $\gr^i(F)$ is concentrated in cohomological degree $i$.
\item[(b)] $F$ is complete. 
\end{enumerate}
Conversely, any $F$ satisfying these conditions necessarily lies in the heart: it is clear that $F \in DF^{\leq 0}(R)$ by (a), and the inclusion $F \in DF^{\geq 0}(R)$ follows from the formula $F(i) = \lim_{j \geq i} F(i)/F(j)$ (by (b)), the hypothesis that $\gr^j(F) \in D^{\geq i}(R)$ for $j \geq i$ (by (a)), and the stability of $D^{\geq i}(R) \subset D(R)$ under limits. In particular, there is a natural functor $G:\mathrm{Ch}(R) \to DF(R)^{\heartsuit}$ given by $G(K^\bullet)(i) = K^{\geq i}$ and obvious transition maps; this functor is exact. We shall check that $G$ is fully faithful and essentially surjective by first handling the bounded case, then the bounded above case (by passage to filtered direct limits along the stupid truncation), and then the general case (by passage to cofiltered inverse limits along the stupid truncation).

Let us first check the result in the bounded case. Write $\mathrm{Ch}^b(R) \subset \mathrm{Ch}(R)$ for the full subcategory of bounded chain complexes; this is an abelian subcategory. Similarly, write $DF(R)^{\heartsuit,b} \subset DF(R)^{\heartsuit}$ for the full subcategory spanned by bounded filtrations, i.e., those $F$'s with $\gr^i(F) = 0$ for $|i| \gg 0$. It is clear that $G$ restricts to a functor $G^b:\mathrm{Ch}^b(R) \to DF(R)^{\heartsuit,b}$. It is proven in \cite[Proposition 3.1.8]{BBD} (see also \cite[Proposition A.5]{BeilinsonDF}) that $G^b$ is an equivalence. As the definitions in \cite{BBD} and here are not obviously the same, we briefly sketch a proof.  Note that every $K^\bullet \in \mathrm{Ch}^b(R)$ admits a functorial finite filtration with graded pieces of the form $M[-i]$, where $M$ is an $R$-module, $i$ is an integer, and as usual $M[-i]$ indicates the $R$-complex given by $M$ concentrated in cohomological degree $i$. Similarly, any $F \in DF(R)^{\heartsuit, b}$ admits a functorial finite filtration with graded pieces of the form $L_i(M[-i])$, where $L_i$ is the functor from (1), $M$ is an $R$-module, and $i$ is an integer. Moreover, these pieces match up: for an $R$-module and an integer $i$, we have $G(M[-i]) = L_i(M[-i])$, as one readily checks by unwinding definitions. By Lemma~\ref{AbCatCriterion}, it is enough to show the following: for $R$-modules $M$ and $N$ and integers $i$ and $j$, the functor $G$ induces isomorphisms 
\[ \mathrm{Ext}^a_{\mathrm{Ch}(R)}(M[-i], N[-j]) \cong \mathrm{Ext}^a_{DF(R)}(L_i(M[-i]), L_j(N[-j])).\]
Using the definition of $L_i$ as a left-adjoint as well as the explicit definition of $L_j$, one computes that $\mathrm{Ext}^a_{DF(R)}(L_i(M[-i]), L_j(N[-j]))$ vanishes if $i > j$ and equals $\mathrm{Ext}^{a-i+j}_R(M,N)$ if $i \leq j$. On the other hand, by twisting, the left side above identifies with $\mathrm{Ext}^a_{\mathrm{Ch}(R)}(M, N[i-j])$. The claim now follows from Proposition~\ref{ExtCh} applied with $c = i-j$.

Let us now extend the result to complexes that are bounded above. Let $\mathrm{Ch}^{-}(R) \subset \mathrm{Ch}(R)$ be the full subcategory of bounded above complexes $K^\bullet$  (i.e., $K^i = 0$ for $i \gg 0$). Any such $K^\bullet$ can be written functorially as the filtered colimit $\colim_i K^{\geq -i}$ of bounded complexes. Here $K^{\geq -i} \to K^\bullet$ is the displayed truncation of $K^\bullet$, and can be viewed as the universal object in $\mathrm{Ch}(R)$ mapping to $K$ which vanishes in degrees $< -i$.   Similarly, write $DF(R)^{\heartsuit,-} \subset DF(R)^{\heartsuit}$ for the full subcategory spanned by those $F$ which are bounded above (i.e., $\gr^i(F) = 0$ for $i \gg 0$). Any such $F$ can be written functorially as the filtered colimit $\colim_i F^{\geq -i}$ of bounded filtrations. Here $F^{\geq -i} \in DF(R)^{\heartsuit,b}$ is defined by $F^{\geq -i}(j) = F(j)$ if $j \geq -i$ and $F(j) = F(-i)$ if $j \leq -i$, and has a similarly universal property to the one for $K^{\geq -i}$. One then checks by reduction to the bounded case (and using that $G:\mathrm{Ch}(R) \to DF(R)^{\heartsuit}$ commutes with filtered colimits) that $G$ induces an equivalence $\mathrm{Ch}^{-}(R) \simeq DF(R)^{\heartsuit,-}$ on bounded above objects.

Finally, we handle the general case. Any $K^\bullet \in \mathrm{Ch}(R)$ can be written functorially as the $\mathbb{N}$-indexed inverse limit $\lim_i K^{\leq i}$ of bounded above complexes; here $K^\bullet \to K^{\leq i}$ is the displayed truncation of $K$, and is the universal map from $K^\bullet$ into a complex that vanishes in degrees $> i$. Note that the $\mathbb{N}$-indexed diagram $\{K^{\leq i}\}$ is essentially constant in each degree $j$. Similarly, any $F \in DF(R)^{\heartsuit}$ can be written as the $\mathbb{N}$-indexed inverse limit $\lim_i F^{\leq i}$ of bounded above filtrations. Here $F^{\leq i}$ is defined by $F^{\leq i}(j)$ is $0$ if $j > i$ and $F^{\leq i}(j) = F(j)/F(i+1)$ for $j \leq i$, and the map $F \to F^{\leq i}$ is the universal map from $F$ into an object $G$ of ${DF}(R)^{\heartsuit}$ with $\gr^j(G) = 0$ for $j > i$. Note that the $\mathbb{N}$-indexed diagram $\{F^{\leq i}\}$ is essentially constant on applying $\gr^j$ for each $j$. One then checks by reduction to the bounded above case (and using that $G$ carries the $\mathbb{N}$-indexed limit diagrams in $\mathrm{Ch}(R)$ which are essentially constant in each degree to $\mathbb{N}$-indexed limit diagrams in $DF(R)^{\heartsuit}$ that are essentially constant after applying each $\gr^j$) that $G$ induces an equivalence $\mathrm{Ch}(R) \simeq DF(R)^{\heartsuit}$ of abelian categories.

The final statement follows from Lemma~\ref{DFbasics} (5).
\end{proof}

The proof above used the following description of $\mathrm{Ext}$-groups in the abelian category of chain complexes of $R$-modules.

\begin{proposition}
\label{ExtCh}
Let $R$ be a commutative ring. For an integer $c$, write $K^\bullet \mapsto K[c]^\bullet$ for the ``shift to the left by $c$'' autoequivalence of the abelian category $\mathrm{Ch}(R)$ of chain complexes, i.e., $K[c]^i = K^{i+c}$. Then for $R$-modules $M$ and $N$ regarded as complexes with trivial differential, $\mathrm{Ext}^i_{\mathrm{Ch}(R)}(M, N[c])=0$ for all $i\in \mathbb Z$ if $c > 0$, and identifies with $\mathrm{Ext}^{i-c}_R(M,N)$ if $c \leq 0$. 
\end{proposition}

\begin{proof}
We work in the abelian category of $\mathbb{Z}$-graded $R$-modules. For a graded $R$-module $K^\bullet$, write $K^\bullet\{c\}$ for the ``shift to the left by $c$'' autoequivalence of graded $R$-modules, i.e., $(K^\bullet\{c\})^i = K^{i+c}$.  Write $S$ for the graded ring $R[\epsilon]/(\epsilon^2)$ where $\epsilon$ has degree $1$, so $S = R \oplus R\{-1\}$ as a graded $R$-module. Then $\mathrm{Ch}(R)$ can be thought of as the abelian category of graded $S$-modules in the abelian category of graded $R$-modules: restriction of scalars along $R \to S$ gives the underlying graded $R$-module, while the action of $\epsilon \in R$ yields the differential. Under this correspondence, the twisting notations are compatible. Thus, we must compute $\mathrm{Ext}^i_{S,gr}(M, N\{c\})$ for $R$-modules $M$ and $N$ (regarded as graded $S$-modules placed in degree $0$ with $\epsilon$ acting as $0$). We shall use the standard infinite resolution
\[ \Big(... \to S\{-i\} \otimes_R M \to S\{-(i-1)\} \otimes_R M \to .. \to S\{-1\} \otimes_R M \to S \otimes_R M\Big) \stackrel{\sim}{\to} M \]
of graded $S$-modules, where all the transition maps are induced by multiplication by $\epsilon$. Applying $\mathrm{RHom}_{S,gr}(-,N)$ to this resolution and noting that $\mathrm{RHom}_R(M,N\{i\}) = 0$ for $i \neq 0$ for grading reasons, we learn that $\mathrm{RHom}_{S,gr}(M,N\{c\})$ vanishes if $c > 0$, and equals $\mathrm{RHom}_R(M,N)[c]$ if $c \leq 0$, as wanted.
\end{proof}

The following lemma was also used above.

\begin{lemma}
\label{AbCatCriterion}
Let $G:\mathcal{A} \to \mathcal{B}$ be an exact functor between abelian categories. Assume that there exists a collection $S \subset \mathcal{A}$ of objects of $\mathcal{A}$ with the following properties:
\begin{enumerate}
\item Each object of $\mathcal{A}$ admits a finite filtration with graded pieces in $S$.
\item Each object of $\mathcal{B}$ admits a finite filtration with graded pieces in $G(S)$.
\item For $X,Y \in S$, the functor $G$ induces bijections $\mathrm{Ext}_{\mathcal{A}}^*(X,Y) \cong \mathrm{Ext}^*_{\mathcal{B}}(G(X), G(Y))$.
\end{enumerate}
Then $G$ is an equivalence.
\end{lemma}
\begin{proof}
Let us first show that if $X \in S$, then $\mathrm{Ext}^*_{\mathcal{A}}(X,Z) \cong \mathrm{Ext}_{\mathcal{B}}^*(G(X), G(Z))$ for all $Z \in \mathcal{A}$. This holds true for $Z \in S$ by assumption. Applying (1) and the $5$-lemma using (3) then implies the claim for all $Z$. Next, holding $Z$ fixed but letting $X$ vary through all of $\mathcal{A}$ and repeating the previous argument gives $\mathrm{Ext}^*_{\mathcal{A}}(W,Z) \cong \mathrm{Ext}_{\mathcal{B}}^*(G(W), G(Z))$ for all $W,Z \in \mathcal{A}$. In particular, we have shown full faithfulness. Essential surjectivity follows from (2) by induction on the length of the filtration using the statement about $\mathrm{Ext}^1$-groups just proven to facilitate the induction.
\end{proof}

Theorem~\ref{TruncationBeilinson} (3) is somewhat surprising at first glance: it extracts an honest chain complex out from a construction involving derived categories, thus implementing a ``strictification'' procedure. Another such construction is the Berthelot-Ogus-Deligne $L\eta$-functor that played a central role in \cite{BMS} (see Proposition 6.12 in {\em op.cit.} for an explicit example of the ``strictification'' implemented by $L\eta$). We now explain why the latter is a special case of the former by explaining a description of the $L\eta$-functor in terms of filtered derived categories. This result is crucial to the sequel and will be used in particular in Corollary \ref{NygaardSmooth}.

\begin{proposition}
\label{LetaDF}
Let $R$ be a ring, and let $I \subset R$ be an ideal defining a Cartier divisor. Fix $K\in D(R)$. Let $I^\f \otimes K\in DF(R)$ be the $I$-adic filtration on $K$, i.e., the $i$-th level of the filtration is $I^i \otimes_R K$ with obvious maps. Then $L\eta_I K$ identifies with the $R$-complex underlying $\tau^{\leq 0}_B (I^\f \otimes K)$.
\end{proposition}

The reader familiar with \cite{DeligneHodgeII} will have no difficulty deducing Proposition~\ref{LetaDF} from Remark~\ref{rmk:FiltDec}: the object $\eta_I K^\bullet$ defined below coincides with $\mathrm{Dec}(F)^0(K^\bullet)$ in the notation of \cite[\S 1.3.3]{DeligneHodgeII}, where $F$ denotes the $I$-adic filtration on $K$.

\begin{proof}
Choose a complex $K^\bullet$ representing $K$ such that each $K^i$ is $I$-torsion-free. Then we have an evident filtered complex $(I^\f K^\bullet)$ representing $I^\f \otimes K \in DF(R)$. By definition, $\eta_I K^\bullet\subset K^\bullet[1/I]$ is the subcomplex with $(\eta_I K^\bullet)^n = \{x\in I^n K^n\mid dx\in I^{n+1} K^{n+1}\}$.

Define a filtration $G^{\f,\bullet}$ on $\eta_I K^\bullet$ via $G^{i,\bullet} = I^i K^\bullet \cap \eta_I K^\bullet$ as subcomplexes of $K^\bullet[1/I]$. Then there is an evident inclusion
\[ \widetilde{\beta}: G^{\f,\bullet} \to I^\f K^\bullet\]
of filtered complexes, and hence a map 
\[ \beta: G^\f \to I^\f K\]
in $DF(R)$. We shall check that this map is a connective cover map for the Beilinson-$t$-structure, which will prove the proposition.

To check this, we need to check that $\gr^i \beta: \gr^i G^\f\to \gr^i I^\f K$ identifies the source with $\tau^{\leq i}$ of the target, and that $\beta(\infty)$ is an equivalence. Note that $\beta(\infty)$ is an equivalence as in any given degree $n$, the map $G^{i,\bullet}\to I^i K^\bullet$ is an isomorphism in degrees $i>n$.

On the other hand, by construction the map of complexes $\gr^i G^{\f,\bullet}\to \gr^i I^\f K^\bullet$ is injective, and the inclusions $I^{n+1} K^n\subset (\eta_I K^\bullet)^n\subset I^n K^n$ imply that is an isomorphism for $n<i$ and the left-hand side is zero for $n>i$. It remains to see that in degree $i$, the image is precisely the set of cocycles. But this follows from the exact definition of $\eta_I K^\bullet$.
\end{proof}

\begin{remark}
The interpretation of $L\eta_I$ coming from Proposition~\ref{LetaDF} gives a concrete measure of the failure of $L\eta_I$ to preserve exact triangles: if $K \to L \to M$ is an exact triangle in $D(R)$, then the induced sequence on applying $L\eta_I$ is an exact triangle if the boundary map $H^0_B(M) \to H^1_B(K)$ is the $0$ map. Via Theorem~\ref{TruncationBeilinson} (3), the latter is equivalent to requiring that the boundary map $H^i(M \dotimes_R R/I) \to H^{i+1}(K \dotimes_R R/I)$ be the zero map for all $i$.
\end{remark}

\begin{corollary}
\label{Letasymmmon}
With notation from Proposition~\ref{LetaDF}, the functor $L\eta_I: D(R)\to D(R)$ of $\infty$-categories has a natural structure as a lax symmetric monoidal functor. In particular, it takes $E_\infty$-$R$-algebras to $E_\infty$-$R$-algebras.
\end{corollary}

\begin{proof} By the previous proposition, the functor $L\eta_I$ can be written as a composite of the following three functors:
\begin{enumerate}
\item The functor $K\mapsto I^\f\dotimes_R K: D(R)\to DF(R)$.
\item The connective cover functor $\tau_B^{\leq 0}: DF(R)\to DF(R)$.
\item The functor $F\mapsto F(\infty): DF(R)\to D(R)$.
\end{enumerate}
It is a general fact that the connective cover functor is lax symmetric monoidal. In fact, a right adjoint to a symmetric monoidal functor is always lax symmetric monoidal by \cite[Corollary 7.3.2.7]{HA}, and $\tau_B^{\leq 0}: DF(R)\to DF^{\leq 0}(R)$ is right adjoint to the symmetric monoidal inclusion $DF^{\leq 0}(R)\subset DF(R)$.

The functor $F\mapsto F(-\infty): DF(R)\to D(R)$ is symmetric monoidal; for this, note that
\[\begin{aligned}
(F\dotimes_R G)(-\infty) &= \colim_{i\to -\infty} \colim_{j+k\geq i} F(j)\dotimes_R G(k) = \colim_{j,k\to -\infty} F(j)\dotimes_R G(k)\\
&= \colim_{j\to-\infty} F(j)\dotimes_R \colim_{k\to-\infty} G(k) = F(-\infty)\dotimes_R G(-\infty)\ .
\end{aligned}\]
Finally, the functor $K\mapsto I^\f\dotimes_R K$ can be written as the composite of the symmetric monoidal functor $K\mapsto L_0(K)$ (from the proof of Theorem \ref{TruncationBeilinson}) and the functor $F\mapsto I^\f\dotimes_R F$ that is lax symmetric monoidal as $I^\f\in DF(R)$ has a natural structure as $E_\infty$-algebra in $DF(R)$. In fact, $I^\f$ has a strict commutative ring structure on the level of filtered $R$-modules (thus, of filtered chain complexes).
\end{proof}

\subsection{De Rham complexes and negative cyclic homology}\label{subsection_de_rham_neg}

Now we return to the quasisyntomic site, and fix a base ring $R\in \Qs$. Our goal in this section is to prove Theorem~\ref{thm:main7} relating negative cyclic homology to de~Rham cohomology.

\begin{example}[Hodge-completed derived de Rham complex]
\label{HCDDR}
Consider the $\widehat{DF}(R)$-valued pre\-sheaf on $\Qs_{R}^\sub{op}$ determined by the $p$-adic completion $(\widehat{L\Omega}_{-/R}, \widehat{L\Omega}_{-/R}^{\geq \f})$ of the Hodge-completed derived de~Rham complex. We claim that this is a sheaf. By closure of the sheaf property under limits and the behaviour of limits in $\widehat{DF}(R)$, we are reduced to checking that $A \mapsto (\wedge^i_A L_{A/R})^\wedge_p$ is a sheaf on $\Qs_R^\sub{op}$ for all $i$, which follows from Theorem~\ref{FlatDescentCC}.
\end{example}

\begin{example}[$p$-completed derived de Rham complex]
\label{pCDDR} Consider the $D(R)$-valued presheaf on $\qs_{R}^\sub{op}$ determined by the $p$-adic completion of the derived de~Rham complex; we will simply denote this as $L\Omega_{-/R}$ and leave the $p$-adic completion implicit. We claim that this is a sheaf. It is enough to check that $A \mapsto L\Omega_{A/R} \dotimes_{\mathbb Z} \mathbb Z/p\mathbb Z$ is a sheaf. The conjugate filtration on derived de Rham cohomology modulo $p$ endows $L\Omega_{A/R}\dotimes_{\mathbb Z} \mathbb Z/p\mathbb Z$ with a functorial increasing exhaustive $\mathbb{N}$-indexed filtration with graded pieces $\wedge^i_A L_{A/R}[-i] \dotimes_{\mathbb{Z}} \mathbb{Z}/p\mathbb Z$. As any $A\in \qs_{R}$ is quasisyntomic over $R$, each graded piece, and hence each finite level of the filtration, takes values in $D^{\geq -1}(R)$. The claim now follows as sheaves valued in $D^{\geq -1}(R)$ are closed under filtered colimits in the corresponding presheaf category.

If $R$ is perfectoid or $R=\mathbb Z_p$, the same discussion applies to the larger site $\Qs_R^\sub{op}$, using Lemma~\ref{QSynCC}.
\end{example}

The next example is a toy example of the key construction of this paper:

\begin{example}[Recovering Hodge- and $p$-completed derived de Rham complex from $\HC^{-}$]
\label{HCddR}\

\noindent We shall need the following result describing the Hochschild homology of quasiregular semiperfectoid $R$-algebras:

\begin{lemma} 
\label{HKRSperf}
Fix $S \in \qsp_{R}$; if $R$ is perfectoid or $R=\mathbb Z_p$, we allow more generally $S\in \Qsp_R$.
\begin{enumerate}
\item For each $i \geq 0$, the $S$-complex $\wedge^i_S L_{S/R}[-i]$ is $p$-completely flat, and in particular its $p$-completion is concentrated in degree $0$.
\item We have $\pi_{odd} \HH(S/R;\mathbb Z_p) = 0$, and there is multiplicative identification of $\pi_{2i} \HH(S/R;\mathbb Z_p)$ with the $p$-completion of $\wedge^i L_{S/R}[-i]$.
\end{enumerate}
\end{lemma}

\begin{proof} It suffices to prove part (1): The HKR filtration then implies (2) as the $p$-completion of $\gr^i_\sub{HKR} \HH(S/R) \simeq \wedge^i_S L_{S/R}[i]$ is concentrated in degree $2i$ by part (1). Moreover, by stability of $p$-completely flat modules under divided powers, it suffices to handle the case $i=1$.

If $R=\mathbb Z_p$ and $S\in \Qsp_R$, then $L_{S/R}[-1]$ is $p$-completely flat by Remark~\ref{rem:qsptor}. If $R$ is perfectoid and $S\in \Qsp_R$, we use Lemma~\ref{QspAlternative} for the same conclusion. Finally, if $R\in \Qs$ is general and $S\in \qsp_{R}$, then $L_{S/R}$ has $p$-complete Tor-amplitude in $[-1,0]$ but $\Omega^1_{S/R}=0$, so $L_{S/R}[-1]$ is $p$-completely flat.
\end{proof}

In particular, for $S \in \qsp_{R}$, as $\pi_* \HH(S/R;\mathbb Z_p)$ lives only in even degrees, the homotopy fixed point spectral sequence calculating $\HC^{-}(S/R;\mathbb Z_p)$ degenerates to yield a complete descending multiplicative filtration on $\pi_0 \HC^{-}(S/R;\mathbb Z_p)$ with the $i$-th graded piece being the $p$-completion of $\wedge^i_S L_{S/R}[-i]$. By the same reasoning used in Example~\ref{HCDDR}, it follows that $\pi_0 \HC^{-}(-/R)$ is a $D(R)$-valued sheaf on $\qsp_{R}$, and thus unfolds to a sheaf $(\pi_0 \HC^{-}(-/R;\mathbb Z_p))^{\beth}$ on $\qs_{R}$ by Proposition~\ref{qsqspextend}. Again, if $R$ is perfectoid or $R=\mathbb Z_p$, the dicussion applies also to $\Qs_R$.

\begin{proposition}\label{prop:hcddr}
The sheaf $(\pi_0 \HC^{-}(-/R;\mathbb Z_p))^{\beth}$ on $\qs_{R}$ is canonically identified with the $p$-adic completion $\widehat{L\Omega}_{-/R}$ of the Hodge-completed derived de Rham complex from Example~\ref{HCDDR}.
\end{proposition}

\begin{proof}
It is convenient to use filtrations. Thus, for $S \in \qsp_{R}$, viewing $\pi_0 \HC^{-}(S/R;\mathbb Z_p)$ with the filtration defined via homotopy fixed point spectral sequence as above gives a $\widehat{DF}(R)$-valued sheaf $F$ on $\qsp_{R}$. This unfolds to a $\widehat{DF}(R)$-valued sheaf $F^\beth$ on $\qs_{R}$; the underlying sheaf of complexes coincides with the sheaf $(\pi_0 \HC^{-}(-/R;\mathbb Z_p))^{\beth}$ of interest.  In the paragraph above, for a quasismooth $R$-algebra $A$, we have identified $\gr^i(F^\beth)(A)$ with the $p$-adic completion of $\Omega^i_{A/R}[-i]$; as $R$ has bounded $p^\infty$-torsion, so does $A$ by Lemma~\ref{QSProperties}, and hence this graded piece is concentrated in cohomological degree $i$ by quasismoothness and Lemma~\ref{BoundedTorsionFlat}. In particular, $F^\beth(A) \in \widehat{DF}(R)^{\heartsuit}$. As the equivalence in Proposition~\ref{qsqspextend} is symmetric monoidal, Theorem \ref{TruncationBeilinson}(3) tells us that $F^\beth(A)$ is given by a commutative differential graded $R$-algebra of the form
\[ A \to ({\Omega}^1_{A/R})^\wedge_p \to ({\Omega}^2_{A/R})^\wedge_p \to \cdots\]
By checking in the example of $A = \widehat{R[x]}$, one concludes that the differential coincides with the de Rham differential (see \cite[Lemma IV.4.7]{NikolausScholze} for a similar calculation). In other words, $F^\beth$ coincides with $\widehat{L\Omega}_{-/R}$ on the category of quasismooth $R$-algebras. Hence, their left Kan extensions (as functors to $\widehat{DF}(R)$) to all $p$-complete simplicial commutative $R$-algebras also coincide. But these extensions, when restricted to $\qs_{R}$, agree with the original functors as the same holds true for the associated graded functors of either functor (as they are given by the $p$-completions of $\wedge^i L_{-/R}[-i]$). The result follows.
\end{proof}

In fact, we can now prove Theorem~\ref{thm:main7}.

\begin{proof}[Proof of Theorem~\ref{thm:main7}] We analyze the sheaf $(\pi_{2n} \HC^-(-/R;\mathbb Z_p))^\beth$ on $\qs_{R}$ for any $n\in \mathbb Z$. For $n\leq 0$, periodicity (given by multiplication by the generator of $\pi_{2n} \HC^-(R/R;\mathbb Z_p)\simeq R$) shows that it gets identified with $(\pi_0 \HC^-(-/R;\mathbb Z_p))^\beth$, as desired. For $n>0$, the analysis of the previous proof shows that on quasismooth $R$-algebras $A$, it is given by a complex
\[
 ({\Omega}^n_{A/R})^\wedge_p \to ({\Omega}^{n+1}_{A/R})^\wedge_p \to \ldots\ ,
\]
and one can identify this as a subcomplex of the complex for $(\pi_0 \HC^-(-/R;\mathbb Z_p))^\beth$ via multiplication by the generator of $\pi_{-2n} \HC^-(R/R;\mathbb Z_p)\cong R$. By left Kan extension, we get this description in general.

It remains to see that the filtration is complete and exhaustive. Completeness can be checked locally on $\qs_{R}$, and is evident on $\qsp_{R}$ as the Postnikov filtration is complete. To see that it is exhaustive, we note that on any homotopy group $\pi_i \Fil^n \HC^-(A/R;\mathbb Z_p)$, the filtration is eventually constant and equal to $\pi_i \HC^-(A/R;\mathbb Z_p)$; indeed, it suffices to take $n$ sufficiently negative so that $i\geq 2n$.

The case of $\HP$ is similar, but easier by $2$-periodicity.
\end{proof}
\end{example}

\newpage

\section{$\THH$ over perfectoid base rings}
\label{sec:TCperfectoid}

The first goal in this section, realized in \S \ref{subsec:THHTCTPPerfectoid}, is to analyze the theories $\THH$, $\TC^-$ and $\TP$ for perfectoid rings, and in particular prove Theorem~\ref{thm:TCperfectoid}. In fact, with little extra effort, we can also identify $\THH$ of a smooth algebra over a perfectoid ring in \S \ref{subsec:THHsmoothperfectoid}, which generalizes a result of Hesselholt; the key tool here is Theorem~\ref{TPHPDeformation}, which explains why the topological theory provides a $1$-parameter deformation of the algebraic theory. The formulation of these theorems in explicit terms entails making certain choices; one can formulate the results in a more invariant way in the language of Breuil-Kisin twists, which is discussed in \S \ref{subsec:BKtwist}.

For the rest of this section, fix a perfectoid ring $R$, and set $A_{\inf} = A_{\inf}(R)$ with the map $\theta=\theta_R: A_{\inf}\to R$. 

\subsection{$\THH, \TC^{-}$ and $\TP$ for perfectoid rings}
\label{subsec:THHTCTPPerfectoid}

B\"{o}kstedt calculated $\pi_* \THH(\mathbb{F}_p)$ to be a polynomial ring on a degree $2$ generator \cite{Boekstedt}. Using his result, we can prove the analog for any perfectoid ring:

\begin{theorem} The ring
$\pi_* \THH(R;\mathbb Z_p) = R[u]$ is a polynomial ring, where $u \in \pi_2 \THH(R;\mathbb Z_p)\cong \pi_2\HH(R;\mathbb Z_p)=\ker(\theta)/\ker(\theta)^2$ is a generator of $\ker(\theta)/\ker(\theta)^2$.
\end{theorem}

\begin{proof}
We first claim that $\pi_i\HH(R;\mathbb Z_p)\cong R$ if $i\geq 0$ is even and $=0$ else (however without identifying the multiplicative structure, which would be a divided power algebra). This follows from the HKR filtration, as the graded pieces are given by $(\wedge^i_R L_{R/\mathbb Z_p})^\wedge_p[i]\simeq R[2i]$, cf.~Proposition~\ref{prop:basicperfd}.

In particular, $\HH(R;\mathbb Z_p)$ is a pseudocoherent complex of $R$-modules, i.e.~it can be represented by a complex of finite free $R$-modules that is bounded to the right (but not to the left). Thus, the same is true for
\[
\THH(R;\mathbb Z_p)\otimes_{\THH(\mathbb Z)} \mathbb Z=\HH(R;\mathbb Z_p)\ ,
\]
where we use Lemma \ref{THHvsHH}. By induction using the finiteness in Lemma \ref{THHvsHH}, all $\THH(R;\mathbb Z_p)\otimes_{\THH(\mathbb Z)} \tau_{\leq n} \THH(\mathbb Z)$ are pseudocoherent, which implies that $\THH(R;\mathbb Z_p)$ itself is pseudocoherent.

Next, we check that for any map $R\to R^\prime$ between perfectoid rings the induced map
\[
\THH(R;\mathbb Z_p)\dotimes_R R^\prime\to \THH(R^\prime;\mathbb Z_p)
\]
is an equivalence. It suffices to check the assertion after tensoring over $\THH(\mathbb Z)$ with $\mathbb Z$ (as then by induction it follows for the tensor product over $\THH(\mathbb Z)$ with $\tau_{\leq n} \THH(\mathbb Z)$, and one can pass to the limit). Thus, it suffices to see that
\[
\HH(R;\mathbb Z_p)\dotimes_R R^\prime\to \HH(R^\prime;\mathbb Z_p)
\]
is an equivalence, which the HKR filtration reduces to $(\wedge_R^i L_{R/\mathbb Z_p})^\wedge_p\dotimes_R R^\prime\simeq (\wedge_{R^\prime}^i L_{R^\prime/\mathbb Z_p})^\wedge_p$; but this follows from the description $(L_{R/\mathbb Z_p})^\wedge_p[-1]\simeq (\ker \theta)/(\ker \theta)^2 = R\cdot u$, which is compatible with base change \cite[Lemma 3.14]{BMS}.

We know by B\"okstedt's theorem that the theorem holds true for $R=\mathbb F_p$, cf.~\cite[Theorem IV.4.4]{NikolausScholze}. Thus, the base change property implies that it holds if $R$ is of characteristic $p$.

In general, we argue by induction on $i$, so assume the result is known in degrees $<i$. As then $\tau_{<i} \THH(R;\mathbb Z_p)$ is a perfect complex of $R$-modules, it follows that $\tau_{\geq i} \THH(R;\mathbb Z_p)$ is still pseudocoherent, and in particular $M=\pi_i \THH(R;\mathbb Z_p)$ is a finitely generated $R$-module. Consider the map
\[
M^\prime = \begin{cases} R\cdot u^{i/2} & i\text{ \rm even}\geq 0,\\ 0 & \mathrm{else}\end{cases} \to M=\pi_i \THH(R;\mathbb Z_p)
\]
and let $\overline{R}$ is the direct limit perfection of $R/p$. Then $R\to\overline{R}$ is surjective, the kernel lies in the Jacobson radical, and $\overline{R}$ is a perfect ring of characteristic $p$. By the base change property and freeness of $\pi_j \THH(R;\mathbb Z_p)$ for $j<i$, we see that
\[
M\otimes_R \overline{R} = \pi_i \THH(\overline{R};\mathbb Z_p),
\]
which by the known case of characteristic $p$ is given by $M^\prime\otimes_R \overline{R}=\overline{R}\cdot u^{i/2}$ (if $i$ is even, or $0$ else). Thus, the map
\[
M^\prime\otimes_R \overline{R}\to M\otimes_R \overline{R}
\]
is an isomorphism.

In particular, if $i$ is odd, then $M\otimes_R \overline{R}=0$, which by Nakayama's lemma implies that $M=0$, as desired. If $i$ is even, then Nakayama's lemma implies that $M^\prime\cong R\to M$ is surjective. To see that it is an isomorphism, it suffices to see that the rank of $M$ at all points of $\Spec R$ is at least $1$, as $R$ is reduced. All points of characteristic $p$ lie in $\Spec \overline{R}\subset \Spec R$, and we know that $M\otimes_R \overline{R}\cong \overline{R}$, giving the result in that case. On the other hand, rationally we have
\[
M\otimes_{\mathbb Z} \mathbb Q = \pi_i \THH(R;\mathbb Z_p)\otimes_{\mathbb Z} \mathbb Q = \pi_i \HH(R;\mathbb Z_p)\otimes_{\mathbb Z} \mathbb Q\simeq R\otimes_{\mathbb Z} \mathbb Q
\]
using $\THH(\mathbb Z)\otimes_{\mathbb Z} \mathbb Q = \mathbb Q$ (Lemma \ref{THHvsHH}), showing that the rank of $M$ at characteristic $0$ points is also at least $1$.
\end{proof}

Write $\varphi:\THH(R;\mathbb Z_p) \to \THH(R;\mathbb Z_p)^{tC_p}$ for the cyclotomic Frobenius, and $\varphi^{h\T}:\TC^{-}(R;\mathbb Z_p) \to \TP(R;\mathbb Z_p)\simeq (\THH(R;\mathbb Z_p)^{tC_p})^{h\T}$ for the induced map (using \cite[Lemma II.4.2]{NikolausScholze}). These fit into a commutative diagram
\begin{equation}
\label{TCTPTHH}
 \xymatrix{ \TC^{-}(R;\mathbb Z_p) \ar[r]^-{\varphi^{h\T}} \ar[d]_-{\mathrm{can}} & \TP(R;\mathbb Z_p) \ar[d]^-{\mathrm{can}} \\
		   \THH(R;\mathbb Z_p) \ar[r]_-{\varphi} & \THH(R;\mathbb Z_p)^{tC_p} }
		   \end{equation}
of $E_\infty$-ring spectra, where the vertical maps are the usual ones (cf.~also \cite[Corollary I.4.3]{NikolausScholze} for the right vertical map).

\begin{proposition}
\label{TCTPTHHpi}
The square obtained by applying $\pi_*$ to the square \eqref{TCTPTHH} above is given by 
\[ \xymatrix{ A_{\inf}[u,v]/(uv-\xi) \ar[rrr]_{\varphi-\text{linear}}^-{u \mapsto \sigma,\, v \mapsto \varphi(\xi) \sigma^{-1}} \ar[d]_{\theta-\text{linear}}^-{u \mapsto u,\, v \mapsto 0} & & &  A_{\inf}[\sigma,\sigma^{-1}] \ar[d]_{\widetilde{\theta}-\text{linear}}^-{\sigma \mapsto \sigma} \\
		  R[u] \ar[rrr]_{R-\text{linear}}^-{u \mapsto \sigma} & & & R[\sigma,\sigma^{-1}]. }\]
Here $\xi$ has degree $0$ and is a generator of the ideal $\ker(\theta)$, $u$ and $\sigma$ have degree $2$, while $v$ has degree~$-2$. 
\end{proposition}

Before embarking on the proof, we note that in addition, we also have the canonical map
\begin{equation}
\label{TCTP}
\TC^{-}(R;\mathbb Z_p) \xrightarrow{\mathrm{can}} \TP(R;\mathbb Z_p) . 
\end{equation}

\begin{proposition}
The map on $\pi_*$ obtained from the canonical map in \eqref{TCTP} above is given by
\[ \xymatrix{ A_{\inf}[u,v]/(uv-\xi) \ar[rrr]_{A_{\inf}-\text{linear}}^-{u \mapsto \xi \sigma,\, v \mapsto \sigma^{-1}} & & &  A_{\inf}[\sigma,\sigma^{-1}], }\]
where we use the presentations from \ref{TCTPTHHpi}.
\end{proposition}

More precisely, the statement is that generators $u$, $v$, $\sigma$ and $\xi$ can be chosen such that these descriptions hold true. Now we prove both propositions together.

\begin{proof} First, we identify $F(R) := \pi_0 \TP(R;\mathbb Z_p)$. Note that by the Tate spectral sequence, $\TP(R;\mathbb Z_p)$ is concentrated in even degrees, and $F(R)$ has a multiplicative complete descending filtration $\Fil^i F(R)\subset F(R)$ with $\gr^i F(R)\cong \pi_{2i} \THH(R;\mathbb Z_p)\cong R$ in degrees $i\geq 0$, and $=0$ else. In particular, $F(R)\to \pi_0 \THH(R;\mathbb Z_p) = R$ is a $p$-adically complete pro-nilpotent thickening. By the universal property of $A_{\inf}$ \cite[\S 1.2]{Fontaine1994}, we get a unique map $A_{\inf}\to F(R)$ over $R$. Moreover, this sends the ideal $\ker(\theta)$ into $\Fil^1 F(R) = \ker(F(R)\to R)$, and thus by multiplicativity $\ker(\theta)^i$ into $\Fil^i F(R)$. We claim that this induces a graded isomorphism $A_{\inf}\to F(R)$. For this, we need to check that the maps on $\gr^i$ are isomorphisms, i.e.~certain maps $R\to R$ are isomorphisms. This can be checked after base change to perfect fields of characteristic $p$. As all constructions are functorial in $R$, we can therefore assume that $R=k$ is a perfect field of characteristic $p$. But then there is a map $\mathbb F_p\to k$, and using functoriality again, we are reduced to the case of $\mathbb F_p$, where it follows from \cite[Corollary IV.4.8]{NikolausScholze}.

Moreover, the Tate spectral sequence implies that $\TP(R;\mathbb Z_p)$ is $2$-periodic, so we find an isomorphism $\pi_\ast \TP(R;\mathbb Z_p)\cong A_{\inf}[\sigma,\sigma^{-1}]$ by choosing a generator $\sigma\in \pi_2 \TP(R;\mathbb Z_p)$.

Looking at the homotopy fixed point spectral sequence for $\TC^-(R;\mathbb Z_p)$, which maps to the Tate spectral sequence via the canonical map, we see again that everything is concentrated in even degrees, and that generators in degree $2$ and $-2$ multiply to a generator for $\Fil^1 F(R) = \ker(\theta)\subset F(R) = A_{\inf}$; thus, we can find an isomorphism $\pi_\ast \TC^-(R;\mathbb Z_p)\cong A_{\inf}[u,v]/(uv-\xi)\subset A_{\inf}[\sigma^{\pm 1}]$ under which $v\mapsto \sigma^{-1}$ and $u\mapsto \xi \sigma$ under the canonical map.

Next, we identify
\[
\pi_0 \varphi^{h\T}: \pi_0 \TC^-(R;\mathbb Z_p) = A_{\inf}\to \pi_0 \TP(R;\mathbb Z_p)=A_{\inf}\ .
\]
For this, we look at the commutative diagram
\[\xymatrix{
& \pi_0 \TC^-(R;\mathbb Z_p)\ar[r]\ar[d] & \pi_0\TP(R;\mathbb Z_p)\ar[d]\ar[r] & \pi_0 R^{t\T} = R\ar[d]\\
R\ar@{=}[r] & \pi_0 \THH(R;\mathbb Z_p)\ar[r] & \pi_0\THH(R;\mathbb Z_p)^{tC_p}\ar[r] & \pi_0 R^{tC_p} = R/p\ .
}\]
By \cite[Corollary IV.2.4]{NikolausScholze}, the lower composite is given by the Frobenius map $x\mapsto x^p$. The left vertical map is given by $\theta: A_{\inf}\to R$ by construction. The right upper horizontal map is also $\theta: A_{\inf}\to R$, and the right-most vertical map is the canonical reduction map $R\to R/p$. It follows that the map $f=\pi_0 \varphi^{h\T}$ makes the diagram
\[\xymatrix{
A_{\inf}\ar[r]^f\ar[d]_{\theta} & A_{\inf}\ar[d]^{\theta}\\
R\ar[r]_{\varphi}&  R/p
}\]
commute. As $A_{\inf}$ is the universal $p$-adically complete pro-nilpotent thickening of $R/p$, this shows that $f$ must be the Frobenius map $\varphi$.

Now we claim that the map
\[
\pi_2 \varphi^{h\T}: \pi_2 \TC^-(R;\mathbb Z_p) = A_{\inf}\cdot u\to \pi_2 \TP(R;\mathbb Z_p) = A_{\inf}\cdot \sigma
\]
is an isomorphism. Again, this can be checked after replacing $R$ by a perfect field, and then by $\mathbb F_p$, where it follows from \cite[Proposition IV.4.9]{NikolausScholze}. In particular, $u$ maps to $\alpha\sigma$ for some unit $\alpha\in A_{\inf}$. Replacing $\xi$ by $\varphi^{-1}(\alpha)\xi$, we can then arrange that $u$ maps to $\sigma$ on the nose. By multiplicativity, it follows that $v$ maps to $\varphi(\xi) \sigma^{-1}$.

It remains to get the description of $\THH(R;\mathbb Z_p)^{tC_p}$. We have a commutative diagram of rings
\[ \xymatrix{ A_{\inf}[u,v]/(uv-\xi) \ar[rrr]_{\varphi-\text{linear}}^-{u \mapsto \sigma,\, v \mapsto \varphi(\xi) \sigma^{-1}} \ar[d]_{\theta-\text{linear}}^-{u \mapsto u,\, v \mapsto 0} & & &  A_{\inf}[\sigma,\sigma^{-1}] \ar[d]\\
R[u] \ar[rrr] & & & \pi_\ast \THH(R;\mathbb Z_p)^{tC_p}. }\]
As $v\mapsto 0$ on the left (for degree reasons), this induces a natural map
\[
A_{\inf}[\sigma,\sigma^{-1}]/\varphi(\xi) \sigma^{-1} = R[\sigma,\sigma^{-1}]\to \pi_\ast \THH(R;\mathbb Z_p)^{tC_p}\ ,
\]
which is $\tilde{\theta}$-linear. It remains to see that this is an isomorphism. For this, it is enough to see that the natural map of $E_\infty$-ring spectra
\[
\TP(R;\mathbb Z_p)\otimes_{\TC^-(R;\mathbb Z_p)} \THH(R;\mathbb Z_p)\to \THH(R;\mathbb Z_p)^{tC_p}
\]
is an equivalence, i.e.~\eqref{TCTPTHH} is a pushout. More generally, for any $\T$-equivariant $\THH(R;\mathbb Z_p)$-module $M$ (such as $M=\THH(R;\mathbb Z_p)^{tC_p}$ via $\varphi$), we claim that the natural map
\[
M^{h\T}\otimes_{\TC^-(R;\mathbb Z_p)} \THH(R;\mathbb Z_p)\to M
\]
is an equivalence. For this, we note that $\THH(R;\mathbb Z_p) = \TC^-(R;\mathbb Z_p)/v$ is a perfect $\TC^-(R;\mathbb Z_p)$-module, so both sides commute with limits in $M$. By Postnikov towers, we can therefore assume that $M$ is bounded below, so after shifting coconnective. In that case, both sides commute with filtered colimits in $M$, and we may assume that $M$ is bounded above as well, and then by induction concentrated in degree $0$. But then the result follows from $M^{h\T}/v = M$, which is the first part of \cite[Lemma IV.4.12]{NikolausScholze}.
\end{proof}

We note that the final paragraph of this proof actually implies the following result for $R$-algebras $A$.

\begin{proposition}\label{prop:TCtoTHH} For any connective $E_\infty$-$R$-algebra $A$, the natural maps
\[
\TC^-(A;\mathbb Z_p)/v = \TC^-(A;\mathbb Z_p)\otimes_{\TC^-(R;\mathbb Z_p)} \THH(R;\mathbb Z_p)\to \THH(A;\mathbb Z_p)
\]
and
\[
\TP(A;\mathbb Z_p)/\tilde\xi = \TP(A;\mathbb Z_p)\otimes_{\TP(R;\mathbb Z_p)} \THH(R;\mathbb Z_p)^{tC_p}\to \THH(A;\mathbb Z_p)^{tC_p}
\]
are equivalences of $E_\infty$-ring spectra.
\end{proposition}

In particular, the map $\varphi: \THH(A;\mathbb Z_p)\to \THH(A;\mathbb Z_p)^{tC_p}$ can be recovered from the map $\varphi^{h\T}: \TC^-(A;\mathbb Z_p)\to \TP(A;\mathbb Z_p)$ by modding out by $v$. In traditional approaches to the cyclotomic structure on $\THH$, one would first analyze $\varphi: \THH(A;\mathbb Z_p)\to \THH(A;\mathbb Z_p)^{tC_p}$ by hand; here, we will not do such an analysis but instead identify directly $\varphi^{h\T}$. The present discussion shows that the identification of $\varphi^{h\T}$ then also leads to an identification of $\varphi$.

\begin{proof} This follows by applying the equivalence
\[
M^{h\T}\otimes_{\TC^-(R;\mathbb Z_p)}\THH(R;\mathbb Z_p)\to M
\]
valid for any $\T$-equivariant $\THH(R;\mathbb Z_p)$-module $M$ to $M=\THH(A;\mathbb Z_p)$ resp.~$M=\THH(A;\mathbb Z_p)^{tC_p}$. In the second case, we also use that
\[
\TP(R;\mathbb Z_p)\otimes_{\TC^-(R;\mathbb Z_p)} \THH(R;\mathbb Z_p)\to \THH(R;\mathbb Z_p)^{tC_p}
\]
is an equivalence.
\end{proof}

\subsection{Breuil-Kisin twists}
\label{subsec:BKtwist}

Before going on, we want to make the previous identifications more canonical. Let $A_{\inf}\{1\}:=\pi_2 \TP(R;\mathbb Z_p)$; this is a free $A_{\inf}$-module of rank $1$. For any $A_{\inf}$-module $M$ and $i\in \mathbb Z$, we define the Breuil-Kisin twist $M\{i\}=M\otimes_{A_{\inf}} A_{\inf}\{1\}^{\otimes i}$. If $M$ is an $R$-module, then $M\{i\}$ denotes the corresponding twist when $M$ is considered as $A_{\inf}$-module via $\tilde\theta$.

With these notations, there is a natural isomorphism of graded rings
\[
\pi_\ast \TP(R;\mathbb Z_p) = \bigoplus_{i\in \mathbb Z} A_{\inf}\{i\}
\]
where $A_{\inf}\{i\}$ sits in degree $2i$. The canonical map $\pi_\ast \TC^-(R;\mathbb Z_p)\to \pi_\ast \TP(R;\mathbb Z_p)$ is an isomorphism in negative degrees, and has image $(\ker\theta)^i A_{\inf}\{i\}$ in degree $2i\geq 0$. On the other hand, the Frobenius map $\varphi: \pi_\ast \TC^-(R;\mathbb Z_p)\to \pi_\ast \TP(R;\mathbb Z_p)$ induces on $\pi_{-2}$ a $\varphi$-linear Frobenius endomorphism $\varphi_{A_{\inf}\{-1\}}: A_{\inf}\{-1\}\to A_{\inf}\{-1\}$ that becomes an isomorphism after inverting $\xi$ on the source, respectively $\tilde\xi$ on the target. In particular, we have a map $\varphi_{A_{\inf}\{1\}}: A_{\inf}\{1\}[\tfrac 1\xi]\to A_{\inf}\{1\}[\tfrac 1{\tilde\xi}]$. It sends $\xi A_{\inf}\{1\}$ into $A_{\inf}\{1\}$. Below, we will relate this to the Breuil-Kisin-Fargues twist from \cite[Example 4.24]{BMS}.

Defining the Nygaard filtration on $A_{\inf}$ as $\calN^{\geq i} A_{\inf} = \xi^i A_{\inf}$ for $i\geq 0$ and $\calN^{\geq i} A_{\inf} = A_{\inf}$ for $i\leq 0$, we see that
\[
\pi_\ast \TC^-(R;\mathbb Z_p) = \bigoplus_{i\in \mathbb Z} \calN^{\geq i} A_{\inf}\{i\}\ .
\]
Moreover, if we set $\calN^i A_{\inf} = \calN^{\geq i} A_{\inf} / \calN^{\geq i+1} A_{\inf}\cong R\cdot \xi^i$, then the formula $\THH(A;\mathbb Z_p) = \TC^-(A;\mathbb Z_p)/v$ implies that
\[
\pi_\ast \THH(R;\mathbb Z_p) = \bigoplus_{i\geq 0} \calN^i A_{\inf}\{i\}\ .
\]
But we know that there is a canonical isomorphism $\pi_2 \THH(R;\mathbb Z_p)\cong (\ker \theta)/(\ker \theta)^2\cong \calN^1 A_{\inf}$. It follows that $\calN^1 A_{\inf}\{1\}\simeq \calN^1 A_{\inf}$ canonically, or in other words $A_{\inf}\{1\}\otimes_{A_{\inf},\theta} R\cong R$ canonically. This explains our choice above to define $R\{1\} = A_{\inf}\{1\}\otimes_{A_{\inf},\tilde\theta} R$ as the base change via $\tilde\theta$; the base change via $\theta$ is canonically trivial.

On the other hand, $\THH(R;\mathbb Z_p)^{tC_p} = \TP(R;\mathbb Z_p)/\tilde\xi$, and so
\[
\pi_\ast \THH(R;\mathbb Z_p)^{tC_p} = \bigoplus_{i\in \mathbb Z} (A_{\inf}/\tilde\xi)\{i\} = \bigoplus_{i\in \mathbb Z} R\{i\}\ .
\]

Next, we know that
\[
\varphi: \THH(R;\mathbb Z_p)\to \THH(R;\mathbb Z_p)^{tC_p}
\]
identifies the source with the connective cover of the target. We see that $R\{1\}\cong \calN^1 A_{\inf} = (\ker \theta)/(\ker\theta)^2$ canonically.

Let us summarize the discussion. 

\begin{proposition}\label{prop:breuilkisintwist} Consider the $A_{\inf}$-module $A_{\inf}\{1\} = \pi_2 \TP(R;\mathbb Z_p)$ with the $\varphi$-linear map
\[
\varphi_{A_{\inf}\{1\}} = \pi_2 \varphi: A_{\inf}\{1\}[\tfrac 1\xi]\cong A_{\inf}\{1\}[\tfrac 1{\tilde\xi}]\ ,
\]
which induces an isomorphism $\xi A_{\inf}\{1\}\cong A_{\inf}\{1\}$.
\begin{enumerate}
\item There are natural isomorphisms
\[\begin{aligned}
A_{\inf}\{1\}\otimes_{A_{\inf},\theta} R&\cong R\ ,\\
A_{\inf}\{1\}\otimes_{A_{\inf},\tilde\theta} R&\cong (\ker \theta)/(\ker \theta)^2=R\{1\}\ .
\end{aligned}\]
\item There are natural isomorphisms
\[\begin{aligned}
\pi_\ast \THH(R;\mathbb Z_p) &= \bigoplus_{i\geq 0} R\{i\} = \bigoplus_{i\geq 0} \calN^i A_{\inf}\ ,\\
\pi_\ast \TC^-(R;\mathbb Z_p) &= \bigoplus_{i\in\mathbb Z} \calN^{\geq i} A_{\inf}\{i\}\ ,\\
\pi_\ast \TP(R;\mathbb Z_p)&= \bigoplus_{i\in\mathbb Z} A_{\inf}\{i\}\ ,
\end{aligned}\]
under which the canonical map $\TC^-\to \TP$ corresponds to the inclusion $\calN^{\geq i} A_{\inf}\to A_{\inf}$, and the Frobenius map $\TC^-\to \TP$ corresponds to the Frobenius $A_{\inf}\{i\}[\tfrac1\xi]\to A_{\inf}\{i\}[\tfrac1{\tilde\xi}] $, which sends $\calN^{\geq i} A_{\inf}\{i\}$ into $A_{\inf}\{i\}$.
\end{enumerate}
\end{proposition}

\begin{remark} In this remark, we show that these Breuil-Kisin twists agree with those of \cite[Example 4.24]{BMS}; this is essentially the only spot in the paper that uses ``classical'' results about topological Hochschild homology. We recall the construction of {\it loc.cit.}~which works in the case that $R$ is $p$-torsion-free. Starting from the description $A_{\inf} = \varprojlim_F W_r(R)$ identifying the projection $A_{\inf}\to W_r(R)$ with $\tilde\theta_r: A_{\inf}\to W_r(R)$ as in \cite[Lemma 3.2]{BMS}, where the kernel of $\tilde\theta_r$ is generated by the non-zero-divisor $\tilde\xi_r = \tilde\xi \varphi(\tilde\xi)\cdots \varphi^{r-1}(\tilde\xi)$, one has canonical isomorphisms
\[
A_{\inf}\{1\}\otimes_{A_{\inf},\tilde\theta_r} W_r(R) = (\ker \tilde\theta_r)/(\ker \tilde\theta_r)^2\ .
\]
Varying $r$, the natural maps on the right correspond to $p$ times the natural map on the left. This determines the transition maps when $R$ is $p$-torsion-free, and then $A_{\inf}\{1\}$ as
\[
A_{\inf}\{1\} = \varprojlim_r A_{\inf}\{1\}\otimes_{A_{\inf},\tilde\theta_r} W_r(R)\ .
\]

Coming back to $\THH$, we know that as $\THH(R;\mathbb Z_p)\to \THH(R;\mathbb Z_p)^{tC_p}$ is a connective cover, also the map
\[
\THH(R;\mathbb Z_p)^{hC_{p^r}}\to (\THH(R;\mathbb Z_p)^{tC_p})^{hC_{p^r}}\simeq \THH(R;\mathbb Z_p)^{tC_{p^{r+1}}}
\]
induces an equivalence of connective covers. Moreover, the same input implies that the map from the genuine fixed points
\[
\TR^{r+1}(R;\mathbb Z_p) = \THH(R;\mathbb Z_p)^{C_{p^r}}\to \THH(R;\mathbb Z_p)^{hC_{p^r}}
\]
is again a connective cover by a result of Tsalidis, \cite{Tsalidis}, cf.~also \cite[Corollary II.4.9]{NikolausScholze}. By \cite[Theorem 3.3]{Hesselholt1997}, there is a natural isomorphism $\pi_0 \TR^{r+1}(R;\mathbb Z_p)\cong W_{r+1}(R)$ under which the transition maps for varying $r$ correspond to the Frobenius $F: W_{r+1}(R)\to W_r(R)$. Thus, $\THH(R;\mathbb Z_p)^{tC_{p^{r+1}}}$ is an even $2$-periodic ring spectrum with $\pi_0$ given by $W_{r+1}(R)$. The equivalence
\[
\TP(R;\mathbb Z_p)\simeq \lim_r \THH(R;\mathbb Z_p)^{tC_{p^r}}
\]
from the proof of \cite[Lemma II.4.2]{NikolausScholze} then induces an isomorphism $A_{\inf}\cong \varprojlim_r W_r(R)\cong A_{\inf}$ on the level of $\pi_0$. This must be the identity by compatibility with $\tilde\theta$ and the universal property of $A_{\inf}$. This implies that the map
\[
A_{\inf} = \pi_0 \TP(R;\mathbb Z_p)\to \pi_0 \THH(R;\mathbb Z_p)^{tC_{p^r}} = W_r(R)
\]
is given by $\tilde\theta_r$, and in particular is surjective. As both spectra are $2$-periodic, it follows that
\[
\THH(R;\mathbb Z_p)^{tC_{p^r}} = \TP(R;\mathbb Z_p)/\tilde\xi_r\ ,
\]
and thus
\[
\pi_\ast \THH(R;\mathbb Z_p)^{tC_{p^r}}\cong \bigoplus_{i\in \mathbb Z} W_r(R)\{i\}\ .
\]
On the other hand,
\[
\TP(R;\mathbb Z_p)\simeq (\THH(R;\mathbb Z_p)^{tC_{p^r}})^{h(\T/C_{p^r})}
\]
by \cite[Lemma II.4.1, II.4.2]{NikolausScholze}. Looking at the resulting spectral sequence computing $\pi_0$ (whose abutment filtration is determined by multiplicativity to be given by powers of $\ker \tilde\theta_r$), we see that there is a canonical isomorphism
\[
W_r(R)\{1\}\cong H^2(\T/C_{p^r},W_r(R)\{1\})\cong (\ker \tilde\theta_r)/(\ker \tilde\theta_r)^2\ .
\]
Moreover, the natural transition maps on the left correspond to multiplication by $p$ on the right (the factor of $p$ coming from the covering $\T/C_{p^r}\to \T/C_{p^{r+1}}$). This shows that $A_{\inf}\{1\}$ has the description given in \cite[Example 4.24]{BMS}; we leave it to the reader to check compatibility with the Frobenius map.
\end{remark}

\subsection{$\THH$ for smooth algebras over perfectoid rings}
\label{subsec:THHsmoothperfectoid}

Let $R$ be a perfectoid ring. The following theorem expresses why the topological theory yields a deformation of the algebraic theory; it will be useful in controlling the topological theory. Here and in the following, when a perfectoid base ring is fixed, we will usually omit Breuil-Kisin twists.

\begin{theorem}
\label{TPHPDeformation}
Let $A$ be an $R$-algebra. Then there is a $\T$-equivariant cofiber sequence
\[ \THH(A;\mathbb Z_p)[2] \xrightarrow{u} \THH(A;\mathbb Z_p) \to \HH(A/R;\mathbb Z_p)\]
of $\THH(A;\mathbb Z_p)$-module spectra. In particular, by passage to fixed points, there is an induced cofiber sequence
\[ \TC^{-}(A;\mathbb Z_p)[2] \xrightarrow{u} \TC^{-}(A;\mathbb Z_p) \to \HC^{-}(A/R;\mathbb Z_p)\]
of $\TC^{-}(A;\mathbb Z_p)$-module spectra. Likewise, by passage to the Tate construction, there is an induced cofiber sequence
\[ \TP(A;\mathbb Z_p)[2] \xrightarrow{ \xi \cdot \sigma} \TP(A;\mathbb Z_p) \to \HP(A/R;\mathbb Z_p)\]
of $\TP(A;\mathbb Z_p)$-module spectra.
\end{theorem}

\begin{proof}
As $\HH(A/R) = \THH(A) \otimes_{\THH(R)} R$ by Lemma \ref{THHvsHH}, and $\HH(R/R) = R$, it is enough to prove the first statement for $R$ itself. In this case, note that $u \in \pi_2 \TC^{-}(R;\mathbb Z_p)$ can be viewed a $\T$-equivariant map $\mathbb{A}[2] \to \THH(R;\mathbb Z_p)$, and hence a $\THH(R;\mathbb Z_p)$-linear $\T$-equivariant map  $\THH(R;\mathbb Z_p)[2] \xrightarrow{u} \THH(R;\mathbb Z_p)$. The cofiber of this map is the discrete $\THH(R;\mathbb Z_p)$-module $R$ non-equivariantly, and thus also $\T$-equivariantly: any discrete module over a $\T$-equivariant connective $E_\infty$-ring carries a unique $\T$-action (the trivial one).
\end{proof}

Next, we shall describe $\pi_* \THH(A;\mathbb Z_p)$ for a quasismooth $R$-algebra $A$. First, we give a general construction relating differential forms and $\THH$.

\begin{construction}
\label{cons:antisym}
For any $R$-algebra $A$, we shall construct a natural graded $A$-algebra map
\[ \mu_A:(\Omega^*_{A/R})^\wedge_p \to \pi_* \THH(A; \mathbb{Z}_p) \]
of graded derived $p$-complete $A$-modules; here the left side denotes the graded $A$-module obtained as $H^0$ of the termwise derived $p$-completion of the exterior algebra $\Omega^*_{A/R}$. To see this, observe that we have a natural (often called ``antisymmetrization'') $A$-module map
\[ \Omega^1_{A/\mathbb{Z}} \to \pi_1 \HH(A) \]
for any ring $A$\footnote{The $S^1$-action on $\HH(A)$ endows $\pi_* \HH(A)$ with the structure of a commutative differential graded algebra whose $0$-th term is $A$; the differential is usually called the Connes differential. As $\HH(A)$ can be computed by a simplicial commutative ring, $\pi_* \HH(A)$ is strictly graded commutative (i.e., odd degree elements square to $0$). The universal property of the de Rham complex gives a map $\Omega^*_{A/\mathbb{Z}} \to \pi_* \HH(A)$ carrying the de Rham differential to the Connes differential.}. Now the canonical map $\THH(A) \to \HH(A)$ is an isomorphism on $\tau_{\leq 2}$ \cite[Proposition IV.4.2]{NikolausScholze}. Applying this observation for $\tau_{\leq 1}$ thus gives an $A$-module map
\[ \Omega^1_{A/\mathbb{Z}} \to \pi_1 \THH(A). \]
Applying the observation for $\tau_{\leq 2}$, and using that $\pi_* \HH(A)$ is an anticommutative graded ring, the preceding map extends to a map
\[ \Omega^*_{A/\mathbb{Z}} \to \pi_* \THH(A)\]
of graded $A$-algebras. Composing with $p$-completions gives a graded $A$-algebra map
\[ \Omega^*_{A/\mathbb{Z}} \to \pi_* \THH(A; \mathbb{Z}_p).\]
By the universal property of $H^0$ of derived $p$-completions, this gives a graded $A$-algebra map
\[ (\Omega^*_{A/\mathbb{Z}})^\wedge_p \to \pi_* \THH(A; \mathbb{Z}_p),\]
where the left side is defined as $H^0$ of the termwise derived $p$-completion of the graded ring $\Omega^*_{A/\mathbb{Z}}$. To finish constructing $\mu_A$, it is now enough to show that for any $R$-algebra $A$, the natural map $\Omega^i_{A/\mathbb{Z}} \to \Omega^i_{A/R}$ induces an isomorphism on $H^0$ after applying derived $p$-completion. Note that the map $\Omega^i_{A/\mathbb{Z}} \to \Omega^i_{A/R}$ is surjective with $p$-divisible kernel (as this holds true for $i=1$ since $\Omega^1_{R/\mathbb{Z}}$ is $p$-divisible by the perfectoid nature of $R$). But then the homotopy fiber of the map $\widehat{\Omega^i_{A/\mathbb{Z}}} \to \widehat{\Omega^i_{A/R}}$ in $D(A)$ obtained by applying the derived $p$-completion functor lies in $D^{\leq -1}$, so applying $H^0$ gives the claim.
\end{construction}

The map constructed above linearizes to an isomorphism in favorable cases:

\begin{corollary}[Hesselholt]
\label{THHFSmooth}
For any $R$-algebra $A$, the map in Construction~\ref{cons:antisym} linearizes to give a map
\[ (\Omega^*_{A/R})^\wedge_p \otimes_R \pi_* \THH(R;\mathbb Z_p) \to \pi_* \THH(A;\mathbb Z_p)\]
of graded $A \otimes_R \pi_* \THH(R; \mathbb{Z}_p)$-algebras.  If $A$ is quasismooth, this map is an isomorphism.
\end{corollary}

\begin{proof} Only the last statement requires proof. We begin by noting that the composite
\[ (\Omega^*_{A/R})^\wedge_p \to \pi_* \THH(A;\mathbb Z_p) \to \pi_* \HH(A/R;\mathbb Z_p)\]
is an isomorphism of graded rings by the HKR filtration. This implies that the long exact sequence on $\pi_*$  obtained from the first fiber sequence in Theorem~\ref{TPHPDeformation} decomposes into short exact sequences
\[ 0 \to \pi_{i-2} \THH(A;\mathbb Z_p) \xrightarrow{u} \pi_i \THH(A;\mathbb Z_p) \to \pi_i \HH(A/R;\mathbb Z_p) \cong (\Omega^i_{A/R})^\wedge_p \to 0\]
where the surjective map comes equipped with a preferred section, and the final isomorphism comes from Remark~\ref{HKRQuasismooth}. This easily implies the assertion in the corollary by induction on $i$.
\end{proof}

The following filtration will only play a technical role.

\begin{corollary}
\label{THHPostnikovPerfectoid}
The functor $\THH(-;\mathbb Z_p)$ on the category of $p$-complete $R$-algebras admits a complete descending multiplicative $\mathbb N$-indexed filtration $P^\f{(-)}$ with $\gr^n_P \THH(-;\mathbb Z_p)$ being naturally identified with
\[\bigoplus_{\substack{0\leq i\leq n\\ i-n\ \mathrm{even}}} (\wedge^i L_{-/R})^\wedge_p[n].\]
\end{corollary}

\begin{proof}
The assertion of the corollary holds true on the category of quasismooth $R$-algebras by Corollary~\ref{THHFSmooth} simply by using the Postnikov filtration. By left Kan extension in $p$-complete spectra, one gets a filtration $P^\f{(-)}$ as in the statement above as $\THH(-)$ commutes with sifted colimits of $R$-algebras; the completeness of $\THH(-)$ with respect to $P^\f{(-)}$ is a consequence of the fact that $P^n_A$ is $n$-connective for any $p$-complete $R$-algebra $A$ (by left Kan extension from the quasismooth case).
\end{proof}

\newpage
\section{$p$-adic Nygaard complexes}
\label{sec:pAdicNygaard}

Let $R$ be a perfectoid ring and write $A_{\inf} = A_{\inf}(R)$. We shall explain in \S \ref{subsec:PrismDef} how to extract an $A_{\inf}$-valued cohomology theory $\widehat{\Prism}_S$ for quasismooth $R$-algebras $S$ by unfolding $\pi_0 \TC^{-}(-;\mathbb Z_p)$. The abutment filtration for the homotopy fixed point spectral sequence unfolds to give a filtration, called the Nygaard filtration, on $\widehat{\Prism}_S$ that will be crucial in the sequel. To carry out the unfolding effectively, we describe $\TC^{-}$ for quasiregular semiperfectoid rings in \S \ref{subsec:TCqrsp}. This description is also used in \S \ref{ss:MotFilt} to prove Theorem~\ref{thm:main5}.

\subsection{$\TC^{-}$ for quasiregular semiperfectoids}
\label{subsec:TCqrsp}

First, we discuss the topological Hochschild homology of a quasiregular semiperfectoid $R$-algebra.

\begin{theorem}
\label{THHqrs}
Let $S \in \Qsp_R$ and let $M=\pi_1 (L_{S/R})^\wedge_p$ be the associated $p$-completely flat $S$-module.
\begin{enumerate}
\item $\pi_* \THH(S;\mathbb Z_p)$ is concentrated in even degrees. 
\item Multiplication by the generator $u\in \pi_2 \THH(R;\mathbb Z_p)$ gives a natural injective map 
\[ \pi_{2i-2} \THH(S;\mathbb Z_p)\xrightarrow{u} \pi_{2i} \THH(S;\mathbb Z_p).\]
\item Write $\pi_\infty \THH(S;\mathbb Z_p) = \colim_i \pi_{2i} \THH(S;\mathbb Z_p) = \pi_0 \THH(S;\mathbb Z_p)[u^{-1}]$ for the colimit of multiplication by $u$; we may view this object as an increasingly filtered commutative $R$-algebra. There is a functorial identification
\[ (\Gamma^*_S M)^\wedge_p \cong \gr_* \pi_\infty \THH(S;\mathbb Z_p)\]
of graded rings (where the left side denotes the $p$-completion in graded rings). In particular, each $\pi_{2i} \THH(S;\mathbb Z_p)$ admits a finite increasing filtration with graded pieces given in ascending order by $(\Gamma^j_S M)^\wedge_p$ for $0 \leq j \leq i$.
\item Each $\pi_{2i} \THH(S;\mathbb Z_p)$ is $p$-completely flat over $S$.
\end{enumerate}
\end{theorem}

\begin{proof}
By Corollary~\ref{THHPostnikovPerfectoid}, the spectrum $\THH(S;\mathbb Z_p)$ admits a complete descending multiplicative $\mathbb{N}$-indexed filtration with $\gr^n \THH(S)$ being
\[ \bigoplus_{\substack{0\leq i\leq n\\ i-n\ \mathrm{even}}} (\wedge_S^i L_{S/R})^\wedge_p[n].\]
Note that $(\wedge_S^i L_{S/R})^\wedge_p$ has $p$-complete Tor amplitude concentrated in homological degree $i$ by Lemma~\ref{HKRSperf}, and hence it lives in degree $i$ by Lemma~\ref{BoundedTorsionFlat}. But then the displayed terms above live in degree $i+n$, which is even. This implies (1) by completeness of the filtration.

For (2) and (3), we use the $\T$-equivariant fiber sequence 
\[ \THH(S;\mathbb Z_p)[2] \xrightarrow{u} \THH(S;\mathbb Z_p) \to \HH(S/R;\mathbb Z_p)\]
from Theorem~\ref{TPHPDeformation}. The preceding paragraph shows that $\pi_* \THH(S;\mathbb Z_p)$ lives in even degrees, and the same holds for $\HH(S/R;\mathbb Z_p)$ by Lemma~\ref{HKRSperf}. Thus, the long exact sequence on homotopy for the previous fiber sequence gives short exact sequences
\[ 0 \to \pi_{2i-2} \THH(S;\mathbb Z_p) \xrightarrow{u} \pi_{2i} \THH(S;\mathbb Z_p) \to \pi_{2i} \HH(S/R;\mathbb Z_p) \to 0.\]
Using the identification $\pi_{2i} \HH(S/R;\mathbb Z_p)\cong \pi_i (\wedge_S^i L_{S/R})^\wedge_p\cong (\Gamma^i_S M)^\wedge_p$ from Lemma~\ref{HKRSperf}, we can write this as
\[ 0 \to \pi_{2i-2} \THH(S) \xrightarrow{u} \pi_{2i} \THH(S) \to (\Gamma^i_R M)^\wedge_p \to 0.\]
This proves (2) and (3) by induction; the assertion about multiplicativity is a consequence of the multiplicativity of the map $\THH(S;\mathbb Z_p) \to \HH(S/R;\mathbb Z_p)$.

Finally, (4) follows from the last exact sequence above by induction as $(\Gamma^i_S M)^\wedge_p$ is a $p$-completely flat $S$-module.
\end{proof}

For any $R$-algebra $A$, we view $\pi_* \TC^{-}(A;\mathbb Z_p)$ resp.~$\pi_* \TP(A;\mathbb Z_p)$ as a graded algebra over the graded ring
\[
\pi_* \TC^{-}(R;\mathbb Z_p) = A_{\inf}[u,v]/(uv-\xi) \quad \text{resp.}\quad \pi_* \TP(R;\mathbb Z_p) \cong A_{\inf}[\sigma,\sigma^{-1}]\ .
\]
In particular, $\pi_* \TP(A;\mathbb Z_p)$ is $2$-periodic. By passing to fixed points, Theorem~\ref{THHqrs} yields:

\begin{theorem}
\label{TCqrsp}
Let $S \in \Qsp_R$.
\begin{enumerate}
\item The homotopy fixed point spectral sequence calculating $\TC^{-}(S;\mathbb Z_p)$ and the Tate spectral sequence calculating $\TP(S;\mathbb Z_p)$ degenerate. Both $\pi_* \TC^{-}(S;\mathbb Z_p)$ and $\pi_* \TP(S;\mathbb Z_p)$ live only in even degrees. Moreover, the canonical map $\pi_* \TC^{-}(S;\mathbb Z_p) \xrightarrow{\mathrm{can}} \pi_* \TP(S;\mathbb Z_p)$ is injective in all degrees, and an isomorphism in degrees $\leq 0$.
\item The (degenerate) homotopy fixed point spectral calculating $\TC^{-}(R;\mathbb Z_p)$ or the (degenerate) Tate spectral sequence calculating $\TP(R;\mathbb Z_p)$ endows
\[
\widehat{\Prism}_S := \pi_0 \TC^{-}(S;\mathbb Z_p) \stackrel{\mathrm{can}}{\cong} \pi_0 \TP(S;\mathbb Z_p)
\]
with the same complete descending $\mathbb N$-indexed filtration $\calN^{\geq \f}\widehat{\Prism}_S$, called the {\em Nygaard filtration}, for which it is complete. There are natural identifications of the associated graded $\calN^i \widehat{\Prism}_S\cong \pi_{2i} \THH(S;\mathbb Z_p)$ for all $i \geq 0$.
\item The filtration level $\calN^{\geq i} \widehat{\Prism}_S \subset \widehat{\Prism}_S = \pi_0 \TC^{-}(S;\mathbb Z_p)$ is identified with $\pi_{2i} \TC^{-}(S;\mathbb Z_p)$ via multiplication by the element $v^i\in \pi_{-2i} \TC^-(R;\mathbb Z_p)$,
\[ \pi_{2i} \TC^{-}(S;\mathbb Z_p) \xrightarrow{v^i} \pi_0 \TC^{-}(S;\mathbb Z_p).\]
\item The cyclotomic Frobenius $\pi_* \TC^{-}(S;\mathbb Z_p) \xrightarrow{\varphi_S^{h\T}} \pi_* \TP(S;\mathbb Z_p)$ induces an endomorphism $\varphi_S:\widehat{\Prism}_S \to \widehat{\Prism}_S$ by (2). This endomorphism maps $\calN^{\geq i}\widehat{\Prism}_S$ to $\tilde\xi^i \widehat{\Prism}_S$. This gives a natural divided Frobenius $\varphi_{S,i}:\calN^{\geq i}\widehat{\Prism}_S \to \widehat{\Prism}_S$ such that
\[
\varphi_S|_{\calN^{\geq i}\widehat{\Prism}_S} = \tilde\xi^i \varphi_{S,i}\ .
\]
\item There is a natural isomorphism of $R$-algebras $\widehat{\Prism}_S/\xi \cong \widehat{L\Omega}_{S/R}$, and $\widehat\Prism_S$ is $\xi$-torsion-free.
\end{enumerate}
\end{theorem}

\begin{proof}
As $\pi_* \THH(S;\mathbb Z_p)$ lives in even degrees, (1) and (2) are immediate. Part (3) follows by unwinding the statement that $\THH(S;\mathbb Z_p)$ is a $\T$-equivariant $\THH(R;\mathbb Z_p)$-module spectrum at the level of the homotopy fixed point spectral sequences.

For (4), we use the last statement of (2) and the identity $\varphi(v) = \tilde\xi \sigma^{-1}$.

For (5), we use Theorem~\ref{TPHPDeformation} to obtain $\pi_0\TC^{-}(S;\mathbb Z_p)/\xi \cong \pi_0 \HC^{-}(S/R;\mathbb Z_p)$, Proposition~\ref{prop:hcddr} then implies that $\pi_0 \TC^{-}(S;\mathbb Z_p)/\xi \cong \widehat{L\Omega}_{S/R}$. Moreover, that theorem shows that any $\xi$-torsion in $\pi_0\TP(S;\mathbb Z_p)$ would be detected by $\HP_1(S/R;\mathbb Z_p)$, but this is $0$ by the Tate spectral sequence and the fact that $\HH_\sub{odd}(S/R;\mathbb Z_p)=0$ by Lemma \ref{HKRSperf}(2).
\end{proof}

\begin{remark}\label{rem:basering}
Let $S \in \Qsp$ but do not fix a perfectoid ring mapping to $S$. Then (1) and (4) in Theorem~\ref{THHqrs}, and (1) and (2) in Theorem~\ref{TCqrsp} continue to hold, i.e., do not depend on the choice of a perfectoid ring mapping to $S$. 
\end{remark}

\subsection{Unfolding to $\widehat{\Prism}_{(-)}$}
\label{subsec:PrismDef}

We begin by unfolding $\THH$:

\begin{construction}[Unfolding $\pi_{2i} \THH$]
\label{GradedTHH}
By Theorem~\ref{THHqrs} (3), for each $S \in \Qsp_{R}$, the $S$-module $\pi_{2i} \THH(S;\mathbb Z_p)$ admits a functorial finite increasing filtration with graded pieces given by $(\wedge_S^j L_{S/R})^\wedge_p[-j]$ for $0 \leq j \leq i$ in ascending order. Theorem~\ref{FlatDescentCC} then implies that $\pi_{2i} \THH(-;\mathbb Z_p)$ is a $D(R)$-valued sheaf on $\Qsp_{R}^\sub{op}$. By Proposition~\ref{qsqspextend}, it unfolds to a $D(R)$-valued sheaf $(\pi_{2i} \THH(-;\mathbb Z_p))^\beth$ on $\Qs^\sub{op}_{R}$; this sheaf admits a similar filtration by functoriality of unfolding. In particular, it takes values in $D^{\leq i}(R)$.
\end{construction}

A tangible consequence of this discussion is the construction of the ``motivic'' filtration on $\THH$:

\begin{proposition}
\label{MotivicFiltTHH}
 For any $A\in \Qs_{R}$, the spectrum $\THH(A;\mathbb Z_p)$ admits a functorial complete descending $\mathbb{N}$-indexed $\T$-equivariant filtration such that $\gr^i \THH(A;\mathbb Z_p)$ is canonically an $A$-module spectrum with trivial $\T$-action that admits a finite increasing filtration with graded pieces given by $(\wedge^j_A L_{A/R})^\wedge_p[2i-j]$ for $0 \leq j \leq i$ in ascending order.
\end{proposition}

\begin{proof}
The claim holds true on $\Qsp_{R}$ simply by using the double speed Postnikov filtration thanks to Construction~\ref{GradedTHH}. It then follows in general thanks to Theorem~\ref{FlatDescentCC} and Corollary~\ref{FlatDescentHH} and functoriality of unfolding.
\end{proof}

\begin{remark}
By left Kan extension in $p$-complete $\T$-equivariant spectra, Proposition~\ref{MotivicFiltTHH} extends to all $p$-complete $R$-algebras.
\end{remark}

We now lift the discussion to $\TC^{-}$. First, let us give the analog of Construction~\ref{GradedTHH} by constructing $p$-adic Nygaard complexes; these are the main objects of interest from the perspective of a comparison with integral $p$-adic Hodge theory.

\begin{construction}[Unfolding $\pi_0 \TC^{-}$]
Consider the $\widehat{DF}(A_{\inf})$-valued functor on $\Qsp_{R}$ given by $(\widehat{\Prism}_{(-)}, \calN^{\geq \f}\widehat{\Prism}_{(-)})$ with notation as in Theorem~\ref{TCqrsp}. By the same theorem, this functor is a sheaf, and thus unfolds to a sheaf $(\widehat{\Prism}_{(-)}, \calN^{\geq \f} \widehat{\Prism}_{(-)})$ on $\Qs_{R}^\sub{op}$. As the equivalence in Proposition~\ref{qsqspextend} is symmetric monoidal, this sheaf is valued in $E_\infty$-algebras in $\widehat{DF}(A_{\inf})$. By construction, for any $A\in \Qs_R$, the underlying $E_\infty$-$A_{\inf}$-algebra $\widehat{\Prism}_A$ is $(p,\xi)$-complete (as it is given by a limit of the values for objects of $\Qsp_R$, which are all $(p,\xi)$-complete) and comes equipped with a complete descending multiplicative $\mathbb{N}$-indexed filtration $\calN^{\geq \f}\widehat{\Prism}_A$. Write $\calN^i\widehat{\Prism}_A$ for the $i$-th graded piece. The cyclotomic Frobenius induces a Frobenius semilinear map $\varphi_A:\widehat{\Prism}_A \to \widehat{\Prism}_A$. When $F$ is a perfectoid $R$-algebra, then $\widehat{\Prism}_F = A_{\inf}(F)$, $\calN^{\geq i}\widehat{\Prism}_F = \ker(\theta_F)^i$, and $\varphi_F$ is the usual Frobenius on $A_{\inf}(F)$.
\end{construction}

The associated graded pieces $\calN^i\widehat{\Prism}_A$ constructed above coincide with those in Construction~\ref{GradedTHH}.

\begin{proposition}
\label{NygaardGradedFilt}
For $A \in \Qs_R$, each $\calN^i\widehat{\Prism}_A\simeq \gr^i \THH(A;\mathbb Z_p)[-2i]$ is functorially an $A$-complex that admits a finite increasing filtration with graded pieces given in ascending order by $(\wedge_A^j L_{A/R})^\wedge_p[-j]$ for $0 \leq j \leq i$.
\end{proposition}

\begin{proof}
As $\calN^i\widehat{\Prism}_{(-)} = (\pi_{2i} \THH(-;\mathbb Z_p))^{\beth}$ by Theorem~\ref{TCqrsp} (2), this assertion is simply a reformulation of Construction~\ref{GradedTHH}.
\end{proof}

The complexes $\widehat{\Prism}_A$ constructed above deform de Rham cohomology across $\theta:A_{\inf} \to R$.

\begin{proposition}
\label{NygaarddR}
For $A \in \Qs_{R}$, there is a natural identification of $E_\infty$-$R$-algebras $\widehat{\Prism}_A/\xi \simeq \widehat{L\Omega}_{A/R}$. 
\end{proposition}

\begin{proof}
This follows from Theorem~\ref{TCqrsp} (5) by descent.
\end{proof}

For the purposes of our later comparison with the $A\Omega$-theory, we record some features of the Nygaard complexes for $p$-adic completions of smooth $R$-algebras.

\begin{corollary}
\label{NygaardSmooth}
Assume $A \in \Qs_R$ is the $p$-adic completion of a smooth $R$-algebra of relative dimension $d$. Then 
\begin{enumerate}
\item For each $i \geq 0$, we have $\calN^i\widehat{\Prism}_A \in D^{[0,\max(i,d)]}(A_{\inf})$ and $\calN^{\geq i}\widehat{\Prism}_A \in D^{[0,d]}(A_{\inf})$. In particular,  we have $(\widehat{\Prism}_A, \calN^{\geq \f} \widehat{\Prism}_A) \in DF^{\leq 0}(A_{\inf})$.

\item The ring $H^0(\widehat{\Prism}_A)$ has no $\varphi^r(\xi)$-torsion for any $r \in \mathbb{Z}$.

\item The linearization of the Frobenius map $\varphi_A$ factors functorially over a map
\[ \widehat{\Prism}_A \to L\eta_{\xi} \varphi_* \widehat{\Prism}_A \simeq \varphi_* L\eta_{\tilde\xi} \widehat{\Prism}_A\]
of $E_\infty$-algebras in $D(A_{\inf})$.
\end{enumerate}
\end{corollary}

\begin{proof}
For (1), everything follows from Proposition~\ref{NygaardGradedFilt}.

For (2), we may assume by Zariski localization that $A$ admits an \'etale map to a torus, so that we can choose a quasisyntomic cover $A \to F$ in $\Qs_{R}$ with $F$ perfectoid by extracting $p$-power roots of the coordinates on the torus. As $\widehat{\Prism}_{(-)}$ takes values in $D^{\geq 0}$, the map $H^0(\widehat{\Prism}_A) \to H^0(\widehat{\Prism}_F)$ is injective by the sheaf property. But $\widehat{\Prism}_F \simeq A_{\inf}(F)$, and this ring has no $\varphi^r(\xi)$-torsion for any $r \in \mathbb{Z}$: as $F$ is perfectoid, the image of $\xi \in A_{\inf}(F)$ is a nonzerodivisor and $\varphi$ is an automorphism.

For (3), we shall use Proposition~\ref{LetaDF} (including its notation). Note that for any $A \in \Qs_{R}$, the Frobenius map $\varphi:\widehat{\Prism}_A \to \widehat{\Prism}_A$ defines a map 
\[ \calN^{\geq \ast} \widehat{\Prism}_A \to \varphi_* \tilde\xi^\ast \widehat{\Prism}_A = \xi^\ast \varphi_* \widehat{\Prism}_A\]
of $E_\infty$-algebras in $\widehat{DF}(A_{\inf})$: this is clear for $A \in \Qsp_{R}$ by Theorem~\ref{TCqrsp} and thus follows in general by descent. For $A$ as in the corollary, the left side lies in the connective part $DF^{\leq 0}$ by (1), so the map above factors uniquely over $\tau^{\leq 0}_B$ of the target. This gives the desired map by Proposition~\ref{LetaDF}.
\end{proof}

\begin{remark}
\label{NygaardSmoothFrobeniusIterate}
Iterating Corollary~\ref{NygaardSmooth} (3) gives a functorial map
\[ \widehat{\Prism}_A \to (L\eta_{\xi}  \varphi_*)^{\circ r} \widehat{\Prism}_A \simeq L\eta_{\xi_r}  \varphi^r_* \widehat{\Prism}_A.\]
factoring $r$-fold Frobenius on $\widehat{\Prism}_A$; here $\xi_r=\xi\varphi^{-1}(\xi)\cdots \varphi^{-r+1}(\xi)$ generates the kernel of $\theta_r:A_{\inf} \to W_r(R)$, and the natural identification of functors $(L\eta_{\xi}  \varphi_*)^{\circ r} \simeq L\eta_{\xi_r}  \varphi^r_*$ falls out immediately by expanding both sides. For instance, when $r=2$, we have
\[ L\eta_{\xi} \varphi_* L\eta_{\xi} \varphi_* = L\eta_{\xi} L\eta_{\varphi^{-1}(\xi)} \varphi^2_* \simeq L\eta_{\xi_2} \varphi^2_*, \]
where the last isomorphism uses $L\eta_f L\eta_g \simeq L\eta_{fg}$, cf.~\cite[Lemma 6.11]{BMS}
\end{remark}

We shall also need the following non-Nygaard-completed variant of $\widehat{\Prism}$ in the sequel.

\begin{construction}[Non-complete variant of $\widehat{\Prism}$]
\label{NoncompleteNygaard}
For $A$ the $p$-adic completion of a smooth $R$-algebra, we have $\widehat{\Prism}_A/\xi \simeq L\Omega_{A/R}$ by Proposition~\ref{NygaarddR} and the fact that the combined Hodge and $p$-adic filtration is commensurate with the $p$-adic filtration for smooth $R$-algebras. By left Kan extension in $(p,\xi)$-complete $A_{\inf}$-complexes, we obtain a new functor $A \mapsto {\Prism}_A$ on all $p$-complete simplicial commutative $R$-algebras. By construction, we have an identification ${\Prism}_{(-)}/\xi \simeq L\Omega_{-/R}$. This implies ${\Prism}_{(-)}$ is a $D(A_{\inf})$-valued sheaf on $\Qs_{R}^\sub{op}$ (by Example~\ref{pCDDR}) and that it takes discrete values on $\Qsp_{R}^\sub{op}$. We warn the reader that unlike $\widehat{\Prism}_A$, the $E_\infty$-algebra ${\Prism}_A$ depends on the choice of the perfectoid ring $R$ mapping to $A$, at least a priori.
\end{construction}

\subsection{Motivic filtrations}
\label{ss:MotFilt}

The ``motivic'' filtration for $\TC^{-}$ is given by the following proposition, which proves most of Theorem~\ref{thm:main5} when working over a fixed perfectoid base ring.

\begin{proposition} 
For any $A \in \Qs_{R}$, we have:
\begin{enumerate}
\item The spectrum $\TC^{-}(A;\mathbb Z_p)$ admits a functorial complete and exhaustive descending multiplicative $\mathbb{Z}$-indexed filtration with $\gr^i \TC^{-}(A;\mathbb Z_p) = \calN^{\geq i}\widehat{\Prism}_A[2i]$. In particular, there exists a spectral sequence
\[ E_2^{ij}:H^{i-j}(\calN^{\geq -j} \widehat{\Prism}_A) \Rightarrow \pi_{-i-j} \TC^{-}(A;\mathbb Z_p).\]
\item The spectrum $\TP(A;\mathbb Z_p)$ admits a functorial complete and exhaustive descending multiplicative $\mathbb{Z}$-indexed filtration with $\gr^i \TP(A;\mathbb Z_p) = \widehat{\Prism}_A[2i]$. In particular, there exists a spectral sequence
\[ E_2^{ij}:H^{i-j}(\widehat{\Prism}_A) \Rightarrow \pi_{-i-j} \TP(A;\mathbb Z_p),\]
\end{enumerate}
For $A \in \Qsp_R$ (and thus for $A = R$ itself), both filtrations are given by the double speed Postnikov filtration on the corresponding spectra.
\end{proposition}

\begin{proof}
(1) For each $n \in \mathbb{Z}$, the functor $A \mapsto \tau_{\geq 2n} \TC^{-}(A;\mathbb Z_p)$ on $\Qsp_R$ is a sheaf by Theorem~\ref{TCqrsp} and Theorem~\ref{THHqrs} (and stability of sheaves under limits). Write $\mathrm{Fil}^n \TC^{-}(-;\mathbb Z_p)$ for its unfolding. As $n$ varies, this gives a $\widehat{DF}(A_{\inf})$-valued sheaf $\mathrm{Fil}^\ast \TC^{-}(-;\mathbb Z_p)$ on $\Qs_R$; here the completeness follows from the completeness of the Postnikov filtration and the fact that a $DF(A_{\inf})$-valued sheaf on $\Qs_R$ takes complete values if and only if its restriction to $\Qsp_R$ does so. The $i$-th graded piece of this sheaf is $\big(\pi_{2i} \TC^{-}(-;\mathbb Z_p)[2i]\big)^\beth \cong \calN^{\geq i}\widehat{\Prism}_{(-)}[2i]$. It remains to prove that the filtration is exhaustive; but on any homotopy group $\pi_i \Fil^n \TC^-(A;\mathbb Z_p)$, the filtration is eventually constant and equal to $\pi_i \TC^-(A;\mathbb Z_p)$; indeed, it suffices to take $n$ sufficiently negative so that $i\geq 2n$.

For part (2), the argument is identical.
\end{proof}

The use of the perfectoid base ring $R$ above is rather mild: the spectra $\TC^{-}(A;\mathbb{Z}_p)$ and $\TP(A; \mathbb{Z}_p)$ as well as their Postnikov filtrations are obviously independent of the choice of $R$, and the only role played by $R$ is in making sense of the Breuil-Kisin twist. In fact, this can also be done in a direct way, thus proving Theorem~\ref{thm:main5} in general:

\begin{proof}[Proof of Theorem~\ref{thm:main5}] Parts (1) and (2) clear; part (4) follows formally by reduction to the case of a perfectoid base ring once (3) is known, and part (5) follows formally from part (4). Thus, it remains to prove part (3).

Assume first that $A$ is an $R$-algebra with $R$ perfectoid. Then the Breuil-Kisin twist $\widehat{\Prism}_A\{1\}\cong \widehat{\Prism}_R\{1\}\otimes_{\widehat{\Prism}_R} \widehat{\Prism}_A$ is trivial by the above discussion. After base change along $\widehat{\Prism}_A\to A$, it is even canonically trivial: The map
\[
\gr^\ast \TP(A;\mathbb Z_p)\otimes_{\widehat{\Prism}_A} A\to \gr^\ast \HP(A/A;\mathbb Z_p)
\]
is an equivalence, and thus
\[
\widehat{\Prism}_A\{1\}\otimes_{\widehat{\Prism}_A} A = \gr^1 \TP(A;\mathbb Z_p)[-2]\otimes_{\widehat{\Prism}_A} A = \gr^1 \HP(A/A;\mathbb Z_p)[-2] = A
\]
canonically.

In the general case, it suffices to prove that $\widehat{\Prism}_A\{1\}$ is an invertible $\widehat{\Prism}_A$-module in the presentably symmetric monoidal stable $\infty$-category $\widehat{DF}(\mathbb Z)$, and commutes with base change: As tensoring with the invertible module $\widehat{\Prism}_{\mathbb Z_p}\{1\}$ is an equivalence on the category of completed filtered $\widehat{\Prism}_{\mathbb Z_p}$-modules, and in particular commutes with all limits, all other statements of Theorem~\ref{thm:main5} follow via descent.

Write $\widehat{DF}_{\geq 0}(\mathbb{Z})$ for the $\infty$-category of $\mathbb{N}$-filtered complexes of abelian groups. This is a presentably symmetric monoidal stable $\infty$-subcategory of $\widehat{DF}(\mathbb{Z})$. Write $\mathrm{Gr}(\mathbb{Z})_{\geq 0} := \mathrm{Fun}(\mathbb{N}, D(\mathbb{Z}))$ for the $\infty$-category of $\mathbb{N}$-graded objects in $D(\mathbb{Z})$; this is also a presentably symmetric monoidal stable $\infty$-category (via the Day convolution symmetric monoidal structure). Taking associated graded gives an exact and conservative symmetric monoidal functor
\[ \mathrm{gr}^*:\widehat{DF}(\mathbb{Z})_{\geq 0} \to \mathrm{Gr}(\mathbb{Z})_{\geq 0}.\]
In particular, if $A \in \mathrm{CAlg}(\widehat{DF}(\mathbb{Z})_{\geq 0})$, then $\mathrm{gr}^*(A) \in \mathrm{CAlg}(\mathrm{Gr}(\mathbb{Z})_{\geq 0})$, and taking associated gradeds gives an exact and conservative symmetric monoidal functor
\[ \mathrm{gr}^*:\mathrm{Mod}_A(\widehat{DF}(\mathbb{Z})_{\geq 0}) \to \mathrm{Mod}_{\mathrm{gr^*(A)}}( \mathrm{Gr}(\mathbb{Z})_{\geq 0}).\]
We need the following lemma:

\begin{lemma}
Fix $M, N \in \mathrm{Mod}_A(\widehat{DF}(\mathbb{Z})_{\geq 0})$ with a map $\eta:M \otimes_A N \to A$ in $\mathrm{Mod}_A(\widehat{DF}_{\geq 0}(\mathbb{Z}))$. Assume the following:
\begin{enumerate}
\item The natural map $\mathrm{gr}^0(M) \otimes_{\mathrm{gr}^0(A)} \mathrm{gr}^*(A) \to \mathrm{gr}^*(M)$ is an equivalence, and similarly for $N$.
\item The map $\eta$ induces an isomorphism $\mathrm{gr}^0(M) \otimes_{\mathrm{gr}^0(A)} \mathrm{gr}^0(N) \to \mathrm{gr}^0(A)$. 
\end{enumerate}
Then $\eta$ is an equivalence. In particular, both $M$ and $N$ are invertible $A$-modules.
\end{lemma}

\begin{proof}
As $\mathrm{gr}^*$ is conservative, it is enough to show that $\mathrm{gr}^*(\eta)$ is an equivalence. By (1), this reduces to checking that $\gr^0(\eta)$ is an equivalence, but this is exactly ensured by (2). 
\end{proof}

We apply the lemma to $M=\widehat{\Prism}_A\{1\}$ and $N=\widehat{\Prism}_A\{-1\}$ as completed filtered modules over $\widehat{\Prism}_A$. By the above discussion, we know that there is a canonical isomorphism $\mathrm{gr}^0(\widehat{\Prism}_A\{1\})\simeq \mathrm{gr}^0(\widehat{\Prism}_A)$ locally for the quasisyntomic topology, which thus glues to such an isomorphism by descent. In particular, condition (2) follows (as there is a compatible such isomorphism for $\mathrm{gr}^0(\widehat{\Prism}_A\{-1\})$). On the other hand, condition (1) can be checked locally in the quasisyntomic topology, and for quasiregular semiperfectoid $A$, it follows by $2$-periodicity of the Tate spectral sequence for $\TP(A;\mathbb Z_p)$.
\end{proof}

\subsection{The syntomic sheaves $\mathbb Z_p(i)$ and $K$-theory}\label{subsection_syntomic}
As in the statement of Theorem \ref{thm:main5}(4), for any quasisyntomic ring $A$ we introduce its ``syntomic cohomology'' \[\mathbb Z_p(i)(A):=\gr^n\TC(A;\mathbb Z_p)[-2i]=\mathrm{hofib}(\varphi - \mathrm{can}: \calN^{\geq i}\widehat{\Prism}_A\{i\}\to \widehat{\Prism}_A\{i\}).\] In the case of $S\in\Qsp$, this is given by the two term complex \[\mathbb Z_p(i)(S)=(\tau_{[2i-1,2i]}\TC(S;\mathbb Z_p))[-2i]=\mathrm{hofib}(\varphi - \mathrm{can}: \pi_{2i}\TC^-(S;\mathbb Z_p)\to \pi_{2i}\TP(S;\mathbb Z_p))\] with cohomology \[H^0(\mathbb Z_p(i)(S))=\TC_{2i}(S;\mathbb Z_p),\qquad H^1(\mathbb Z_p(i)(S))=\TC_{2i-1}(S;\mathbb Z_p).\]

We can relate these $\mathbb Z_p(i)$ to algebraic $K$-theory using the following theorem; this will appear in forthcoming work of Clausen, Mathew, and the second author. We denote by $K(-)$ the connective algebraic $K$-theory of a ring, and by $K(-;\mathbb Z_p)$ its $p$-completion.

\begin{theorem}[{\cite{CMM}}]\label{theorem_CMM}
Let $S$ be a ring which is Henselian along $pS$ and such that $S/pS$ is semiperfect, e.g., $S\in\Qsp$. Then the trace map $K(S;\mathbb Z_p)\to\tau_{\ge0}\TC(S;\mathbb Z_p)$ is an equivalence.
\end{theorem}

Using this, we can identify the complexes $\mathbb Z_p(n)$ for $n\leq 1$. First, we handle the case $n\leq 0$.

\begin{proposition} For $n<0$, the sheaf of complexes $\mathbb Z_p(n)=0$ vanishes. For $n=0$, there is a natural isomorphism $\mathbb Z_p(0) = \varprojlim_r \mathbb Z/p^r\mathbb Z$.
\end{proposition}

\begin{proof} For any connective ring spectrum $A$, one has $\pi_i \TC(A;\mathbb Z_p)=0$ for $i<-1$ by comparison with the classical definition of $\TC$, cf.~\cite[Theorem II.4.10]{NikolausScholze} noting that the spectra $\TR^r(A)$ are all connective. Moreover, using the identification $\pi_0 \TR^r(A) = W_r(A)$, one sees that $\pi_{-1} \TC(A;\mathbb Z_p)$ is given by the cokernel of $F-1: W(A)\to W(A)$. As one can extract infinite sequences of Artin-Schreier covers in $\Qsp$, this map is locally surjective. Thus, $\mathbb Z_p(0)$ is locally concentrated in degree $0$.

Finally, locally in $\Qsp$, the ring $S$ is w-local (in the sense of \cite{BhattScholzePro}, so in particular any Zariski cover of $\Spec S$ is split), by passing to the $p$-completion of the ind-\'etale w-localization of \cite{BhattScholzePro}. As any Zariski cover splits, it follows that the rank function from $K_0(S)$ to locally constant functions from $\Spec S$ to $\mathbb Z$ is an isomorphism. This implies the identification $\mathbb Z_p(0) = \varprojlim_r \mathbb Z/p^r\mathbb Z$ by passage to the $p$-completion, using Theorem~\ref{theorem_CMM}.
\end{proof}

Using Theorem~\ref{theorem_CMM} and results on algebraic $K$-theory in low degrees more seriously, we can identify $\mathbb Z_p(1)$.

\begin{proposition}
The sheaf of complexes $\mathbb Z_p(1)$ on $\Qsp$ is locally concentrated in degree $0$, given by $T_p\mathbb G_m$.
\end{proposition}

\begin{proof}
We begin by proving that, for any $S\in\Qsp$ which is both $w$-local and for which $S^\times$ is $p$-divisible, there are natural isomorphisms
\[K_2(S;\mathbb Z_p)\cong T_p(S^\times),\qquad K_1(S;\mathbb Z_p)= 0\,.\]
It is classical that, for any local ring $B$, the symbol map $B^\times\to K_1(B)$ (splitting the determinant) is an isomorphism, that the resulting product $B^\times\otimes_{\mathbb Z}B^\times \to K_2(B)$ is surjective, and that $K_0(B)$ is torsion-free (it is $\cong\mathbb Z$). Since any Zariski cover of $\Spec S$ is split, these properties remain true for $S$. Therefore $K_1(S)\cong S^\times$ and $K_2(S)$ are both $p$-divisible and $K_0(S)$ is $p$-torsion-free. The desired identities immediately follow.

It remains to show that such rings $S$ provide a basis for $\Qsp$. Given any $S\in \Qsp$, let $S\to S^Z$ denote its w-localization, which is a faithfully flat, ind-Zariski-localization, whence $L_{S^Z/S}\simeq 0$ and $S^Z/pS^Z$ is still semiperfect. Next, denote by $S^{1/p}$ the $S$-algebra obtained by formally adjoining $p^\sub{th}$-roots of all units, i.e., \[S^{1/p}=S[X_u:u\in S^\times]/(X_u-u^p:u\in S^\times)\,.\] Then $S\to S^{1/p}$ is a composition of an ind-smooth map followed by a quotient by a quasiregular ideal, whence $L_{S^{1/p}/S}$ has Tor amplitude in $[-1,0]$. Iterating these two processes countably many times, we set \[S^q:=\mathrm{colim}(S^Z\to (S^Z)^{1/p}\to ((S^Z)^{1/p})^Z\to (((S^Z)^{1/p})^Z)^{1/p}\to\cdots).\] Observe that $L_{S^q/S}$ has Tor-amplitude in $[-1,0]$, that $S\to S^q$ is faithfully flat, and that the units in $S^q$ are $p$-divisible; moreover, $S^q$ is w-local, since $w$-localisation is a left adjoint, \cite[Lemma 2.2.4]{BhattScholzePro}, and hence commutes with all colimits.

Let $\widehat{S^q}$ be the $p$-adic completion of $S$; then $\widehat{S^q}$ is a quasisyntomic semiperfectoid which is a quasisyntomic cover of $S$. Moreover $\widehat{S^q}$ is still w-local: indeed, since it is $p$-adically complete, this is equivalent to the w-locality of $\widehat{S^q}/p=S^q/p$, which follows from that of $S^q$ \cite[Lemmas 2.1.3 \& 2.1.7]{BhattScholzePro}. Finally observe that all units of $\widehat{S^q}$ admit a $p^\sub{th}$-root, by using Hensel's lemma to lift a root from $S^q/p^2$ ($S^q/2^3$ in the case $p=2$).
\end{proof}

The previous proposition proves the case $n=1$ of the following conjecture, which will be proved in characteristic $p$ in \S\ref{subsection_log_dRW}.

\begin{conjecture}\label{conjecture_syntomic}
The sheaf of complexes $\mathbb Z_p(i)$ on $\Qsp$ is locally concentrated in degree $0$, given by a $p$-torsion-free sheaf.
\end{conjecture}

\begin{remark}
$K$-theoretically, the conjecture predicts that on $\Qsp$ the sheafification of $K_{2i}(-;\mathbb Z_p)$ is $p$-torsion-free and that the sheafification of $K_{2i-1}(-;\mathbb Z_p)$ vanishes; this vanishing is equivalent to the surjectivity of $\varphi_i-1:\mathcal N^{\ge i}\widehat\Prism_{(-)}\{i\}\to \widehat\Prism_{(-)}\{i\}$.
\end{remark}

\begin{remark}
Let $S\in\Qsp$. Once $\widehat\Prism_S$ has been identified with the prismatic cohomology of \cite{BS} (see Remark \ref{remark_prism}), the $p$-torsion-freeness part of the conjecture will follow, as we now explain.

That identification will show that the Frobenius $\varphi$ on $\widehat{\Prism}_S$ arises from the finer structure of a $p$-derivation in the sense of J.~Buium, i.e., there exists $\delta:\widehat{\Prism}_S\to \widehat{\Prism}_S$ satisfying $p\delta(x)=\varphi(x)-x^p$, $\delta(1)=0$, $\delta(xy)=x^p\delta(y)+y^p\delta(x)+p\delta(x)\delta(y)$, $\delta(x+y)=\delta(x)+\delta(y)+\tfrac{x^p+y^p-(x+y)^p}p$. A standard lemma about $p$-derivations shows that if $px=0$ then $\varphi(x)=0$.

Given a $p$-torsion element $x\in H^0(\mathbb Z_p(i)(S))=\ker(\mathcal N^{\ge i}\widehat{\Prism}_S\{i\}\stackrel{\varphi_i-1}{\longrightarrow}\widehat{\Prism}_S\{i\})$, we can now show that $x=0$. Let $R\to S$ be a perfectoid ring mapping to $S$ so that we can identify $\tilde\xi^i\varphi_i$ with $\varphi:\mathcal N^{\ge i}\widehat{\Prism}_S\to \widehat{\Prism}_S$. Then the lemma on $p$-derivations tells us that $\varphi(x)=0$, whence $\tilde{\xi}^ix=\tilde\xi^i\varphi_i(x)=\varphi(x)=0$. But $\tilde\xi=\varphi(\xi)\equiv\xi^p$ mod $p$, so we deduce $\xi^{pi} x=0$, whence $x=0$ thanks to the $\xi$-torsion-freeness of $\widehat\Prism_S$ (Theorem \ref{TCqrsp}(5)).
\end{remark}

\begin{proposition} 
\label{cor:CocontinuousTateTwist}
The sheaf $\mathbb Z/p^r\mathbb (i):=\mathbb Z_p(i)/p^r$ commutes with filtered colimits in $\Qs$.
\end{proposition}

We give two proofs. The first proof (which was our original proof) uses the above remark along with $K$-theory and \cite{CMM}, while the second proof (discovered while this paper was being refereed) is more elementary and self-contained, relying ultimately on the contracting property of the Frobenius in characteristic $p$ (Lemma~\ref{SimpleFormulaZ/p}).

\begin{proof}[Proof via $K$-theory]
It is enough to prove this for a filtered colimit in $\Qsp$ by passing to functorial quasiregular semiperfectoid covers and quasisyntomic descent. But given $S\in \Qsp$, we have explained in the previous remark that $H^0(\mathbb Z_p(i)(S))$ is $p$-torsion-free for all $i$. Therefore $\mathbb Z/p^r\mathbb Z(i)(S)\simeq (\tau_{[2i-1,2i]}\TC(S;\mathbb Z/p^r\mathbb Z))[-2i]$, which commutes with filtered colimits of rings by Theorem~\ref{theorem_CMM} and the commutation of algebraic $K$-theory with filtered colimits.
\end{proof}

\begin{proof}[Direct proof of Proposition~\ref{cor:CocontinuousTateTwist}]
By induction, we may assume $r=1$. Recall that we have defined 
 \begin{equation}
 \label{eq:TateTwist}
 \mathbb{Z}/p\mathbb{Z}(i)(A):= \mathrm{hofib}(\varphi_i - 1: \calN^{\geq i}\widehat{\Prism}_A\{i\}/p\to \widehat{\Prism}_A\{i\}/p).
 \end{equation}

Fix the perfectoid field $C=\mathbb Q_p^{\mathrm{cycl}}$ of characteristic $0$. By quasisyntomic descent, we can assume that $A$ is a $\mathcal O_C$-algebra.

First, we observe that the desired compatibility with filtered colimits on the category of $p$-torsionfree quasisyntomic $\mathcal{O}_C$-algebras follows immediately from Lemma~\ref{SimpleFormulaZ/p}: the lemma implies that, for $m \gg 0$, we can work modulo $\calN^{\geq m} \widehat{\Prism}_A\{i\}/p$ when computing $\mathbb{Z}/p\mathbb{Z}(i)(A)$ via \eqref{eq:TateTwist}, and it is easy to see that $\Big(\calN^{\geq i}\widehat{\Prism}_A\{i\}/p\Big)/\Big( \calN^{\geq m}\widehat{\Prism}_A\{i\}/p\Big)$ commutes with filtered colimits for all $m \geq i$.

In fact, by left Kan extension and Nygaard completion, one can define an endomorphism $\varphi_i$ of $\calN^{\geq m} \widehat{\Prism}_A\{i\}/p$ for any $m\geq \frac{pi+1}{p-1}$ and any quasisyntomic ring $A$ over $\mathcal O_C$. This still has the property that the resulting map  \[ \varphi_i - 1: \calN^{\geq m}\widehat{\Prism}_A\{i\}/p\to \calN^{\geq m}\widehat{\Prism}_A\{i\}/p \]
 is an equivalence. Thus, one can repeat the above argument for any $A$.
\end{proof}

\begin{lemma}
\label{SimpleFormulaZ/p}
Fix a perfectoid field $C$ of characteristic $0$ as well as a $p$-torsionfree quasisyntomic $\mathcal{O}_C$-algebra $A$. For $m \geq \frac{pi+1}{p-1}$, both $\varphi_i$ and $1$ preserve $\calN^{\geq m} \widehat{\Prism}_A\{i\}/p$ functorially in $A$, and the resulting map
 \[ \varphi_i - 1: \calN^{\geq m}\widehat{\Prism}_A\{i\}/p\to \calN^{\geq m}\widehat{\Prism}_A\{i\}/p \]
 is an equivalence.
\end{lemma}

\begin{proof}
As we work over our fixed perfectoid ring $\mathcal{O}_C$, we can ignore the Breuil-Kisin twists. By quasisyntomic descent, we may assume $A$ is quasiregular semiperfectoid and $p$-torsionfree. In particular, each $\calN^k \widehat{\Prism}_A/p$ is concentrated in degree $0$. The complexes $\calN^{\geq k} \widehat{\Prism}_A/p$ are then also concentrated in degree $0$, and moreover are $\xi$-torsionfree, where $\xi \in A_{\inf}$  generates $\mathrm{\ker}(\theta)$. Setting $\tilde{\xi} := \phi(\xi)$, we can then compute the map 
 \[ \varphi_i: \calN^{\geq i} \widehat{\Prism}_A \to \widehat{\Prism}_A\]
 as induced by the map $\tilde{\xi}^{-i} \phi$ of $\widehat{\Prism}_A[1/\tilde{\xi}]$. One then computes that 
 \[\varphi_i(\calN^{\geq m} \widehat{\Prism}_A) \subset \tilde{\xi}^{m-i} \widehat{\Prism}_A\] for all $m \geq i$. Working modulo $p$ and using that $\tilde{\xi} = \xi^p \mod pA_{\inf}$, this gives 
 \[\varphi_i(\calN^{\geq m} \widehat{\Prism}_A/p) \subset \xi^{p(m-i)}\widehat{\Prism}_A/p.\]
 As $\xi^k \widehat{\Prism}_A \subset \calN^{\geq k} \widehat{\Prism}_A$ for all $k$,  taking $m \geq \frac{pi+1}{p-1}$ then shows that 
 \[\varphi_i(\calN^{\geq m} \widehat{\Prism}_A/p) \subset \calN^{\geq m+1} \widehat{\Prism}_A/p.\]
Thus, for such $m$, not only does $\varphi_i$ preserve $\calN^{\geq m} \widehat{\Prism}_A/p$, but in fact it induces a topologically nilpotent endomorphism of $\calN^{\geq m} \widehat{\Prism}_A/p$. But then $\varphi_i - 1$ is an automorphism of $\calN^{\geq m} \widehat{\Prism}_A/p$, as wanted.
\end{proof}

\newpage
\section{The characteristic $p$ situation}
\label{sec:charp}

The goal of this section is to specialize the previous discussion to $\mathbb F_p$-algebras and prove Theorem~\ref{thm:main3} as well as Theorem~\ref{thm:main6} (1). We begin in \S \ref{subsection_Nygaard} by discussing the Nygaard filtration on the de Rham-Witt complex in multiple different ways. This discussion is put to use in \S \ref{subsection_derived_Nygaard} where we record some structural features of $\mathbb A_{\mathrm{crys}}(S)$ for $S$ quasiregular semiperfect. These tools are then used to prove Theorem~\ref{thm:main3} in \S \ref{subsection_TC_Acrys}. Finally, the explicit description of the Nygaard filtration on $\mathbb{A}_{\mathrm{crys}}(S)$ obtained in \S \ref{subsection_derived_Nygaard} is employed in \S \ref{subsection_log_dRW} to prove Theorem~\ref{thm:main6} (1).

\subsection{The Nygaard filtration on the de~Rham-Witt complex}
\label{subsection_Nygaard}

As preparation, we recall the Nygaard filtration on the de~Rham-Witt complex. Let $k$ be a perfect field of characteristic $p$ and $A$ a smooth $k$-algebra; let $W\Omega_A^\bullet = W\Omega_{A/k}^\bullet$ be the usual de Rham-Witt complex of Bloch--Deligne--Illusie \cite{Illusie1979}. Various versions of the Nygaard filtration have appeared in the literature \cite[II.1]{Kato1987a}, \cite[III.3]{IllusieRaynaud1983}, \cite{Nygaard}; here we fix the version of interest to us and explain its relation to the filtered $L\eta_p$ functor. The general theme will be that the Nygaard filtration is the filtration by the subobjects where $\varphi$ is divisible by $p^i$. As we are dealing with complexes, it is not a priori clear what this means, but it is true on the level of the actual de Rham-Witt complex (for smooth algebras), on the level of the de Rham-Witt complex in the derived category (for smooth algebras) when formulated in terms of the filtered $L\eta_p$, and also on the level of the derived de Rham-Witt complex for quasiregular semiperfect algebras, where the derived de Rham-Witt complex is concentrated in degree $0$.

\begin{definition}\label{definition_Nygaard}
Let $\calN^{\geq i}W\Omega_A^\bullet\subseteq W\Omega_A^\bullet$ be the subcomplex
\[
p^{i-1}VW(A)\to p^{i-2}VW\Omega^1_A\to\cdots\to pVW\Omega^{i-2}_A\to VW\Omega^{i-1}_A\to W\Omega_A^i\to W\Omega_A^{i+1}\to\cdots.\]
This defines a descending, complete multiplicative $\mathbb N$-indexed filtration on $W\Omega_A^\bullet$. We define
\[
\calN^i W\Omega_A^\bullet = \calN^{\geq i} W\Omega_A^\bullet / \calN^{\geq i+1} W\Omega_A^\bullet
\]
as the associated graded.
\end{definition}

Recalling that the groups $W\Omega^j_A$ are $p$-torsion-free, that $FV=p$, and that $\varphi=p^jF$, one sees immediately that the restriction of the absolute Frobenius $\varphi:W\Omega^\bullet_A\to W\Omega^\bullet_A$ to $\calN^{\geq i}W\Omega^\bullet_A$ is uniquely divisible by $p^i$, thereby defining the divided Frobenius
\[
\varphi_i = \frac{\varphi}{p^i}:\calN^{\geq i}W\Omega^\bullet_A\to W\Omega^\bullet_A.
\]
In fact, the proof of Proposition \ref{proposition_frobenius_of_nygaard} below even shows that $\calN^{\geq i}W\Omega^\bullet_A$ is the largest subcomplex of $W\Omega^\bullet_A$ on which $\varphi$ is divisible by $p^i$.

Both the conjugate and Hodge filtration on $\Omega^\bullet_A$ can be recovered from the Nygaard filtration; we begin with the conjugate filtration:

\begin{lemma}
\label{lem:nygaardconj}
The composition
\[
\calN^{\geq i}W\Omega^\bullet_A\xrightarrow{\varphi_i} W\Omega^\bullet_A\to\Omega^\bullet_A
\]
lands in $\tau^{\le i}\Omega^\bullet_{A/k}$ and kills $\calN^{\geq i+1}W\Omega^\bullet_A$. Moreover, the induced map
\[
\varphi_i \mod p:\calN^iW\Omega^\bullet_A\to \tau^{\le i}\Omega^\bullet_A
\]
is a quasi-isomorphism.
\end{lemma}

\begin{proof}
The comments immediately above show that $\varphi_i$ is injective with image given by the complex
\[
W(A)\to\cdots\to W\Omega^{i-1}_A\to FW\Omega^i_A\to pFW\Omega^{i+1}_A\to p^2FW\Omega^{i+2}_A\to\cdots
\]
Since $dF=pFd$, the composition to $\Omega^\bullet_A$ has image in $\tau^{\le i}\Omega^\bullet_A$. Similarly, the restriction of $\varphi_i$ to $\calN^{\geq i+1}W\Omega^\bullet_A$ has image
\[
pW(A)\to\cdots\to pW\Omega^{i-1}_A\to pW\Omega^i_A\to pFW\Omega^{i+1}_A\to p^2FW\Omega^{i+2}_A\to\cdots,\]
which vanishes in $\Omega^\bullet_A$.

Therefore $\varphi_i$ sends $\calN^iW\Omega^\bullet_A$ isomorphically to
\[
W(A)/p\to\cdots\to W\Omega^{i-1}_A/p\to FW\Omega^i_A/pW\Omega^i_A\to 0\to\cdots,
\]
which is precisely the canonical truncation $\tau^{\le i}(W\Omega^\bullet_A/p)$ (which maps quasi-isomorphically to $\tau^{\le i}\Omega^\bullet_A$) since $d^{-1}(pW\Omega^{i+1}_A)=FW\Omega^i_A$ by \cite[Eqn. I.3.21.1.5]{Illusie1979}.
\end{proof}

Noting that clearly $p\calN^{\geq i}W\Omega^\bullet_A\subseteq\calN^{\geq i+1}W\Omega^\bullet_A$, and secondly that the canonical projection $W\Omega^\bullet_A\to\Omega^\bullet_A\to \Omega^{\le i}_A$ kills $\calN^{\geq i+1}W\Omega^\bullet_A$, we now explain how to recover the Hodge filtration:

\begin{lemma}\label{lemma_p_freeness_of_N}
The sequence
\[
W\Omega^\bullet_A/\calN^{\geq i}W\Omega^\bullet_A\xrightarrow{p}W\Omega^\bullet_A/\calN^{\geq i+1}W\Omega^\bullet_A\to \Omega^{\le i}_{A/k}
\]
is a cofiber sequence.
\end{lemma}

\begin{proof}
The indicated multiplication by $p$ map is clearly injective with cokernel \[W(A)/p\xrightarrow{d}\cdots\xrightarrow{d}W\Omega^{i-1}_A/p\xrightarrow{d}W\Omega_A^i/VW\Omega_A^i\xrightarrow{d}0\xrightarrow{d}\cdots.\] The natural map from this to $\Omega^{\le i}_{A/k}$ is a quasi-isomorphism by \cite[Corol.~II.3.20]{Illusie1979}. 
\end{proof}

In the following we recall the well-known result that the divided-Frobenius-fixed points on the Nygaard filtration recover the dlog forms in the de Rham--Witt complex. For any smooth $k$-scheme $X$, denote by $W_r\Omega_{X,\sub{log}}^i\subseteq W_r\Omega_X^i$ the pro-\'etale subsheaf given by the image of the map of pro-\'etale sheaves \[\mathrm{dlog}[\cdot]:\mathbb G_{m,X}^{\otimes i}\to W_r\Omega_{X}^i,\quad f_1\otimes\cdots\otimes f_i\mapsto \frac{d[f_1]}{[f_1]}\wedge\cdots\wedge\frac{d[f_i]}{[f_i]},\] and set $W\Omega^i_{X,\sub{log}}:=\lim_rW_r\Omega^i_{X,\sub{log}}$ as a pro-\'etale sheaf.

\begin{proposition}\label{prop:logforms}
Let $X$ be a smooth $k$-scheme. Then the sequence of complexes of pro-\'etale sheaves
\[0\to W\Omega_{X,\sub{log}}^i[-i]\to\calN^{\ge i}W\Omega^\bullet_X\xrightarrow{\varphi_i-1}W\Omega^\bullet_X\to 0\] is exact (i.e., exact in each degree). Moreover, $W\Omega_{X,\sub{log}}^i$ coincides with the derived inverse limit $\mathrm{Rlim}_rW_r\Omega^i_{X,\sub{log}}$.
\end{proposition}

\begin{proof}
Let $\Spec A$ be an affine open of $X$. Then, in degrees $n>i$, the map $\varphi_i-1$ is given by $p^{n-i}F-1:W\Omega^n_A\to W\Omega^n_A$, which is an isomorphism since $p^{n-i}F$ is $p$-adically contracting. Meanwhile, in degrees $n<i$ there is a commutative diagram
\[\xymatrix@C=2cm{
W\Omega_A^n\ar@/_1cm/[rr]_{1-p^{i-1-n}V}\ar[r]^{p^{i-1-n}V}_{\cong}&\calN^{\ge i}W\Omega^n_A\ar[r]^{\varphi_i-1}&W\Omega^n_A,
}\]
in which the curved arrow (hence also $\varphi_i-1$) is an isomorphism since $W\Omega_A^n$ is $p$-adically complete (resp.\ $V$-adically complete in the boundary case $i=n-1$).

It remains only to analyse the behaviour of $F-1:W\Omega^i_X\to W\Omega^i_X$. To do this, we recall that the sequence of pro-sheaves on $X_\sub{\'et}$ \[0\to \{W_n\Omega_{X,\sub{log}}^i\}_n\to \{W_n\Omega^i_X\}_n\xrightarrow{F-1}\{W_n\Omega^i_X\}_n\to 0\]
is exact \cite[Thm.~I.5.7.2]{Illusie1979}. In particular, taking the inverse limit of this sequence of pro sheaves gives an exact sequence of pro-\'etale sheaves and shows that $W\Omega_{X,\sub{log}}^i$ coincides with the derived inverse limit $\mathrm{Rlim}_rW_r\Omega^i_{X,\sub{log}}$.
\end{proof}

We continue with two different perspectives on the Nygaard filtration.

\subsubsection{Nygaard filtration via $L\eta$}
First, we explain that the Nygaard filtration naturally appears as the canonical filtration on the $L\eta$-functor (Proposition~\ref{LetaDF}) via Ogus's generalization of Mazur's theorem.

\begin{proposition}\label{proposition_frobenius_of_nygaard}
The absolute Frobenius $\varphi:W\Omega^\bullet_A\to W\Omega^\bullet_A$ induces an isomorphism
\[
\varphi: W\Omega^\bullet_A\to \eta_p W\Omega^\bullet_A\subset W\Omega^\bullet_A
\]
of complexes as well as isomorphisms
\[
\varphi:\calN^{\geq i}W\Omega^\bullet_A\cong \Fil^i\eta_pW\Omega^\bullet_A\ ,
\]
of complexes for all $i\ge 0$, where $\Fil^i \eta_p W\Omega^\bullet_A = p^i W\Omega^\bullet_A\cap \eta_p W\Omega^\bullet_A$ is the filtration on $\eta_p$ defined in (the proof of) Proposition~\ref{LetaDF}.
\end{proposition}

\begin{proof}
Since $W\Omega^n_A$ is $p$-torsion-free for all $n\ge0$, the standard relations $\varphi=p^\bullet F$ and $FV=VF=p$ on the de Rham-Witt complex show that $\varphi:\calN^{\geq i} W\Omega^\bullet_A\to W\Omega^\bullet_A$ is injective with image given by the subcomplex \[p^iW(A)\to p^iW\Omega^1_A\to\cdots\to p^iW\Omega^{i-1}_A\to p^iFW\Omega_A^i\to p^{i+1}FW\Omega_A^i\to\cdots.\] But \cite[Eqn.~I.3.21.1.5]{Illusie1979} states that $d^{-1}(pW\Omega^{n+1}_A)=FW\Omega^n_A$, whence this complex is precisely $\Fil^i\eta_pW\Omega^\bullet_A$.
\end{proof}

Given a smooth $k$-variety $X$, we define the Nygaard filtration on $Ru_*\calO_{X/W(k)}^{\mathrm{crys}}$ to be that induced by the Nygaard filtration via Illusie's comparison quasi-isomorphism $Ru_*\calO_{X/W(k)}^{\mathrm{crys}}\simeq W\Omega^\bullet_A$. Here $u: X_{\mathrm{crys}}\to X_{\mathrm{Zar}}$ denotes the projection from the crystalline site to the Zariski site.

\begin{corollary}
Let $X$ be a smooth $k$-variety. Then Berthelot--Ogus' quasi-isomorphism $\varphi:Ru_*\calO_{X/W(k)}^{\mathrm{crys}}\simeq L\eta_pRu_*\calO_{X/W(k)}^{\mathrm{crys}}$ may be upgraded to a filtered quasi-isomorphism, in which the source has the Nygaard filtration and the target has the filtered d\'ecalage filtration.
\end{corollary}

\begin{proof}
This is the content of the previous lemma since $\varphi$ is given by the absolute Frobenius after identifying $Ru_*\calO_{X/W(k)}^{\mathrm{crys}}$ with $W\Omega^\bullet_A$.
\end{proof}

\subsubsection{Nygaard filtration in the presence of smooth lift}
Recall that if $\tilde{A}$ is the $p$-adic completion of a smooth $W(k)$-algebra lifting $A$, there is a natural quasi-isomorphism
\[
\Omega_{\tilde{A}/W(k)}^\bullet\to W\Omega_A^\bullet\ ,
\]
where the left-hand side is understood to be $p$-completed. Assuming that a Frobenius lift $\tilde{\varphi}: \tilde{A}\to \tilde{A}$ has been chosen, we explain how to identify the Nygaard filtration under this isomorphism. We expect that the Nygaard filtration (in filtration degrees $\geq p$) cannot be obtained without the choice of a Frobenius lift.

In the following proposition, the complex $p^{\max(i-\bullet,0)}\Omega^\bullet_{\tilde A/W(k)}$ denotes
\[
p^i\tilde A\xrightarrow{d}p^{i-1}\Omega^1_{\tilde A/W(k)}\xrightarrow{d}\cdots\xrightarrow{d}p\Omega_{\tilde A/W(k)}^{i-1}\xrightarrow{d}\Omega_{\tilde A/W(k)}^i\xrightarrow{d}\Omega_{\tilde A/W(k)}^{i+1}\xrightarrow{d}\cdots
\]

\begin{proposition}\label{prop:nygaardlift}
The comparison map $\sigma:\Omega^\bullet_{\tilde A/W(k)}\to W\Omega^\bullet_A$ induces quasi-isomorphisms \[\sigma:p^{\max(i-\bullet,0)}\Omega^\bullet_{\tilde A/W(k)}\to\calN^{\geq i}W\Omega^\bullet_{A}\] for all $i\ge0$.
\end{proposition}

\begin{proof}
We first recall that construction of $\sigma$. The lifted Frobenius $\tilde\varphi$ induces the (unique) Dieudonn\'e-Cartier-Witt homomorphism $\delta:\tilde A\to W(A)$ compatible with the Frobenius maps and the projections to $A$. This in turn induces Illusie's comparison map
\[
\Omega^\bullet_{\tilde A/W(k)}\to W\Omega^\bullet_A
\]
(which is a quasi-isomorphism). In fact, it may be quickly seen that $\sigma$ is a quasi-isomorphism as the composition $\Omega^\bullet_{\tilde A/W(k)}/p\xrightarrow{\sigma} W\Omega^\bullet_A/p\simeq \Omega^\bullet_{A/k}$ is the identity map, whence $\sigma$ modulo $p$ is a quasi-isomorphism, which is enough to deduce that it is a quasi-isomorphism.

Note that $\sigma$ indeed maps $p^{\max(i-\bullet,0)}\Omega^\bullet_{\tilde A/W(k)}$ to $\mathcal N^{\geq i} W\Omega^\bullet_{A}$, since $p=VF$. To prove that it is a quasi-isomorphism we proceed by induction on $i\ge0$, the case $i=0$ having already been treated by the previous paragraph. Easily calculating the graded pieces of the filtration (in particular, we point out that the graded pieces of the filtration on $\Omega^\bullet_{\tilde A/W(k)}$ have zero differential), one must check that each of the maps \[p\sigma:\Omega^j_{\tilde A/W(k)}/p\to \tfrac{VW\Omega^{j}_{A}}{pVW\Omega^{j}_{A}}\qquad (0\le j<i),\qquad\sigma:\Omega^i_{\tilde A/W(k)}/p\to \tfrac{W\Omega^i_{A}}{VW\Omega^i_{A}}\] induces an isomorphism
\[\Omega^j_{\tilde A/W(k)}/p\to H^j:=H^j\big(\tfrac{VW(A)}{pVW(A)}\xrightarrow{pd}\tfrac{VW\Omega^1_{A}}{pVW\Omega^1_{A}}\xrightarrow{pd}\cdots \xrightarrow{pd}\tfrac{VW\Omega^{i-1}_{A}}{pVW\Omega^{i-1}_{A}}\xrightarrow{d}\tfrac{W\Omega^i_{A}}{VW\Omega^i_{A}}\big)\] for $0\le j\le i$. The de Rham--Witt identities already used in the proof of Proposition~\ref{proposition_frobenius_of_nygaard} easily show that \[H^j=\begin{cases}\tfrac{pW\Omega^i_A}{pVW\Omega^i_A+pdV\Omega^{i-1}_A} & j<i\\\tfrac{W\Omega^i_A}{VW\Omega^i_A+dV\Omega^{i-1}_A} & j=i,\end{cases}\] which is isomorphic via the restriction map (and dividing out the extraneous copy of $p$ when $j<i$) to $\Omega^j_{A/k}$. Therefore the map which must be checked to be an isomorphism is simply the canonical identification $\sigma:\Omega^j_{\tilde A/W(k)}/p\stackrel{=}\to\Omega^j_A$, completing the proof.
\end{proof}

\subsection{The case of quasiregular semiperfect rings}\label{subsection_derived_Nygaard}

As in Construction~\ref{LKE}, we define the derived de~Rham-Witt complex $LW\Omega_{(-)}$ and its Nygaard filtration $\calN^{\geq \f}LW\Omega_{(-)}$ on the category of all simplicial $\mathbb F_p$-algebras via left Kan extension from the category of smooth $\mathbb F_p$-algebras, as functors to the $\infty$-category of $p$-complete $E_\infty$-algebras in $DF(\mathbf{Z}_p)$.

Our goal will be to study these in the case of {\em quasiregular semiperfect} $\mathbb F_p$-algebras (Definition \ref{definition_qrsp}), i.e.~quasiregular semiperfectoid rings of characteristic $p$. As is relatively well-known, for such rings the above theories are closely related to divided powers and crystalline period rings. However, here we want to emphasize that the relevant filtration is not the Hodge filtration (corresponding to the divided power filtration) but rather the Nygaard filtration.

The results of \S \ref{subsection_Nygaard} immediately induce derived analogues, as we now explain. By taking the first part of Lemma~\ref{lem:nygaardconj} and left Kan extension, we obtain a natural fiber sequence
\begin{equation}
\calN^{\geq i+1} LW\Omega_{(-)}\to \calN^{\geq i} LW\Omega_{(-)}\xrightarrow{\varphi_i\mathrm{ mod }p} L\tau^{\le i}\Omega^\bullet_{-/\mathbb F_p}
\label{eqn_N1}
\end{equation}

Secondly, Lemma~\ref{lemma_p_freeness_of_N} implies the existence of a natural fiber sequence
\begin{equation}
LW\Omega^\bullet_{(-)}/\calN^{\geq i}LW\Omega^\bullet_{(-)}\xrightarrow{p} LW\Omega^\bullet_{(-)}/\calN^{\geq i+1}LW\Omega^\bullet_{(-)}\to L\Omega^{\le i}_{-/\mathbb F_p},\label{eqn_N2}
\end{equation}
where the quotients in the first and middle terms really denote cofibers.

Now we wish to compute the derived de Rham--Witt cohomology of the following class of rings:

\begin{definition}\label{definition_qrsp}
An $\mathbb F_p$-algebra $S$ is called {\em semiperfect} if and only if the Frobenius $\varphi:S\to S$ is surjective; in other words, the canonical map $S^\flat\to S$ is surjective, where $S^\flat:=\lim_\varphi S$ is the inverse limit perfection of $S$.

A semiperfect $\mathbb F_p$-algebra $S$ is {\em quasiregular} if and only if $L_{S/\mathbb F_p}$ (which we note is $\simeq\mathbb L_{S/S^\flat}$) is a flat $S$-module supported in homological degree $1$; in other words, if and only if $S=S^\flat/I$ where $I$ is a quasiregular ideal of $S^\flat$.
\end{definition}

In particular, an $\mathbb F_p$-algebra $S$ is quasiregular semiperfect if and only if it is quasiregular semiperfectoid in the sense of Definition \ref{defqsp}.

\begin{definition}
Given a semiperfect $\mathbb F_p$-algebra $S$, let $\mathbb A_\mathrm{crys}^\circ(S)$ be the divided power envelope of $W(S^\flat)\to S$ (where our divided powers are required to be compatible with those on $(p) \subset W(S^\flat)$), and let $\mathbb A_\mathrm{crys}(S)$ be its $p$-adic completion. Note that $\mathbb A_\mathrm{crys}(S)/p=D_{S^\flat}(I)$ is the divided power envelope of $S^\flat$ along the ideal $I\subseteq S^\flat$.

Denote by $\varphi:\mathbb A_\mathrm{crys}(S)\to\mathbb A_\mathrm{crys}(S)$ the endomorphism induced via functoriality from the absolute Frobenius $\varphi:S\to S$, and define the decreasing {\em Nygaard filtration} on $\mathbb A_\mathrm{crys}(S)$ by
\[
\calN^{\geq i}\mathbb A_\mathrm{crys}(S)=\{x\in \mathbb A_\mathrm{crys}(S):\varphi(x)\in p^i\mathbb A_\mathrm{crys}(S)\}
\]
for $i\ge0$. Let $\widehat{\mathbb A}_\mathrm{crys}(S)$ denote the completion of $\mathbb A_\mathrm{crys}(S)$ with respect to the Nygaard filtration, with its completed Nygaard filtration $\calN^{\geq i}\widehat{\mathbb A}_{\mathrm{crys}}(S)$.
\end{definition}

As usual, we write $\calN^i \mathbb A_{\mathrm{crys}}(S) = \calN^{\geq i}\mathbb A_{\mathrm{crys}}(S)/\calN^{\geq i+1}\mathbb A_{\mathrm{crys}}(S)$ for the induced graded of the Nygaard filtration.

 As a consequence of a comparison with derived de~Rham--Witt cohomology, we will eventually see in Theorem~\ref{theorem_structure_Acrys} that if $S$ is quasiregular semiperfect, then $\mathbb A_\mathrm{crys}(S)$ is $p$-torsion-free. However we first need an additional piece of structural information about $\mathbb A_\mathrm{crys}(S)$, namely the conjugate filtration on $\mathbb A_\mathrm{crys}(S)/p$.

\begin{definition}
If $A$ is an $\mathbb F_p$-algebra and $I\subseteq A$ an ideal with divided power envelope $D_A(I)$, the increasing {\em conjugate} filtration \[0=\Fil_{-1}^\mathrm{conj}D_A(I)\subseteq \Fil_0^\mathrm{conj}D_A(I)\subseteq\cdots\] on $D_A(I)$ is the filtration by $A$-submodules defined by letting $\Fil_n^\mathrm{conj}D_A(I)$ be the $A$-submodule generated by elements of the form $a_1^{[l_1]}\cdots a_m^{[l_m]}$, where $m\ge0$, $a_1,\dots,a_m\in I$, and $\sum_{i=1}^ml_i<(n+1)p$.
\end{definition}

\begin{proposition}\label{proposition_conjugate_filtration}
The conjugate filtration on $D_A(I)$ has the following properties:
\begin{enumerate}
\item It is multiplicative and exhaustive.
\item $\Fil^\mathrm{conj}_nD_A(I)$ is the $A$-submodule of $D_A(I)$ generated by elements of the form $a_1^{[pk_1]}\cdots a_m^{[pk_m]}$, where $m\ge0$, $a_1,\dots,a_m\in I$, and $\sum_{i=1}^mk_i\le n$.
\item There is a well-defined surjective map of graded $A$-algebras
\[\Gamma^*_{A/I}(I/I^2)\otimes_{A/I,\varphi}A/\varphi(I)\to \gr_*^\mathrm{conj}D_A(I)\ ,\qquad a_1^{[k_1]}\cdots a_m^{[k_m]}\otimes 1\mapsto \big(\prod_{i=1}^m \frac{(pk_i)!}{p^{k_i} k_i!}\big) a_1^{[pk_1]}\cdots a_m^{[pk_m]}.\]
\end{enumerate}
\end{proposition}

We remark that the rational number $\frac{(pk)!}{p^k k!}$ lies in $\mathbb Z_{(p)}^\times$, and in particular is a unit in $\mathbb F_p$. Moreover, one checks that the map in part (3) is compatible with divided powers.

\begin{proof} (1) The filtration is clearly multiplicative and exhaustive. 

(2) Recall that, for any $a\in I$, $k\ge1$, and $0\le r<p$, there is a divided power relation \[a^{[pk+r]}=a^{[pk]}a^{[r]}\tfrac{(pk)!r!}{(pk+r)!},\] where the fraction $(pk!)/(pk+r)!$ on the right side is a $p$-adic unit; since $r!a^{[r]}=a^r\in A$, we have shown that $a^{[pk+r]}$ belongs to the $A$-submodule generated by $a^{[pk]}$. By writing the exponent of each generator as $l_i=pk_i+r_i$, one easily proves (2).

(3) Note that, if $a,b\in I$, then $(ab)^{[pk]}=p!(a^k)^{[p]}b^{[pk]}=0$ for all $k\ge1$; the same argument as in (2) then shows that $(ab)^{[l]}\in \Fil_0^\mathrm{conj}D_A(I)$ for all $l\ge0$ (it even vanishes if $l\ge p$). Recalling the behaviour of divided powers on the sum of two elements now reveals that, if $x\in I^2$, then $x^{[l]}\in \Fil_0^\mathrm{conj}D_A(I)$ for all $l\ge 0$. This shows that the desired map is well-defined, i.e., depends only on the residue classes of $a_i$ modulo $I^2$.

Moreover, the map is surjective, as all generators of the divided power algebra are in the image. The factor in front is chosen so as to make the map compatible with multiplication.
\end{proof}

The conjugate filtration is actually related to the conjugate filtration on (derived) de~Rham cohomology.

\begin{proposition}\label{prop:Acrysmodp} Let $S$ be a quasiregular semiperfect $\mathbb F_p$-algebra and $I=\ker(S^\flat\to S)$. Then $L\Omega_{S/\mathbb F_p}\simeq L\Omega_{S/S^\flat}$ is concentrated in degree $0$, and there is a natural isomorphism
\[
\mathbb A_{\mathrm{crys}}(S)/p\cong L\Omega_{S/\mathbb F_p}\ .
\]
Under this isomorphism, the conjugate filtration on $\mathbb A_\mathrm{crys}(S)/p = D_{S^\flat}(\ker(S^\flat\to S))$ agrees with the conjugate filtration $L\tau^{\leq n}\Omega_{S/\mathbb F_p}$, and the surjective map
\[
\Gamma_S^*(I/I^2)\to \gr_*^\mathrm{conj}(\mathbb A_\mathrm{crys}(S)/p)
\]
from Proposition~\ref{proposition_conjugate_filtration}(3) is an isomorphism (note: since $A=S^\flat$ is perfect, we may omit the $\otimes_{A/I,\varphi}A/\varphi(I)$ from \ref{proposition_conjugate_filtration}(3).)

Moreover, the divided power filtration on $\mathbb A_{\mathrm{crys}}(S)/p$ gets identified with the Hodge filtration $L\Omega^{\geq i}_{S/\mathbb F_p}$.
\end{proposition}

\begin{proof} Concentration in degree $0$ follows from $L_{S/\mathbb F_p}\simeq L_{S/S^\flat}$ being a flat module in degree $1$. By \cite[Proposition 3.25]{BhargavPAdicddR}, there is a comparison map
\[
L\Omega_{S/\mathbb F_p}\to \mathbb A_{\mathrm{crys}}(S)/p\ .
\]
By \cite[Theorem 3.27]{BhargavPAdicddR}, this is an isomorphism if $S$ is the quotient of a perfect $\mathbb F_p$-algebra by a regular sequence; we actually only need the case $S=\mathbb F_p[X^{1/p^\infty}]/X$ and tensor products thereof (and the case of perfect rings themselves). In fact, the proof of \cite[Theorem 3.27]{BhargavPAdicddR} even shows compatibility with the conjugate and Hodge filtrations, and the divided power structures on associated gradeds for the conjugate filtration. Thus, the proposition holds true in this case.

Now for any quasiregular semiperfect $\mathbb F_p$-algebra $S$, we may consider the quasiregular semiperfect ring $\tilde{S}=S^\flat[X_i^{1/p^\infty}, i\in I]/(X_i, i\in I)$ where $i$ ranges over all $i\in I=\ker(S^\flat\to S)$; this maps surjectively onto $S$ via sending $X_i^{1/p^n}$ to the image of $i^{1/p^n}$ in $S$. Then also $L_{\tilde{S}/\mathbb F_p}[-1]=\tilde{I}/\tilde{I}^2\to L_{S/\mathbb F_p}[-1]=I/I^2$ is surjective, and hence the same is true on all divided powers. It follows that both $L\Omega_{\tilde{S}/\mathbb F_p}\to L\Omega_{S/\mathbb F_p}$ and $\mathbb A_{\mathrm{crys}}(\tilde{S})/p\to \mathbb A_{\mathrm{crys}}(S)/p$ are surjective, and in fact the maps on all graded pieces are surjective. The result is true for $\tilde{S}$ as it is a filtered colimit of tensor products of algebras for which we know the result.

This has the consequence that $L\Omega_{S/\mathbb F_p}\to \mathbb A_{\mathrm{crys}}(S)/p$ is surjective, and preserves the filtrations; on associated gradeds for the conjugate filtration, one gets the surjections
\[
\Gamma_S^*(I/I^2)\cong \gr_* L\Omega_{S/\mathbb F_p}\to \gr_*^{\mathrm{conj}}(\mathbb A_{\mathrm{crys}}(S)/p)\ .
\]

To finish the proof, it suffices to show that $L\Omega_{S/\mathbb F_p}\to S$ is a divided power thickening; indeed, this will induce an inverse map $\mathbb A_{\mathrm{crys}}(S)/p\to L\Omega_{S/\mathbb F_p}$ by the universal property. For this, we will actually show that $LW\Omega_{S}\to S$ is a divided power thickening. Note that $LW\Omega_{S}$ is concentrated in degree $0$ and flat over $\mathbb Z_p$, as its reduction modulo $p$ is $L\Omega_{S}$. Thus, it suffices to see that for any $x\in \ker(LW\Omega_{S}\to S)$, all $x^n$ lie in $n!\ker(LW\Omega_{S}\to S)$. For this, we can again replace $S$ by $\tilde{S}$. But then we have a natural map
\[
LW\Omega_{\tilde{S}}\to \mathbb A_\mathrm{crys}(\tilde{S})
\]
by left Kan extension of the equivalence $W\Omega_A\simeq R\Gamma_{\mathrm{crys}}(A/\mathbb Z_p)$ in the case of smooth $\mathbb F_p$-algebras. This map is an equivalence as it is an equivalence modulo $p$ by what we have already shown, and $\mathbb A_{\mathrm{crys}}(\tilde{S})$ is $p$-torsion free for a ring $\tilde{S}$ of the form $S^\flat[X_i^{1/p^\infty}, i\in I]/(X_i, i\in I)$, cf.~\cite[Proposition 4.1.11]{ScholzeWeinstein} and the discussion before it. In particular, we have the desired divided powers.
\end{proof}

At the end of the proof, we have used the derived de~Rham-Witt cohomology of $S$. Its basic properties are as follows.

\begin{proposition}\label{proposition_ddrw_of_qrsp}
Let $S$ be a quasiregular semiperfect $\mathbb F_p$-algebra. Then:
\begin{enumerate}
\item The derived de~Rham-Witt complex $LW\Omega_S$ is concentrated in degree $0$ and flat over $\mathbb Z_p$.
\item The Nygaard filtration $\calN^{\geq i} LW\Omega_S$ is concentrated in degree $0$ and a submodule of $LW\Omega_{S/\mathbb F_p}$.
\item The map
\[
\varphi_i\mod p: \calN^{\geq i} LW\Omega_S\cong L\tau^{\leq i} \Omega_{S/\mathbb F_p}\to LW\Omega_S/p = L\Omega_{S/\mathbb F_p}
\]
is injective.
\item The map $LW\Omega_S \to S$ is a divided power thickening.
\end{enumerate}
\end{proposition}

\begin{proof} Part (1) follows from $LW\Omega_S/p = L\Omega_{S/\mathbb F_p}$. The second part follows from the description of the graded pieces $\calN^i LW\Omega_S$ in terms of wedge powers of the cotangent complex which are all concentrated in degree $0$; the same holds for part (3). The last part was proved at the end of the proof of the last proposition.
\end{proof}

The following result is the main structural result about $\mathbb A_\mathrm{crys}(S)$ in case $S$ is quasiregular semiperfect; related results may be found in \cite{BhargavPAdicddR, FontaineJannsen}.

\begin{theorem}\label{theorem_structure_Acrys}
Let $S$ be a quasiregular semiperfect ring.
\begin{enumerate}
\item The ring $\mathbb A_\mathrm{crys}(S)$ is $p$-torsion-free.
\item The map
\[
\varphi_i\mod p: \calN^i\mathbb A_\mathrm{crys}(S)\to \mathbb A_\mathrm{crys}(S)/p
\]
is injective and has image $\Fil^{\mathrm{conj}}_i(\mathbb A_\mathrm{crys}(S)/p)$, for each $i\ge0$.
\item There is a natural $\varphi$-equivariant isomorphism $\mathbb A_{\mathrm{crys}}(S)\cong LW\Omega_S$ compatible with the Nygaard filtrations.
\item The image of $\calN^{\geq i} \mathbb A_{\mathrm{crys}}(S)\to \mathbb A_{\mathrm{crys}}(S)/p\cong L\Omega_{S/\mathbb F_p}$ agrees with the Hodge filtration $L\Omega^{\geq i}_{S/\mathbb F_p}$. In particular, the Nygaard-completed $\widehat{\mathbb A}_{\mathrm{crys}}(S)$ reduces modulo $p$ to the Hodge-completed derived de~Rham complex:
\[
\widehat{\mathbb A}_{\mathrm{crys}}(S)/p\cong \widehat{L\Omega}_{S/\mathbb F_p}\ ,
\]
which by Proposition~\ref{prop:Acrysmodp} is also the divided power completion of $\mathbb A_{\mathrm{crys}}(S)/p$.
\item The map $\varphi\mod p: \mathbb A_{\mathrm{crys}}(S)/p\to \mathbb A_{\mathrm{crys}}(S)/p$ satisfies $\varphi(x)=x^p$ for all $x\in \mathbb A_{\mathrm{crys}}(S)/p$.
\end{enumerate}
\end{theorem}

We warn the reader that the Nygaard completion differs from the divided power completion: for $p=2$, the divided power completion of the Nygaard complete ring $\mathbb{Z}_2$ is simply $\mathbb F_2$ as $(2^{2^n}/(2^n)!) = (2)$ in $\mathbb{Z}_2$ for all $n\geq 1$.

\begin{proof} By part (4) of the previous proposition, there is a natural map $\mathbb A_{\mathrm{crys}}(S)\to LW\Omega_S$. At the end of the proof of Proposition~\ref{prop:Acrysmodp}, we proved that this is an isomorphism for $\tilde{S}$ of the form $S^\flat[X_i^{1/p^\infty}, i\in I]/(X_i, i\in I)$. Moreover, in that case, the map is compatible with the isomorphism of Proposition~\ref{prop:Acrysmodp}. By surjectivity of the maps induced by $\tilde{S}\to S$, we see that in general, the map is compatible with that isomorphism. As the target is $p$-torsion-free, it follows that the map $\mathbb A_{\mathrm{crys}}(S)\to LW\Omega_S$ is an isomorphism in general, and in particular part (1) follows.

It follows by induction from part (3) of Proposition~\ref{proposition_ddrw_of_qrsp} that the Nygaard filtration on $LW\Omega_S$ is precisely the submodule on which $\varphi$ is divisible by $p^i$. Thus, we see that the Nygaard filtrations are identified, proving part (3). Then part (2) follows from Proposition~\ref{proposition_ddrw_of_qrsp}~(3).

Part (4) is now a consequence of the derived version of Lemma~\ref{lemma_p_freeness_of_N}. For part (5), we may again reduce to the case of $S^\flat[X_i^{1/p^\infty},i\in I]/(X_i,i\in I)$, and then to $S=\mathbb F_p[X^{1/p^\infty}]/X$ by decomposition into tensor products. Then as an $\mathbb F_p[X^{1/p^\infty}]$-algebra, $\mathbb A_{\mathrm{crys}}(S)/p$ is generated by $X^i/i!$, and both $\varphi(X^i/i!) = X^{pi}/i!$ and $(X^i/i!)^p = X^{pi}/(i!)^p$ are divisible by $p$, as $(pi)!$ is divisible by $p (i!)^p$.
\end{proof}

\begin{remark}[Canonical representative for crystalline cohomology]
\label{rmk:CanRepCrys}
Let $A$ be a regular $\mathbb{F}_p$-algebra, and let $S = A_{\mathrm{perf}}$ denote the perfection of $A$, i.e., the direct limit of $A \xrightarrow{\varphi} A \xrightarrow{\varphi} \cdots$. Write $(S/A)^*$ for the Cech nerve of $A \to S$. We shall explain why\footnote{A careful exposition with additional context has recently been provided by Drinfeld \cite{DrinfeldStacky}.} $R\Gamma_{\mathrm{crys}}(A/\mathbb{Z}_p) \in D(\mathbb{Z}_p)$ is computed by the cochain complex
\[ \mathbb{A}_{\mathrm{crys}}((S/A)^*) := \mathbb{A}_{\mathrm{crys}}(S) \to \mathbb{A}_{\mathrm{crys}}(S \otimes_A S) \to \mathbb{A}_{\mathrm{crys}}(S \otimes_A S \otimes_A S) \to ..., \]
thus giving a canonical cochain complex calculating $R\Gamma_{\mathrm{crys}}(A/\mathbb{Z}_p)$. Note that $A,S \in \Qs_{\mathbb{F}_p}$, and the map $A \to S$ is a quasisyntomic cover of $A$ by a perfect ring: the faithful flatness of $A \to S$ follows from the regularity of $A$ (by the easy direction of Kunz's theorem), whilst Popescu's theorem \cite[Tag 07GB]{StacksProject} ensures that $L_{S/A} \simeq L_{A/\mathbb{F}_p}[1]$ has Tor amplitude concentrated in degree $-1$. By quasisyntomic descent, it is thus enough to show that the $D(\mathbb{Z}_p)$-valued presheaf $\mathbb{A}_{\mathrm{crys}}(-)$ on $\Qsp_{\mathbb{F}_p}$ is a sheaf, and that its unfolding $\mathbb{A}_{\mathrm{crys}}(-)^\beth$ coincides with $R\Gamma_{\mathrm{crys}}(-/\mathbb{Z}_p)$ on regular $\mathbb{F}_p$-algebras. By Theorem~\ref{theorem_structure_Acrys} (3), the first assertion reduces to checking that $B \mapsto LW\Omega_B$ is a sheaf on $\Qs_{\mathbb{F}_p}$, which follows from Example~\ref{pCDDR} and the isomorphism $LW\Omega_B/p \simeq L\Omega_{B/\mathbb{F}_p}$. The second assertion then reduces (by another application of Popescu) to checking that $LW\Omega_B \simeq R\Gamma_{\mathrm{crys}}(B/\mathbb{Z}_p)$ for smooth $\mathbb{F}_p$-algebras $B$. But for such $B$, the canonical map gives an quasi-isomorphism $LW\Omega_B \simeq W\Omega^\bullet_B$: this reduces to the the analogous isomorphism $L\Omega_{B/\mathbb{F}_p} \simeq \Omega^\bullet_{B/\mathbb{F}_p}$ for derived de Rham complex, cf.~\cite[Corollary 3.14]{BhargavPAdicddR}. It remains to observe that there is a canonical isomorphism $W\Omega^\bullet_B \simeq R\Gamma_{\mathrm{crys}}(B/\mathbb{Z}_p)$ by Illusie~\cite[\S II.1]{Illusie1979}.
\end{remark}

\begin{question}[Drinfeld]
Remark~\ref{rmk:CanRepCrys} gives a canonical cochain complex $\mathbb{A}_{\mathrm{crys}}((S/A)^*)$ computing the crystalline cohomology $R\Gamma_{\mathrm{crys}}(A/\mathbb{Z}_p)$ of a regular $\mathbb{F}_p$-algebra $A$. Another such complex is given by the Illusie's de Rham-Witt complex $W\Omega^\bullet_A$. Thus, there is a canonical isomorphism $\mathbb{A}_{\mathrm{crys}}((S/A)^*) \simeq W\Omega^\bullet_A$ in the derived category $D(\mathbb{Z}_p)$.  Is there a natural map (as opposed to a zig-zag) between these two complexes realizing this isomorphism in the derived category?
\end{question}

\subsection{Relation to $\TC^-$}
\label{subsection_TC_Acrys}

Finally, we want to obtain the relation to $\TC^-$, as follows. Recall that by Theorem~\ref{TCqrsp}, for a quasiregular semiperfect $S$, the ring $\widehat{\Prism}_S = \pi_0\TC^-(S;\mathbb Z_p)$ is a $p$-complete $p$-torsion-free $\mathbb Z_p$-algebra complete for the Nygaard filtration $\calN^{\geq i}\widehat{\Prism}_S\subset \widehat{\Prism}_S$, and there are compatible divided Frobenius maps
\[
\varphi_i = \frac{\varphi}{p^i}: \calN^{\geq i} \widehat{\Prism}_S\to \widehat{\Prism}_S\ .
\]
One subtlety is that in addition to the cyclotomic Frobenius map $\varphi$, there is also the map $\varphi^\prime$ induced by the Frobenius endomorphism of $S$, and moreover $\widehat{\Prism}_S/p$ has its own Frobenius endomorphism. These maps turn out to be all the same a posteriori, but we need to distinguish between them in the proof.

The main result is as follows.

\begin{theorem}\label{thm:tccharp} The maps $\varphi$ and $\varphi^\prime$ on $\widehat{\Prism}_S$ agree, and induce the Frobenius map $x\mapsto x^p$ on $\widehat{\Prism}_S/p$. There is a functorial $\varphi$-equivariant isomorphism $\widehat{\Prism}_S\cong \widehat{\mathbb A}_{\mathrm{crys}}(S)\cong \widehat{LW\Omega}_S$ with the Nygaard completion of the derived de Rham-Witt complex that identifies Nygaard filtrations.

The isomorphism $\widehat{\Prism}_S\cong \widehat{LW\Omega}_S$ lifts the isomorphism
\[
\widehat{\Prism}_S/p\cong \pi_0 \HC^-(S/\mathbb F_p)\cong \widehat{L\Omega}_{S/\mathbb F_p}
\]
from Theorem~\ref{TCqrsp}~(5) and Proposition~\ref{prop:hcddr}.
\end{theorem}

This theorem implies Theorem~\ref{thm:main3} by quasisyntomic descent, as $\widehat{LW\Omega}_S$ unfolds to $\widehat{LW\Omega}_A$ for all $A\in \Qs_{\mathbb F_p}$, and restricts to the de Rham-Witt complex on smooth algebras.

\begin{proof} We give the proof as a series of steps. The key step of the proof is the identification for $S=\mathbb F_p[T^{\pm 1/p^\infty}]/(T-1)$. Once that case has been settled, one can show that $\widehat{\Prism}_S\to S$ is a pd thickening, which provides us with a functorial map $\mathbb A_{\mathrm{crys}}(S)\to\widehat{\Prism}_S$, which can be shown to extend to the Nygaard completion and be an isomorphism.\vspace{0.1in}

{\em Preliminaries.} If $S$ is perfect, then the result follows from Proposition~\ref{TCTPTHHpi}. Moreover, for general $S$, we know that modulo $p$, there is a functorial isomorphism
\[
\widehat{\Prism}_S/p\cong \widehat{L\Omega}_{S/\mathbb F_p}\cong \widehat{\mathbb A}_{\mathrm{crys}}(S)/p
\]
by Theorem~\ref{TCqrsp}~(5) and Theorem~\ref{theorem_structure_Acrys}~(4). In particular, by functoriality in $S$, this identifies $\varphi^\prime$ with the Frobenius map of $\widehat{\Prism}_S/p$ by Theorem~\ref{theorem_structure_Acrys}~(5).\vspace{0.1in}

{\em The case of $S=\mathbb F_p[T^{\pm 1/p^\infty}]/(T-1) = \mathbb F_p[\mathbb Q_p/\mathbb Z_p]$.} For this, we use an argument that we learned from Akhil Mathew. Consider the $E_\infty$-ring spectrum $B=\mathbb S[\mathbb Q_p/\mathbb Z_p]$, a spherical group algebra. Then
\[
\THH(S) = \THH(B)\otimes_{\mathbb S} \THH(\mathbb F_p)
\]
as $\THH$ is a symmetric monoidal functor, cf.~\cite[\S IV.2]{NikolausScholze}. On the other hand,
\[\begin{aligned}
\THH(B)\otimes_{\mathbb S}\mathbb Z&=(\THH(B)\otimes_{\mathbb S}\THH(\mathbb Z))\otimes_{\THH(\mathbb Z)} \mathbb Z\\
&=\THH(B\otimes_{\mathbb S}\mathbb Z)\otimes_{\THH(\mathbb Z)} \mathbb Z\\
&=\THH(\mathbb Z[\mathbb Q_p/\mathbb Z_p])\otimes_{\THH(\mathbb Z)} \mathbb Z\\
&=\HH(\mathbb Z[\mathbb Q_p/\mathbb Z_p])
\end{aligned}\]
using Lemma~\ref{THHvsHH}. By \cite[Corollary IV.4.10]{NikolausScholze}, there is a natural $\T$-equivariant map $\mathbb Z\to \THH(\mathbb F_p)$ of $E_\infty$-ring spectra, and so we get a natural $\T$-equivalence
\[\begin{aligned}
\THH(S) &= \THH(B)\otimes_{\mathbb S}\THH(\mathbb F_p) = (\THH(B)\otimes_{\mathbb S}\mathbb Z)\otimes_{\mathbb Z}\THH(\mathbb F_p)\\
&=\HH(\mathbb Z[\mathbb Q_p/\mathbb Z_p])\otimes_{\mathbb Z}\THH(\mathbb F_p)\ .
\end{aligned}\]
Now we claim that for any connective $\T$-equivariant $M\in D(\mathbb Z)$ (such as $M=\HH(\mathbb Z[\mathbb Q_p/\mathbb Z_p])$), the map
\[
M\to M\otimes_{\mathbb Z} \THH(\mathbb F_p)
\]
induces a map
\[
M^{t\T}\to (M\otimes_{\mathbb Z} \THH(\mathbb F_p))^{t\T}
\]
that identifies the target as the $p$-completion of the source. Indeed, the target is always $p$-complete, so we need to prove that it is an equivalence after modding out by $p$. By \cite[Lemma IV.4.12]{NikolausScholze}, it is thus enough to prove that
\[
M^{tC_p}\to (M\otimes_{\mathbb Z}\THH(\mathbb F_p))^{tC_p}
\]
is an equivalence. Both sides commute with writing $M$ as the limit of $\tau_{\leq n} M$ by using Lemma~\ref{Postnikov}, so we can assume that $M$ is bounded, and then by induction concentrated in one degree. We can also assume that $M$ is killed by $p$, and thus an $\mathbb F_p$-vector space. The result is true if $M=\mathbb Z$ by \cite[Corollary IV.4.13]{NikolausScholze}, and thus for $M=\mathbb F_p$. It then follows formally that it holds for arbitrary products of copies of $\mathbb F_p$, and thus for every $\mathbb F_p$-vector space by passing to direct summands.

Applied to $M=\HH(\mathbb Z[\mathbb Q_p/\mathbb Z_p])$, we arrive at an equivalence of $E_\infty$-ring spectra
\[
\HP(\mathbb Z[\mathbb Q_p/\mathbb Z_p];\mathbb Z_p)\simeq \TP(\mathbb F_p[\mathbb Q_p/\mathbb Z_p])\ .
\]
On the other hand, Theorem~\ref{thm:main7} gives an isomorphism
\[
\pi_0 \HP(\mathbb Z[\mathbb Q_p/\mathbb Z_p];\mathbb Z_p)\cong (\widehat{L\Omega}_{\mathbb Z[\mathbb Q_p/\mathbb Z_p]/\mathbb Z})^\wedge_p\ .
\]
If $\tilde{R}$ is a flat $\mathbb Z_p$-algebra with a Frobenius lift and $R=\tilde{R}/p$, there is a natural quasi-isomorphism $(\Omega^\bullet_{\tilde{R}/\mathbb Z_p})^\wedge_p\to W\Omega_R^\bullet$ by Proposition~\ref{prop:nygaardlift} that moreover intertwines the combined $p$-adic and Hodge filtration on the left with the Nygaard filtration on the right. By left Kan extension and using the natural Frobenius lift on $\mathbb Z[\mathbb Q_p/\mathbb Z_p]$, this implies that there is a natural isomorphism
\[
(\widehat{L\Omega}_{\mathbb Z[\mathbb Q_p/\mathbb Z_p]/\mathbb Z})^\wedge_p\simeq \widehat{LW\Omega}_{\mathbb F_p[\mathbb Q_p/\mathbb Z_p]}\ .
\]
In summary,
\[
\pi_0 \TP(\mathbb F_p[\mathbb Q_p/\mathbb Z_p])\cong \pi_0\HP(\mathbb Z[\mathbb Q_p/\mathbb Z_p];\mathbb Z_p)\cong \widehat{LW\Omega}_{\mathbb F_p[\mathbb Q_p/\mathbb Z_p]}\ .
\]
It follows from the construction that this isomorphism is compatible with the isomorphism
\[
\pi_0\TP(S)/p\cong \pi_0\HP(S/\mathbb F_p)\cong \widehat{L\Omega}_{S/\mathbb F_p}\ .
\]
This shows that for $S=\mathbb F_p[\mathbb Q_p/\mathbb Z_p]$, there is indeed an isomorphism of rings
\[
\pi_0 \TP(S)\cong \widehat{\mathbb A}_{\mathrm{crys}}(S)\ .
\]
The same arguments apply to $\mathbb F_p[T^{\pm 1/p^\infty}]$ (which is a perfect ring for which we already know the result), and so we see by functoriality that the map
\[
W(\mathbb F_p[T^{\pm 1/p^\infty}]) = \pi_0 \TP(\mathbb F_p[T^{\pm 1/p^\infty}])\to \pi_0\TP(S) = \widehat{\mathbb A}_{\mathrm{crys}}(S)
\]
is the natural injective map. On its image, we know that $\varphi=\varphi^\prime$. As the map
\[
W(\mathbb F_p[T^{\pm 1/p^\infty}])[\tfrac 1p]\to \widehat{\mathbb A}_{\mathrm{crys}}(S)[\tfrac 1p]
\]
has dense image, it follows that $\varphi=\varphi^\prime$ on $\widehat{\Prism}_S = \pi_0 \TP(S)$, and agrees with the Frobenius of $\widehat{\mathbb A}_{\mathrm{crys}}(S)$. As on $\calN^{\geq i}\widehat{\Prism}_S$, the Frobenius $\varphi$ is divisible by $p^i$, it follows that it maps into $\calN^{\geq i} \widehat{\mathbb A}_{\mathrm{crys}}(S)$. To prove that they agree, we argue by induction and use the short exact sequence~\eqref{eqn_N2}. Using this, it is enough to prove that
\[
\widehat{\Prism}_S/(p\widehat{\Prism}_S + \calN^{\geq i+1}\widehat{\Prism}_S) = L\Omega^{\leq i}_{S/\mathbb F_p}
\]
as quotients of $\widehat{\Prism}_S/p = \widehat{\mathbb A}_{\mathrm{crys}}(S)/p = \widehat{L\Omega}_{S/\mathbb F_p}$. But the Nygaard filtration was defined in terms of the abutment filtration for the Tate spectral sequence, and modulo $p$ this reduces to the abutment filtration for the Tate spectral sequence for $\pi_0 \HP(S/\mathbb F_p)$ (equivalently, the homotopy fixed point spectral sequence for $\pi_0 \HC^-(S/\mathbb F_p)$), which was identified with the Hodge filtration in the proof of Proposition~\ref{prop:hcddr}. Thus, the theorem holds true for $S=\mathbb F_p[T^{\pm 1/p^\infty}]/(T-1)$, or equivalently for $S=\mathbb F_p[T^{1/p^\infty}]/T$. \vspace{0.1in}

{\em More reductions.} If the result holds true for $S_1$ and $S_2$, then it holds true for $S_1\otimes_{\mathbb F_p} S_2$, as both $\widehat{\Prism}_{S_1\otimes S_2}$ and $\widehat{\mathbb A}_{\mathrm{crys}}(S_1\otimes S_2)$ are given by completions of $\widehat{\Prism}_{S_1}\otimes_{\mathbb Z_p} \widehat{\Prism}_{S_2}$ respectively $\widehat{\mathbb A}_{\mathrm{crys}}(S_1)\otimes_{\mathbb Z_p} \widehat{\mathbb A}_{\mathrm{crys}}(S_2)$ for the tensor product of the Nygaard filtrations; indeed, this can be checked modulo $p$, where it follows from the above discussion. By passage to tensor products and filtered colimits, the theorem holds true for any algebra of the form $R[X_i^{1/p^\infty},i\in I]/(X_i,i\in I)$, where $R$ is any perfect algebra.\vspace{0.1in}

{\em The general case.} For a general ring $S$, we know now that if $\tilde{S}=S^\flat[X_i^{1/p^\infty},i\in I]/(X_i,i\in I)$, where $I=\ker(S^\flat\to S)$, which has its natural surjection $\tilde{S}\to S$, then the theorem holds true for $\tilde{S}$. The natural map $\widehat{\Prism}_{\tilde{S}}\to \widehat{\Prism}_S$ is surjective, and is surjective on all steps of the Nygaard filtration (by checking on associated gradeds for the Nygaard filtration). In particular, we see that $\varphi=\varphi^\prime$ on $\widehat{\Prism}_S$. Moreover, we see that the ideal $\ker(\widehat{\Prism}_S\to S)$ has divided powers, by reduction to the case of $\tilde{S}$. In particular, we get a functorial map $\mathbb A_{\mathrm{crys}}(S)\to \widehat{\Prism}_S$. This map is compatible with the Nygaard filtration, again by reduction to the case of $\tilde{S}$. Therefore, it induces a functorial map $\widehat{\mathbb A}_{\mathrm{crys}}(S)\to \widehat{\Prism}_S$. By reduction to the case of $\tilde{S}$, this is surjective, and induces surjections on all steps of the Nygaard filtration. To finish the proof, it remains to see that the map $\widehat{\mathbb A}_{\mathrm{crys}}(S)/p\to \widehat{\Prism}_S/p$ is an isomorphism. But this is an endomorphism $\widehat{L\Omega}_{S/\mathbb F_p}\to \widehat{L\Omega}_{S/\mathbb F_p}$, and we know that for $\tilde{S}$ it is the identity endomorphism. Thus, the same holds true for $S$, as desired.
\end{proof}

A consequence of the discussion is the following description of $\THH$ and $\THH^{tC_p}$, and a version of the Segal conjecture.

\begin{corollary}\label{cor:segalcharp} Let $A$ be a smooth $k$-algebra of dimension $d$. For all $i\in \mathbb Z$, there are natural isomorphisms
\[
\gr^i \THH(A)\simeq (\tau^{\leq i} \Omega_{A/k})[2i]
\]
and
\[
\gr^i \THH(A)^{tC_p}\simeq \Omega_{A/k}[2i]\ ,
\]
where the filtration on $\THH(-)^{tC_p}$ is defined as usual via quasisyntomic descent of the double-speed Postnikov filtration. Under this equivalence, the map $\varphi: \gr^i \THH(A)\to \gr^i \THH(A)^{tC_p}$ is the natural map $\tau^{\leq i} \Omega_{A/k}\to \Omega_{A/k}$.

In particular,
\[
\varphi: \THH(A)\to \THH(A)^{tC_p}
\]
is an equivalence in degrees $\geq d$.
\end{corollary}

The last part was observed earlier by Hesselholt \cite[Proposition 6.6]{HesselholtZeta}.

\begin{proof} The identificaton $\THH(A)^{tC_p}\simeq \TP(A)/p$ from Proposition~\ref{prop:TCtoTHH} is compatible with filtrations (by checking for quasiregular semiperfect rings) and thus induces equivalences $\gr^i \THH(A)^{tC_p}\simeq W\Omega_A/p[2i] = \Omega_{A/k}[2i]$. Under the general equivalence between $\calN^i\widehat{\Prism}_A$ and $\gr^i \THH(A)[-2i]$, the compatibility with filtrations (Nygaard respectively $L\eta$) of the equivalence $\widehat{\Prism}_A\simeq L\eta_p \widehat{\Prism}_A$ is equivalent to the assertion that the maps $\gr^i \THH(A)\to \gr^i \THH(A)^{tC_p}$ induced by $\varphi$ induce isomorphisms $\gr^i \THH(A)\simeq \tau^{\leq -i} \gr^i \THH(A)^{tC_p}$, giving the result.
\end{proof}

\subsection{$K$-theory, $\TC$, and logarithmic de~Rham-Witt sheaves in characteristic $p$}
\label{subsection_log_dRW}
Here we apply the results obtained so far in this section to analyse the syntomic sheaves from \S\ref{subsection_syntomic} in characteristic $p$, identify them in terms of algebraic $K$-theory, and show that their pushforwards to the \'etale world recover $p$-adic motivic cohomology in its guise as a logarithmic de~Rham-Witt sheaf. 

\begin{lemma}\label{lem:TCcharp}$ $
\begin{enumerate}
\item For any quasiregular semiperfect $\mathbb F_p$-algebra $S$ and $i>0$, the operator \[\varphi_i-1:\calN^{\ge i}\widehat{\mathbb A}_\sub{crys}(S)\to \widehat{\mathbb A}_\sub{crys}(S)\] is surjective.
\item For any $i\ge0$, the operator \[\varphi_i-1:\calN^{\ge i}\widehat{\mathbb A}_\sub{crys}(-)\to \widehat{\mathbb A}_\sub{crys}(-)\] is surjective as a map of sheaves on $\Qsp_{\mathbb F_p}$.
\end{enumerate}
\end{lemma}

\begin{proof}
(1) By $p$-completeness of both sides it is enough to prove surjectivity modulo $p$. When restricted to $\calN^{\geq i+1} \widehat{\mathbb A}_\sub{crys}(S)$, the map
\[
\varphi_i-1: \calN^{\geq i+1} \widehat{\mathbb A}_\sub{crys}(S)\to \widehat{\mathbb A}_\sub{crys}(S)/p
\]
agrees with minus the canonical map, as $\varphi_i$ is divisible by $p$ on $\calN^{\geq i+1} \widehat{\mathbb A}_\sub{crys}(S)$. It follows that the image
\[
\widehat{L\Omega}^{\geq i+1}_{S/\mathbb F_p}\subset \widehat{L\Omega}_{S/\mathbb F_p} = \widehat{\mathbb A}_{\sub{crys}}(S)/p
\]
of $\calN^{\geq i+1} \widehat{\mathbb A}_\sub{crys}(S)$ lies in the image of $\varphi_i-1$; this can be identified with the divided power filtration $\Fil^{i+1}_\sub{pd} \widehat{\mathbb A}_\sub{crys}(S)/p\subset \widehat{\mathbb A}_\sub{crys}(S)/p$ by Theorem \ref{theorem_structure_Acrys} (4).

When restricted to $p\calN^{\geq i-1} \widehat{\mathbb A}_\sub{crys}(S)$, the map
\[
\varphi_i-1: p\calN^{\geq i-1}\widehat{\mathbb A}_\sub{crys}(S)\to \widehat{\mathbb A}_\sub{crys}(S)/p
\]
agrees with
\[
\varphi_{i-1} = \varphi_{i-1}-p: \calN^{\geq i-1}\widehat{\mathbb A}_\sub{crys}(S)\to \widehat{\mathbb A}_\sub{crys}(S)/p\ .
\]
This factors over the map
\[
\calN^{i-1}\widehat{\mathbb A}_{\sub{crys}}(S) = L\tau^{\leq i-1}\Omega_{S/\mathbb F_p}\to \widehat{L\Omega}_{S/\mathbb F_p} = \widehat{\mathbb A}_\sub{crys}(S)/p\ .
\]
Thus, also $L\tau^{\leq i-1}\Omega_{S/\mathbb F_p} = \Fil_{i-1}^\sub{conj} \widehat{\mathbb A}_\sub{crys}(S)/p$ lies in the image of $\varphi_i-1$. But for $i>0$, one has
\[
\Fil_{i-1}^\sub{conj} \widehat{\mathbb A}_\sub{crys}(S)/p + \Fil^{i+1}_\sub{pd} \widehat{\mathbb A}_\sub{crys}(S)/p = \widehat{\mathbb A}_\sub{crys}(S)/p
\]
as in general for all $j\geq 1$,
\[
\Fil_{j-1}^\sub{conj} \widehat{\mathbb A}_\sub{crys}(S)/p + \Fil^{pj}_\sub{pd} \widehat{\mathbb A}_\sub{crys}(S)/p = \widehat{\mathbb A}_\sub{crys}(S)/p\ ,
\]
giving the result.

(2) The case $i>0$ is covered by part (1), which also shows that $\varphi-1\mod p:\widehat{\mathbb A}_\sub{crys}(S)\to \widehat{\mathbb A}_\sub{crys}(S)/p$ hits all of $\Fil^1_\sub{pd}\widehat{\mathbb A}_\sub{crys}(S)/p$. Moreover, the composition \[\widehat{\mathbb A}_\sub{crys}(S)\xrightarrow{\varphi-1} \widehat{\mathbb A}_\sub{crys}(S)\to S\] is surjective, when viewed as a sheaf over $S\in\Qsp_{\mathbb F_p}$, since Artin--Schreier extensions exist in $\Qsp_{\mathbb F_p}$. Since $\widehat{\mathbb A}_\sub{crys}(S)$ is $p$-adically complete and the union of a tower of Artin--Schreier extensions is still a cover in $\Qsp_{\mathbb F_p}$, this proves the desired surjectivity.
\end{proof}

Combining the previous lemma with the identifications of Theorem \ref{thm:tccharp}, we obtain exact sequences of sheaves \[0\to\pi_{2i}\TC(-)\to\pi_{2i}\TC^-(-)\xrightarrow{\varphi^{h\T}-1}\pi_{2i}\TP(-)\to 0\] on $\Qsp_{\mathbb F_p}$. In particular, this shows that the syntomic sheaf $\mathbb Z_p(i)$ on $\Qsp_{\mathbb F_p}$ is concentrated in degree $0$ and identifies with $\pi_{2i} \TC(-)$, which is $p$-torsion-free (since we know from Theorem \ref{thm:tccharp} that $\pi_{2i}\TC^-(S)\cong\mathcal N^{\geq i}\widehat{\mathbb A}_\sub{crys}(S)$, which is $p$-torsion-free by Theorem \ref{theorem_structure_Acrys} -- to be precise, it easily follows from the definition of the Nygaard filtration that $p$-torsion-freeness of $\mathbb A_\sub{crys}(S)$ implies the same for its Nygaard completion). We have proved most of:

\begin{proposition}\label{proposition_Z_p(i)}
Conjecture \ref{conjecture_syntomic} is true in characteristic $p$. More precisely, for any $S\in\Qsp_{\mathbb F_p}$,  the complex $\mathbb Z_p(i)(S)$ is concentrated in degree $0$ and given by the $p$-torsion-free group $\mathbb A_\sub{crys}(S)^{\varphi=p^i}$.
\end{proposition}

\begin{proof}
We claim that the natural map
\[
\alpha: \ker(\mathcal N^{\geq i}\mathbb A_\sub{crys}(S)\xrightarrow{\varphi_i-1} \mathbb A_\sub{crys}(S))\to \ker(\mathcal N^{\geq i}\widehat{\mathbb A}_\sub{crys}(S)\xrightarrow{\varphi_i-1} \widehat{\mathbb A}_\sub{crys}(S))
\]
is an isomorphism. This implies the result, as the left-hand side is in fact equal to $\mathbb A_\sub{crys}(S)^{\varphi = p^i}$, using the definition of the Nygaard filtration. By the definition of the Nygaard filtration, the Frobenius map $\mathbb A_\sub{crys}(S)\to \mathbb A_\sub{crys}(S)$ factors canonically over the Nygaard completion $\widehat{\mathbb A}_\sub{crys}(S)$. In fact, we have a natural commutative diagram
\[\xymatrix{
\mathcal N^{\geq i}\mathbb A_\sub{crys}(S)\ar[r]^{\varphi_i}\ar[d]^\alpha & \mathbb A_\sub{crys}(S)\ar[d]^\alpha\\
\mathcal N^{\geq i}\widehat{\mathbb A}_\sub{crys}(S)\ar[r]^{\varphi_i}\ar[ur]^\beta & \widehat{\mathbb A}_\sub{crys}(S)\ .
}\]
This implies that $\alpha$ is injective. Indeed, assume that $x\in \mathcal N^{\geq i}\mathbb A_\sub{crys}(S)$ satisfies $\varphi_i(x)=x$, and that $\alpha(x)=0$. Then in particular $\varphi_i(x)=\beta(\alpha(x))=0$, and thus $x=\varphi_i(x)=0$.

On the other hand, if $y\in \mathcal N^{\geq i}\widehat{\mathbb A}_\sub{crys}(S)$ satisfies $\varphi_i(y)=y$, then $x=\beta(y)\in \mathbb A_\sub{crys}(S)$ maps to $y$ and satisfies $\varphi(x)=\varphi(\beta(y))=\beta(\varphi(y)) = p^i\beta(y) = p^ix$, and therefore lies in
\[
\mathbb A_\sub{crys}(S)^{\varphi = p^i} = \ker(\mathcal N^{\geq i}\mathbb A_\sub{crys}(S)\xrightarrow{\varphi_i-1} \mathbb A_\sub{crys}(S))\ ,
\]
as desired.
\end{proof}

For smooth $A$, passing to quasisyntomic cohomology yields the following corollary.

\begin{corollary}\label{cor:logdRW} Assume that $A$ is a smooth $k$-algebra and $X=\Spec A$. Let $\lambda: \qs_A^\sub{op}\to X_\sub{pro\'{e}t}$ be the natural map of sites. For all $i\geq 0$, there is a natural isomorphism
\[
R\lambda_\ast \mathbb Z_p(i)\simeq W\Omega_{X,\sub{log}}^i[-i]\ .
\]
\end{corollary}

\begin{proof} This follows from the description of $\gr^i \TC^-$ and $\gr^i \TP$ and Proposition~\ref{prop:logforms}.
\end{proof}

\begin{remark}
In the setting of the previous corollary, it is classical that the projection map $X_{\sub{fppf}}\to X_\sub{\'{e}t}$ sends the sheaf $\mu_{p^n}$ to $W_n\Omega_{X,\sub{log}}^1[-1]$, cf.\ \cite[\S II.5]{Illusie1979}. The previous corollary may be viewed as an analogue for higher weight $p$-adic motivic cohomology in characteristic $p$, as was conjectured by Milne \cite[Remark 1.12]{Milne}, except that we use the quasisyntomic rather than the flat topology.
\end{remark}

We also record the following calculation of connective algebraic $K$-theory. It may be applied, for example, when $S=\mathcal O_{\mathbb C_p}/p$, in which case $\bigoplus_{i\ge0}(\mathbb A_\sub{crys}(S)^{\varphi=p^i})[\tfrac1p]=\bigoplus_{i\ge0}\mathbb B_\sub{crys}^+(S)^{\varphi=p^i}$ is the graded ring defining the Fargues--Fontaine curve, \cite{FarguesFontaine}. This was conjectured in 2013 by the third author (based on evidence in degrees $\leq 2$) and sparked much of this work. Indeed, the formula for $\TC$ as Frobenius fix points on something else made it natural to guess that this ``something else'' should have homotopy groups $\mathbb A_\sub{crys}(S)$, and this is what we have realized here in terms of $\TC^-$ and $\TP$.\footnote{It was also one of the inspirations for \cite{NikolausScholze} as the classical $\TR$ or $\TF$ do not have the right form.}

\begin{corollary}\label{cor:Ktheory}
For any quasiregular semiperfect $\mathbb F_p$-algebra $S$, the $K$-theory $K_*(S;\mathbb Z_p)$ vanishes in odd degrees, while
\[
K_\sub{even}(S;\mathbb Z_p)\cong\bigoplus_{i\ge0}\mathbb A_\sub{crys}(S)^{\varphi=p^i}\ .
\]
\end{corollary}

\begin{proof}
As in \S\ref{subsection_syntomic}, Theorem~\ref{theorem_CMM} allows us to identify
\[
H^0(\mathbb Z_p(i)(S))\cong K_{2i}(S;\mathbb Z_p)\ ,\quad H^1(\mathbb Z_p(i)(S))\cong K_{2i-1}(S;\mathbb Z_p)\ ,
\]
so this follows from Proposition \ref{proposition_Z_p(i)}.
\end{proof}

\newpage
\section{The mixed-characteristic situation}
\label{sec:mixedchar}

The goal of this section is to prove Theorem~\ref{thm:main2}, i.e., we compare $\widehat{\Prism}_A$ with $A\Omega_A$ in case $A$ is the $p$-adic completion of a smooth $\calO_C$-algebra. This goal is realized in \S \ref{subsection_AOmegaComp} using some lemmas in almost mathematics collected together in \S \ref{sec:AlmostHA}. Next, in \S \ref{subsec:NygaardAOmega}, we check that the identification $\widehat{\Prism}_A \simeq A\Omega_A$ constructed earlier carries the Nygaard filtration on the left to the filtration on $A\Omega_A$ coming from its definition via Proposition~\ref{LetaDF}. Finally, in \S \ref{subsec:AdamsOps}, we check that the Adams operations act with weight $i$ on $\gr^i \TC^{-}(A;\mathbb{Z}_p)$ (and variants) for any $A \in \Qs$; even though this statement has nothing to do with $\mathcal{O}_C$, its proof reduces to the case where $A$ is as above, whence we can use Theorem~\ref{thm:main2}.

The following notation will be held fixed throughout this section.

\begin{notation}
Let $C/\mathbb{Q}_p$ be a perfectoid field containing $\mu_{p^\infty}$, and set $A_{\inf} = A_{\inf}(\calO_C)$. Let $\mu := [\underline{\epsilon}] - 1 \in A_{\inf}$ where $\underline{\epsilon} \in \calO_C^\flat = \lim_{x \mapsto x^p} \calO_C/p$ is a compatible system of primitive $p$-power roots of unity. Then $\varphi^{-1}(\mu) \mid \mu$. For $r \geq 1$, set $\xi_r = \frac{\mu}{\varphi^{-r}(\mu)} \in A_{\inf}$ and $\xi = \xi_1$, so $\xi_r \mid \xi_{r+1} \mid \mu$ for all $r$. 
\end{notation}

\subsection{Some almost homological algebra}
\label{sec:AlmostHA}

As preparation, we recall some facts from almost mathematics. We shall be interested in almost mathematics over $A_{\inf}$ in the $p$-complete setting, i.e., we are interested in the quotient $\widehat{D}(A_{\inf}^a)$ of the $\infty$-category $\widehat{D}(A_{\inf})$ of $p$-complete $A_{\inf}$-complexes by the full subcategory of those whose cohomology groups are killed by $W(\mathfrak{m}^\flat)$. Recall that such complexes form a stable subcategory, so that one can pass to the Verdier quotient (for the $\infty$-categorical version, cf.~e.g.~\cite[\S I.3]{NikolausScholze}). In fact, complexes whose cohomology groups are killed by $W(\mathfrak{m}^\flat)$ form a $\otimes$-ideal, so $\widehat{D}(A_{\inf}^a)$ is a symmetric monoidal stable $\infty$-category by \cite[Theorem I.3.6]{NikolausScholze}. The following lemma allows us to replace $W(\mathfrak{m}^\flat)$ in the preceding definition by much smaller ideals:

\begin{lemma}
\label{AlmostIdeal}
For $d \geq 1$, set $J_d = \cup_r (\varphi^{-r}(\mu)^d) \subset A_{\inf}$, so $J_d \subset J_1 \subset W(\mathfrak{m}^\flat)$ for all $d$. Then $p$ is a nonzero divisor on $A_{\inf}/J_d$ for all $d$. Moreover, the $p$-adic completion of any $J_d$ coincides with $W(\mathfrak{m}^\flat)$. 
\end{lemma}

\begin{proof}
Note that $\varphi^{-(r+1)}(\mu) \mid \varphi^{-r}(\mu)$, so we may regard each $J_d$ as a filtering union of the principal ideals $(\varphi^{-r}(\mu)^d)$. To show that $p$ is a nonzerodivisor on $A_{\inf}/J_d$, it thus suffices to show that $p$ is a nonzerodivisor modulo $\varphi^{-r}(\mu)^d$. By Frobenius twisting, we may assume $r=0$. Note that $\mu$ is a nonzero divisor modulo $p$, so $(p,\mu)$ and then also $(p,\mu^d)$ forms a regular sequence. We are now done by the general fact that if $(x,y)$ form a regular sequence in a commutative ring $A$, then $x$ is a nonzerodivisor modulo $y$: if $xa = by$, then $x \mid b$ as $y$ is regular mod $x$, whence $y \mid a$ (as we can divide $xa = x\frac{b}{x} y$ by $x$ as $x$ is a nonzerodivisor in $A$), and thus $x$ is a nonzerodivisor modulo $y$.

It follows from the previous paragraph that $B_d := A_{\inf}/J_d$ is a $p$-torsionfree ring. Moreover, since it is clear that $(J_d,p) = (J_1,p) = (p, W(\mathfrak{m}^\flat))$ as ideals of $A_{\inf}$, the ring $B_d/p$ identifies with $A_{\inf}/(W(\mathfrak{m}^\flat),p) \simeq k$, and is thus perfect and independent of $d$.  But then the $p$-adic completion of $B_d$ is a $p$-torsionfree and $p$-complete ring lifting $k$, and must thus coincide with $W(k)$ for all $d$. This implies in particular that $J_d \subset J_1 \subset W(\mathfrak{m}^\flat)$ give the same ideal on $p$-adic completion, as wanted.
\end{proof}

Consider now the evident natural transformations
\[ L\eta_\mu \to ... \to L\eta_{\xi_{r+1}} \to L\eta_{\xi_r} \to ... \to L\eta_\xi\]
of endofunctors on the full subcategory $D^{\geq 0}_{tf}(A_{\inf})$ of $D^{\geq 0}(A_{\inf})$ where $H^0$ is torsion-free.

\begin{lemma}
\label{LetaAlmost}
Fix $K \in D^{\geq 0}_{tf}(A_{\inf})$ that is $p$-complete. Then each cohomology group of the cofiber $Q$ of the natural map $L\eta_\mu K \to R\lim_r L\eta_{\xi_r} K$ of $p$-complete complexes is killed by $W(\mathfrak{m}^\flat)$.
\end{lemma}

\begin{proof}
As $K$ is $p$-complete, the same holds true for $L\eta_f K$ for any nonzero $f \in A_{\inf}$ by \cite[Lemma 6.19]{BMS}. Applying this for $f = \mu,\xi_r$ and using stability of $p$-completeness under limits, it follows that both $L\eta_\mu K$ and $R\lim_r L\eta_{\xi_r} K$ are $p$-complete, and hence so is $Q$. By Lemma~\ref{AlmostIdeal}, it is enough to show that for each $i$, there is some $d \geq 0$ such that $H^i(Q)$ is annihilated by $J_d$.

As $L\eta$ preserves cohomological amplitude by \cite[Corollary 6.5]{BMS} while $R\lim$ changes it by at most $1$, we may assume after shift that $K \in D^{[0,d]}(A_{\inf})$ for some $d \geq 1$ with $H^0(K)$ torsion-free. We may then represent $K$ by some $K^\bullet$ with $K^i = 0$ for $i < 0$ and for $i > d$, and $K^i$ torsionfree. Consider the following diagram of subcomplexes of $K^\bullet$:
\[ \eta_\mu K^\bullet \subset \ldots \subset \eta_{\xi_s} K^\bullet \subset \ldots \subset \eta_{\xi_{r+1}} K^\bullet \subset \eta_{\xi_r} K^\bullet.\]
As $\mu = \varphi^{-r}(\mu) \cdot \xi_r$, we can write $\eta_\mu K^\bullet = \eta_{\varphi^{-r}(\mu)} \eta_{\xi_r} K^\bullet$. As $K$ has cohomological amplitude $d$, multiplication by $\varphi^{-r}(\mu)^d$ on $\eta_{\xi_r} K^\bullet$ thus factors over $\eta_\mu K^\bullet$. It formally follows that multiplication by $\varphi^{-r}(\mu)^d$ on both the source and target of the map
\[ L\eta_\mu K \to R\lim_{s \geq r} L\eta_{\xi_s} K \simeq R\lim_s L\eta_{\xi_s} K \]
factors over the map. But then $\varphi^{-r}(\mu)^d$ annihilates each $H^i(Q)$. As this is true for all $r$, we have shown that $J_d \cdot H^i(Q) = 0$ for all $i$, as wanted.
\end{proof}

The following technical result shall be used later.

\begin{lemma}
\label{AlmostEltReal}
If $M$ is the $(p,\xi)$-completion of a free $A_{\inf}$-module, then the natural map
\[
M\to \mathrm{Hom}_{A_{\inf}}(W(\mathfrak m^\flat),M)
\]
is an isomorphism.
\end{lemma}

\begin{proof}
As everything in sight is $\xi$-complete and $\xi$-torsionfree, this reduces to checking that $M/\xi \to \mathrm{Hom}(\mathfrak{m}, M/\xi)$ is an isomorphism. Injectivity is clear as $M/\xi$ is a torsionfree $\calO_C$-module. For surjectivity, write $M/\xi \cong \widehat{\bigoplus}_{i \in I} \calO_C$ as the $p$-adic completion of a free $\calO_C$-module. Regard $M/\xi$ as a submodule of $N = \prod_{i \in I} \calO_C$ in the usual way: $M/\xi \subset N$ is the set of sequences $(a_i) \in \prod_{i \in I} \calO_C^\flat$ such that for all $n \geq 0$, we have $|a_i| \leq |p^n|$ for all but finitely many $i \in I$.

Now it is clear for valuative reasons that $N\cong\mathrm{Hom}(\mathfrak{m},N)$. Under this identification, the subgroup $\mathrm{Hom}(\mathfrak{m}, M/\xi) \subset \mathrm{Hom}(\mathfrak{m}, N)$  corresponds to the set of sequences $(a_i) \in N = \prod_{i \in I} \calO_C$ such that, for each $\epsilon \in \mathfrak{m}$, we have $(\epsilon \cdot a_i) \in M/\xi$, i.e., for each such $\epsilon$ and each $n \geq 0$, we must have $|\epsilon \cdot a_i| \leq |p^n|$ for all but finitely many $i \in I$. Applying this condition for $\epsilon = p$ then shows that for all $n \geq 0$, we have $|a_i| \leq |p^{n-1}|$ for all but finitely many $i \in I$; as this holds true for all $n$, we immediately get $(a_i) \in M/\xi \subset N$, as wanted.
\end{proof}

For future reference, we note that the functor $\mathrm{RHom}_{A_{\inf}}(W(\mathfrak{m^\flat}),-)$ kills all the ``almost zero'' objects, i.e., those $M \in \widehat{D}(A_{\inf})$ whose cohomology groups are killed by $W(\mathfrak{m^\flat})$: this follows because 
\[W(\mathfrak{m^\flat}) \widehat{\dotimes}_{A_{\inf}} A_{\inf}/W(\mathfrak{m^\flat}) \simeq 0.\]
In particular, we may regard $\mathrm{RHom}_{A_{\inf}}(W(\mathfrak{m^\flat}),-)$ as a functor $\widehat{D}(A_{\inf}^a) \to \widehat{D}(A_{\inf})$. One can show that this functor is right adjoint to the quotient map. 

\subsection{The comparison map}
\label{subsection_AOmegaComp}

Our goal now is to compare the $\widehat{\Prism}_{(-)}$ theory constructed by unfolding $\pi_0 \TC^{-}$ to the $A\Omega$-theory from \cite{BMS}. We shall need the following variant of the latter that makes sense for all $p$-complete $\calO_C$-algebras:

\begin{construction}[Noncomplete $A\Omega$-complexes for arbitrary rings]
\label{NoncompleteAOmega}
Recall that \cite[Theorem 1.10]{BMS} gives a functor 
\[ A \mapsto A\Omega_A := L\eta_\mu R\Gamma(\mathrm{Spf}(A)_C, \mathbb A_{\inf})\]
from $p$-adic completions of smooth $\calO_C$-algebras to $E_\infty$-$A_{\inf}$-algebras. Two important features of this construction are:
\begin{enumerate}
\item $A\Omega_A$ is $(p,\xi)$-complete.
\item There is a natural isomorphism $A\Omega_A/\xi \simeq L\Omega_{A/\calO_C}$. 
\end{enumerate}
Following Construction~\ref{NoncompleteNygaard}, by left Kan extension in $(p,\xi)$-complete $A_{\inf}$-complexes, we obtain a functor $A\Omega_{(-)}$ on all $p$-complete simplicial commutative $\calO_C$-algebras. This functor satisfies the obvious analogs of (1) and (2) above.  As in Construction~\ref{NoncompleteNygaard}, it follows that $A\Omega_{(-)}$ is a sheaf of $E_\infty$-$A_{\inf}$-algebras on $\Qs_{\calO_C}^\sub{op}$ and takes discrete values on $\Qsp_{\calO_C}^\sub{op}$.
 \end{construction}

We can now state and prove the main theorem.

\begin{theorem}\label{main_theorem}
Let $A$ be an $\calO_C$-algebra that can be written as the $p$-adic completion of a smooth $\calO_C$-algebra. There is a natural isomorphism $\widehat{\Prism}_A \stackrel{\simeq}{\to} A\Omega_A$ of $E_\infty$-$A_{\inf}$-algebras that is compatible with Frobenius.
\end{theorem}

\begin{proof}
Before explaining the proof, let us explain the idea informally. As we understand $\widehat{\Prism}_S$ for $S$ perfectoid, it is easy to construct a comparison map $\widehat{\Prism}_A \to R\Gamma(\mathrm{Spf}(A)_C, \mathbb{A}_{\inf})$. To factor this over $L\eta_\mu$ of the target (and thus producing a map to $A\Omega_A$), we use the criterion from Lemma~\ref{LetaAlmost} as well as the behaviour of Frobenius on the Nygaard filtration on $\widehat{\Prism}_A$ coming from Corollary~\ref{NygaardSmooth}. At the end, this only gives a factorization in the almost category, so we employ a trick involving left Kan extensions to topologically free objects in $\Qsp_{\calO_C}$ to get back to the real world.

Let us now explain the proof as a series of steps. \vspace{0.1in}

{\em A primitive comparison map.}
Let us first construct a functorial $\varphi$-equivariant comparison map 
\[ b_A:\widehat{\Prism}_A \to R\Gamma(\mathrm{Spf}(A)_C, \mathbb A_{\inf})\] 
for the $p$-adic completion $A$ of a smooth $\calO_C$-algebra. For this, observe that for every map $A \to R$ with $R$ perfectoid, we have an induced functorial map $\widehat{\Prism}_A \to \widehat{\Prism}_R \cong A_{\inf}(R)$. As $R\Gamma(\mathrm{Spf}(A)_C, \mathbb A_{\inf})$ can be regarded as a limit of the functor $R\mapsto A_{\inf}(R)$ on a subcategory of perfectoid $A$-algebras, we formally obtain the map $b_A$. \vspace{0.1in}

{\em Constructing the comparison map in the almost category.}
We shall now refine $b_A$ to obtain a functorial comparison map 
\[ c^a_A:\widehat{\Prism}_A^a \to A\Omega_A^a\]
of $E_\infty$-algebras in $\widehat{D}(A_{\inf}^a)$ for the $p$-adic completion $A$ of a smooth $\calO_C$-algebra. Consider the $\infty$-category $\mathcal C$ of $E_\infty$-algebras in $\widehat{D}(A_{\inf}^a)$. The $\infty$-category $\mathcal C$ comes equipped with an endofunctor $F: \mathcal C\to \mathcal C$ given by $A\mapsto L\eta_\xi \varphi_\ast A$. Given any $\infty$-category with an endofunctor $F$, we have the $\infty$-category of fixed points $\mathcal C^F$ of pairs of an object $X\in \mathcal C$ and an equivalence $X\simeq F(X)$. Moreover, we have the $\infty$-category $\mathcal C^{\to F}$ of objects $X\in \mathcal C$ with a map $X\to F(X)$; and the $\infty$-category $\mathcal C^{F\to}$ of objects $X\in \mathcal C$ with a map $F(X)\to X$. If $\mathcal C$ admits sequential limits, then there is an endofunctor $R$ of $\mathcal C^{F\to}$ given by sending $F(X)\to X$ to the inverse limit $R(X)$ of $\ldots \to F(F(X))\to F(X)\to X$ with its natural map $F(R(X))\to R(X)$. If $Y\to F(Y)$ is an object of $\mathcal C^{\to F}$ and $F(X)\to X$ an object of $\mathcal C^{F\to}$ together with a commutative diagram
\[\xymatrix{
Y\ar[r]\ar[d]^f& F(Y)\ar[d]^{F(f)}\\
X&\ar[l] F(X)\ ,
}\]
then this factors canonically over a similar map from $Y\to F(Y)$ to $F(R(X))\to R(X)$.

In our case, $\widehat{\Prism}_A^a$ lies in $\mathcal C^{\to F}$ as there is a natural map
\[
\widehat{\Prism}_A\to L\eta_\xi\varphi_\ast \widehat{\Prism}_A
\]
by Corollary~\ref{NygaardSmooth}. On the other hand, $R\Gamma(\mathrm{Spf}(A)_C,\mathbb A_{\inf})$ lies in $\mathcal C^{F\to }$, as $\varphi$ is an automorphism and there is a natural map $L\eta_\xi R\Gamma(\mathrm{Spf}(A)_C,\mathbb A_{\inf})\to R\Gamma(\mathrm{Spf}(A)_C,\mathbb A_{\inf})$ since $R\Gamma(\mathrm{Spf}(A)_C,\mathbb A_{\inf})\in D_{tf}^{\geq 0}(A_{\inf})$. Moreover, the diagram
\[\xymatrix{
\widehat{\Prism}_A\ar[r]\ar[dd]^-{b_A}& L\eta_\xi\varphi_\ast \widehat{\Prism}_A\ar[r]^{\mathrm{can}} \ar[d]^-{L\eta_\xi \varphi_* b_A} & \varphi_* \widehat{\Prism}_A\ar[d]^-{\varphi_* b_A}\\
& L\eta_\xi\varphi_\ast R\Gamma(\mathrm{Spf}(A)_C,\mathbb A_{\inf})\ar[r]^-{\mathrm{can}} \ar@{-->}[dl] & \varphi_\ast R\Gamma(\mathrm{Spf}(A)_C,\mathbb A_{\inf})\\
R\Gamma(\mathrm{Spf}(A)_C,\mathbb A_{\inf})\ar[rru]^{\varphi}_\simeq
}\]
commutes without the dashed arrow, and thus also with the dashed arrow by inverting the lower equivalence. Thus, by passing to the almost category, we are in the abstract setup, and get a map of $E_\infty$-algebras
\[
\widehat{\Prism}_A^a\to \varprojlim (L\eta_\xi\varphi_\ast)^{\circ r} R\Gamma(\mathrm{Spf}(A)_C,\mathbb A_{\inf})^a
\]
in $\widehat{D}(A_{\inf}^a)$ that commutes with the respective maps to/from their $L\eta_\xi\varphi_\ast -$. But the right-hand side is equivalent to $A\Omega_A^a$ by Lemma~\ref{LetaAlmost}. More precisely, applying the same argument for $A\Omega_A$ with its equivalence to $L\eta_\xi A\Omega_A$, we get a natural map of $E_\infty$-algebras
\[
A\Omega_A^a=L\eta_\mu R\Gamma(\mathrm{Spf}(A)_C,\mathbb A_{\inf})^a\to \varprojlim (L\eta_\xi\varphi_\ast)^{\circ r} R\Gamma(\mathrm{Spf}(A)_C,\mathbb A_{\inf})^a
\]
in $\widehat{D}(A_{\inf}^a)$ that commutes with the respective maps to/from $L\eta_\xi \varphi_\ast -$, and this map is an equivalence by Lemma~\ref{LetaAlmost}.\vspace{0.1in}

{\em Lifting the almost comparison map $c_A^a$ to the real world.} By left Kan extension in $(p,\xi)$-complete $A_{\inf}$-complexes, we obtain an almost map $c_A^a:{\Prism}^{a}_A \to A\Omega^{a}_A$ for any $p$-complete $\calO_C$-algebra $A$ (see Construction~\ref{NoncompleteNygaard} and Construction~\ref{NoncompleteAOmega} for the definitions). For $S\in \Qsp_{\calO_C}$, both ${\Prism}_S$ and $A\Omega_S$ are discrete. Hence, we can identify $c_S^a$ with an honest map ${\Prism}_S \to \mathrm{Hom}_{A_{\inf}}(W(\mathfrak m^\flat), A\Omega_S)$ for $S \in \Qsp_{\calO_C}$, recalling that $\mathrm{Hom}_{A_{\inf}}(W(\mathfrak m^\flat),-)$ is the right adjoint to the forgetful functor $M\mapsto M^a$. Now if $S \in \qsp^\sub{proj}_{\calO_C}$ (see Variant~\ref{FreeQSP} for the definition), then Lemma~\ref{QRSPFree} identifies this with an honest map
\[d_R:{\Prism}_S \to A\Omega_S.\]
In other words, on the category $\qsp_{\calO_C}^\sub{proj} \subset \Qsp_{\calO_C}$ from Variant~\ref{FreeQSP}, we have constructed the comparison map $d_S$ as above. Using the equivalence in the last statement of Variant~\ref{FreeQSP}, we may unfold the map $d_S$ to a functorial comparison map
\[ d_A:{\Prism}_A\to A\Omega_A\]
for any $A\in \qs_{\calO_C}^\sub{proj}$. As $p$-adic completions of smooth $\calO_C$-algebras lie in $\qs_{\calO_C}^\sub{proj}$, this construction restricts to a functorial comparison map on the category of $p$-adic completions of smooth $\calO_C$-algebras. It is also clear by descent that $d_A$ is a map of $E_\infty$-$A_{\inf}$-algebras that is compatible with Frobenius. \vspace{0.1in}

{\em Showing $d_A$ is an isomorphism.} Note that both $\widehat{\Prism}_A/\xi$ and $A\Omega_A/\xi$ are naturally identified with $L\Omega_{A/\calO_C}$. By completeness, to show $d_A$ is an isomorphism, it is enough to check that $d_A/\xi$ is an isomorphism. Using Lemma~\ref{CompCrit}, we must verify the following: 

\begin{claim} If $A$ is the $p$-adic completion of a smooth $\calO_C$-algebra, the map $H^0(d_A/(p,\xi))$ induces the identity map on $(A/p)^{(1)} := A/p \otimes_{\calO_C/p,\varphi} \calO_C/p$ under the isomorphisms 
\[ H^0(\widehat{\Prism}_A/(p,\xi)) \cong H^0(L\Omega_{(A/p)/(\calO_C/p)}) \cong (A/p)^{(1)}\]
and
\[ H^0(A \Omega_A/(p,\xi)) \cong H^0(L\Omega_{(A/p)/(\calO_C/p)}) \cong (A/p)^{(1)}\] 
coming from the Cartier isomorphism. 
\end{claim}

But this follows from the analogous statement in the perfectoid case. Indeed, after a localization, we may choose a cover $A \to R$ in $\qs^\sub{proj}_{\calO_C}$ with $R$ perfectoid. Then $d_{R}$ is the identity map by the construction of the primitive comparison map $b_A$, and hence $H^0(d_{R}/(p,\xi))$ is also the identity map. As $A\to R$ is injective modulo $p$, the map $L\Omega_{(A/p)/(\calO_C/p)} \to L\Omega_{(R/p)/(\calO_C/p)}$ also induces an injective map on $H^0$, and hence it follows that $H^0(d_A/(p,\xi))$ is the identity map.
\end{proof}

The following two lemmas were used above.

\begin{lemma}
\label{QRSPFree}
Assume $S\in \qsp^\sub{proj}_{\calO_C}$ (see Variant~\ref{FreeQSP} for the definition). Then
\[
A\Omega_S \cong \mathrm{Hom}_{A_{\inf}}(W(\mathfrak m^\flat), A\Omega_S)\ .
\]
\end{lemma}

\begin{proof}
As $S$ is $\calO_C$-flat, we have $A\Omega_S/(p,\xi) \simeq L\Omega_{(S/p)/(\calO_C/p)}$. By assumption $S/p$ is a free $\calO_C/p$-module and $L_{(S/p)/(\calO_C/p)}[-1]$ is a projective $S/p$-module. But then $\wedge^i L_{(S/p)/(\calO_C/p)}[-i] \cong \Gamma^i_{(S/p)}(\pi_1 L_{(S/p)/(\calO_C/p)})$ is a projective $S/p$-module (and free over $\calO_C/p$) for all $i$. By (non-cano\-ni\-cal\-ly) splitting the conjugate filtration, one sees that $L\Omega_{(S/p)/(\calO_C/p)}$ is also a free $\calO_C/p$-module. It follows that $A\Omega_S$ is the $(p,\xi)$-completion of a free $A_{\inf}$-module: the sequence $(p,\xi)$ is regular in $A\Omega_S$ as the derived quotient $A\Omega_S/(p,\xi)=L\Omega_{(S/p)/(\calO_C/p)}$ is discrete, and then a basis of $A\Omega_S/(p,\xi)$ lifts to a topological basis of $A\Omega_S$. The claim now follows from Lemma~\ref{AlmostEltReal}.
\end{proof}

\begin{lemma}
\label{CompCrit}
Let $A$ be the $p$-adic completion of a smooth $\calO_C$-algebra. Let $\eta:\Omega_{A/\calO_C} \to \Omega_{A/\calO_C}$ be a map of $p$-completed $E_\infty$-$\calO_C$-algebras with mod $p$ reduction $\overline{\eta}$. If $H^0(\overline{\eta})$ is the identity, then $H^*(\overline{\eta})$ is the identity, and thus $\eta$ is an isomorphism.
\end{lemma}

\begin{proof}
View $H^*(\overline{\eta})$ as a graded endomorphism of a graded ring $R^* := H^*(\Omega_{(A/p)/(\calO_C/p)})$. By the Cartier isomorphism, $R^*$ is generated in degree $1$. As we have assumed $H^0(\overline{\eta})$ is the identity on $R^0$,  it is enough to show that the resulting $R^0$-linear map $H^1(\overline{\eta}):R^1 \to R^1$ is also the identity. Now $H^*(\overline{\eta})$ is compatible with the Bockstein differential $\beta_p:R^0 \to R^1$, so the map $H^1(\overline{\eta})$ acts as the identity on $\beta_p(R^0) \subset R^1$.  But, by the Cartier isomorphism, $R^1$ is generated as an $R^0$-module by $\beta_p(R^0)$: under the Cartier isomorphism, the Bockstein corresponds to the de Rham differential. As $H^1(\overline{\eta})$ is $R^0$-linear, the claim follows.
\end{proof}

\subsection{Nygaard filtrations}
\label{subsec:NygaardAOmega}

Moreover, we identify the Nygaard filtration.

\begin{proposition}\label{prop:identnygaard} Let $A$ be the $p$-adic completion of a smooth $\calO_C$-algebra. The map $\widehat{\Prism}_A\to L\eta_\xi \varphi_\ast \widehat{\Prism}_A$ from Corollary~\ref{NygaardSmooth}~(3) is an isomorphism, and identifies the Nygaard filtration $\calN^{\geq i}\widehat{\Prism}_A$ with the filtration on $L\eta_\xi$ from Proposition~\ref{LetaDF}.
\end{proposition}

In other words, the equivalence $\widehat{\Prism}_A\simeq A\Omega_A$ carries the Nygaard filtration $\calN^{\geq i}\widehat{\Prism}_A$ to the filtration on $A\Omega_A$ coming from the equivalence $A\Omega_A\simeq L\eta_\xi \varphi_\ast A\Omega_A$.

\begin{proof} As the equivalence $\widehat{\Prism}_A\simeq A\Omega_A$ commutes with the maps to their $L\eta_\xi \varphi_\ast -$, the first statement follows from the corresponding statement for $A\Omega_A$.

For the statement on Nygaard filtrations, we need to see that the maps of associated gradeds is an equivalence. We know that $\calN^i\widehat{\Prism}_A$ is a complex of $A$-modules concentrated in degrees $[0,i]$ with cohomology groups $\Omega^j_{A/\calO_C}$, $0\leq j\leq i$, by Proposition~\ref{NygaardGradedFilt}. For the right-hand side, we use the equivalence $L\eta_\xi \varphi_\ast \widehat{\Prism}_A\simeq \varphi_\ast L\eta_{\tilde\xi} A\Omega_A$ to see that the graded pieces are isomorphic to $\varphi_\ast \tau^{\leq i} A\Omega_A/\tilde\xi$. By \cite[Theorem 8.3, Theorem 9.4 (i)]{BMS}, its cohomology groups are also given by $\Omega_{A/\calO_C}^j$ for $0\leq j\leq i$, and $0$ else. Thus, we must check that certain endomorphisms $\Omega^j_{A/\calO_C}\to \Omega^j_{A/\calO_C}$ are isomorphisms. This can be checked after base extension along $A\to A\otimes_{\calO_C} k$, where $k$ is the residue field of $\calO_C$. Then it follows from the results in characteristic $p$.
\end{proof}

\begin{remark}\label{remark_nygaard_in_coords} We briefly explain how to make the Nygaard filtration explicit in coordinates. Recall that if one fixes a framing $\square: \calO_C\langle T_1^{\pm 1},\ldots,T_d^{\pm 1}\rangle\to A$ that is the $p$-adic completion of an \'etale map, one gets a corresponding flat deformation $\tilde{A}$ of $A$ to $A_{\inf}$ (along $\theta: A_{\inf}\to \calO_C$) and there is an explicit complex computing $A\Omega_A$, given by a $q$-de~Rham complex
\[
q\text-\Omega^\bullet_{\tilde{A}/A_{\inf}} = \tilde{A}\to \bigoplus_{i=1}^d \tilde{A}\to \ldots\to \tilde{A}\to 0
\]
that can be defined as a Koszul complex $K_{\tilde{A}}(\frac{\partial_q}{\partial_q \log(T_1)},\ldots,\frac{\partial_q}{\partial_q \log(T_d)})$, cf.~\cite[Definition 9.5]{BMS}. Here, $q=[\epsilon]-1\in A_{\inf}$. Under this equivalence, the map
\[
A\Omega_A\to L\eta_\xi \varphi_\ast A\Omega_A = \varphi_\ast L\eta_{\tilde\xi} A\Omega_A
\]
is given by the map of complexes
\[
\varphi: q\text-\Omega^\bullet_{\tilde{A}/A_{\inf}}\to \varphi_\ast \eta_{\tilde\xi}q\text-\Omega^\bullet_{\tilde{A}/A_{\inf}}\subset \varphi_\ast q\text-\Omega^\bullet_{\tilde{A}/A_{\inf}}
\]
induced by the map $\varphi: \tilde{A}\to \tilde{A}$ sending all $T_i$ to $T_i^p$. Now a direct computation shows that this implies that one can describe the Nygaard filtration as the filtration
\[
\xi^{\max(i-\bullet,0)} q\text-\Omega^\bullet_{\tilde{A}/A_{\inf}}\subset q\text-\Omega^\bullet_{\tilde{A}/A_{\inf}}
\]
as in Proposition~\ref{prop:nygaardlift}.
\end{remark}

Having identified the Nygaard filtration, we can now identify $\THH$ and $\THH^{tC_p}$ more precisely, and verify a version of the Segal conjecture. Recall the complex $\widetilde{\Omega}_A = A\Omega_A\otimes_{A_{\inf},\tilde\theta} A$ from \cite[\S 8]{BMS}, whose cohomology groups are $\Omega^i_{A/\calO_C}\{-i\}$.

\begin{corollary} Let $A$ be the $p$-adic completion of a smooth $\calO_C$-algebra of dimension $d$. For all $i\in \mathbb Z$, there are natural isomorphisms
\[
\gr^i \THH(A;\mathbb Z_p)\simeq (\tau^{\leq i} \widetilde{\Omega}_A\{i\})[2i]
\]
and
\[
\gr^i \THH(A;\mathbb Z_p)^{tC_p}\simeq \widetilde{\Omega}_A\{i\}[2i]\ ,
\]
where the filtration on $\THH(-;\mathbb Z_p)^{tC_p}$ is defined as usual via quasisyntomic descent of the double-speed Postnikov filtration. Under this equivalence, the map $\varphi: \gr^i \THH(A;\mathbb Z_p)\to \gr^i \THH(A;\mathbb Z_p)^{tC_p}$ is the natural map $(\tau^{\leq i} \widetilde{\Omega}_A)\{i\}[2i]\to \widetilde{\Omega}_A\{i\}[2i]$.

In particular,
\[
\varphi: \THH(A;\mathbb Z_p)\to \THH(A;\mathbb Z_p)^{tC_p}
\]
is an equivalence in degrees $\geq d$.
\end{corollary}

\begin{proof} The identificaton $\THH(A;\mathbb Z_p)^{tC_p}\simeq \TP(A;\mathbb Z_p)/\tilde\xi$ from Proposition~\ref{prop:TCtoTHH} is compatible with filtrations (by checking for quasiregular semiperfectoids) and thus induces equivalences
\[
\gr^i \THH(A;\mathbb Z_p)^{tC_p}\simeq A\Omega_A/\tilde\xi\{i\}[2i] = \widetilde{\Omega}_A\{i\}[2i]\ .
\]
Under the general equivalence between $\calN^i\widehat{\Prism}_A$ and $\gr^i \THH(A;\mathbb Z_p)[-2i]$, Proposition~\ref{prop:identnygaard} is equivalent to the assertion that the maps $\gr^i \THH(A;\mathbb Z_p)\to \gr^i \THH(A;\mathbb Z_p)^{tC_p}$ induced by $\varphi$ induce isomorphisms $\gr^i \THH(A;\mathbb Z_p)\simeq \tau^{\leq -i} \gr^i \THH(A;\mathbb Z_p)^{tC_p}$, giving the result.
\end{proof}

\subsection{Adams operations}
\label{subsec:AdamsOps}
A consequence of the functorial identification between $\widehat{\Prism}_A$ and $A\Omega_A$ is the identification of the Adams operations. Let us recall their construction first.

\begin{construction}
Note that $p$-completed $\THH(A;\mathbb Z_p)$ can also be defined as the $p$-completion of $A\otimes_{E_\infty\text-k} \T^\wedge_p$, where we consider the $p$-completion $\T^\wedge_p=K(\mathbb Z_p,1)$ of the circle. This implies that the automorphisms $\mathbb Z_p^\times$ of $\T^\wedge_p$ act functorially on the $E_\infty$-algebra in cyclotomic spectra $\THH(A;\mathbb Z_p)$. In particular, there are natural $\mathbb Z_p^\times$-actions on all objects considered throughout, such as the quasisyntomic sheaves $\widehat{\Prism}_{(-)}$ and the Breuil-Kisin twist $\widehat{\Prism}_{(-)}\{1\}$; these operations are called the Adams operations.
\end{construction}

We can now identify the Adams operations.

\begin{proposition}\label{prop:adams} The $\mathbb Z_p^\times$-action on the quasisyntomic sheaf $\widehat{\Prism}_{(-)}$ is trivial, and the action on $\widehat{\Prism}_{(-)}\{1\}$ is given by the natural multiplication action. The same holds true for all steps of the Nygaard filtration. In particular, $\gamma\in \mathbb Z_p^\times$ acts via multiplication with $\gamma^i$ on $\gr^i \THH(-;\mathbb Z_p)$, $\gr^i \TC^-(-;\mathbb Z_p)$, $\gr^i \TP(-;\mathbb Z_p)$ and $\gr^i \TC(-;\mathbb Z_p)$, for all $i\in \mathbb Z$.
\end{proposition}

\begin{proof} First, note that as $\widehat{\Prism}_S$ is concentrated in degree $0$ for $S$ quasiregular semiperfectoid, the triviality of the action is a condition, not a datum. Moreover, $\calN^{\geq i}\widehat{\Prism}_S\subset \widehat{\Prism}_S$ is an ideal in this case, so if the action is trivial on $\widehat{\Prism}_S$, then this also holds for the Nygaard filtration. Similar remarks apply to the Breuil-Kisin twist.

Assume first that $R$ is perfectoid. Then the universal property of $A_{\inf}(R)=\widehat{\Prism}_R\to R$ as the universal $p$-complete pro-infinitesimal thickening, together with the triviality of the $\mathbb Z_p^\times$-action on $R$, implies that the $\mathbb Z_p^\times$-action on $A_{\inf}(R)$ is trivial. To identify the action on $\widehat{\Prism}_R\{1\}$, we use the identification of Breuil-Kisin twists after Proposition~\ref{prop:breuilkisintwist}; this shows that there are natural isomorphisms
\[
H^2(\T/C_{p^r},\widehat{\Prism}_R\{1\}/\tilde\xi_r) = \ker\tilde\theta_r/(\ker\tilde\theta_r)^2\ ,
\]
equivariant for the $\mathbb Z_p^\times$-action. In particular, the $\mathbb Z_p^\times$-action is trivial on the right. As $\mathbb Z_p^\times$ acts through multiplication by the inverse on $H^2(\T,\mathbb Z_p)=H^2(\T^\wedge_p,\mathbb Z_p)$, we see that $\mathbb Z_p^\times$ must act through multiplication on $\widehat{\Prism}_R\{1\}/\tilde\xi_r$, and then also on the inverse limit $\widehat{\Prism}_R\{1\}$.

By the base change property of $\widehat{\Prism}_{(-)}\{1\}$, it remains to show that in general the $\mathbb Z_p^\times$-action on $\widehat{\Prism}_S$ is trivial if $S$ is quasiregular semiperfectoid. We may assume that $S$ is an $\calO_C$-algebra by passing to a quasisyntomic cover. Going through the proof of the equivalence between $\widehat{\Prism}_A$ and $A\Omega_A$, we see that all maps are equivariant for the $\mathbb Z_p^\times$-action when the source is equipped with the Adams operations and the target with the trivial action. Moreover, this equivariance persists for the Nygaard filtration. Thus, the $\mathbb Z_p^\times$-action on the $E_\infty$-algebra $\widehat{\Prism}_A$ in $\widehat{DF}(A_{\inf})$ is functorially trivial on the category of $p$-adic completions of smooth $\calO_C$-algebras. By left Kan extension, it follows that the $\mathbb Z_p^\times$-action on $\widehat{\Prism}_S$ is trivial if $S$ is quasiregular semiperfectoid and admits an $\calO_C$-algebra structure.
\end{proof}

\newpage
\section{$p$-adic nearby cycles}
\label{sec:nearby}

Our goal in this section is to prove the following theorem. On the one hand, this gives a very precise assertion relating $p$-adic nearby cycles to syntomic cohomology; on the other hand, as explained by the second author in \cite{MorrowNearby}, it is also closely related to the results of Geisser-Hesselholt, \cite{GeisserHesselholt}.

In the following result, $\mathbb Z/p^n\mathbb Z(i)$ denotes the usual \'etale sheaf on the space $X$ on which $p$ is invertible, and on the $p$-adic formal scheme $\mathfrak X$, it denotes the syntomic sheaf of complexes from \S\ref{subsection_syntomic}
\[
\mathbb Z/p^n\mathbb Z(i) = \mathbb Z_p(i)/p^n = \mathrm{hofib}(\varphi - 1: \mathcal N^{\geq i}\widehat{\Prism}\{i\}\to \widehat{\Prism}\{i\})/p^n\ ,
\]
which we in fact restrict to the \'etale site of $\mathfrak X$, where Theorem \ref{main_theorem} tells us that we get the complex
\[
\mathbb Z/p^n\mathbb Z(i) = \mathrm{hofib}(\varphi_i-1: \mathcal N^{\geq i}A\Omega\{i\}\to A\Omega\{i\})/p^n\ .
\]
So far, these objects are entirely unrelated.

\begin{theorem}\label{thm:nearbycycles} Let $\mathfrak X$ be a smooth formal scheme over $\mathrm{Spf}\ \calO_C$ with generic fibre $X$. Consider the map of sites $\psi: X_{\sub{\'et}}\to \mathfrak X_{\sub{\'et}}$. There is a natural equivalence
\[
\mathbb Z/p^n \mathbb Z(i) \simeq  \tau^{\leq i} R\psi_\ast \mathbb Z/p^n\mathbb Z(i)
\]
of sheaves of complexes on $\mathfrak X_{\sub{\'et}}$, compatible in $n$.
\end{theorem}

\begin{remark} 
The complexes $R\psi_* \mathbb{Z}/p^n\mathbb{Z}(i)$ appearing in Theorem~\ref{thm:nearbycycles} are essentially the {\em nearby cycles} complexes for the morphism $\mathfrak{X} \to \mathrm{Spf}(\mathcal{O}_C)$, as introduced by Deligne in \cite[\S XIII.1.3]{SGA7}. The key differences are: (a) $\mathfrak{X}$ is merely a formal scheme in our context (and thus $X$ is a rigid space), while {\em loc.\ cit.} works with schemes, and (b) we can ignore the passage to the algebraic closure that is present in {\em loc. cit.} (as the residue fields of $\mathcal{O}_C$ are algebraically closed). We do not use any non-trivial theory concerning nearby cycles complexes in the sequel. 
\end{remark}

\begin{remark}
From Theorem~\ref{thm:nearbycycles}, one obtains a similar statement on the pro-\'etale site after passing to the inverse limit over $n$. The statement at finite level $n$ has the advantage that the compatibility for varying $n$ implies that the individual \'etale nearby cycle sheaves $R^i\psi_\ast \mathbb Z/p^n\mathbb Z\simeq R^i\psi_\ast \mathbb Z/p^n\mathbb Z(i)$ are flat over $\mathbb Z/p^n\mathbb Z$.
\end{remark}

\begin{remark}
\label{ExtendFormal}
For future use, note that any sheaf $F$ on $\Qs_{\mathcal{O}_C}$ naturally yields a sheaf on $\mathfrak{X}_{\sub{\'et}}$: for any \'etale map $\mathfrak{U} \to \mathfrak{X}$, one defines $F(\mathfrak{U}) = \lim F(R)$, where the limit runs over the category of affine open subsets $\mathrm{Spf}(R) \subset \mathfrak{U}$. This gives a sheaf because  $p$-completely \'etale (or even smooth) covers give quasisyntomic covers.
\end{remark}

By Remark~\ref{ExtendFormal}, we obtain sheaves $ \mathcal{N}^{\geq i} A\Omega\{i\}/p^n$ on $\mathfrak{X}_{\sub{\'et}}$. 
We start by showing that as a sheaf of complexes on $\mathfrak X_{\sub{\'et}}$, the complex
\[
\mathrm{hofib}(\varphi_i - 1: \calN^{\geq i}A\Omega\{i\}\to A\Omega\{i\})/p^n
\]
is concentrated in degrees $\leq i$; for this, we may assume that $n=1$. We can assume that $\mathfrak X=\mathrm{Spf}\ A$ is affine and that $A$ admits a framing $\square: \calO_C\langle T_1^{\pm 1},\ldots,T_d^{\pm 1}\rangle\to A$ that is the $p$-adic completion of an \'etale map. This induces a flat deformation $\tilde{A}$ of $A$ to $A_{\inf}$, which is formally \'etale over $A_{\inf}\langle T_1^{\pm 1},\ldots,T_d^{\pm 1}\rangle$. In that case, we have equivalences
\[
A\Omega_A = q\text-\Omega^\bullet_{\tilde{A}/A_{\inf}}
\qquad \calN^{\geq i} A\Omega_A = \xi^{\max(i-\bullet,0)} q\text-\Omega^\bullet_{\tilde{A}/A_{\inf}}\ ,
\]
as explained in Remark \ref{remark_nygaard_in_coords}. Trivializing the Breuil-Kisin twist, the map
$
\varphi_i: \calN^{\geq i} A\Omega_A\to A\Omega_A
$
is given by
\[
\tilde\xi^{-i}\varphi: \xi^{\max(i-\bullet,0)} q\text-\Omega^\bullet_{\tilde{A}/A_{\inf}}\to q\text-\Omega^\bullet_{\tilde{A}/A_{\inf}}\ .
\]
Recall that in degree $j$, the $\tilde{A}$-module $q\text-\Omega^j_{\tilde{A}/A_{\inf}}$ is free with basis given by $d_q\log(T_{a_1})\wedge\ldots\wedge d_q\log(T_{a_j})$ for varying integers $1\leq a_1<\ldots<a_j\leq d$. On this basis, $\varphi$ acts by multiplication by $\tilde\xi^j$ as
\[
\varphi(d_q\log(T_a)) = \tilde\xi d_q\log(T_a)\ .
\]
In particular, in degree $i$, the basis elements $d_q\log(T_{i_1})\wedge\ldots\wedge d_q\log(T_{i_j})$ are fixed points of $\varphi_i = \tilde\xi^{-i} \varphi$.

Using these representatives, it suffices to see that the map of complexes
\[
\tilde{\xi}^{-i} \varphi - 1: q\text-\Omega^{\geq i}_{\tilde{A}/A_{\inf}}/p\to q\text-\Omega^{\geq i}_{\tilde{A}/A_{\inf}}/p
\]
is an isomorphism on $q\text-\Omega^j_{\tilde{A}/A_{\inf}}/p$ for $j>i$ and is \'etale locally surjective for $j=i$. Writing everything in terms of the above basis, we need to see that
\[
\varphi-1: \tilde{A}/p\to \tilde{A}/p
\]
is \'etale locally surjective, and for $j>i$ the map
\[
\tilde\xi^{j-i} \varphi-1: \tilde{A}/p\to \tilde{A}/p
\]
is an automorphism. The former statement follows from the existence of Artin-Schreier covers in $\mathfrak X_\sub{\'et}$ (noting that this \'etale site is equivalent to the \'etale site of $\mathrm{Spf}\ \tilde{A}/p$, both being equivalent to the \'etale site of $\Spec A/p$), and the latter statement follows from the existence of the inverse operator
\[
- 1 - \tilde\xi^{j-i}\varphi - \tilde\xi_2^{j-i}\varphi^2 - \ldots - \tilde\xi_r^{j-i} \varphi^r - \ldots : \tilde{A}/p\to \tilde{A}/p\ .
\]
In summary, we get the following result.

\begin{proposition}\label{prop:largedegrees} The natural map is an equivalence of sheaves of complexes on $\mathfrak X_{\sub{\'et}}$
\[
\tau^{\leq i} \mathrm{hofib}(\tau^{\leq i}\mathcal N^{\geq i}A\Omega\{i\}/p^n\xrightarrow{\varphi_i-1} \tau^{\leq i} A\Omega\{i\}/p^n)\to \mathrm{hofib}(\mathcal N^{\geq i}A\Omega\{i\}\xrightarrow{\varphi_i-1} A\Omega\{i\})/p^n\ .
\]$\hfill \Box$
\end{proposition}

We note that $\tau^{\leq i} \mathcal N^{\geq i} A\Omega\simeq \tau^{\leq i} A\Omega$ via $\varphi_i$. Under this equivalence, the homotopy fibre above may be rewritten as
\[
\tau^{\leq i}\mathrm{hofib}(\tau^{\leq i}A\Omega/p^n\xrightarrow{1-\xi^i\varphi^{-1}}\tau^{\leq i}A\Omega/p^n)\ ,
\]
where $\xi^i \varphi^{-1}: \tau^{\leq i}A\Omega\to \tau^{\leq i}A\Omega$ is the composite
\[
\tau^{\leq i} A\Omega = L\eta_\mu \tau^{\leq i}R\nu_\ast \mathbb A_{\sub{inf},X}\stackrel{\varphi^{-1}}{\simeq} L\eta_{\varphi^{-1}(\mu)} \tau^{\leq i} R\nu_\ast \mathbb A_{\sub{inf},X}\xrightarrow{\xi^i} L\eta_\xi L\eta_{\varphi^{-1}(\mu)} \tau^{\leq i} R\nu_\ast \mathbb A_{\sub{inf},X} = \tau^{\leq i} A\Omega\ .
\]
We will continue with this description.

We formulate the next step somewhat more abstractly for convenience. Assume that $A$ is a $p$-power-torsion ring in some topos equipped with an automorphism $\varphi$, and suppose that $\mu,\xi\in A$ are non-zero-divisors satisfying the relation $\mu=\xi\varphi^{-1}(\mu)$; set $\xi_r=\xi\varphi^{-1}(\xi)\cdots\varphi^{1-r}(\xi)$ so that $\mu=\xi_r\varphi^{-r}(\mu)$ for all $r\ge1$.

Suppose that $C\in D^{\geq 0}(A)$ is a complex such that $H^0(C)$ is $\mu$-torsion-free, and that $\varphi:C\simeq C$ is a given $\varphi$-semi-linear quasi-isomorphism. In our application, we will take $A=A_{\inf}/p^n$ on the topos $\mathfrak X_{\sub{\'et}}^\sim$ and $C=R\nu_\ast \mathbb A_{\sub{inf},X}/p^n$. In the next two lemmas we make the following assumption.

\begin{quote}
(As) The map $1-\xi^i\varphi^{-1}:C/\mu^jC\to C/\mu^jC$ is a quasi-isomorphism for all $i\ge j\ge 0$.
\end{quote}

\begin{remark}\label{rem:as} The assumption (As) is satisfied for $C=R\nu_\ast \mathbb A_{\sub{inf},X}/p^n$; indeed, for this, it suffices to see that $1-\xi^i\varphi^{-1}$ is an automorphism of $\mathbb A_{\sub{inf},X}/(p^n,\mu^j)$. It is enough to handle the case $n=1$, where one gets $\mathbb A_{\sub{inf},X}/(p,\mu^j) = \widehat{\calO}_X^{\flat +}/\mu^j$. The map $1-\xi^i\varphi^{-1}$ is injective, as if $f\in \widehat{\calO}_X^{\flat +}$ satisfies $f-\xi^i\varphi^{-1}(f)\in \mu^j \widehat{\calO}_X^{\flat +}$, then $g=\tfrac f{\mu^j}$ satisfies $g-\mu^{i-j}\varphi^{-1}(g)\in \widehat{\calO}_X^{\flat +}$, which by integral closedness of $\widehat{\calO}_X^{\flat +}\subset \widehat{\calO}_X^\flat$ implies that $g\in \widehat{\calO}_X^{\flat +}$. For surjectivity, it is enough to see that $1-\xi^i\varphi^{-1}$ is surjective on $\widehat{\calO}_X^{\flat +}$, which by integral closedness of $\widehat{\calO}_X^{\flat +}\subset \widehat{\calO}_X^\flat$ reduces to surjectivity of $1-\xi^i\varphi^{-1}$ on $\widehat{\calO}_X^\flat$. This follows from the existence of Artin-Schreier covers in characteristic $p$, and the tilting equivalence. Some further details of these arguments may be found in \cite{MorrowNearby}.
\end{remark}

\begin{lemma} Assume (As). Let $i\ge 0$ and $j\ge 1$. Then
\begin{enumerate}
\item the endomorphism $1-\xi^{\ell+j}\varphi^{-1}$ of $H^i(C)/\mu^jH^i(C)$ is injective for $\ell=0$ and an automorphism for $\ell>0$;
\item the endomorphism $1-\xi^{\ell}\varphi^{-1}$ of $H^i(C)[\mu^j]$ is surjective for $\ell=0$ and an automorphism for $\ell>0$;
\item the endomorphism $1-\xi^\ell\varphi^{-1}$ of $H^i(C)[\mu^j]/H^i(C)[\mu]$ is an automorphism for $\ell\ge 0$. 
\end{enumerate}
\end{lemma}

\begin{proof}
(1) \& (2): The fibre sequence $C\xrightarrow{\mu^j}C\to C/\mu^j C$ is compatible with the endomorphisms $1-\xi^\ell\varphi^{-1}$, $1-\xi^{\ell+j}\varphi^{-1}$, $1-\xi^{\ell+j}\varphi^{-1}$ respectively; therefore the Bockstein sequence \[0\to H^i(C)/\mu^jH^i(C)\to H^i(C/\mu^j C)\to H^{i+1}(C)[\mu^j]\to 0\] is compatible with the endomorphisms $1-\xi^{\ell+j}\varphi^{-1}$, $1-\xi^{\ell+j}\varphi^{-1}$, $1-\xi^{\ell}\varphi^{-1}$ respectively. Since the endomorphism on the middle term is an automorphism by assumption, the desired injectivity and surjectivity claims in (1) and (2) follow. Moreover, to prove the rest of (i) and (ii) it remains only to show that $1-\xi^\ell\varphi^{-1}$ is injective on $H^{i+1}(C)[\mu^j]$ for all $\ell>0$ (because then we deduce the desired surjectivity on the left term thanks to the Bockstein sequence). Considering the exact sequence \[0\to H^{i+1}(C)[\mu]\to H^{i+1}(C)[\mu^j]\xrightarrow{m}H^{i+1}(C)[\mu^{j-1}],\] which is compatible with the operators $1-\xi^\ell\varphi^{-1}$, $1-\xi^\ell\varphi^{-1}$, $1-\xi^{\ell+1}\varphi^{-1}$ respectively, a trivial induction reduces us to the case $j=1$. But the map \[\xi^\ell\varphi^{-1}:H^{i+1}(C)[\mu]\to H^{i+1}(C)[\mu]\] is the restriction of \[p\xi^{\ell-1}\varphi^{-1}:H^{i+1}(C)\to H^{i+1}(C)\] since $\xi\equiv p$ mod $\varphi^{-1}(\mu)$. So, finally, we must show that $1-p\xi^{\ell-1}\varphi^{-1}$ is injective on $H^{i+1}(C)$; but this operator is even an automorphism of $C$ (hence of $H^{i+1}(C)$), since $C$ was assumed to be derived $p$-adically complete.

(3): The assertion is trivial if $j=1$, so assume $j>1$. The injection \[\mu:H^i(C)[\mu^{j}]/H^i(C)[\mu]\hookrightarrow H^i(C)[\mu^{j-1}]\] is compatible with the endomorphisms $1-\xi^\ell\varphi^{-1}$, $1-\xi^{\ell+1}\varphi^{-1}$ respectively. But the endomorphism on the right side is injective by (2), whence the endomorphism is also injective on the left side; but it is also surjective by (2).
\end{proof}

For any $i\ge 0$, we define $\xi^i\varphi^{-1}:\tau^{\le i}L\eta_\mu C\to \tau^{\le i}L\eta_\mu C$ to be the composition \[\tau^{\le i}L\eta_\mu C\stackrel{\varphi^{-1}}\simeq\tau^{\le i}L\eta_{\varphi^{-1}(\mu)} C\xrightarrow{\xi^i}\tau^{\le i}L\eta_{\xi}L\eta_{\varphi^{-1}(\mu)} C=\tau^{\le i}L\eta_\mu C.\] The $\xi^i\varphi^{-1}$-fixed-points are essentially unchanged under $L\eta$.

\begin{lemma}\label{lemma2} Assume (As). Associated to the commutative diagram
\[\xymatrix@C=2cm{\tau^{\le i}L\eta_\mu C\ar[r]^{1-\xi^i\varphi^{-1}}\ar[d]_{\iota} & \tau^{\le i}L\eta_\mu C\ar[d]^\iota\\
C\ar[r]_{1-\xi^i\varphi^{-1}} & C,
}\]
the induced map \[\tau^{\le i}\mathrm{hofib}(1-\xi^i\varphi^{-1}\text{ \rm on }\tau^{\le i}L\eta_\mu C)\to \tau^{\le i}\mathrm{hofib}(1-\xi^i\varphi^{-1}\text{ \rm on } C)\] is a quasi-isomorphism.
\end{lemma}

\begin{proof}
It is enough (we leave it to the reader to draw the necessary nine-term diagram of fibre sequences) to show that $1-\xi^i\varphi^{-1}$ acts automorphically on the kernel and cokernel of $\iota:H^j(L\eta_\mu C)\to H^j(C)$ for all $j<i$, automorphically on the kernel when $j=i$, and injectively on the cokernel when $j=i$.

For $j=0$, we have $H^0(L\eta_\mu C)=H^0(C)$ as $H^0(C)$ is $\mu$-torsion-free, so there is nothing to prove. Assume now that $j>0$. Recalling the isomorphisms $\mu^j:H^j(C)/H^j(C)[\mu]\cong H^j(L\eta_\mu C)$ \cite[Lemma 6.4]{BMS}, we see that for each $0<j\le i$ the canonical map $\iota$ fits into a natural exact sequence \[0\to H^j(C)[\mu]\to H^j(C)[\mu^j]\to H^j(L\eta_\mu C)\xrightarrow{\iota}H^j(C)\to H^j(C)/\mu^j H^j(C)\to 0,\] which is compatible with the operators $1-\xi^{i-j}\varphi^{-1}$, $1-\xi^{i-j}\varphi^{-1}$, $1-\xi^i\varphi^{-1}$, $1-\xi^i\varphi^{-1}$, $1-\xi^i\varphi^{-1}$ respectively. Then all desired properties of $1-\xi^i\varphi^{-1}$ on the kernel and cokernel of $\iota$ follow immediately from the previous lemma.
\end{proof}

Finally, we can finish the proof of Theorem~\ref{thm:nearbycycles}. By Proposition~\ref{prop:largedegrees},
\[
\mathrm{hofib}(\mathcal N^{\geq i}A\Omega\{i\}/p^n\xrightarrow{\varphi_i-1} A\Omega\{i\}/p^n) = \tau^{\leq i} \mathrm{hofib}(\tau^{\leq i}\mathcal N^{\geq i}A\Omega\{i\}/p^n\xrightarrow{\varphi_i-1} \tau^{\leq i} A\Omega\{i\}/p^n)\ .
\]
Now Lemma~\ref{lemma2} says that under the identification
\[
\tau^{\leq i} \mathrm{hofib}(\tau^{\leq i}\mathcal N^{\geq i}A\Omega\{i\}/p^n\xrightarrow{\varphi_i-1} \tau^{\leq i} A\Omega\{i\}/p^n) = \tau^{\leq i} \mathrm{hofib}(\tau^{\leq i}A\Omega/p^n\xrightarrow{1-\xi^i\varphi^{-1}} \tau^{\leq i} A\Omega/p^n)\ ,
\]
one has
\[
\tau^{\leq i} \mathrm{hofib}(\tau^{\leq i}A\Omega/p^n\xrightarrow{1-\xi^i\varphi^{-1}} \tau^{\leq i} A\Omega/p^n) = \tau^{\leq i} \mathrm{hofib}(R\nu_*\mathbb A_{\sub{inf},X}/p^n\xrightarrow{1-\xi^i\varphi^{-1}} R\nu_*\mathbb A_{\sub{inf},X}/p^n)\ .
\]
On the other hand, on the pro-\'etale site of the generic fibre $X$, the map
\[
\mathbb A_{\sub{inf},X}/p^n\xrightarrow{1-\xi^i\varphi^{-1}} \mathbb A_{\sub{inf},X}/p^n
\]
is surjective (by the argument of Remark~\ref{rem:as}), and the kernel is given by $\mathbb Z/p^n\mathbb Z(i)\simeq \mathbb Z/p^n\mathbb Z\cdot \mu^i$. In summary,
\[
\mathrm{hofib}(\mathcal N^{\geq i}A\Omega\{i\}/p^n\xrightarrow{\varphi_i-1} A\Omega\{i\}/p^n)  = \tau^{\leq i} R\nu_\ast \mathbb Z/p^n\mathbb Z(i) = \tau^{\leq i} R\psi_\ast \mathbb Z/p^n\mathbb Z(i)\ ,
\]
as desired.

\newpage
\section{Breuil-Kisin modules}
\label{sec:breuilkisin}

In this section, we use relative $\THH$ to prove Theorem~\ref{thm:main1}. Before explaining what we do, let us gather the relevant notation.

\begin{notation}
Let $K$ be a discretely valued extension of $\mathbb Q_p$ with perfect residue field $k$, ring of integers $\calO_K$ and fix a uniformizer $\varpi\in \calO_K$. Let $K_\infty$ be the $p$-adic completion of $K(\varpi^{1/p^\infty})$ and let $C$ be the completion of an algebraic closure of $K_\infty$ and $A_{\inf} = A_{\inf}(\calO_C)$. Let $\mathfrak{S} = W(k) \llbracket z \rrbracket$; there is a surjective map $\tilde{\theta}:\mathfrak{S} \to \mathcal{O}_K$ determined via the standard map on $W(k)$ and $z \mapsto \varpi$. The kernel of this map is generated by an Eisenstein polynomial $E(z) \in \mathfrak{S}$ for $\varpi$. Let $\varphi$ be the endomorphism of $\mathfrak{S}$ determined by the Frobenius on $W(k)$ and $z \mapsto z^p$. We regard $\mathfrak{S}$ as a $\varphi$-stable subring of $A_{\inf}(\mathcal{O}_{K_\infty})$ or $A_{\inf}$ by the Frobenius on $W(k)$ and sending $z$ to $[\varpi^\flat]^p$ where $\varpi^\flat = (\varpi, \varpi^{\frac{1}{p}}, \varpi^{\frac{1}{p^2}}, ....) \in \mathcal{O}_{K_\infty}^\flat$ is our chosen compatible system of $p$-power roots of $\varpi$; the resulting map $\mathfrak{S} \to A_{\inf}$ is faithfully flat and even topologically free (see \cite[Lemma 4.30 and its proof]{BMS}). Write $\theta = \tilde{\theta} \circ \varphi:\mathfrak{S} \to \mathcal{O}_K$. The embedding $\mathfrak{S} \subset A_{\inf}$ is compatible with the $\theta$ and $\tilde{\theta}$ maps.
 \end{notation}
 
Our goal in this section is to prove the following more precise local assertion that implies Theorem~\ref{thm:main1}.
 
 \begin{theorem}
 \label{thm:BKlocal}
To any smooth affine formal scheme $\mathrm{Spf}(A)/\mathcal{O}_K$, one can functorially attach a $(p,z)$-complete $E_\infty$-algebra $\widehat{\Prism}_{A/\mathfrak{S}} \in D(\mathfrak{S})$ together with a $\varphi$-linear Frobenius endomorphism $\varphi:\widehat{\Prism}_{A/\mathfrak{S}} \to \widehat{\Prism}_{A/\mathfrak{S}}$ inducing an isomorphism $\widehat{\Prism}_{A/\mathfrak{S}} \otimes_{\mathfrak{S},\varphi} \mathfrak{S}[\frac{1}{E}] \simeq \widehat{\Prism}_{A/\mathfrak{S}}[\frac{1}{E}]$, and having the following features:
\begin{enumerate}

\item ($A\Omega$ comparison) After base extension to $A_{\inf}$, there is a functorial Frobenius equivariant isomorphism
\[ \widehat{\Prism}_{A/\mathfrak{S}} \widehat{\otimes_{\mathfrak{S}}^{\mathbb L}} A_{\inf} \simeq A\Omega_{A_{\mathcal{O}_C}}\]
of $E_\infty$-$A_{\inf}$-algebras.

\item (de Rham comparison) After scalar extension along $\theta:=\tilde\theta\circ \varphi: \mathfrak S\to \calO_K$, there is a functorial isomorphism
\[
\widehat{\Prism}_{A/\mathfrak{S}} \dotimes_{\mathfrak{S},\theta} \mathcal{O}_K \simeq \big(\Omega_{A/\mathcal{O}_K}\big)^\wedge
\]
of $E_\infty$-$\mathcal{O}_K$-algebras.

\item (Crystalline comparison) After scalar extension along the map $\mathfrak S\to W(k)$ which is the Frobenius on $W(k)$ and sends $z$ to $0$, there is a functorial Frobenius equivariant isomorphism
\[ \widehat{\Prism}_{A/\mathfrak{S}} \dotimes_{\mathfrak{S}} W(k) \simeq R\Gamma_{\mathrm{crys}}(A_k/W(k)).\]
\end{enumerate}
In particular, for a proper smooth formal scheme $\mathfrak{X}/\calO_K$, setting $R\Gamma_{\mathfrak{S}}(\mathfrak{X}) := R\Gamma(\mathfrak{X}, \widehat{\Prism}_{-/\mathfrak{S}})$ (see Remark~\ref{ExtendFormal}) gives the cohomology theory wanted in Theorem~\ref{thm:main1}.\footnote{To see that it is a perfect complex of $\mathfrak S$-modules, reduce modulo $p$ and $z$ and use the crystalline comparison.}
\end{theorem}

As explained already in \S \ref{ss:THHtoBK}, the construction of $\widehat{\Prism}_{A/\mathfrak{S}}$ uses relative $\THH$. Thus, let $\mathbb S[z] :=\mathbb S[\mathbb N]$ be the free $E_\infty$-ring spectrum generated by the commutative monoid $\mathbb N$, so $\pi_* (\mathbb{S}[z]) = (\pi_* \mathbb{S})[z]$; we regard $\mathfrak{S}$ as an $\mathbb{S}[z]$-algebra via $z \mapsto z$, and thus $\calO_K$ is an $\mathbb S[z]$-algebra via $z\mapsto \varpi$. Roughly, we construct $\widehat{\Prism}_{A/\mathfrak{S}}$ by repeating the construction of $\widehat{\Prism}_A$ using  $\THH(-/\mathbb{S}[z])$ instead of $\THH(-)$. More precisely, we use this idea in \S \ref{ss:BK} to construct the Frobenius pullback $\varphi^* \widehat{\Prism}_{A/\mathfrak{S}}$ (Corollary~\ref{cor:BKtwisted}). In \S \ref{ss:BKdescent}, we then descend this construction along the Frobenius on $\mathfrak{S}$ to construct $\widehat{\Prism}_{A/\mathfrak{S}}$; this additional descent uses the structure of $\THH(\calO_K/\mathbb{S}[z])$ (and variants) as well as the analog of the Segal conjecture proven in Corollary~\ref{cor:segalcharp}. To carry out this outline, we need a good handle on $\THH(-/\mathbb{S}[z])$, and we record the relevant features in \S \ref{ss:RelTHH}.

\subsection{Relative $\THH$}
\label{ss:RelTHH}
We recall the structure of $\THH(\mathbb{S}[z])$ and use that to endow $\THH(-/\mathbb{S}[z])$ with a cyclotomic structure. Recall the definition: for any $E_\infty$-$\mathbb{S}[z]$-algebra $A$, the spectrum
\[
\THH(A/\mathbb S[z]) = \THH(A)\otimes_{\THH(\mathbb S[z])} \mathbb S[z] \cong A \otimes_{E_\infty-\mathbb{S}[z]} \T
\]
is the universal $\T$-equivariant $E_\infty$-$\mathbb S[z]$-algebra with a non-equivariant map from $A$. To endow $\THH(A/\mathbb S[z])$ with a Frobenius map
\[
\varphi_p: \THH(A/\mathbb S[z])\to \THH(A/\mathbb S[z])^{tC_p}\ .
\]
we recall a few results about $\THH(\mathbb S[z])$.

\begin{proposition}
\label{prop:THHN}
The $\T$-equivariant $E_\infty$-ring spectrum $\THH(\mathbb S[z])$ is given by $\mathbb S[B^\sub{cy} \mathbb N]$, where $B^\sub{cy}$ denotes the cyclic bar construction, with its natural $\T$-action. Concretely, $B^\sub{cy}\mathbb N$ is equivalent to the topological abelian monoid which is the closed submonoid of $S^1\times\mathbb Z$ given by the union of $S^1\times \mathbb N_{>0}$ and the zero element $0=(1,0)\in S^1\times \mathbb Z$; the $\T$-action is given via letting $t\in \T$ act on $(s,n)\in S^1\times \mathbb Z$ via $(s,n)\mapsto (t^n s,n)$, where we write the group structure on $S^1$ multiplicatively.

The map $\THH(\mathbb S[z])=\mathbb S[B^\sub{cy} \mathbb N]\to \mathbb S[z]=\mathbb S[\mathbb N]$ is induced by the map $B^\sub{cy} \mathbb N\to \mathbb N: (s,n)\mapsto n$. The map $\varphi_p: \THH(\mathbb S[z])\to \THH(\mathbb S[z])^{tC_p}$ factors naturally over $\THH(\mathbb S[z])^{hC_p}=\mathbb S[B^\sub{cy} \mathbb N]^{hC_p}$, and in fact over $\mathbb S[(B^\sub{cy} \mathbb N)^{hC_p}]$, and is induced by the $\T\cong \T/C_p$-equivariant map $B^\sub{cy} \mathbb N\to (B^\sub{cy} \mathbb N)^{hC_p}$ given by $(s,n)\mapsto (s^p,pn)$. In particular, the diagram
\[\xymatrix{
\THH(\mathbb S[z])\ar[r]^{\varphi_p}\ar[d] & \THH(\mathbb S[z])^{tC_p}\ar[d]\\
\mathbb S[z]\ar[r]^{z\mapsto z^p} & \mathbb S[z]^{tC_p}\ .
}\]
is a $\T$-equivariant commutative diagram of $E_\infty$-ring spectra.
\end{proposition}

\begin{remark} 
By the Segal conjecture $\mathbb S^{tC_p}\simeq \mathbb S^\wedge_p$, so one can compute that $\mathbb S[z]^{tC_p}\simeq (\mathbb S[z])^\wedge_p$ is the $p$-completion of $\mathbb S[z]$ as a spectrum.
\end{remark}

\begin{proof} See for example \cite[Lemma IV.3.1]{NikolausScholze}. The explicit description of $B^\sub{cy} \mathbb N$ can be deduced from the description of $B^\sub{cy} \mathbb Z$ as the free loop space of $S^1$ (which is equivalent to $S^1\times \mathbb Z$ with given $\mathbb T$-action), cf.~\cite[Proposition IV.3.2]{NikolausScholze}. The final commutative diagram follows formally.
\end{proof}

Next, let us explain why relative $\THH$ carries a cyclotomic structure.

\begin{construction}[Construction of the cyclotomic structure on {$\THH(-/\mathbb{S}[z])$}] 
 For any $E_\infty$-$\mathbb S[z]$-algebra $A$, there is a natural map
\[\begin{aligned}
\THH(A/\mathbb S[z]) &= \THH(A)\otimes_{\THH(\mathbb S[z])} \mathbb S[z]\xrightarrow{\varphi_p\otimes 1} \THH(A)^{tC_p}\otimes_{\THH(\mathbb S[z])} \mathbb S[z]\\
&\to \THH(A)^{tC_p}\otimes_{\THH(\mathbb S[z])^{tC_p}} \mathbb S[z]^{tC_p}\to \THH(A/\mathbb S[z])^{tC_p}\ ,
\end{aligned}\]
where the first map comes from the cyclotomic structure of $\THH(A)$, the second map exists thanks to the commutative square in Proposition~\ref{prop:THHN}, and the last map comes from the lax symmetric monoidal nature of $(-)^{tC_p}$. By construction, this map is $\T\cong\T/C_p$-equivariant and linear over $\mathbb S[z]\to \mathbb S[z]$ given by $z\mapsto z^p$ provided we regard the target as $\mathbb S[z]$-algebra via the natural map $\mathbb S[z]\to \mathbb S[z]^{tC_p}\to \THH(A/\mathbb S[z])^{tC_p}$.

In particular, $\THH(A/\mathbb S[z])$ is a cyclotomic spectrum in the sense of \cite{NikolausScholze}, and in fact a cyclotomic $E_\infty$-algebra over the cyclotomic $E_\infty$-ring spectrum $\mathbb S[z]$, where $\mathbb S[z]$ is equipped with the trivial $\T$-action and the $\T\cong \T/C_p$-equivariant map $\varphi_p: \mathbb S[z]\to \mathbb S[z]^{tC_p}$ sending $z$ to $z^p$.
\end{construction}

There are two simple comparisons between the relative theory and the absolute theory. The first describes the specialization $z \mapsto 0$ and is the main source of the crystalline comparison.

\begin{lemma} 
Assume that $A$ is an $\calO_K$-algebra, regarded as $\mathbb S[z]$-algebra via $z\mapsto \varpi$. Base extension along $\mathbb S[z]\to \mathbb S$ sending $z$ to $0$ gives
\[
\THH(A/\mathbb S[z])\otimes_{\mathbb S[z]} \mathbb S\simeq \THH(A\otimes_{\calO_K} k)\ ,
\]
compatibly with the $\T$-action and $\varphi_p$.
\end{lemma}
\begin{proof}
The base change property and its compatibility with the $\T$-actions follows from the base change property of $\THH$ together with the observation that both commutative squares in
\[ \xymatrix{ \mathbb{S}[z] \ar[r] \ar[d]^{z \mapsto 0} & \calO_K \ar[r] \ar[d] & A \ar[d] \\
		   \mathbb{S} \ar[r] & k \ar[r] & A \otimes_{\calO_K} k }\]
are Cartesian in $E_\infty$-rings. The compatibility with $\varphi_p$ follows by observing that the map $\mathbb{S}[z] \xrightarrow{z \mapsto 0} \mathbb{S}$ intertwines the Frobenius map on $\mathbb{S}[z]$ (determined by $z \mapsto z^p$ and used to define the cyclotomic structure) with the identity on $\mathbb{S}$.
\end{proof}

 For the second, we observe the following proposition.

\begin{proposition} After $p$-completion, the map
\[
\THH(\mathbb S[z^{1/p^\infty}])\to \mathbb S[z^{1/p^\infty}]
\]
is an equivalence.
\end{proposition}

\begin{proof} By tensoring with $\THH(\mathbb Z)$ over $\mathbb S$ and using Lemma~\ref{THHvsHH}, this reduces to the same question for $\HH(\mathbb Z[z^{1/p^\infty}])$, which in turn follows from the vanishing of the $p$-completion of $L_{\mathbb Z[z^{1/p^\infty}]/\mathbb Z}=0$ by the HKR filtration.
\end{proof}

From this, we learn what relative $\THH$ looks like after base change to $\mathbb{S}[z^{\frac{1}{p^\infty}}]$. 

\begin{corollary}
\label{cor:BKPerfBC}
For any $\calO_K$-algebra $A$, the $p$-completion of
\[
\THH(A/\mathbb S[z])\otimes_{\mathbb S[z]} \mathbb S[z^{1/p^\infty}]\simeq \THH(A[\varpi^{1/p^\infty}]/\mathbb S[z^{1/p^\infty}])
\]
agrees with
\[
\THH(A[\varpi^{1/p^\infty}];\mathbb Z_p)=\THH(A\otimes_{\calO_K} \calO_{K_\infty};\mathbb Z_p)\ ,
\]
compatibly with the $\T$-action and $\varphi_p$. 
\end{corollary}

\begin{remark}
Philosophically, the equality (after $p$-completion) between relative $\THH$ over ``perfect'' base rings such as $\mathbb S[z^{1/p^\infty}]$ and absolute $\THH$ is the reason that one can define the $A\Omega$-theory in terms of absolute $\THH$ while one needs relative $\THH$ for the Breuil-Kisin descent.
\end{remark}

\subsection{Frobenius twisted Breuil-Kisin modules}
\label{ss:BK}

In this section, we construct complexes that will end up equalling $\varphi^* \widehat{\Prism}_{A/\mathfrak{S}}$ in the context of Theorem~\ref{thm:BKlocal}. As the latter complexes have not yet been defined, this notation does not yet make sense; instead, we rename the map $\varphi:\mathfrak{S} \to \mathfrak{S}$ as the map $\mathfrak{S} \to \mathfrak{S}^{(-1)}$, and construct complexes $\widehat{\Prism}_{A/\mathfrak{S}^{(-1)}}$ that will eventually descend to $\mathfrak{S}$.

Thus, let us write $\mathfrak S^{(-1)}$ for a copy of $\mathfrak S$, which we regard as $\mathfrak S$-algebra via $\varphi: \mathfrak S\to \mathfrak S=\mathfrak S^{(-1)}$. We write $\theta^{(-1)}: \mathfrak S^{(-1)}\to \calO_K$ for the usual map $\mathfrak S\to \calO_K$, $z\mapsto \varpi$, that was denoted $\tilde\theta: \mathfrak S\to \calO_K$ before. Then there is a natural inclusion $\mathfrak S^{(-1)}\hookrightarrow A_{\inf}(\calO_{K_\infty})$ which is $W(k)$-linear and sends $z$ to $[\varpi^\flat]$; on $\mathfrak S\subset \mathfrak S^{(-1)}$, this is the inclusion $\mathfrak S\hookrightarrow A_{\inf}(\calO_K)$ fixed earlier. The diagram
\[\xymatrix{
\mathfrak S^{(-1)}\ar[r]\ar[d]^{\theta^{(-1)}} & A_{\inf}(\calO_{K_\infty})\ar[d]^\theta\\
\calO_K\ar[r] & \calO_{K_\infty}
}\]
commutes, and there is a natural diagram
\[\xymatrix{
\mathfrak S^{(-1)}\ar[r]\ar[d]^{\tilde{\theta}^{(-1)}} & A_{\inf}(\calO_{K_\infty})\ar[d]^{\tilde{\theta}}\\
\calO_K[\varpi^{1/p}]\ar[r] & \calO_{K_\infty}\ ,
}\]
where $\tilde{\theta}^{(-1)}: \mathfrak S^{(-1)}\to \calO_K[\varpi^{1/p}]$ is defined to make the diagram commute; in particular, $z\mapsto \varpi^{1/p}$. Using the base change properties for relative $\THH$, it is easy to check the following by reduction to the perfectoid case:

\begin{proposition} On homotopy groups,
\[
\pi_\ast \THH(\calO_K/\mathbb S[z];\mathbb Z_p)\cong \calO_K[u]
\]
where $u$ is of degree $2$,
\[
\pi_\ast \TC^-(\calO_K/\mathbb S[z];\mathbb Z_p)\cong \mathfrak S^{(-1)}[u,v]/(uv-E)
\]
where $u$ is of degree $2$ and $v$ is of degree $-2$,
\[
\pi_\ast \TP(\calO_K/\mathbb S[z];\mathbb Z_p) = \mathfrak S^{(-1)}[\sigma^{\pm 1}]
\]
and
\[
\pi_\ast \THH(\calO_K/\mathbb S[z];\mathbb Z_p)^{tC_p} = \calO_K[\varpi^{1/p}][\sigma^{\pm 1}]\ ,
\]
where $\sigma$ has degree $2$. The canonical map
\[
\pi_\ast \TC^-(\calO_K/\mathbb S[z];\mathbb Z_p)\to \pi_\ast \TP(\calO_K/\mathbb S[z];\mathbb Z_p)
\]
sends $u$ to $E\sigma$ and $v$ to $\sigma^{-1}$. The diagram
\[\xymatrix{
\pi_\ast \TC^-(\calO_K/\mathbb S[z];\mathbb Z_p)\ar[d]\ar[r]^{\varphi_p^{h\T}} & \pi_\ast \TP(\calO_K/\mathbb S[z];\mathbb Z_p) \ar[d] \\
\pi_\ast \THH(\calO_K/\mathbb S[z];\mathbb Z_p)\ar[r]^{\varphi_p} & \pi_\ast \THH(\calO_K/\mathbb S[z];\mathbb Z_p)^{tC_p}
}\]
is given by
\[\xymatrix{
\mathfrak S^{(-1)}[u,v]/(uv-E)\ar[d]^{\theta^{(-1)},u\mapsto u, v\mapsto 0}\ar[rrrr]^{\varphi,u\mapsto \sigma,v\mapsto \varphi(E)\sigma^{-1}} &&&& \mathfrak S^{(-1)}[\sigma^{\pm 1}]\ar[d]^{\tilde\theta^{(-1)},\sigma\mapsto \sigma} \\
\calO_K[u]\ar[rrrr]^{u\mapsto \sigma} &&&& \calO_K[\varpi^{1/p}][\sigma^{\pm 1}]
}\]
\end{proposition}
\begin{proof}
This follows from the results of \S \ref{sec:TCperfectoid} by base change. 
\end{proof}

Moreover, if $S$ is quasiregular semiperfectoid, then also $S\hat{\otimes}_{\calO_K} \calO_{K_\infty}$ is quasiregular semiperfectoid. This implies the following proposition, using Theorem~\ref{TCqrsp}.

\begin{proposition} If $S\in \Qsp_{\calO_K}$ is quasiregular semiperfectoid, then $\THH(S/\mathbb S[z];\mathbb Z_p)$, $\TC^-(S/\mathbb S[z];\mathbb Z_p)$ and $\TP(S/\mathbb S[z];\mathbb Z_p)$ are concentrated in even degrees, and $S\mapsto \pi_i \THH(S/\mathbb S[z];\mathbb Z_p)$ respectively $S\mapsto \pi_i \TC^-(S/\mathbb S[z];\mathbb Z_p)$, $S\mapsto \pi_i \TP(S/\mathbb S[z];\mathbb Z_p)$ define sheaves on $\Qsp_{\calO_K}$ with vanishing cohomology on any $S\in \Qsp_{\calO_K}$.
\end{proposition}

In particular, we can define a sheaf
\[ \gr^0 \TC^-(-/\mathbb S[z];\mathbb Z_p) \stackrel{\mathrm{can}}{\simeq} \gr^0 \TP(-/\mathbb S[z];\mathbb Z_p)\]
 of $E_\infty$-$\mathfrak S^{(-1)}$-algebras on $\Qs_{\calO_K}$ by unfolding $\pi_0 \TC^{-}(-/\mathbb{S}[z];\mathbb{Z}_p)$; it is equipped with a natural Frobenius endomorphism compatible with the one on $\mathfrak{S}^{(-1)}$. This construction proves Theorem~\ref{thm:BKlocal} up to a missing Frobenius pullback:

\begin{corollary} 
\label{cor:BKtwisted}
Let $\mathfrak X = \mathrm{Spf}(A)$ be an affine smooth formal scheme over $\calO_K$. The complex $\widehat{\Prism}_{A/\mathfrak{S}^{(-1)}} := \gr^0 \TC^-(A/\mathbb S[z];\mathbb Z_p)$ is a $(p,z)$-complete object of $D(\mathfrak S^{(-1)})$ that admits a natural Frobenius endomorphism $\varphi$. This construction has the following properties:
\begin{enumerate}
\item There is a natural $\varphi$-equivariant isomorphism
\[
\widehat{\Prism}_{A/\mathfrak{S}^{(-1)}} \widehat{\otimes^{\mathbb L}_{\mathfrak{S}^{(-1)}}} A_{\inf}(\calO_{K_\infty}) \simeq \widehat{\Prism}_{A_{\calO_{K_\infty}}}
\]
of $E_\infty$-$A_{\inf}(\calO_{K_\infty})$-algebras, and thus a $\varphi$-equivariant isomorphism
\[
\widehat{\Prism}_{A/\mathfrak{S}^{(-1)}} \widehat{\otimes^{\mathbb L}_{\mathfrak{S}^{(-1)}}} A_{\inf}(\calO_C) \simeq A\Omega_{A_{\calO_C}}
\]
of $E_\infty$-$A_{\inf}(\calO_C)$-algebras.
\item There is a natural isomorphism
\[
\widehat{\Prism}_{A/\mathfrak{S}^{(-1)}} \dotimes_{\mathfrak S^{(-1)},\theta^{(-1)}} \calO_K\simeq \big(\Omega_{A/\calO_K}\big)^\wedge
\]
of $E_\infty$-$\calO_K$-algebras. 
\item After scalar extension along the map $\mathfrak S^{(-1)} \to W(k)$ which is the identity on $W(k)$ and sends $z$ to $0$, there is a functorial Frobenius equivariant isomorphism
\[
\widehat{\Prism}_{A/\mathfrak{S}^{(-1)}} \dotimes_{\mathfrak S^{(-1)}} W(k)\simeq R\Gamma_{\mathrm{crys}}(A_k/W(k))
\]
of $E_\infty$-$W(k)$-algebras.
\item The Frobenius $\varphi$ induces an isomorphism
\[ \widehat{\Prism}_{A/\mathfrak{S}^{(-1)}} \otimes_{\mathfrak{S}^{(-1)},\varphi} \mathfrak{S}^{(-1)} [\tfrac{1}{\varphi(E)}] \simeq \widehat{\Prism}_{A/\mathfrak{S}^{(-1)}}[\tfrac{1}{\varphi(E)}] \]
\end{enumerate}
All completions are above are with respect to $(p,z)$. 
\end{corollary}

\begin{proof} Part (2) comes from the natural equivalence $\gr^0 \TP(A/\mathbb S[z];\mathbb Z_p)/E\simeq \gr^0 \HP(A/\calO_K;\mathbb Z_p)$ and Theorem~\ref{thm:main7}. The first equality of part (1) now comes from the identification
\[
\gr^0 \TP(A/\mathbb S[z];\mathbb Z_p)\widehat{\otimes}_{\mathbb S[z]} \mathbb S[z^{1/p^\infty}]\simeq \gr^0 \TP(A\hat{\otimes}_{\calO_K} \calO_{K_\infty};\mathbb Z_p)\ .
\]
obtained by passing to graded pieces in the similar statement for $\TP$ itself (which follows from the same statement for $\THH$). The second equality of part (1) follows from Theorem~\ref{thm:main2} and the identification
\[
\widehat{\Prism}_B\widehat{\otimes}_{A_{\inf}(\calO_{K_\infty})} A_{\inf}(\calO_C)\simeq \widehat{\Prism}_{B\hat{\otimes}_{\calO_{K_\infty}} \calO_C}
\]
for any $B\in \Qs_{\calO_{K_\infty}}$, for which it suffices to observe that modulo $\ker\theta$, both sides compute Hodge-completed derived de~Rham cohomology, which satisfies the required base change. Part (3) now comes from Theorem~\ref{thm:main3} and the identification
\[
\gr^0 \TP(A/\mathbb S[z];\mathbb Z_p)\otimes_{\mathbb S[z]} \mathbb S\simeq \gr^0 \TP(A\otimes_{\calO_K} k)
\]
since $A$ is flat over $\calO_K$. For part (4), we shall check that the cofiber of $\varphi^* \widehat{\Prism}_{A/\mathfrak{S}^{(-1)}} \to \widehat{\Prism}_{A/\mathfrak{S}^{(-1)}}$ becomes acyclic on inverting $\varphi(E)$. The map $\mathfrak{S}^{(-1)} \to A_{\inf}(\calO_C)$ is a topological direct summand, so by part (1), it suffices to show that the cofiber of $\varphi^* A\Omega_{A_{\calO_C}} \to A\Omega_{A_{\calO_C}}$ is acyclic on inverting $\varphi(\xi) = \tilde{\xi}$; this follows immediately from the definition of $A\Omega$.
\end{proof}

\begin{remark}
The complex $\widehat{\Prism}_{A/\mathfrak{S}^{(-1)}}$ from Corollary~\ref{cor:BKtwisted} does not give the complex $\widehat{\Prism}_{A/\mathfrak{S}}$ desired in Theorem~\ref{thm:BKlocal}.  Concretely, one cannot recover the consequence discussed in Remark~\ref{rmk:TorsiondR} from Corollary~\ref{cor:BKtwisted}.
\end{remark}

 \begin{remark}
 \label{rmk:NygaardBK}
Our construction naturally equips $\widehat{\Prism}_{A/\mathfrak{S}^{(-1)}}$ with a complete descending multiplicative $\mathbf{N}$-indexed filtration $\calN^{\geq \f} \widehat{\Prism}_{A/\mathfrak{S}^{(-1)}}$ coming from the homotopy fixed point spectral sequence.  This filtration is compatible with the Nygaard filtration via the identifications in Corollary~\ref{cor:BKtwisted} (1) and (3), and with the Hodge filtration via the identification in Theorem~\ref{thm:BKlocal} (2).  It is thus reasonable to refer to this as the Nygaard filtration on $\widehat{\Prism}_{A/\mathfrak{S}^{(-1)}}$. 
\end{remark}

\subsection{Frobenius descent}
\label{ss:BKdescent}

We now explain how to descend the complex $\widehat{\Prism}_{A/\mathfrak{S}^{(-1)}}$ from Corollary~\ref{cor:BKtwisted} along $\mathfrak S\to \mathfrak S^{(-1)}$. This relies on the following observation.

\begin{proposition} 
\label{prop:BKnc}

For any $\calO_K$-algebra $A$, the map $\varphi: \TC^-(A/\mathbb S[z];\mathbb Z_p)\to \TP(A/\mathbb S[z];\mathbb Z_p)$ extends naturally to a map
\[
\TC^-(A/\mathbb S[z];\mathbb Z_p)[\tfrac 1u]\otimes_{\mathbb S[z]} \mathbb S[z^{1/p}]\to \TP(A/\mathbb S[z];\mathbb Z_p)\ .
\]
If $A$ is quasiregular semiperfectoid, then the source is concentrated in even degrees, and $A\mapsto \pi_0 (\TC^-(A/\mathbb S[z];\mathbb Z_p)[\tfrac 1u])$ defines a sheaf with vanishing cohomology on $\Qsp_{\calO_K}$. Denote its unfolding to $\Qs_{\calO_K}$ by $\gr^0(\TC^{-}(-/\mathbb{S}[z];\mathbb{Z}_p)[\frac{1}{u}])$. By functoriality of unfolding, we have a  natural map
\[
\gr^0(\TC^-(A/\mathbb S[z];\mathbb Z_p)[\tfrac 1u])\otimes_{\mathbb Z[z]} \mathbb Z[z^{1/p}]\to \gr^0 \TP(A/\mathbb S[z];\mathbb Z_p)
\]
for $A\in \Qs_{\calO_K}$. If $A$ is the $p$-adic completion of a smooth $\calO_K$-algebra, this map is an equivalence. 
\end{proposition}

\begin{proof} The extension of the map follows from the observation that for $A=\calO_K$, the element $u$ maps to $\sigma$ under $\varphi$, and thus becomes invertible; and that the map is linear over $\mathbb S[z]\to \mathbb S[z]$, $z\mapsto z^p$.

As for $A$ quasiregular semiperfetoid, each $\pi_i \TC^-(A/\mathbb S[z];\mathbb Z_p)$ is a sheaf with vanishing higher cohomology, it follows by passage to filtered colimits that the same is true for $\pi_i \TC^-(A/\mathbb S[z];\mathbb Z_p)[\tfrac 1u]$. To check that
\[
\gr^0(\TC^-(A/\mathbb S[z];\mathbb Z_p)[\tfrac 1u])\otimes_{\mathbb Z[z]} \mathbb Z[z^{1/p}]\to \gr^0 \TP(A/\mathbb S[z];\mathbb Z_p)
\]
is an equivalence for the $p$-adic completion $A$ of a smooth $\calO_K$-algebra, it suffices to see that for $i$ at least the dimension of $A$, the map
\[
\gr^i(\TC^-(A/\mathbb S[z];\mathbb Z_p))\otimes_{\mathbb Z[z]} \mathbb Z[z^{1/p}]\to \gr^i \TP(A/\mathbb S[z];\mathbb Z_p)
\]
is an equivalence. For this, we can reduce modulo $z^{1/p}$; then it suffices to see that for a smooth $k$-algebra $\overline{A}$, the Frobenius map
\[
\gr^i(\TC^-(\overline{A}))\to \gr^i \TP(\overline{A})
\]
is an equivalence for $i$ at least the dimension of $A$. But this follows from the version of the Segal conjecture, Corollary~\ref{cor:segalcharp}, by passing to homotopy-$\T$-fixed points.
\end{proof}

\begin{proof}[Proof of Theorem~\ref{thm:BKlocal}]
For the $p$-adic completion $A$ of a smooth $\calO_K$-algebra, define
\[ \widehat{\Prism}_{A/\mathfrak{S}} := \gr^0(\TC^{-}(A;\mathbb{Z}_p)[\tfrac{1}{u}]) \]
as in Proposition~\ref{prop:BKnc}. It follows from this proposition that we have a natural identification
\begin{equation}
\label{eq:BKdescent}
\widehat{\Prism}_{A/\mathfrak{S}} \otimes_{\mathfrak{S},\mathrm{can}} \mathfrak{S}^{(-1)}  \simeq \widehat{\Prism}_{A/\mathfrak{S}^{(-1)}}.
\end{equation}
We can also define the (linearized) Frobenius 
\[ \widehat{\Prism}_{A/\mathfrak{S}} \otimes_{\mathfrak{S},\varphi} \mathfrak{S} \simeq \widehat{\Prism}_{A/\mathfrak{S}} \otimes_{\mathfrak{S},\mathrm{can}} \mathfrak{S}^{(-1)} \to \widehat{\Prism}_{A/\mathfrak{S}} \]
as the following composition
\begin{align*}
 \widehat{\Prism}_{A/\mathfrak{S}} \otimes_{\mathfrak{S},\mathrm{can}} \mathfrak{S}^{(-1)}  & \simeq \widehat{\Prism}_{A/\mathfrak{S}^{(-1)}} \\
&:= \gr^0(\TP(A/\mathbb S[z];\mathbb Z_p)) \stackrel{\mathrm{can}}{\simeq} \gr^0(\TC^-(A/\mathbb S[z];\mathbb Z_p))\\
&\xrightarrow{invert \ u} \gr^0(\TC^-(A/\mathbb S[z];\mathbb Z_p)[\tfrac 1u])\\
&=: \widehat{\Prism}_{A/\mathfrak{S}}
\end{align*}
One verifies that this base changes to the Frobenius endomorphism of $\widehat{\Prism}_{A/\mathfrak{S}^{(-1)}}$ under \eqref{eq:BKdescent}, thus descending the pair $(\widehat{\Prism}_{A/\mathfrak{S}^{(-1)}},\varphi)$ along $\mathfrak{S} \to \mathfrak{S}^{(-1)}$. All assertions of Theorem~\ref{thm:BKlocal} now follow from Corollary~\ref{cor:BKtwisted}.
\end{proof}

We end with two remarks on the Nygaard filtration. First, we explain why the Nygaard filtration does not descend along $\varphi$.

\begin{remark}
The Nygaard filtration on $\varphi^* \widehat{\Prism}_{A/\mathfrak{S}} \simeq \widehat{\Prism}_{A/\mathfrak{S}^{(-1)}}$ from Remark~\ref{rmk:NygaardBK} does not obviously descend along $\varphi$ to a filtration on $\widehat{\Prism}_{A/\mathfrak{S}}$. In fact, there cannot be a functorial descent. For instance, if the projection $\widehat{\Prism}_{A/\mathfrak{S}^{(-1)}} \to \gr^0 \widehat{\Prism}_{A/\mathfrak{S}^{(-1)}} \simeq A$ descended functorially along $\varphi:\mathfrak{S} \to \mathfrak{S}^{(-1)}$, then one could globalize to conclude that each smooth formal scheme $\mathfrak{X}/\calO_K$ descends canonically to the subring $W(k)[\varpi^p] \subset \mathcal{O}_K$ (which is the image of $\mathfrak{S} \xrightarrow{\theta} \calO_K$), which is clearly nonsensical: any elliptic curve with good reduction whose $j$ invariant lies in $\calO_K - W(k)[\varpi^p]$ gives a counterexample.
\end{remark}

Secondly, we prove that $\varphi^* \widehat{\Prism}_{A/\mathfrak{S}}$ identifies with $L\eta_E \widehat{\Prism}_{A/\mathfrak{S}}$ via the Frobenius, in analogy with $A\Omega$.

\begin{remark}
In the situation of Theorem~\ref{thm:BKlocal}, consider the Frobenius map 
\[ \varphi_A: \varphi^* \widehat{\Prism}_{A/\mathfrak{S}} \simeq \widehat{\Prism}_{A/\mathfrak{S}^{(-1)}} \to \widehat{\Prism}_{A/\mathfrak{S}}\]
of $E_\infty$-algebras in $D(\mathfrak{S})$. The source of this map comes equipped with the Nygaard filtration $\calN^{\geq \f} \widehat{\Prism}_{A/\mathfrak{S}^{(-1)}}$ from Remark~\ref{rmk:NygaardBK}. The target of the map comes equipped with the $E$-adic filtration $(E)^\f \otimes \widehat{\Prism}_{A/\mathfrak{S}}$. We claim that the map $\varphi_A$ above lifts to a map of $E_\infty$-algebras in $DF(\mathfrak{S})$ of the form
\[  \calN^{\geq \f} \widehat{\Prism}_{A/\mathfrak{S}^{(-1)}} \xrightarrow{\tilde{\varphi}_A} (E)^\f \otimes \widehat{\Prism}_{A/\mathfrak{S}},\]
and that this map is a connective cover for the Beilinson $t$-structure. In particular, by Proposition~\ref{LetaDF}, this implies that $\varphi_A$ factors as 
\[ \varphi^* \widehat{\Prism}_{A/\mathfrak{S}} \simeq L\eta_E \widehat{\Prism}_{A/\mathfrak{S}} \xrightarrow{\mathrm{can}} \widehat{\Prism}_{A/\mathfrak{S}}.\]
To prove the above assertion, one checks that $\calN^{\geq \f} \widehat{\Prism}_{A/\mathfrak{S}^{(-1)}} \in DF^{\leq 0}(\mathfrak{S})$ just as in Corollary~\ref{NygaardSmooth}. Once we have shown that $\varphi_A$ lifts to a filtered map $\tilde{\varphi}_A$ as above, the rest will follow by base change to $\calO_C$ (i.e., via Theorem~\ref{thm:BKlocal} (1)) as $\mathfrak{S} \to A_{\inf}$ is topologically free. Thus, we are reduced to checking that the restriction of $\varphi_A$ to $\calN^{\geq i} \widehat{\Prism}_{A/\mathfrak{S}^{(-1)}}$ is functorially divisible by $E^i$ compatibly in $i$. Unwinding definitions and unfolding, this reduces to checking that for $S\in \Qsp_{\calO_K}$, the composite 
\[ \pi_{2i} \TC^{-}(S/\mathbb{S}[z]; \mathbb{Z}_p) \xrightarrow{v^i} \pi_0 \TC^{-}(S/\mathbb{S}[z];\mathbb{Z}_p) \xrightarrow{\mathrm{can}} \pi_0 \big(\TC^{-}(S/\mathbb{S}[z];\mathbb{Z}_p)[\tfrac{1}{u}]\big) \]
is functorially divisible by $E^i$ compatibly in $i$. The divisibility follows as
\[
v^i = \tfrac{E^i}{u^i} \in \pi_* \TC^{-}(\calO_K/\mathbb{S}[z]; \mathbb{Z}_p)[\tfrac{1}{u}]\ .
\]
To get functoriality as well as compatibility in $i$, it is enough to check that $E$ is a nonzerodivisor in $\pi_* \big(\TC^{-}(S/\mathbb{S}[z];\mathbb{Z}_p)[\frac{1}{u}]\big) \simeq \pi_* \big(\TC^{-}(S/\mathbb{S}[z];\mathbb{Z}_p)\big)[\frac{1}{u}]$. As inverting $u$ is flat, we are reduced to showing that $E$ is a nonzerodivisor in the graded ring $\pi_* \big(\TC^{-}(S/\mathbb{S}[z];\mathbb{Z}_p)\big) \subset \pi_* \big(\TP(S/\mathbb{S}[z];\mathbb{Z}_p)\big)$. As the larger graded ring is $2$-periodic and concentrated in even degrees, it is enough to check that the cone of multiplication by $E$ on $\pi_0 \TP(S/\mathbb{S}[z]; \mathbb{Z}_p)$ is discrete. But this cone identifies with $\widehat{L\Omega}_{S/\calO_K}$ (see proof of Corollary~\ref{cor:BKtwisted} (2)), which is discrete as $S \in \Qsp_{\calO_K}$. 
\end{remark}

\newpage
\bibliographystyle{amsalpha}
\bibliography{bms2}
\end{document}